\documentclass[a4paper,oneside,english]{amsart}
\usepackage[T1]{fontenc}
\usepackage[utf8]{inputenc}
\usepackage{babel}
\usepackage{mathtools}
\usepackage{bm}
\usepackage{amstext}
\usepackage{amsthm}
\usepackage{amssymb}
\usepackage{esint}
\PassOptionsToPackage{normalem}{ulem}
\usepackage{ulem}
\usepackage[unicode=true,pdfusetitle,
 bookmarks=true,bookmarksnumbered=false,bookmarksopen=false,
 breaklinks=false,pdfborder={0 0 1},backref=false,colorlinks=false]
 {hyperref}

\makeatletter

\pdfpageheight\paperheight
\pdfpagewidth\paperwidth

\numberwithin{equation}{section}
\theoremstyle{plain}
\newtheorem{thm}{\protect\theoremname}[section]
\theoremstyle{remark}
\newtheorem{rem}[thm]{\protect\remarkname}
\theoremstyle{plain}
\newtheorem{cor}[thm]{\protect\corollaryname}
\theoremstyle{plain}
\newtheorem{lem}[thm]{\protect\lemmaname}
\theoremstyle{plain}
\newtheorem{prop}[thm]{\protect\propositionname}

\makeatother

\providecommand{\corollaryname}{Corollary}
\providecommand{\lemmaname}{Lemma}
\providecommand{\propositionname}{Proposition}
\providecommand{\remarkname}{Remark}
\providecommand{\theoremname}{Theorem}

\begin{document}
\global\long\def\bbC{\mathbb{C}}%
\global\long\def\bbN{\mathbb{N}}%
\global\long\def\bbQ{\mathbb{Q}}%
\global\long\def\bbR{\mathbb{R}}%
\global\long\def\bbZ{\mathbb{Z}}%

\global\long\def\bfD{{\bf D}}%

\global\long\def\bfv{{\bf v}}%

\global\long\def\calA{\mathcal{A}}%
\global\long\def\calC{\mathcal{C}}%
\global\long\def\calE{\mathcal{E}}%
\global\long\def\calF{\mathcal{F}}%
\global\long\def\calH{\mathcal{H}}%
\global\long\def\calI{\mathcal{I}}%
\global\long\def\calK{\mathcal{K}}%
\global\long\def\calL{\mathcal{L}}%
\global\long\def\calM{\mathcal{M}}%
\global\long\def\calN{\mathcal{N}}%
\global\long\def\calU{\mathcal{U}}%
\global\long\def\calO{\mathcal{O}}%
\global\long\def\calP{\mathcal{P}}%
\global\long\def\calR{\mathcal{R}}%
\global\long\def\calS{\mathcal{S}}%
\global\long\def\calT{\mathcal{T}}%
\global\long\def\calU{\mathcal{U}}%
\global\long\def\calZ{\mathcal{Z}}%

\global\long\def\frkc{\mathfrak{c}}%
\global\long\def\frkp{\mathfrak{p}}%

\global\long\def\eps{\epsilon}%
\global\long\def\gmm{\gamma}%
\global\long\def\Gmm{\Gamma}%
\global\long\def\tht{\theta}%
\global\long\def\lmb{\lambda}%
\global\long\def\Lmb{\Lambda}%

\global\long\def\rd{\partial}%
\global\long\def\aleq{\lesssim}%
\global\long\def\ageq{\gtrsim}%

\global\long\def\peq{\mathrel{\phantom{=}}}%
\global\long\def\To{\longrightarrow}%
\global\long\def\weakto{\rightharpoonup}%
\global\long\def\embed{\hookrightarrow}%
\global\long\def\Re{\mathrm{Re}}%
\global\long\def\Im{\mathrm{Im}}%
\global\long\def\chf{\mathbf{1}}%
\global\long\def\td#1{\widetilde{#1}}%
\global\long\def\br#1{\overline{#1}}%
\global\long\def\ul#1{\underline{#1}}%
\global\long\def\wh#1{\widehat{#1}}%
\global\long\def\tint#1#2{{\textstyle \int_{#1}^{#2}}}%
\global\long\def\tsum#1#2{{\textstyle \sum_{#1}^{#2}}}%

\global\long\def\CR{\mathbf{D}_{+}}%
\global\long\def\Mod{\mathbf{Mod}}%
\global\long\def\bbet{\bm{\beta}}%
\global\long\def\out{\phantom{}^{\mathrm{(out)}}}%
\global\long\def\inn{\phantom{}^{\mathrm{(in)}}}%
\global\long\def\loc{\mathrm{loc}}%
\global\long\def\pert{\mathrm{pert}}%
\global\long\def\sg{\mathrm{sg}}%
\global\long\def\rg{\mathrm{rg}}%

\title[Blow-up rigidity for CSS]{Rigidity of smooth finite-time blow-up for equivariant self-dual
Chern--Simons--Schrödinger equation}
\author{Kihyun Kim}
\email{khyun@ihes.fr}
\address{IHES, 35 route de Chartres, Bures-sur-Yvette 91440, France}
\begin{abstract}
We consider the long time dynamics for the self-dual Chern--Simons--Schrödinger
equation (CSS) within equivariant symmetry. (CSS) is a self-dual $L^{2}$-critical
equation having pseudoconformal invariance and solitons. In this paper,
we show that any $m$-equivariant, $m\geq1$, $H^{3}$ finite-time
blow-up solution to (CSS) is a pseudoconformal blow-up solution. More
precisely, such a solution decomposes into the sum of one modulated
soliton that contracts at the pseudoconformal rate $\lmb(t)\sim T-t$,
and a radiation. Applying the pseudoconformal transform in reverse,
we also obtain a refined soliton resolution theorem for $m$-equivariant,
$m\geq1$, sufficiently regular and localized solutions: any such
solutions blow up in the pseudoconformal regime, scatter (to $0$),
or scatter to a modulated soliton with some fixed scale and phase.
To our knowledge, this is the first result on the full classification
of the dynamics of arbitrary smooth and spatially decaying solutions,
in the class of nonlinear Schrödinger equations which are not known
to be completely integrable.

Our analysis not only builds upon the previous works, especially the
soliton resolution theorem by the author, Kwon, and Oh, but also refines
all steps of the arguments typically employed in the forward construction
of blow-up dynamics. The key feature of the proof is that we can identify
the singular and regular parts of any $H^{3}$ finite-time blow-up
solutions, such that the evolution of the singular part is governed
by a universal modulation dynamics while the regular part is kept
$H^{3}$-bounded even up to the blow-up time. As a byproduct, we also
observe that the asymptotic profile has a universal singular structure.
\end{abstract}

\maketitle
\tableofcontents{}

\section{Introduction}

\subsection{Self-dual Chern--Simons--Schrödinger equation}

We consider the self-dual Chern--Simons--Schrödinger equation within
equivariant symmetry. This equation reads 
\begin{equation}
i(\rd_{t}+iA_{t}[u])u+\rd_{rr}u+\frac{1}{r}\rd_{r}u-\Big(\frac{m+A_{\theta}[u]}{r}\Big)^{2}u+|u|^{2}u=0,\tag{CSS}\label{eq:CSS-m-equiv}
\end{equation}
where $m\in\bbZ$ (called \emph{equivariance index}), $u:[0,T)\times(0,\infty)\to\bbC$,
and the connection components $A_{t}[u]$ and $A_{\tht}[u]$ are given
by the formulas 
\begin{equation}
A_{t}[u]=-\int_{r}^{\infty}(m+A_{\tht}[u])|u|^{2}\frac{dr'}{r'},\qquad A_{\tht}[u]=-\frac{1}{2}\int_{0}^{r}|u|^{2}r'dr'.\label{eq:def-A}
\end{equation}

The (full non-equivariant) covariant Chern--Simons--Schrödinger
equation, 
\begin{equation}
\left\{ \begin{aligned}\bfD_{t}\phi & =i(\bfD_{1}\bfD_{1}+\bfD_{2}\bfD_{2})\phi+ig|\phi|^{2}\phi,\\
F_{t1} & =-\Im(\br{\phi}\bfD_{2}\phi),\\
F_{t2} & =\Im(\br{\phi}\bfD_{1}\phi),\\
F_{12} & =-\tfrac{1}{2}|\phi|^{2},
\end{aligned}
\right.\label{eq:CSS-cov}
\end{equation}
with the coupling constant $g>0$, the scalar field $\phi:[0,T)\times\bbR^{2}\to\bbC$,
the covariant derivatives $\bfD_{\alpha}=\rd_{\alpha}+iA_{\alpha}$,
and the curvature components $F_{\alpha\beta}=\rd_{\alpha}A_{\beta}-\rd_{\beta}A_{\alpha}$,
$\alpha,\beta\in\{t,1,2\}$, was introduced by the physicists Jackiw
and Pi \cite{JackiwPi1990PRL} as a gauged 2D nonlinear Schrödinger
equation that admits \emph{explicit static self-dual solitons} within
the particular choice of the coupling constant $g=1$. This shows
a stark contrast with the 2D cubic nonlinear Schrödinger equation
\begin{equation}
i\partial_{t}\psi+\Delta\psi+|\psi|^{2}\psi=0,\qquad\psi:[0,T)\times\bbR^{2}\to\bbC,\tag{NLS}\label{eq:NLS}
\end{equation}
where the equation does not have static (time-independent) finite
energy solutions and no explicit soliton formulas are known. The model
\eqref{eq:CSS-m-equiv} under consideration is derived from the covariant
system \eqref{eq:CSS-cov} after choosing $g=1$, fixing the Coulomb
gauge condition, and imposing \emph{equivariant symmetry} on the scalar
field $\phi$: 
\begin{equation}
\phi(t,x)=u(t,r)e^{im\theta},\qquad m\in\bbZ,\label{eq:def-equiv-symm}
\end{equation}
where $(r,\theta)$ denotes the polar coordinates on $\bbR^{2}$.
We refer to \cite{JackiwPi1990PRL,JackiwPi1990PRD,JackiwPi1991PRD,JackiwPi1992Progr.Theoret.,Dunne1995Springer}
for more physical contents of the Chern--Simons--Schrödinger equation
\eqref{eq:CSS-cov} and to the introduction of \cite{LiuSmith2016,KimKwon2023MAMS}
for the derivation of \eqref{eq:CSS-m-equiv}.

Recent studies on the long-term dynamics of \eqref{eq:CSS-m-equiv}
make this model more interesting. To be further explained in this
introduction, two notable features are \emph{rotational instability}
and \emph{soliton resolution}. Rotational instability is an instability
mechanism of some finite-time blow-up solutions and soliton resolution
asserts a strong rigidity of asymptotic behavior of solutions. The
present paper continues our study on soliton resolution and our main
result gives a complete description of the asymptotic behavior of
$m$-equivariant, $m\geq1$, smooth (and spatially decaying) solutions
to \eqref{eq:CSS-m-equiv}.

We continue our discussion on the model. \eqref{eq:CSS-m-equiv} enjoys
various symmetries and conservation laws analogous to \eqref{eq:NLS}.
Among the most basic symmetries are the time translation $(u(t,r)\to u(t+t_{0},r)$,
$t_{0}\in\bbR$) and the phase rotation symmetries ($u(t,r)\to e^{i\gmm_{0}}u(t,r)$,
$\gmm_{0}\in\bbR/2\pi\bbZ$). Associated to these are the conservation
of \emph{energy} and \emph{mass}\footnote{The physical interpretation of the quantity $M[u]$ is the \emph{total
charge}, but in this paper we shall call it mass following the widespread
convention for NLS.}:
\begin{align}
E[u] & \coloneqq\int\frac{1}{2}|\partial_{r}u|^{2}+\frac{1}{2}\Big(\frac{m+A_{\theta}[u]}{r}\Big)^{2}|u|^{2}-\frac{1}{4}|u|^{2},\label{eq:energy-Coulomb-form}\\
M[u] & \coloneqq\int|u|^{2},\label{eq:charge}
\end{align}
where we denoted $\int f(r)=2\pi\int f(r)rdr$. \eqref{eq:CSS-m-equiv}
moreover admits a \emph{Hamiltonian} structure
\begin{equation}
\rd_{t}u=-i\nabla E[u],\label{eq:CSS-equiv-ham}
\end{equation}
where $\nabla$ (acting on a functional) is the Fréchet derivative
with respect to the real inner product $(u,v)_{r}=\int\Re(\br uv)$.
Of particular importance in this work are the \emph{$L^{2}$-scaling
symmetry} and the \emph{pseudoconformal symmetry};\emph{ }if $u(t,r)$
is a solution to \eqref{eq:CSS-m-equiv}, then the functions $u_{\lmb}$
and $\calC u$ also solve \eqref{eq:CSS-m-equiv}: 
\begin{align}
u_{\lmb}(t,r) & \coloneqq\frac{1}{\lmb}u\Big(\frac{t}{\lmb^{2}},\frac{r}{\lmb}\Big),\qquad\qquad\forall\lmb>0,\label{eq:scaling}\\{}
[\calC u](t,r) & \coloneqq\frac{1}{|t|}u\Big(-\frac{1}{t},\frac{r}{|t|}\Big)e^{\frac{ir^{2}}{4t}},\qquad\forall t\neq0.\label{eq:def-pseudoconf}
\end{align}
Associated to \eqref{eq:scaling} and \eqref{eq:def-pseudoconf} are
the \emph{virial identities}:\emph{ }
\begin{align}
\rd_{t}\int r^{2}|u|^{2} & =4\int\Im(\br u\cdot r\rd_{r}u),\label{eq:virial-1}\\
\rd_{t}\int\Im(\br u\cdot r\rd_{r}u) & =4E[u].\label{eq:virial-2}
\end{align}
We note that the covariant system \eqref{eq:CSS-cov} also shares
very similar symmetries and conservation laws with \eqref{eq:NLS}.

As alluded to above, a notable feature of \eqref{eq:CSS-m-equiv}
distinguished from \eqref{eq:NLS} is the \emph{self-duality}. Let
us introduce the (covariant) \emph{Cauchy--Riemann operator} 
\begin{equation}
\bfD_{u}f\coloneqq\rd_{r}f-\frac{m+A_{\theta}[u]}{r}f,\label{eq:CR-radial}
\end{equation}
which indeed corresponds to the radial part of the operator $\bfD_{+}\coloneqq{\bf D}_{1}+i\bfD_{2}$.
Now the energy functional \eqref{eq:energy-Coulomb-form} can be written
into the self-dual form
\begin{equation}
E[u]=\int\frac{1}{2}|\bfD_{u}u|^{2}.\label{eq:energy-self-dual-form}
\end{equation}
We call the nonlinear operator $u\mapsto\bfD_{u}u$ the \emph{Bogomol'nyi
operator}. Due to \eqref{eq:CSS-equiv-ham} and \eqref{eq:energy-self-dual-form},
any static solutions to \eqref{eq:CSS-m-equiv} are given by solutions
(Jackiw--Pi vortices) to the \emph{Bogomol'nyi equation}
\begin{equation}
\bfD_{Q}Q=0.\label{eq:Bogomol'nyi-eq}
\end{equation}
Since \eqref{eq:Bogomol'nyi-eq} is a first-order (nonlocal) equation,
we say that the equation is \emph{self-dual}. In fact, there is a
covariant version of the Bogomol'nyi equation and Jackiw--Pi \cite{JackiwPi1990PRD}
found a transform of the covariant Bogomol'nyi equation into the Liouville
equation, which is completely integrable. Solving \eqref{eq:Bogomol'nyi-eq},
we obtain a unique (up to symmetries) \emph{explicit} $m$-equivariant
static solution for any $m\geq0$: 
\begin{equation}
Q(r)=\sqrt{8}(m+1)\frac{r^{m}}{1+r^{2m+2}},\qquad m\geq0.\label{eq:Q-formula}
\end{equation}
There is no finite energy solution to \eqref{eq:Bogomol'nyi-eq} when
$m<0$.\footnote{This fact can be proved either by ODE analysis or by showing the coercivity
of energy \cite{KimKwonOh2022arXiv1}.} Applying the pseudoconformal transform \eqref{eq:def-pseudoconf}
to $Q$, we also obtain an explicit finite-time blow-up solution ($m\geq0$)
\[
S(t,r)\coloneqq\frac{1}{|t|}Q\Big(\frac{r}{|t|}\Big)e^{-i\frac{r^{2}}{4|t|}},\qquad t<0,
\]
which is another fundamental object in the long-term dynamics for
\eqref{eq:CSS-m-equiv}. Due to the spatial decay $Q(r)\sim r^{-(m+2)}$,
$S(t)$ has finite energy (i.e., $S(t)$ belongs to $H^{1}$) if and
only if $m\geq1$.

Note that we have suppressed the $m$-dependences in various notation
such as $A_{t}[u]$, $E[u]$, $\bfD_{u}$, $Q$, and $S(t)$. We will
keep suppressing these $m$-dependences throughout this article for
the simplicity of notation.

\subsection{Long-term dynamics for \eqref{eq:CSS-m-equiv}}

Let us begin with some previous works on the Cauchy problem of the
covariant Chern--Simons--Schrödinger equation \eqref{eq:CSS-cov}.
The local well-posedness has been studied by many authors \cite{BergeDeBouardSaut1995Nonlinearity,Huh2013Abstr.Appl.Anal,LiuSmithTataru2014IMRN,Lim2018JDE}.
Lim \cite{Lim2018JDE} proved the $H^{1}$ well-posedness under Coulomb
gauge and Liu--Smith--Tataru \cite{LiuSmithTataru2014IMRN} proved
the small data $H^{0+}$ well-posedness under the heat gauge. These
results still miss the critical $L^{2}$-space. There are also results
on long-term dynamics such as the finite-time blow-up for negative
energy solutions (necessarily $g>1$) with a formal derivation of
the log-log rate \cite{BergeDeBouardSaut1995Nonlinearity,BergeDeBouardSaut1995PRL}
and small data scattering \cite{OhPusateri2015}.

We turn to the long-term dynamics for \eqref{eq:CSS-m-equiv}. First,
it is easy to see that \eqref{eq:CSS-m-equiv} is well-posed in the
critical space $L^{2}$ because there is no derivative nonlinearity
within equivariance and one can use a standard contraction argument
with the $L_{t,x}^{4}$-Strichartz estimates \cite[Section 2]{LiuSmith2016}.
In particular, any small $L^{2}$-solutions exist globally and scatter
both forwards and backwards in time.

For large solutions, the ground state $Q$ provides a natural threshold
for the non-scattering dynamics. Indeed, Liu--Smith \cite{LiuSmith2016}
proved the \emph{subthreshold theorem}: for $m\geq0$, any $m$-equivariant
$L^{2}$-solutions $u$ with $M[u]<M[Q]$ exist globally and scatter.
Following their proof, the same should be true for any negative-equivariant
(i.e., $m<0$) solutions without any restriction on mass; see also
\cite{KimKwonOh2022arXiv1}. Recently, Li--Liu \cite{LiLiu2022AIHP}
considered the threshold dynamics ($M[u]=M[Q]$, necessarily $m\geq0$)
of $H_{m}^{1}$-solutions and proved that $Q$ and $S(t)$ are the
only non-scattering solutions up to symmetries.\footnote{When $m=0$, $S(t)$ does not have finite energy due to the slow spatial
decay of $Q$.}

Recently, author, Kwon, and Oh studied finite-time blow-up dynamics
of \eqref{eq:CSS-m-equiv} with an eye towards the dynamics above
threshold. We remark that the finite-time blow-up dynamics is closely
related to the dynamics of non-scattering global solutions through
the pseudoconformal transform.

In \cite{KimKwon2023MAMS,KimKwon2023AnnPDE,KimKwonOh2020arXiv,KimKwonOh2022arXiv2},
the authors studied the dynamics of finite-time blow-up solutions
to \eqref{eq:CSS-m-equiv} of the form 
\begin{equation}
u(t,r)-\frac{e^{i\gmm(t)}}{\lmb(t)}Q\Big(\frac{r}{\lmb(t)}\Big)\to z^{\ast}(r)\quad\text{ in }L^{2}\quad\text{as }t\to T.\label{eq:intro-decomp}
\end{equation}
The authors in \cite{KimKwon2023MAMS,KimKwon2023AnnPDE} considered
\eqref{eq:CSS-m-equiv} when $m\geq1$ and studied \emph{pseudoconformal
blow-up solutions}, namely, the solutions of the form \eqref{eq:intro-decomp}
with $\lmb(t)/(T-t)\to\ell\in(0,\infty)$ and $\gmm(t)\to\gmm^{\ast}\in\bbR/2\pi\bbZ$.
They in particular constructed pseudoconformal blow-up solutions and
uncovered their instability mechanism (\emph{rotational instability});
see Remark~\ref{rem:rotational-instability} for more discussions.
On the other hand, when $m=0$, $S(t)\notin H^{1}$ so even the existence
of finite energy blow-up solutions is not trivial. In this case, the
authors in \cite{KimKwonOh2020arXiv,KimKwonOh2022arXiv2} constructed
finite energy finite-time blow-up solutions with the rates $\lmb(t)\sim(T-t)/|\log(T-t)|^{2}$
or $\lmb(t)\sim(T-t)^{p}/|\log(T-t)|$ for any $p>1$, where the former
regime can arise from a \emph{smooth} compactly supported initial
data.

With the above blow-up solutions at hand, it is natural to ask the
following.
\begin{itemize}
\item[Q1.] Is there any blow-up scenario other than one-bubble blow-up?
\item[Q2.] Can we classify all possible dynamics of scale and phase?
\end{itemize}
The first question was answered for finite energy solutions by the
author, Kwon, and Oh \cite{KimKwonOh2022arXiv1} (see also Theorem~\ref{thm:asymptotic-description}
below). The result roughly says that any $H_{m}^{1}$ finite-time
blow-up solutions admit the decomposition \eqref{eq:intro-decomp}
with certain upper bounds on $\lmb(t)$. By pseudoconformal symmetry,
one further obtains \emph{soliton resolution} for well-localized regular
data: every such solution to \eqref{eq:CSS-m-equiv} decomposes into
at most one soliton and a radiation. A remarkable fact is that \eqref{eq:CSS-m-equiv}
does not admit multi-soliton configurations and this turns out to
be a consequence of two distinguished features of \eqref{eq:CSS-m-equiv}:
\emph{self-duality} and \emph{non-locality}. We refer to Remarks~\ref{rem:ExteriorDefocusing}
and \ref{rem:soliton resolution} for more explanations.

The present paper concerns the second question, namely, the classification
of the dynamics of scale and phase.

\subsection{Main results}

The main result of this paper is that any $m$-equivariant, $m\geq1$,
$H^{3}$ finite-time blow-up solutions to \eqref{eq:CSS-m-equiv}
must be pseudoconformal blow-up solutions and hence this gives a complete
classification of the dynamics of scale and phase. Via the pseudoconformal
transform, we also give the full description of the dynamics of any
$m$-equivariant, $m\geq1$, $H^{3,3}$-solutions.

Let us briefly introduce some important notation. We denote the modulated
soliton by 
\[
Q_{\lmb,\gmm}(r)\coloneqq\frac{e^{i\gmm}}{\lmb}Q\Big(\frac{r}{\lmb}\Big),\qquad\lmb\in(0,\infty),\ \gmm\in\bbR/2\pi\bbZ.
\]
For $k\in\bbN$, we denote by $H_{m}^{k}$ (resp., $H_{m}^{k,k}$),
the space of $m$-equivariant $H^{k}(\bbR^{2})$ (resp., $H^{k,k}(\bbR^{2})$)
functions. Here, $f\in H^{k,k}(\bbR^{2})$ means that $\langle x\rangle^{\ell}f\in H^{k-\ell}(\bbR^{2})$
for all $0\leq\ell\leq k$. See Section~\ref{subsec:Notation} for
more pieces of notation.

We begin by recalling the soliton resolution theorem of \cite{KimKwonOh2022arXiv1}
for $H_{m}^{1}$ finite-time blow-up solutions.
\begin{thm}[Soliton resolution for $H_{m}^{1}$ finite-time blow-up solutions
\cite{KimKwonOh2022arXiv1}]
\label{thm:asymptotic-description}\ 
\begin{itemize}
\item When $m<0$, there are no $H_{m}^{1}$ finite-time blow-up solutions.
\item When $m\geq0$, if $u$ is a $H_{m}^{1}$-solution to \eqref{eq:CSS-m-equiv}
that blows up forwards in time at $T<+\infty$, then $u(t)$ admits
the decomposition 
\begin{equation}
u(t,\cdot)-Q_{\lmb(t),\gmm(t)}\to z^{\ast}\text{ in }L^{2}\text{ as }t\to T,\label{eq:thm1-decomp}
\end{equation}
for some $\lmb(t)\in(0,\infty)$, $\gmm(t)\in\bbR/2\pi\bbZ$, and
$z^{\ast}\in L^{2}$ with the following properties:
\begin{itemize}
\item (Further regularity of the asymptotic profile) We have 
\begin{equation}
\partial_{r}z^{\ast},\ \tfrac{1}{r}z^{\ast}\in L^{2}.\label{eq:thm1-asym-reg}
\end{equation}
\item (Bound on the blow-up speed) As $t\to T$, we have 
\begin{equation}
\lmb(t)\aleq_{M[u]}\sqrt{E[u]}(T-t).\label{eq:thm1-lmb-upper-bound}
\end{equation}
When $m=0$, we further have an improved bound 
\begin{equation}
\lmb(t)\aleq_{M[u]}\frac{\sqrt{E[u]}(T-t)}{|\log(T-t)|^{\frac{1}{2}}}.\label{eq:thm1-lmb-upper-bound-radial}
\end{equation}
\end{itemize}
\end{itemize}
\end{thm}

\begin{rem}[Exterior defocusing nature]
\label{rem:ExteriorDefocusing}A remarkable fact in the decomposition
\eqref{eq:thm1-decomp} is that only one soliton appears. This is
a consequence of the \emph{self-duality} and \emph{non-locality},
which render \eqref{eq:CSS-m-equiv} \emph{defocusing} (in the form
of strict positivity of energy) after extracting out the soliton profile.
Hence no multi-soliton configurations appear. See the nonlinear coercivity
of energy in \cite{KimKwonOh2022arXiv1} or Lemma~\ref{lem:NonlinCoerEnergy}
of this paper.
\end{rem}

\begin{rem}[Soliton resolution]
\label{rem:soliton resolution}Via pseudoconformal transform, we
also obtain soliton resolution for $H_{m}^{1,1}$-solutions to \eqref{eq:CSS-m-equiv},
which says that at most one soliton can appear in the resolution.
Soliton resolution, in a more general context, says that (under a
genericity assumption) every solution decomposes into the sum of decoupled
solitons and a radiation. This has been known for a wide range of
completely integrable equations. However, it is mostly open for non-integrable
equations with exception of some nonlinear wave equations within symmetry,
such as energy-critical radial nonlinear wave equation and equivariant
wave maps \cite{DuyckaertsKenigMerle2013CambJMath,DuyckaertsKenigMerle2023Acta,DuyckaertsKenigMartelMerle2022CMP,CollotDuyckaertsKenigMerle2022arXiv,JendrejLawrie2021arXiv,JendrejLawrie2023AnnPDE}.
In these equations, multi-solitons are known to exist \cite{Jendrej2019AJM}.
Thus the previously mentioned exterior defocusing nature is special
to \eqref{eq:CSS-m-equiv}.
\end{rem}

From now on, we consider the $m\geq1$ case. Pseudoconformal blow-up
solutions constructed in \cite{KimKwon2023MAMS,KimKwon2023AnnPDE}
saturate the upper bound \eqref{eq:thm1-lmb-upper-bound} for $\lmb(t)$,
but the question whether an exotic blow-up speed other than the pseudoconformal
one exists or not was left open. Our main result claims that such
an exotic regime does not exist and the blow-up must be pseudoconformal
if one assumes $H^{3}$-data.
\begin{thm}[Rigidity of $H_{m}^{3}$ finite-time blow-up when $m\geq1$]
\label{thm:blow-up-rigidity}Let $m\geq1$. Let $u$ be a $H_{m}^{3}$-solution
to \eqref{eq:CSS-m-equiv} which blows up forwards in time at $T<+\infty$.
Then, the following holds.
\begin{itemize}
\item (Rigidity of blow-up rates) In the decomposition \eqref{eq:thm1-decomp},
we can choose 
\begin{equation}
\lmb(t)=\ell\cdot(T-t)\qquad\text{and}\qquad\gmm(t)=\gmm^{\ast}\label{eq:rigidity-blow-up}
\end{equation}
for some $\ell=\ell(u)\in(0,\infty)$ and $\gmm^{\ast}=\gmm^{\ast}(u)\in\bbR/2\pi\bbZ$.
\item (Structure of the asymptotic profile) We can decompose 
\begin{equation}
z^{\ast}=z_{\sg}^{\ast}+z_{\rg}^{\ast}\label{eq:z-ast-decom}
\end{equation}
for some $z_{\sg}^{\ast},z_{\rg}^{\ast}\in L^{2}$ such that the singular
part $z_{\sg}^{\ast}$ is given by 
\begin{equation}
z_{\sg}^{\ast}(r)=\begin{cases}
-e^{i\gmm^{\ast}}\ell^{2}\frac{\sqrt{2}}{8}\chi(r)\cdot r & \text{if }m=1,\\
e^{i\gmm^{\ast}}\ell^{3}\frac{\sqrt{2}}{64}i\chi(r)\cdot r^{2} & \text{if }m=2,\\
0 & \text{if }m\geq3,
\end{cases}\label{eq:z-ast-sing-pt}
\end{equation}
where $\chi(r)$ is a smooth cutoff function near $r=0$ (see Section~\ref{subsec:Notation}),
and the regular part $z_{\rg}^{\ast}$ satisfies 
\begin{equation}
|z_{\rg}^{\ast}|_{-3}\coloneqq\sup_{0\leq k\leq3}|r^{-(3-k)}\rd_{r}^{k}z_{\rg}^{\ast}|\in L^{2}.\label{eq:z-ast-reg-pt}
\end{equation}
In particular, when $m\in\{1,2\}$, $z^{\ast}(r)$ satisfies the near-origin
asymptotics 
\begin{equation}
z^{\ast}(r)=qr^{m}+o_{r\to0}(r^{2})\label{eq:z-ast-near-orig-asym}
\end{equation}
for some nonzero constant $q=q(u)\neq0$.
\end{itemize}
\end{thm}

The above theorem can also be viewed as a refined soliton resolution
theorem for $H_{m}^{3}$ finite-time blow-up solutions. As a standard
application of pseudoconformal symmetry, we obtain a refined soliton
resolution theorem for arbitrary $H_{m}^{3,3}$-solutions.
\begin{cor}[Soliton resolution for equivariant $H_{m}^{3,3}$-data when $m\geq1$]
\label{cor:H33-soliton-resolution}Let $m\geq1$. Let $u$ be a $H_{m}^{3,3}$-solution
to \eqref{eq:CSS-m-equiv}.
\begin{itemize}
\item (Finite-time blow-up) If $u$ blows up forwards in time at $T<+\infty$,
then $u$ satisfies the properties as in Theorem~\ref{thm:blow-up-rigidity}.
\item (Global solutions) If $u$ exists globally forwards in time, then
either $u$ scatters forwards in time in $L^{2}$, or $u$ scatters
to $Q_{\lmb^{\ast},\gmm^{\ast}}$ forwards in time in $L^{2}$ for
some $\lmb^{\ast}\in(0,\infty)$ and $\gmm^{\ast}\in\bbR/2\pi\bbZ$.
\end{itemize}
\end{cor}

In view of time reversal symmetry, the analogous statements for Theorem~\ref{thm:blow-up-rigidity}
and Corollary~\ref{cor:H33-soliton-resolution} hold for backward-in-time
evolutions.
\begin{rem}[Rigidity of long-term dynamics]
Theorem~\ref{thm:blow-up-rigidity} completely classifies the possible
dynamics of the scale and phase for $H_{m}^{3}$ finite-time blow-up
solutions. In view of pseudoconformal symmetry, even the long-term
dynamics of arbitrary $H_{m}^{3,3}$ solutions is rigid as stated
in Corollary~\ref{cor:H33-soliton-resolution}. Moreover, all the
scenarios in Corollary~\ref{cor:H33-soliton-resolution} indeed occur,
thanks to the blow-up construction in \cite{KimKwon2023AnnPDE}.

Our result might be compared to the first part of the series of works
\cite{MartelMerleRaphael2014Acta,MartelMerleRaphael2015JEMS,MartelMerleRaphael2015AnnSci}
by Martel, Merle, and Raphaël, who considered the near-soliton dynamics
for the critical gKdV equation
\[
\rd_{t}u+\rd_{x}(\rd_{xx}u+u^{5})=0,\qquad u:[0,T)\times\bbR\to\bbR.\tag{gKdV}
\]
The result can be roughly stated as follows: (i) if a $H^{1}$-solution
$u$ with strong spatial decay on the right starts from an initial
data $H^{1}$-close to the ground state $Q$, then $u$ blows up in
finite time with $\lmb(t)\sim T-t$, $u$ exits the soliton tube,\footnote{If the minimal mass blow-up solution to critical gKdV scatters backwards
in time, then $u$ scatters to zero \cite{MartelMerleRaphael2015JEMS}.} or $u$ scatters to a modulated $Q$ with some fixed scale and velocity.
(ii) Moreover, the former two regimes are stable.

We remark that there is no near-soliton assumption in Corollary~\ref{cor:H33-soliton-resolution}.
For \eqref{eq:CSS-m-equiv} with $m\geq1$, the equation however allows
complex-valued functions and has additional phase rotation invariance.
In contrast to the gKdV case, the (pseudoconformal) blow-up regime
is believed to be unstable \cite{KimKwon2023MAMS}. Thus the `scattering
to 0' regime is the only stable regime for \eqref{eq:CSS-m-equiv}.
See however Remark~\ref{rem:rotational-instability} for rotational
instability of pseudoconformal blow-up solutions, which we believe
is an interesting nonlinear dynamics that generically occur.
\end{rem}

\begin{rem}[Stabilization of scale and phase for global $H_{m}^{3,3}$-solutions]
Although the soliton resolution theorem for global non-scattering
$H_{m}^{1,1}$ solutions \cite{KimKwonOh2022arXiv1} only says that
$u(t)-Q_{\lmb(t),\gmm(t)}-e^{it\Delta}u^{\ast}\to0$ with $\lmb(t)\aleq1$,
an interesting fact from Corollary~\ref{cor:H33-soliton-resolution}
is that $\lmb(t)$ and $\gmm(t)$ \emph{stabilize} for $H_{m}^{3,3}$
solutions when $m\geq1$.

We believe that the same result might be false without spatial decay
assumptions. A similar statement in \cite{MartelMerleRaphael2014Acta}
for the critical gKdV assumes a fast spatial decay on the right of
solutions, and it is even shown in \cite{MartelMerleRaphael2015AnnSci}
that a slow spatial decay on the right might lead to both finite-time
and infinite-time blow-up. Moreover, for $2$-equivariant harmonic
map heat flow, which is closely related to $1$-equivariant \eqref{eq:CSS-m-equiv}
(see the discussion in Section~\ref{subsec:Linear-conjugation-identities}),
exotic behaviors such as infinite-time blow-up, oscillation, or relaxation
can appear for solutions having slow spatial decay \cite{GustafsonNakanishiTsai2010CMP}.
It would be interesting to construct exotic solutions having smooth
but slowly decaying data.
\end{rem}

\begin{rem}[Comparison with the mass-critical NLS]
Let us begin with some general comparisons between \eqref{eq:NLS}
and \eqref{eq:CSS-m-equiv}. As alluded to above, \eqref{eq:CSS-m-equiv}
and \eqref{eq:NLS} share almost the same conservation laws and symmetries.
Moreover, the (sub-)threshold theorem for \eqref{eq:NLS} holds the
same \cite{Killip-Tao-Visan2009JEMS,Dodson2015AdvMath,KillipLiVisanZhang2009,LiZhang2012_d_geq_2}
as in the \eqref{eq:CSS-m-equiv} case \cite{LiuSmith2016,LiLiu2022AIHP}.

However, the dynamics above threshold rather shows a drastic difference
between the \eqref{eq:NLS} and \eqref{eq:CSS-m-equiv} case. We note
that standing waves for \eqref{eq:NLS} have a nonzero speed of phase
rotation and have the form $e^{it}R(x)$ up to symmetry. The profile
$R$ solves $\Delta R-R+R^{3}=0$ or equivalently $\nabla E_{\mathrm{NLS}}[R]+\nabla M[R]=0$
with the NLS energy $E_{\mathrm{NLS}}[\psi]=\frac{1}{2}\int|\nabla\psi|^{2}-\frac{1}{4}\int|\psi|^{4}$.
Notice that $\nabla E[Q]=0$ for \eqref{eq:CSS-m-equiv}. The difference
between these soliton equations ultimately leads to differences in
underlying linearized operators and modulation dynamics.

As studied in a series of works \cite{MerleRaphael2005AnnMath,MerleRaphael2003GAFA,MerleRaphael2004InventMath,Raphael2005MathAnn,MerleRaphael2006JAMS,MerleRaphael2005CMP}
by Merle and Raphaël, \eqref{eq:NLS} has a \emph{stable }almost self-similar\emph{
}blow-up regime $\lmb(t)\sim\sqrt{2\pi(T-t)/\log|\log(T-t)|}$, called
the log-log blow-up. This regime essentially arises from negative
energy solutions and hence does not appear in the \eqref{eq:CSS-m-equiv}
case \cite{KimKwonOh2022arXiv1} (see \eqref{eq:thm1-lmb-upper-bound}).
Nevertheless, the log-log regime is formally expected for the focusing
non-self-dual CSS ($g>1$), which allow negative energy solutions.

The pseudoconformal blow-up solutions constructed in \cite{KimKwon2023MAMS,KimKwon2023AnnPDE}
are the \eqref{eq:CSS-m-equiv}-analogues of Bourgain--Wang solutions
\cite{BourgainWang1997} to \eqref{eq:NLS}. However, their instability
mechanisms, studied in \cite{MerleRaphaelSzeftel2013AJM} for \eqref{eq:NLS}
and in \cite{KimKwon2023MAMS} for \eqref{eq:CSS-m-equiv}, turn out
to be completely different. While Bourgain--Wang solutions to \eqref{eq:NLS}
arise as the border of log-log blow-up solutions and global scattering
solutions, generic perturbations of pseudoconformal blow-up solutions
to \eqref{eq:CSS-m-equiv} become global scattering solutions and
rather exhibit rotational instability; see Remark~\ref{rem:rotational-instability}
below.

To directly compare our main theorem with the \eqref{eq:NLS} case,
let us briefly recall a result in \cite{Raphael2005MathAnn}: if a
$H^{1}$-solution $\psi(t)$ to \eqref{eq:NLS} with slightly supercritical
mass (i.e., $M[\psi]-M[R]\ll1$) blows up in finite time, then either
$\psi(t)$ blows up in the log-log regime or $\lmb(t)\aleq\sqrt{E[u]}(T-t)$.
The latter regime might be compared with \eqref{eq:thm1-lmb-upper-bound}
and our Theorem~\ref{thm:blow-up-rigidity} improves \eqref{eq:thm1-lmb-upper-bound}
to the rigidity \eqref{eq:rigidity-blow-up} in the \eqref{eq:CSS-m-equiv}
case when $m\geq1$. The question whether this upper bound for \eqref{eq:NLS}
can be replaced by the rigidity \eqref{eq:rigidity-blow-up} for $\lmb(t)$
or not is a long-standing open problem.\footnote{We do not expect the convergence of the phase $\gmm(t)$ for \eqref{eq:NLS}
because Bourgain--Wang solutions do not have this property.}
\end{rem}

\begin{rem}[Structure of the asymptotic profile]
\label{rem:structure-of-asymp-prof}Theorem~\ref{thm:blow-up-rigidity}
provides a more precise understanding of the structure of $z^{\ast}$
for $H_{m}^{3}$ finite-time blow-up solutions compared to the $H^{1}$-regularity
of $z^{\ast}$ for $H_{m}^{1}$ blow-up solutions of Theorem~\ref{thm:asymptotic-description}.
We note that the choice of $\chi$ itself is not important; only the
$r\to0$ behavior of $z^{\ast}$ has meaning.

When $m\in\{1,2\}$, $z^{\ast}(r)$ always has a nontrivial singular
part $z_{\sg}^{\ast}$. Note that this is not a contradiction to the
existence of the minimal mass blow-up solution $S(t)$ (which has
zero asymptotic profile) because $S(t)\notin H_{m}^{3}$ when $m\in\{1,2\}$.
We may also write $z^{\ast}\notin H_{-m-2}^{3}$ when $m\in\{1,2\}$,\footnote{Due to the non-locality of \eqref{eq:CSS-m-equiv}, the dynamics outside
of the soliton region is similar to the dynamics of $-(m+2)$-equivariant
\eqref{eq:CSS-m-equiv}, where $-(m+2)=m+A_{\tht}[Q](+\infty)$. See
\cite{KimKwon2023MAMS,KimKwonOh2022arXiv1}.} thus the asymptotic profile does not necessarily have the same regularity
as the solution $u(t)$. Such an observation is already well-known
in various contexts and the regularity of the asymptotic profile is
closely related to the blow-up speed; we refer to the discussions
in \cite{MerleRaphael2005CMP,RaphaelRodnianski2012Publ.Math.,MerleRaphaelRodnianski2013InventMath,MartelMerleRaphael2014Acta,JendrejLawrieRodriguez2022ASENS}.

An interesting fact in Theorem~\ref{thm:blow-up-rigidity} is that
the singular part of $z^{\ast}$ has a universal structure \eqref{eq:z-ast-near-orig-asym}.
See also \cite{MerleRaphael2005CMP,MartelMerleRaphael2014Acta}. In
particular, an asymptotic profile of the form $z^{\ast}(r)\approx qr^{\nu}$
with $\nu\in(1,2)$ cannot be achieved from $H_{m}^{3}$ finite-time
blow-up solutions if $m\geq1$. It would be very interesting to fully
understand the singular structure of the asymptotic profile arising
from smooth initial data (say $H_{m}^{\infty}$).

Our decomposition of $z^{\ast}$ is tied to the decomposition $u=P^{\sharp}+\eps^{\sharp}$
of the full solution $u(t)$ to be used in our modulation analysis.
See Remark~\ref{rem:structure-of-full-sol}.
\end{rem}

\begin{rem}[Structure of the full solution]
\label{rem:structure-of-full-sol}The key observation of this paper
is that we can realize the decomposition $u=P^{\sharp}+\eps^{\sharp}$
(see \eqref{eq:str-dec} below) such that the evolution of the singular,
blow-up part $P^{\sharp}$ is governed by a universal modulation dynamics
while \emph{the regular part $\eps^{\sharp}$ remains $H^{3}$-bounded}
(see also \eqref{eq:str-bdd}) up to the blow-up time. In the limit
$t\to T$, the regular part $\eps^{\sharp}$ converges to $z_{\rg}^{\ast}$
and the blow-up part $P^{\sharp}$ detects $z_{\sg}^{\ast}$, leading
to the structure of $z^{\ast}$.

The $H^{3}$-boundedness of $\eps^{\sharp}$ is not limited to finite-time
blow-up solutions. For any $H_{m}^{3}$-solutions $u$ lying in a
neighborhood of modulated solitons (i.e., $Q_{\lmb,\gmm}$), we show
in Corollary~\ref{cor:integ-monot} that the $H^{3}$-norm of $\eps^{\sharp}$
cannot blow up in finite time. In this sense, we essentially prove
that the regular parts $\eps^{\sharp}$ of nearby solutions are \emph{uniformly}
$H^{3}$-bounded on any fixed finite time interval.

To prove the $H^{3}$-boundedness, we not only refine the modified
profile construction but also introduce a new energy-Morawetz functional
to control $\eps^{\sharp}$ forwards in time. Several novelties regarding
the proof are explained in the strategy below (Section~\ref{subsec:Strat-Proof}).
\end{rem}

\begin{rem}[Other equivariance]
When $m<0$, there is no $H^{1}$ finite-time blow-up as in Theorem~\ref{thm:asymptotic-description},
so the $m=0$ case is the interesting one to consider.

When $m=0$, there are no finite energy pseudoconformal blow-up solutions
\cite{KimKwonOh2022arXiv1} (or see \eqref{eq:thm1-lmb-upper-bound-radial}).
Recall that $S(t)$ has infinite energy when $m=0$ due to the slow
spatial decay of $Q$. Instead, the authors in \cite{KimKwonOh2020arXiv}
construct smooth (finite energy) finite-time blow-up solutions with
the blow-up rate $\lmb(t)\sim(T-t)/|\log(T-t)|^{2}$.

We do not even expect the uniqueness of blow-up rates for solutions
with smooth data. In view of the blow-up constructions for example
in \cite{RaphaelSchweyer2014AnalPDE,HadzicRaphael2019JEMS,CollotGhoulMasmoudiNguyen2022CPAM}
for some critical parabolic equations and in \cite{MerleRaphaelRodnianski2015CambJMath,Collot2018MemAMS}
for some energy-supercritical dispersive equations, we expect that
(countably) infinitely many blow-up rates of the form $\{(T-t)^{c_{\ell}}/|\log(T-t)|^{d_{\ell}}\}_{\ell\in\bbN}$
can arise from smooth initial data. Identifying such admissible blow-up
rates and proving rigidity in this context are of great interest.
Note that rigidity of this type is obtained for some energy-supercritical
semi-linear heat equations within radial symmetry \cite{Mizoguchi2007MathAnn}
using the intersection comparison argument, which is not available
for dispersive equations.

If one further allows rougher blow-up solutions, then the blow-up
dynamics is no more rigid. In fact, there is a continuum of possible
blow-up rates for $H_{0}^{1}$-solutions \cite{KimKwonOh2022arXiv2};
for any $p>1$ there is a $H_{0}^{1}$ finite-time blow-up solution
such that $\lmb(t)\sim|t|^{p}/|\log|t||$. Let us note that (see Section~\ref{subsec:Linear-conjugation-identities})
the dynamics of $m$-equivariant \eqref{eq:CSS-m-equiv} is closely
related to the dynamics of $(m+1)$-equivariant wave maps and Schrödinger
maps and such a continuum of blow-up rates have already been constructed
in these equations within $1$-equivariance \cite{KriegerSchlagTataru2008Invent,Perelman2014CMP,JendrejLawrieRodriguez2022ASENS}.

We finally note that the stabilization of scale for global-in-time
solutions is \emph{always} false for solutions with $H_{0}^{1,1}$
initial data. Indeed, it is proved in \cite{KimKwonOh2022arXiv1}
that $\lmb(t)\aleq(\log t)^{-1/2}$. This does not contradict the
existence of the static solution $Q$ because $Q$ does not belong
to $H_{0}^{1,1}$ when $m=0$.
\end{rem}

\begin{rem}[Rotational instability]
\label{rem:rotational-instability}By Theorem~\ref{thm:blow-up-rigidity},
(when $m\geq1$) pseudoconformal blow-up is the only blow-up regime
allowed for smooth initial data but it is in fact non-generic and
unstable. Thus the dynamics of solutions near these blow-up solutions
is generic and of interest.

In \cite{KimKwon2023MAMS}, the authors constructed a curve of solutions
$\{u^{(\eta)}\}_{\eta\in[0,\eta^{\ast}]}$ such that the scale and
phase parameters approximately evolve as \eqref{eq:formal-sol}. In
other words, while $u^{(0)}$ is a pseudoconformal blow-up solution,
$u^{(\eta)}$ with $\eta\neq0$ concentrates only up to the spatial
scale $|\eta|$, then exhibits a quick phase rotation by the fixed
angle $\mathrm{sgn}(\eta)(m+1)\pi$ within a time interval of length
$\sim|\eta|$, and then spreads out like a backward-in-time evolution
of pseudoconformal blow-up solutions (\emph{rotational instability}).
On the other hand, a codimension one set of initial data leading to
pseudoconformal blow-up is constructed in \cite{KimKwon2023AnnPDE}.
These two observations naturally suggest the following \emph{rotational
instability conjecture}: there is a codimension one manifold $\calM$
of smooth initial data $u_{0}$ such that (i) if $u_{0}\in\calM$,
then $u(t)$ blows up in finite time (and hence is a pseudoconformal
blow-up) and (ii) if $u_{0}\notin\calM$ but is close to $\calM$,
then $u(t)$ does not blow up and rather exhibits rotational instability.

Rotational instability seems to be not limited to \eqref{eq:CSS-m-equiv}
but rather \emph{universal}. The work \cite{BergWilliams2013} provides
some formal computations and numerical simulations for a quick spatial
rotation by the angle $\pi$ for the $1$-equivariant Landau--Lifschitz--Gilbert
equation. See also \cite{MerleRaphaelRodnianski2013InventMath}. The
same is expected for the critical wave maps. We finally note that
the $H^{3}$-boundedness of the regular part of solution as in Remark~\ref{rem:structure-of-full-sol}
not only holds for finite-time blow-up solutions but also for nearby
solutions. We believe that this observation plays a crucial role for
the resolution of the rotational instability conjecture.
\end{rem}

\subsection{\label{subsec:Strat-Proof}Strategy of the proof}

We use the notation collected in Section~\ref{subsec:Notation}.
As Corollary~\ref{cor:H33-soliton-resolution} is a direct consequence
of pseudoconformal symmetry and Theorem~\ref{thm:blow-up-rigidity},
we only sketch the proof of Theorem~\ref{thm:blow-up-rigidity}.

\medskip{}

Let $u$ be a $H_{m}^{3}$-solution to \eqref{eq:CSS-m-equiv} that
blows up forwards in time at $T<+\infty$. Recall from \cite{KimKwonOh2022arXiv1}
(or Remark~\ref{rem:blow-up-vicinity} of this paper) that after
renormalization $u(t)$ eventually enters the $\dot{H}_{m}^{1}$-vicinity
of the soliton: 
\[
\inf_{\gmm\in\bbR/2\pi\bbZ}\big\|\wh{\lmb}(t)e^{-i\gmm}u(t,\wh{\lmb}(t)\cdot)-Q\big\|_{\dot{H}_{m}^{1}}\to0\quad\text{ as }\quad t\to T,
\]
where $\wh{\lmb}(t)=\|Q\|_{\dot{H}_{m}^{1}}/\|u(t)\|_{\dot{H}_{m}^{1}}$.
We note that there is no $L^{2}$-closeness because we allow $u$
to have arbitrarily large mass.

We study the near-soliton dynamics of \eqref{eq:CSS-m-equiv} using
modulation analysis. As in the previous work \cite{KimKwon2023AnnPDE},
we will write $u(t)$ in the form 
\begin{equation}
u(t,r)=\frac{e^{i\gmm(t)}}{\lmb(t)}[P(\cdot;b(t),\eta(t))+\eps(t,\cdot)]\Big(\frac{r}{\lmb(t)}\Big)\eqqcolon[P+\eps]^{\sharp}(t,r)\label{eq:str-dec}
\end{equation}
for some additional modulation parameters $b,\eta$ and some modified
profile $P=P(y;b,\eta)$ with $P(\cdot;0,0)=Q$. Note that the finite-time
blow-up assumption only says that $\lmb(t)\to0$ as $t\to T$.

Our key ingredient for the proof of the refined descriptions in Theorem~\ref{thm:blow-up-rigidity}
will be the proof of \emph{$H^{3}$-boundedness of $\eps^{\sharp}(t)$
up to the blow-up time}. To be more precise, we will in fact prove
the boundedness of $\eps^{\sharp}(t)$ in a stronger topology, which
in particular guarantees the boundedness of (see \eqref{eq:def-k-derivs}
for the notation $|\cdot|_{-3}$)
\begin{equation}
\big\|\chf_{r\geq\lmb(t)}|\eps^{\sharp}(t)|_{-3}\big\|_{L^{2}}.\label{eq:str-bdd}
\end{equation}
Taking the limit $t\to T$ gives the regularity \eqref{eq:z-ast-reg-pt}
of $z_{\rg}^{\ast}$. Note that the modified profile $P$ should also
be chosen carefully because we a posteriori know from Theorem~\ref{thm:blow-up-rigidity}
that $P^{\sharp}$ should detect the remaining part of $z^{\ast}$,
namely, $z_{\sg}^{\ast}$ in the limit $t\to T$. We also note that
the $H^{3}$-boundedness of $\eps^{\sharp}$ implies the rigidity
of blow-up rates; see Step~5 of this strategy.

As can be seen in the strategy below, our approach shares a similar
spirit with the forward construction of blow-up dynamics as in \cite{RodnianskiSterbenz2010Ann.Math.,RaphaelRodnianski2012Publ.Math.,MerleRaphaelRodnianski2013InventMath,MartelMerleRaphael2014Acta,MerleRaphaelRodnianski2015CambJMath}
and for \eqref{eq:CSS-m-equiv}, in \cite{KimKwon2023AnnPDE,KimKwonOh2020arXiv}
($m\geq1$ and $m=0$, respectively). However, we need further improvements
in all parts of the proof compared to the blow-up construction in
\cite{KimKwon2023AnnPDE}. The main reason is that, since we are dealing
with arbitrary finite-time blow-up solutions, we \emph{cannot assume
dynamical controls} such as 
\begin{equation}
b>0\qquad\text{and}\qquad|\eta|\ll b\label{eq:old-dyn-control-1}
\end{equation}
(meaning that $\lmb$ is shrinking and the instability mode is turned
off, respectively), which can be safely assumed in the blow-up construction
of \cite{KimKwon2023AnnPDE}. We also note that the $H^{3}$-boundedness
of $\eps^{\sharp}(t)$ is not even necessary there. In this sense,
our analysis here focuses more on the general near-soliton dynamics
as in \cite{MartelMerleRaphael2014Acta}. See also \cite{MerleRaphael2005AnnMath,Raphael2005MathAnn,RaphaelSzeftel2011JAMS}
for the mass-critical NLS.

\medskip{}
\emph{1. Conjugated variables and setup for the modulation analysis.}
As in \cite{KimKwonOh2020arXiv}, we view \eqref{eq:CSS-m-equiv}
as a system of equations for nonlinearly conjugated variables $u$,
$u_{1}\coloneqq\bfD_{u}u$, and $u_{2}\coloneqq A_{u}\bfD_{u}u$,
where $A_{u}=\bfD_{u}-\frac{1}{r}$. Being $u_{1}$ and $u_{2}$ naturally
defined (see \eqref{eq:cov-conj-identification} below), the evolution
equations \eqref{eq:u1-eqn} and \eqref{eq:u2-eqn} for $u_{1}$ and
$u_{2}$ still have simple forms.

For the modulation analysis, we will decompose $u,u_{1},u_{2}$ of
the form 
\begin{equation}
\left\{ \begin{aligned}u(t,r) & =\frac{e^{i\gmm(t)}}{\lmb(t)}[P(\cdot;b(t),\eta(t))+\eps(t,\cdot)]\Big(\frac{r}{\lmb(t)}\Big),\\
u_{1}(t,r) & =\frac{e^{i\gmm(t)}}{\lmb^{2}(t)}[P_{1}(\cdot;b(t),\eta(t))+\eps_{1}(t,\cdot)]\Big(\frac{r}{\lmb(t)}\Big),\\
u_{2}(t,r) & =\frac{e^{i\gmm(t)}}{\lmb^{3}(t)}[P_{2}(\cdot;b(t),\eta(t))+\eps_{2}(t,\cdot)]\Big(\frac{r}{\lmb(t)}\Big),
\end{aligned}
\right.\label{eq:str-decom}
\end{equation}
for some modified profiles $P,P_{1},P_{2}$ to be constructed. Four
orthogonality conditions are imposed either on $\eps$ or $\eps_{1}$
to fix the decomposition. To analyze the flow, we introduce the renormalized
coordinates $(s,y)$ as well as the renormalized variable $w$: 
\[
\frac{ds}{dt}=\frac{1}{\lmb^{2}(t)},\quad y=\frac{r}{\lmb(t)},\quad w(s,y)=\lmb e^{-i\gmm}u(t,\lmb y)|_{t=t(s)}.
\]
The associated conjugated variables are $w_{1}\coloneqq\bfD_{w}w$
and $w_{2}\coloneqq A_{w}\bfD_{w}w$. The system of evolution equations
for $w,w_{1},w_{2}$ can be found in \eqref{eq:w-eqn}--\eqref{eq:w2-eqn}.

One notable advantage of using the system of equations for the nonlinearly
conjugated variables $w,w_{1},w_{2}$ is that we can separate the
roles of each evolution equation. For example, the $w$-equation will
only be used to detect the modulation equations for $\lmb$ and $\gmm$,
which allows simpler (or, rougher) ansatz for $P$ compared to $P_{1}$
or $P_{2}$. The $w_{1}$-equation will be used to detect the modulation
equations for $b$ and $\eta$. In fact, these equations become more
apparent in the $w_{1}$-equation rather than the $w$-equation itself,
allowing the authors in \cite{KimKwonOh2020arXiv} to capture the
crucial logarithmic corrections in the $b_{s}$- and $\eta_{s}$-equations.
Finally, the $w_{2}$-equation will be used in the final energy estimates
for the remainder $\eps$. The use of the $w_{2}$-equation not only
simplifies the structure of the nonlinearity but it is also in the
$w_{2}$-equation that the linearized Hamiltonian $\td H_{Q}$ has
a \emph{repulsive }property (see Section~\ref{subsec:Linear-conjugation-identities}).
This enables the authors in \cite{KimKwonOh2020arXiv} to perform
the modified energy method for the variable $\eps_{2}$. We refer
to \cite{KimKwonOh2020arXiv} for more discussions on this approach.

\medskip{}
\emph{2. Construction of the modified profiles.} In \cite{KimKwon2023MAMS,KimKwon2023AnnPDE},
when $m\geq1$, the authors introduced a modified profile, which is
constructed in a nonlinear way and suggests the formal modulation
equations 
\begin{equation}
\frac{\lmb_{s}}{\lmb}+b=0,\quad\gmm_{s}\approx(m+1)\eta,\quad b_{s}+b^{2}+\eta^{2}=0,\quad\eta_{s}=0.\label{eq:str-old-mod-eqn}
\end{equation}
The $\eta=0$ case corresponds to pseudoconformal blow-up and its
instability mechanism can be dictated by varying $\eta$; see Section~\ref{subsec:Old-profile-ansatz}
for more details.

However, when $m\in\{1,2\}$, that profile turns out to be insufficient
to address the $H^{3}$-boundedness of $\eps^{\sharp}$ and we need
a new profile ansatz. We perform the tail computations as in \cite{RaphaelRodnianski2012Publ.Math.,MerleRaphaelRodnianski2013InventMath,MerleRaphaelRodnianski2015CambJMath}.
It turns out that when $m=1$, there are nontrivial \emph{cubic corrections}
to the $b_{s}$- and $\eta_{s}$-modulation equations, say
\[
b_{s}+b^{2}+\eta^{2}+p_{3}^{(b)}=0\qquad\text{and}\qquad\eta_{s}+p_{3}^{(\eta)}=0,
\]
and the new profiles deviate from the old ones at the cubic order.
When $m=2$, the same happens at the quartic order.

Another interesting observation is that \emph{the cutoff radius must
be 
\begin{equation}
y\sim B_{1}\coloneqq(b^{2}+\eta^{2})^{-\frac{1}{2}},\label{eq:str-cutoff-radi}
\end{equation}
which corresponds to $r\sim1$} in the original space variable $r$.
Moreover, in the original spacetime variables, our modified profile
\emph{$P^{\sharp}(t)$ converges to the singular part $z_{\sg}^{\ast}$
as $t\to T$.} This is a sharp contrast to the use of almost self-similar
scales typically employed in many of the prior blow-up constructions
such as \cite{RaphaelRodnianski2012Publ.Math.,MerleRaphaelRodnianski2013InventMath,MerleRaphaelRodnianski2015CambJMath,KimKwon2023AnnPDE}.

\medskip{}
\emph{3. Key a priori bounds from nonlinear coercivity of energy.}
Despite the lack of dynamical controls \eqref{eq:old-dyn-control-1},
we are able to obtain the following \emph{size controls}:
\begin{equation}
\|\eps\|_{L^{2}}\aleq_{M[u]}1\qquad\text{and}\qquad|b|+|\eta|+\|\eps\|_{\dot{H}_{m}^{1}}\aleq_{M[u]}\lmb\sqrt{E[u]}\eqqcolon\mu.\label{eq:size-control-1}
\end{equation}
The former estimate is clear from mass conservation, but the latter,
which we call the \emph{nonlinear coercivity of energy}, is not obvious.
In a simpler context where $b$ and $\eta$ are absent, this estimate
was proved in \cite{KimKwonOh2022arXiv1} as a remarkable consequence
of the \emph{self-duality} and \emph{non-locality} of our problem,
which are two distinguished features of \eqref{eq:CSS-m-equiv}. We
refer to \cite{KimKwonOh2022arXiv1} or Remark~\ref{rem:on-quant-small}
below for more explanations. Recalling its proof in \cite{KimKwonOh2022arXiv1},
it is not difficult to extend the nonlinear coercivity in the presence
of $b$- and $\eta$-modulations. See Lemma~\ref{lem:NonlinCoerEnergy}
below.

\medskip{}
\emph{4. Energy-Morawetz functional and uniform $H^{3}$-bound.} We
propagate the $H^{3}$-boundedness of $\eps^{\sharp}$, which is the
heart of the proof of Theorem~\ref{thm:blow-up-rigidity}. The basic
strategy will be the \emph{energy method with repulsivity}, the idea
being that the non-perturbative but localized error terms appearing
in the energy estimates can be handled by the monotonicity from repulsivity.
This strategy has successfully addressed many of the forward constructions
of blow-up dynamics, for example in \cite{RodnianskiSterbenz2010Ann.Math.,RaphaelRodnianski2012Publ.Math.,MerleRaphaelRodnianski2013InventMath,MerleRaphaelRodnianski2015CambJMath}
and for \eqref{eq:CSS-m-equiv} in \cite{KimKwon2023AnnPDE,KimKwonOh2020arXiv}.

Again, the difficulty is that we cannot assume dynamical controls
such as \eqref{eq:old-dyn-control-1}. However, we have the \emph{size
control} \eqref{eq:size-control-1} and our main assertion will be
that this is sufficient to propagate the $H^{3}$-boundedness of $\eps^{\sharp}$.
We should mention here that in their work \cite{MartelMerleRaphael2014Acta}
on the critical gKdV, the authors were able to design monotonicity
formulas that are sufficient to conclude the rigidity of the near-soliton
dynamics even without the size control \eqref{eq:size-control-1}.
Instead, they used the local virial estimate and properties specific
(e.g., Kato-type smoothing) to gKdV, where the latter seems difficult
to be implemented in the Schrödinger setting.

We propose a new energy functional $\calF$ that propagates the control
$\|\eps^{\sharp}\|_{\dot{H}^{3}\cap\dot{H}^{2}}\aleq1$ on any finite
time intervals only under the size control \eqref{eq:size-control-1}.
As mentioned before, we look at the $\eps_{2}$-variable, whose associated
linearized Hamiltonian $\td H_{Q}$ has a \emph{repulsive} character.
The functional $\calF$ takes the form 
\begin{equation}
\calF=\frac{(\eps_{2},\td H_{Q}\eps_{2})_{r}}{\mu^{6}}-A\frac{(i\eps_{2},\psi'\rd_{y}\eps_{2})_{r}}{\mu^{5}}+A^{2}\frac{\|\eps_{2}\|_{L^{2}}^{2}}{\mu^{4}}\label{eq:str-def-H3en}
\end{equation}
for some large constant $A>1$ and some well-chosen function $\psi'=\psi'(y)$
of Morawetz type (see Lemma~\ref{lem:Morawetz}). Note that $\calF$
is coercive in the sense that $\calF\aleq1$ implies $\|\eps_{2}\|_{\dot{H}^{1}}\aleq\mu^{3}$
and $\|\eps_{2}\|_{L^{2}}\aleq\mu^{2}$, which in turn give $\|\eps\|_{\dot{H}^{3}}\aleq\mu^{3}$
and $\|\eps\|_{\dot{H}^{2}}\aleq\mu^{2}$ using coercivity estimates.
By scaling, we have the $H^{3}$-boundedness $\|\eps^{\sharp}\|_{\dot{H}^{3}\cap\dot{H}^{2}}\aleq1$.
Thus our goal is to propagate the control $\calF\aleq1$.

Propagation of the control $\calF\aleq1$ in this work crucially utilizes
the assumption $m\geq1$ as well as the size control \eqref{eq:size-control-1}.
Indeed, the \emph{Morawetz type repulsivity} (Lemma~\ref{lem:Morawetz}),
i.e., the positivity of the commutator $[\td H_{Q},\Lmb_{\psi}]$
for some choice of $0\leq\psi'\leq1$, where $\Lmb_{\psi}=\psi'\rd_{y}+\frac{\Delta\psi}{2}$,
heavily relies on $m\geq1$. Moreover, $|b|+|\eta|\aleq\mu$ of \eqref{eq:size-control-1}
is used to dominate the non-perturbative (but localized) error terms
arising in $\rd_{s}\{(\eps_{2},\td H_{Q}\eps_{2})_{r}/\mu^{6}\}$
by the Morawetz monotonicity from $\rd_{s}\{(i\eps_{2},\psi'\rd_{y}\eps_{2})_{r}/\mu^{5}\}$.
Furthermore, $\|\eps\|_{\dot{H}_{m}^{1}}\aleq\mu$ of \eqref{eq:size-control-1}
is used to make pure $\eps$-terms perturbative; the contribution
of such pure $\eps$-terms arises in the form $\rd_{s}\calF\aleq\mu^{2}(1+\calF)$,
which gives $\rd_{t}\calF\aleq1+\calF$, which in turn yields $\calF(t)\aleq1+\calF(t_{0})$
on any finite time interval by Grönwall's inequality. More technical
aspects of this energy estimate will be explained in Section~\ref{subsec:Energy-Morawetz}.

As a result of the boundedness of $\calF$, we roughly have 
\begin{equation}
\|\eps\|_{\dot{H}^{k}}\aleq\lmb^{k},\qquad0\leq k\leq3.\label{eq:str-unif-H3}
\end{equation}

\medskip{}
\emph{5. Rigidity of the blow-up rates.} In order to prove rigidity
of the blow-up rate, an important role is played by the quantity 
\begin{equation}
\frac{\beta^{2}}{\lmb^{2}}=\frac{b^{2}+\eta^{2}}{\lmb^{2}}.\label{eq:str-formal-cons}
\end{equation}
This corresponds to the energy of the blow-up part $P^{\sharp}$ and
is conserved under the formal modulation ODEs \eqref{eq:str-old-mod-eqn}.
Note that by \eqref{eq:size-control-1} we only know $\frac{\beta}{\lmb}\aleq1$
at this stage. Our goal is to show the reverse inequality. In fact,
we show that 
\begin{equation}
\frac{\beta}{\lmb}\to\ell\neq0\qquad\text{as }t\to T.\label{eq:str-nonzero-beta}
\end{equation}
We note that $\ell\neq0$ is important because if $\ell=0$ one can
consider an exotic blow-up regime $\lmb(t)\sim(T-t)^{\nu}$ and $b(t)\sim(T-t)^{2\nu-1}$
with $\nu>1$. We also note that energy conservation of the full solution
$u(t)$ is useless because the $\eps^{\sharp}$-part is nontrivial.\footnote{In \cite{RaphaelSzeftel2011JAMS}, where the authors consider minimal
mass blow-up solutions, \eqref{eq:str-nonzero-beta} easily follows
from energy conservation and the improved smallness $\|\eps^{\sharp}\|_{H^{1}}\to0$.
Let us note however that this smallness estimate is nontrivial. Nevertheless,
$\|\eps^{\sharp}\|_{H^{1}}\to0$ does not hold in our context.} Instead, we look at the evolution of the quantity $\frac{\beta}{\lmb}+\lmb$
and obtain \eqref{eq:rmk-diff-ineq}, i.e., 
\[
\Big(\frac{\beta}{\lmb}+\lmb\Big)_{t}\aleq\Big(1+\frac{\|\eps\|_{\loc}^{2}}{\lmb^{4}}\Big)\Big(\frac{\beta}{\lmb}+\lmb\Big)
\]
for some spatially localized norm $\|\cdot\|_{\loc}$. By choosing
$\|\cdot\|_{\loc}=\|\cdot\|_{\dot{H}^{3}}$ and applying the uniform
$H^{3}$-bound \eqref{eq:str-unif-H3}, we obtain $(\frac{\beta}{\lmb}+\lmb)_{t}\aleq\frac{\beta}{\lmb}+\lmb$.
By Grönwall's inequality $\frac{\beta}{\lmb}+\lmb$ cannot converge
to zero in finite time. See also Remark~\ref{rem:H3/2-rigidity}.
Since $\lmb\to0$ on the other hand, we obtain \eqref{eq:str-nonzero-beta}.
Once we have $\frac{\beta}{\lmb}\to\ell\neq0$ at hand, it is not
difficult to conclude $\lmb(t)\approx\ell\cdot(T-t)$ and $\gmm(t)\to\gmm^{\ast}$.
Let us finally note that there is a small technical issue when $m$
is small and we need to introduce corrections $\wh b,\wh{\eta}$ of
$b,\eta$ to run the previous argument.

\medskip{}
\emph{6. Structure of the asymptotic profile.} As the existence of
the asymptotic profile $z^{\ast}$ is a consequence of Theorem~\ref{thm:asymptotic-description},
we focus on the structure of $z^{\ast}$. When $m=1$, the quadratic
term of $P^{\sharp}$ resembles $[\chi_{B_{1}}b^{2}y]^{\sharp}=\chi_{\lmb/\beta}(r)(\frac{b}{\lmb})^{2}r$,
which gives a nontrivial singular part $z_{\sg}^{\ast}$ in the limit
$t\to T$ by the choice \eqref{eq:str-cutoff-radi} of the cutoff
and \eqref{eq:str-nonzero-beta}. Similarly, when $m=2$ the cubic
term of $P^{\sharp}$ converges to $z_{\sg}^{\ast}$. The bound \eqref{eq:z-ast-reg-pt}
follows from taking the $t\to T$ limit of the $H^{3}$-boundedness
\eqref{eq:str-bdd}.

\medskip{}

\noindent \mbox{\textbf{Organization of the paper.} }In Section~\ref{sec:Preliminaries},
we collect notation and preliminaries for our analysis. In Section~\ref{sec:Modified-profiles},
we construct the modified profiles $P,P_{1},P_{2}$. In Section~\ref{sec:Modulation-analysis},
we perform modulation analysis of the near soliton dynamics and prove
the $H^{3}$-boundedness of $\eps^{\sharp}$, which is the heart of
the paper. Finally in Section~\ref{sec:Proof-of-the-main-thm}, we
prove Theorem~\ref{thm:blow-up-rigidity}. For the proof of Corollary~\ref{cor:H33-soliton-resolution},
see Remark~\ref{rem:cor-H33-soliton-res}.

\medskip{}

\noindent \mbox{\textbf{Acknowledgements.} }The author is supported
by  Huawei Young Talents Programme at IHES.

\section{\label{sec:Preliminaries}Preliminaries}

\subsection{\label{subsec:Notation}Notation}

For $A\in\bbC$ and $B\geq0$, we use the standard asymptotic notation
$A\aleq B$ or $A=O(B)$ if there is a constant $C>0$ such that $|A|\leq CB$.
$A\sim B$ means that $A\aleq B$ and $B\aleq A$. The dependencies
of $C$ are specified by subscripts, e.g., $A\aleq_{E}B\Leftrightarrow A=O_{E}(B)\Leftrightarrow|A|\leq C(E)B$.
In this paper, any dependencies on the equivariance index $m$ will
be omitted.

We also use the small $o$-notation such as $o_{\alpha^{\ast}\to0}(1)$
or $o_{t\to T}(1)$, which denotes a quantity that can be made arbitrarily
small by taking $\alpha^{\ast}\to0$ or $t\to T$, respectively. Similarly,
we use the notation $\delta_{M}(\alpha^{\ast})$ to denote a quantity
that goes to $0$ as $\alpha^{\ast}\to0$ but may depend on $M$ (thus
may not be uniform in $M$).

We also use the notation $\langle x\rangle$, $\log_{+}x$, $\log_{-}x$
defined by 
\[
\langle x\rangle\coloneqq(|x|^{2}+1)^{\frac{1}{2}},\quad\log_{+}x\coloneqq\max\{\log x,0\},\quad\log_{-}x\coloneqq\max\{-\log x,0\}.
\]
We let $\chi=\chi(x)$ be a smooth spherically symmetric cutoff function
such that $\chi(x)=1$ for $|x|\leq1$ and $\chi(x)=0$ for $|x|\geq2$.
For $A>0$, we define its rescaled version by $\chi_{A}(x)\coloneqq\chi(x/A)$.
We also define $\chi_{\ageq A}\coloneqq1-\chi_{A}$.

We mainly work with equivariant functions on $\bbR^{2}$, say $\phi:\bbR^{2}\to\bbC$,
or equivalently their radial part $u:\bbR_{+}\to\bbC$ with $\phi(x)=u(r)e^{im\theta}$,
where $\bbR_{+}\coloneqq(0,\infty)$ and $x_{1}+ix_{2}=re^{i\theta}$.
We denote by $\rd_{+}=\rd_{1}+i\rd_{2}$ the Cauchy--Riemann operator,
which acts on $m$-equivariant function via the formula 
\begin{equation}
\rd_{+}[u(r)e^{im\theta}]=[(\rd_{r}-\tfrac{m}{r})u](r)e^{i(m+1)\theta}.\label{eq:def-euc-CR}
\end{equation}
We denote by $\Delta^{(m)}=\rd_{rr}+\tfrac{1}{r}\rd_{r}-\tfrac{m^{2}}{r^{2}}$
the Laplacian acting on $m$-equivariant functions.

The integral symbol $\int$ means 
\[
\int=\int_{\bbR^{2}}\,dx=2\pi\int\,rdr.
\]
For complex-valued functions $f$ and $g$, we define their \emph{real}
inner product by 
\[
(f,g)_{r}\coloneqq\int\Re(\br fg).
\]
For a real-valued functional $F$ and a function $u$, we denote by
$\nabla F[u]$ the \emph{functional derivative} of $F$ at $u$ under
this real inner product.

For $s\in\bbR$, denote by $\Lmb_{s}$ the generator of $\dot{H}^{s}$-scaling:
\begin{align*}
\Lmb_{s} & \coloneqq r\rd_{r}+1-s,\\
\Lmb & \coloneqq\Lmb_{0}=r\rd_{r}+1.
\end{align*}
Note that $\Lmb$ is formally anti-symmetric. We also use a truncated
version of $\Lmb$; for a spherically symmetric function $\psi=\psi(x)$
on $\bbR^{2}$, we define the operator
\begin{equation}
\Lmb_{\psi}\coloneqq\psi'\rd_{r}+\frac{\Delta\psi}{2},\label{eq:def-Lmb-psi}
\end{equation}
where $\psi'=\rd_{r}\psi$ and $\Delta\psi=\rd_{rr}\psi+\frac{1}{r}\rd_{r}\psi$.
Note that $\Lmb_{\psi}$ is formally anti-symmetric.

For $k\in\bbN$, we define 
\begin{equation}
\begin{aligned}|f|_{k} & \coloneqq\max\{|f|,|r\rd_{r}f|,\dots,|(r\rd_{r})^{k}f|\},\\
|f|_{-k} & \coloneqq\max\{|\rd_{r}^{k}f|,|\tfrac{1}{r}\rd_{r}^{k-1}f|,\dots,|\tfrac{1}{r^{k}}f|\}.
\end{aligned}
\label{eq:def-k-derivs}
\end{equation}
We note that $|f|_{k}\sim r^{k}|f|_{-k}$. The following Leibniz rules
hold: 
\[
|fg|_{k}\aleq|f|_{k}|g|_{k},\qquad|fg|_{-k}\aleq|f|_{k}|g|_{-k}.
\]

Given a scaling parameter $\lmb\in\bbR_{+}$, phase rotation parameter
$\gmm\in\bbR/2\pi\bbZ$, and a function $f$, we write 
\[
f_{\lmb,\gmm}(r)\coloneqq\frac{e^{i\gmm}}{\lmb}f\Big(\frac{r}{\lmb}\Big).
\]
When $\lmb$ and $\gmm$ are clear from the context, we will also
denote the above by $f^{\sharp}$ as in \cite{KimKwon2023MAMS}, i.e.,
\[
f^{\sharp}\coloneqq f_{\lmb,\gmm}.
\]
Similarly, we define the $\flat$-operation as 
\[
g^{\flat}\coloneqq g_{1/\lmb,-\gmm}.
\]
When a time-dependent scaling parameter $\lmb(t)$ is given, we define
rescaled spacetime variables $s,y$ by the relations 
\begin{equation}
\frac{ds}{dt}=\frac{1}{\lmb^{2}(t)}\qquad\text{and}\qquad y=\frac{r}{\lmb(t)}.\label{eq:def-renormalized-var}
\end{equation}
The raising operation $\sharp$ converts a function $f=f(y)$ to a
function of $r$: $f^{\sharp}=f^{\sharp}(r)$. The lowering operator
$\flat$ does the reverse. In the modulation analysis in this paper,
the dynamical parameters such as $\lmb,\gmm,b,\eta$ are functions
of either the variable $t$ or $s$ under $\frac{ds}{dt}=\frac{1}{\lmb^{2}}$.

We collect the definitions of various linear operators. Let $A_{\tht}[\psi_{1},\psi_{2}]$
be the polarization of $A_{\tht}[\psi]$: 
\[
A_{\theta}[\psi_{1},\psi_{2}]=-\tfrac{1}{2}\tint 0y\Re(\br{\psi_{1}}\psi_{2})y'dy'.
\]
We will often use the first order operators and their formal $L^{2}$-adjoints:
\begin{align*}
\bfD_{w} & =\rd_{y}-\tfrac{1}{y}(m+A_{\tht}[w]), & \bfD_{w}^{\ast} & =-\rd_{y}-\tfrac{1}{y}(m+1+A_{\tht}[w]),\\
L_{w} & =\bfD_{w}-\tfrac{2}{y}A_{\tht}[w,\cdot], & L_{w}^{\ast} & =\bfD_{w}^{\ast}+w\tint y{\infty}\Re(\br w\cdot)dy',\\
A_{w} & =\bfD_{w}-\tfrac{1}{y}, & A_{w}^{\ast} & =\bfD_{w}^{\ast}-\tfrac{1}{y}.
\end{align*}
The second order operators of particular importance are 
\begin{align*}
\calL_{w} & =\nabla^{2}E[w]\text{, i.e., the Hessian of }E,\\
H_{w} & =-\rd_{yy}-\tfrac{1}{y}\rd_{y}+\tfrac{1}{y^{2}}(m+1+A_{\tht}[w])^{2}-\tfrac{1}{2}|w|^{2}=A_{w}^{\ast}A_{w},\\
\td H_{w} & =-\rd_{yy}-\tfrac{1}{y}\rd_{y}+\tfrac{1}{y^{2}}(m+2+A_{\tht}[w])^{2}+\tfrac{1}{2}|w|^{2}=A_{w}A_{w}^{\ast}.
\end{align*}
See Section~\ref{subsec:Linearization-of-CSS} for more explanations
on these linear operators.

The relevant function spaces such as $\dot{H}_{m}^{k}$, $\dot{\calH}_{m}^{k}$,
and $\dot{V}_{m}^{k}$, and their norms will be defined in Section~\ref{subsec:Adapted-function-spaces}.

\subsection{\label{subsec:Conjugation-identities}Conjugation identities}

In this subsection, we record the evolution equations for $u$, $u_{1}=\bfD_{u}u$,
and $u_{2}=A_{u}\bfD_{u}u$, where we recall $\bfD_{u}=\rd_{r}-\tfrac{1}{r}(m+A_{\tht}[u])$
and $A_{u}=\bfD_{u}-\frac{1}{r}$. As mentioned in the strategy (Section~\ref{subsec:Strat-Proof}),
we will view \eqref{eq:CSS-m-equiv} as a system of equations for
$u$, $u_{1}$, $u_{2}$.

First, we note that $u$, $u_{1}$, and $u_{2}$ are realized as $m$-,
$(m+1)$-, and $(m+2)$-equivariant functions. This is because the
covariant Cauchy--Riemann operator $\bfD_{+}=\bfD_{1}+i\bfD_{2}$
shifts the equivariance index by $1$ (c.f. \eqref{eq:def-euc-CR})
and $u,u_{1},u_{2}$ are identified as 
\begin{equation}
\left\{ \begin{aligned}\phi(x) & =u(r)e^{im\theta},\\{}
[\bfD_{+}\phi](x) & =[\bfD_{u}u](r)e^{i(m+1)\theta}=u_{1}(r)e^{i(m+1)\theta},\\{}
[\bfD_{+}\bfD_{+}\phi](x) & =[A_{u}\bfD_{u}u](r)e^{i(m+2)\theta}=u_{2}(r)e^{i(m+2)\theta}.
\end{aligned}
\right.\label{eq:cov-conj-identification}
\end{equation}
Now, we record the evolution equations for $u$, $u_{1}$, and $u_{2}$.
The following can be derived using the identification \eqref{eq:cov-conj-identification}
and the curvature relations of \eqref{eq:CSS-cov}. 
\begin{gather}
\rd_{t}u+i\bfD_{u}^{\ast}u_{1}+\Big(\int_{r}^{\infty}\Re(\br uu_{1})dr'\Big)iu=0,\label{eq:u-eqn}\\
\rd_{t}u_{1}+iA_{u}^{\ast}u_{2}+\Big(\int_{r}^{\infty}\Re(\br uu_{1})dr'\Big)iu_{1}=0,\label{eq:u1-eqn}\\
\rd_{t}u_{2}+i\td H_{u}u_{2}+\Big(\int_{r}^{\infty}\Re(\br uu_{1})dr'\Big)iu_{2}-i\br uu_{1}^{2}=0.\label{eq:u2-eqn}
\end{gather}
See \cite[Section 2]{KimKwonOh2020arXiv} for more details.

When ($C^{1}$) parameter curves $\lmb(t)$ and $\gmm(t)$ are given,
using the renormalized coordinates $(s,y)$ defined in \eqref{eq:def-renormalized-var},
we consider the renormalized variables $w,w_{1},w_{2}$ of $u,u_{1},u_{2}$:
\begin{align*}
w(s,y) & \coloneqq\lmb e^{-i\gmm}u(t,\lmb r)|_{t=t(s)},\\
w_{1}(s,y) & \coloneqq\lmb^{2}e^{-i\gmm}u_{1}(t,\lmb r)|_{t=t(s)}=\bfD_{w}w,\\
w_{2}(s,y) & \coloneqq\lmb^{3}e^{-i\gmm}u_{2}(t,\lmb r)|_{t=t(s)}=A_{w}\bfD_{w}w.
\end{align*}
These variables solve the following evolution equations:
\begin{align}
\big(\rd_{s}-\frac{\lmb_{s}}{\lmb}\Lmb+\gmm_{s}i\big)w & +i\bfD_{w}^{\ast}w_{1}+\Big(\int_{y}^{\infty}\Re(\br ww_{1})dy'\Big)iw=0,\label{eq:w-eqn}\\
\big(\rd_{s}-\frac{\lmb_{s}}{\lmb}\Lmb_{-1}+\gmm_{s}i\big)w_{1} & +iA_{w}^{\ast}w_{2}+\Big(\int_{y}^{\infty}\Re(\br ww_{1})dy'\Big)iw_{1}=0,\label{eq:w1-eqn}\\
\big(\rd_{s}-\frac{\lmb_{s}}{\lmb}\Lmb_{-2}+\gmm_{s}i\big)w_{2} & +i\td H_{w}w_{2}+\Big(\int_{y}^{\infty}\Re(\br ww_{1})dy'\Big)iw_{2}-i\br ww_{1}^{2}=0.\label{eq:w2-eqn}
\end{align}

\subsection{\label{subsec:Linearization-of-CSS}Linearization}

We quickly record the linearization of \eqref{eq:CSS-m-equiv} around
$Q$. For more detailed exposition, see the corresponding sections
of \cite{KimKwon2023MAMS,KimKwon2023AnnPDE,KimKwonOh2020arXiv}. We
note that the discussion in this subsection formally assumes the decomposition
\[
u(t,r)=Q(r)+\eps(t,r)
\]
for small $\eps$ and hence we will identify $t=s$ and $r=y$. One
may consider the situation where $\lmb(t)\equiv1$ and $\gmm(t)\equiv0$.

\subsubsection{Linearization of the Bogomol'nyi operator}

We first linearize the Bogomol'nyi operator $w\mapsto\bfD_{w}w$.
We can write 
\begin{equation}
\bfD_{w+\eps}(w+\eps)=\bfD_{w}w+L_{w}\eps+N_{w}(\eps),\label{eq:LinearizationBogomolnyi}
\end{equation}
where 
\begin{align*}
L_{w}\eps & \coloneqq\bfD_{w}\eps-\tfrac{2}{y}A_{\tht}[w,\eps]w,\\
N_{w}(\eps) & \coloneqq-\tfrac{2}{y}A_{\tht}[w,\eps]\eps-\tfrac{1}{y}A_{\tht}[\eps]w,
\end{align*}
and $A_{\theta}[\psi_{1},\psi_{2}]$ is defined through polarization
\[
A_{\tht}[\psi_{1},\psi_{2}]\coloneqq-\tfrac{1}{2}\tint 0y\Re(\br{\psi_{1}}\psi_{2})y'dy'.
\]
The $L^{2}$-adjoint $L_{w}^{\ast}$ of $L_{w}$ takes the form 
\[
L_{w}^{\ast}v=\bfD_{w}^{\ast}v+w{\textstyle \int_{y}^{\infty}}\Re(\br wv)\,dy'.
\]
We remark that the operator $L_{w}$ and its adjoint $L_{w}^{\ast}$
are \emph{only} $\bbR$-linear. From \eqref{eq:Bogomol'nyi-eq} and
\eqref{eq:energy-self-dual-form}, we have the following expansion
for the energy: 
\begin{equation}
E[Q+\eps]=\tfrac{1}{2}\|L_{Q}\eps\|_{L^{2}}^{2}+\text{(h.o.t.)}.\label{eq:linearized-energy-expn}
\end{equation}

\subsubsection{Linearization of \eqref{eq:CSS-m-equiv} and generalized kernel relations}

Next, we linearize \eqref{eq:CSS-m-equiv}. Recall the Hamiltonian
form \eqref{eq:CSS-equiv-ham}: $\rd_{t}u+i\nabla E[u]=0$. We decompose
\begin{equation}
\nabla E[w+\eps]=\nabla E[w]+\mathcal{L}_{w}\eps+R_{w}(\eps),\label{eq:grad-energy-lin-nonlin}
\end{equation}
where $\calL_{w}\eps$ collects the linear terms in $\eps$ and $R_{w}(\eps)$
collects the remainders. Note that $\calL_{w}$ is the Hessian of
$E$, i.e., $\nabla^{2}E[w]=\calL_{w}$ and hence $\calL_{w}$ is
formally symmetric with respect to the real inner product. If one
recalls \eqref{eq:energy-self-dual-form}, we have $\nabla E[u]=L_{u}^{\ast}\bfD_{u}u=L_{u}^{\ast}u_{1}$
and hence \eqref{eq:CSS-m-equiv} can also be written as (compare
with \eqref{eq:u-eqn})
\[
i\rd_{t}u+L_{u}^{\ast}u_{1}=0.
\]
Note that \eqref{eq:grad-energy-lin-nonlin} and \eqref{eq:LinearizationBogomolnyi}
yield the following full expression of $\calL_{w}$ 
\[
\calL_{w}\eps=L_{w}^{\ast}L_{w}\eps+(\tfrac{1}{y}\tint 0y\Re(\br w\eps)y'dy')w_{1}+w\tint y{\infty}\Re(\br{\eps}w_{1})dy'+\eps\tint y{\infty}\Re(\br ww_{1})dy',
\]
where $w_{1}=\bfD_{w}w$, but this full expression is not important
in our analysis. We remark that the operator $\calL_{w}$ is only
$\bbR$-linear. From $\bfD_{Q}Q=0$, we observe the \emph{self-dual
factorization} of $i\calL_{Q}$:\footnote{This identity was first observed by Lawrie, Oh, and Shahshahani in
their unpublished note.} 
\begin{equation}
i\calL_{Q}=iL_{Q}^{\ast}L_{Q}.\label{eq:calLQ}
\end{equation}
Thus, the linearization of \eqref{eq:CSS-m-equiv} at $Q$ is 
\begin{equation}
\rd_{t}\eps+i\calL_{Q}\eps\approx0,\quad\text{or}\quad\rd_{t}\eps+iL_{Q}^{\ast}L_{Q}\eps\approx0.\label{eq:lin-CSS-Q}
\end{equation}

We briefly recall the formal generalized kernel relations of the linearized
operator $i\calL_{Q}$. We have 
\[
N_{g}(i\calL_{Q})=\mathrm{span}_{\bbR}\{\Lambda Q,iQ,\tfrac{i}{4}y^{2}Q,\rho\}
\]
with the relations (see \cite[Proposition 3.4]{KimKwon2023MAMS})
\begin{equation}
\left\{ \begin{aligned}i\calL_{Q}(i\tfrac{y^{2}}{4}Q) & =\Lambda Q, & i\calL_{Q}\rho & =iQ,\\
i\calL_{Q}(\Lambda Q) & =0, & i\calL_{Q}(iQ) & =0,
\end{aligned}
\right.\label{eq:gen-kernel-rel}
\end{equation}
where the existence of $\rho$ is given in \cite[Lemma 3.6]{KimKwon2023MAMS}.
In fact, we have: 
\begin{equation}
\left\{ \begin{aligned}L_{Q}(i\tfrac{y^{2}}{4}Q) & =i\tfrac{y}{2}Q, & L_{Q}\rho & =\tfrac{1}{2(m+1)}yQ,\\
L_{Q}^{\ast}(i\tfrac{y}{2}Q) & =-i\Lambda Q, & L_{Q}^{\ast}(\tfrac{1}{2(m+1)}yQ) & =Q,\\
L_{Q}(\Lambda Q) & =0, & L_{Q}(iQ) & =0.
\end{aligned}
\right.\label{eq:gen-kernel-rel-LQ}
\end{equation}
One can uniquely determine $\rho$ by requiring $\rho\aleq y^{2}Q$
and $\rho$ can be explicitly constructed by the formula 
\begin{equation}
\rho\coloneqq\out L_{Q}^{-1}(\tfrac{1}{2(m+1)}yQ),\label{eq:def-rho}
\end{equation}
where $\out L_{Q}^{-1}$ denotes a formal right inverse of $L_{Q}$
whose explicit formula is given in Lemma~\ref{lem:LQ-inv}.

\subsubsection{\label{subsec:Linear-conjugation-identities}Linear conjugation identities
and Morawetz repulsivity}

Assuming $u=Q+\eps$, \eqref{eq:LinearizationBogomolnyi} says that
the first conjugated variable $u_{1}=\bfD_{u}u$ satisfies 
\[
u_{1}=\bfD_{u}u\approx L_{Q}\eps\eqqcolon\wh{\eps}_{1}.
\]
The $u_{1}$-equation \eqref{eq:u1-eqn} can be linearized as 
\begin{equation}
\rd_{t}\wh{\eps}_{1}+iA_{Q}^{\ast}A_{Q}\wh{\eps}_{1}\approx0.\label{eq:lin-e1-eqn}
\end{equation}
In this regard, the operator $H_{w}$ naturally arises: 
\begin{align}
H_{w} & \coloneqq A_{w}^{\ast}A_{w}=-\rd_{yy}-\tfrac{1}{y}\rd_{y}+\tfrac{1}{y^{2}}V_{w},\label{eq:H-factor}\\
V_{w} & \coloneqq(m+1+A_{\tht}[w])^{2}-\tfrac{1}{2}y^{2}|w|^{2}.\label{eq:V-def}
\end{align}
On the other hand, if we simply take $L_{Q}$ to \eqref{eq:lin-CSS-Q},
we also have 
\[
\rd_{t}\wh{\eps}_{1}+iL_{Q}iL_{Q}^{\ast}\wh{\eps}_{1}\approx0.
\]
Comparing this with \eqref{eq:lin-e1-eqn} gives the \emph{linear
conjugation identity} 
\[
L_{Q}iL_{Q}^{\ast}=iA_{Q}^{\ast}A_{Q}=iH_{Q}.
\]

The linear conjugation identity was first observed in \cite{KimKwon2023AnnPDE}.
While the linearized operator $\calL_{w}$ is only $\bbR$-linear
and is nonlocal, $H_{w}$ is now \emph{$\bbC$-linear} and \emph{local}.
Moreover, the kernels $\Lmb Q$ and $iQ$ of $i\calL_{Q}$ have been
removed by taking $L_{Q}$ to \eqref{eq:lin-CSS-Q} and the generalized
kernels $i\tfrac{y^{2}}{4}Q$ and $(m+1)\rho$ have become the \emph{kernels}
$i\tfrac{y}{2}Q$ and $\frac{y}{2}Q$ of $iH_{Q}$, thanks to \eqref{eq:gen-kernel-rel-LQ}
and $A_{Q}(yQ)=0$. More interestingly, $H_{Q}=A_{Q}^{\ast}A_{Q}$
is exactly the same operator arising in the linearization of several
critical geometric equations around harmonic maps within $k=(m+1)$-equivariance,
such as in \cite{RodnianskiSterbenz2010Ann.Math.,RaphaelRodnianski2012Publ.Math.,RaphaelSchweyer2013CPAM,MerleRaphaelRodnianski2013InventMath}.
Therefore, at the linearized level, the dynamics of $m$-equivariant
\eqref{eq:CSS-m-equiv} is closely related to the dynamics of $(m+1)$-equivariant
Schrödinger maps. These earlier works motivated the authors in \cite{KimKwon2023AnnPDE}
to consider further linearly conjugated variable $\wh{\eps}_{2}=A_{Q}\wh{\eps}_{1}=A_{Q}L_{Q}\eps$
below.

The further conjugated variable $u_{2}$ can be linearized as 
\[
u_{2}=A_{u}\bfD_{u}u\approx A_{Q}\wh{\eps}_{1}\eqqcolon\wh{\eps}_{2}
\]
and the linearization of the $u_{2}$-equation \eqref{eq:u2-eqn}
becomes 
\begin{equation}
\rd_{t}\wh{\eps}_{2}+iA_{Q}A_{Q}^{\ast}\wh{\eps}_{2}\approx0.\label{eq:lin-e2-eqn}
\end{equation}
In this regard, we define the linear operator $\td H_{w}$ by 
\begin{align}
\td H_{w} & \coloneqq A_{w}A_{w}^{\ast}=-\rd_{yy}-\tfrac{1}{y}\rd_{y}+\tfrac{1}{y^{2}}\td V_{w},\label{eq:H-td-factor}\\
\td V_{w} & \coloneqq(m+2+A_{\tht}[w])^{2}+\tfrac{1}{2}y^{2}|w|^{2}.\label{eq:V-td-def}
\end{align}
Note that the kernels $iyQ$ and $yQ$ of $iH_{Q}$ have been removed
by taking $A_{Q}$ to \eqref{eq:lin-e1-eqn} and the operator $\td H_{Q}$
has no kernels.

More importantly, as observed in \cite{RodnianskiSterbenz2010Ann.Math.}
for the wave maps (and in \cite{KimKwon2023AnnPDE} in the present
context), $\td H_{Q}$ has a \emph{repulsive} character: 
\begin{equation}
\td V_{Q}\geq m^{2}\qquad\text{and}\qquad-y\rd_{y}\td V_{Q}=y^{2}Q^{2}\geq0.\label{eq:H-td-Q-repul}
\end{equation}
One instance of the usefulness of this repulsivity is the following
spacetime control of $\wh{\eps}_{2}$ obtained through a virial-type
computation:
\begin{align*}
\tfrac{1}{2}\rd_{t}(\wh{\eps}_{2}, & -i\Lmb\wh{\eps}_{2})_{r}=(\wh{\eps}_{2},\tfrac{1}{2}[\td H_{Q},\Lmb]\wh{\eps}_{2})_{r}\\
 & =(\wh{\eps}_{2},\td H_{Q}\wh{\eps}_{2})_{r}+(\wh{\eps}_{2},\tfrac{-\rd_{y}\td V}{2y}\wh{\eps}_{2})_{r}\geq(\wh{\eps}_{2},\td H_{Q}\wh{\eps}_{2})_{r}=\|A_{Q}^{\ast}\wh{\eps}_{2}\|_{L^{2}}^{2}.
\end{align*}
As mentioned in the introduction, the repulsivity \eqref{eq:H-td-Q-repul}
has been used in many of the blow-up constructions. Note that the
virial term $(\wh{\eps}_{2},-i\Lmb\wh{\eps}_{2})_{r}$ above is not
bounded in the usual Sobolev spaces (e.g., for $\wh{\eps}_{2}\in H^{1}$).
In this work, we will use a \emph{Morawetz-type }monotonicity as follows.
\begin{lem}[Morawetz-type repulsivity]
\label{lem:Morawetz}Let $m\geq1$. There exists a smooth function
$\psi:(0,\infty)\to\bbR$ with the following properties:
\begin{itemize}
\item $\psi$ admits a smooth radially symmetric extension on $\bbR^{2}$.
\item (Bounds on $\psi'$) Denote $\psi'\coloneqq\rd_{y}\psi$. We have
$0\leq\psi'\leq1$ and 
\begin{equation}
y\rd_{y}\psi'\aleq\langle y\rangle^{-1}\label{eq:ydy-psi'}
\end{equation}
\item (Morawetz-type repulsivity) Recall the operator $\Lmb_{\psi}=\psi'\rd_{y}+\frac{\Delta\psi}{2}$.
There exists $\frkc>0$ such that for any smooth $(m+2)$-equivariant
function $f$, we have 
\begin{equation}
(\td H_{Q}f,\Lmb_{\psi}f)_{r}\geq\frkc\int\Big(\frac{1}{\langle y\rangle^{2}}|\rd_{y}f|^{2}+\frac{1}{\langle y\rangle}\Big|\frac{f}{y}\Big|^{2}\Big)\eqqcolon\frkc\|f\|_{\dot{V}_{m+2}^{3/2}}^{2}.\label{eq:Morawetz-repul}
\end{equation}
\end{itemize}
\end{lem}

\begin{rem}[On $m\geq1$]
\label{rem:Morawetz}In the proof, we heavily rely on the assumption
$m\geq1$ to obtain the \emph{nonnegativity} of $(\td H_{Q}f,\Lmb_{\psi}f)_{r}$.
When $m=0$, $\td H_{Q}$ is rather similar to the 2D Laplacian $-\Delta_{0}=-\rd_{yy}-\frac{1}{y}\rd_{y}$
near the spatial infinity, and hence there are some subtleties to
obtain a Morawetz estimate due to the zero resonance of $\Delta_{0}$.
Note that the blow-up construction in \cite{KimKwonOh2020arXiv} for
the $m=0$ case (and similarly in \cite{RaphaelRodnianski2012Publ.Math.,MerleRaphaelRodnianski2013InventMath}
for critical wave maps and Schrödinger maps) rather constructs precise
correction terms for the modified energy under the shrinking regime
$\rd_{t}\lmb(t)<0$, which we do not assume in the present context.
\end{rem}

\begin{rem}
The bound \eqref{eq:ydy-psi'} will be used in the energy estimates
of our blow-up analysis (Proposition~\ref{prop:EnergyMorawetz}).
As is well-known, if one only considers a Morawetz estimate for the
linear equation \eqref{eq:lin-e2-eqn} without requiring the bound
\eqref{eq:ydy-psi'}, then the weight $\langle y\rangle^{-2}$ of
\eqref{eq:Morawetz-repul} can be improved to $\langle y\rangle^{-1+}$
by considering for example $\psi(y)=\langle y\rangle-\delta\langle y\rangle^{1-\delta}$
with small $\delta>0$.
\end{rem}

\begin{proof}
We consider $\psi$ of the form 
\begin{equation}
\psi'=\frac{y}{\langle y\rangle}-\delta\frac{y}{\langle y\rangle^{2}}\eqqcolon\wh{\psi}'+\delta\td{\psi}'\label{eq:repul1}
\end{equation}
for some small $0<\delta<1$ to be chosen below. We may explicitly
set $\psi(y)=\langle y\rangle-\frac{\delta}{2}\log(1+y^{2})$ and
hence $\psi$ admits a smooth radially symmetric extension on $\bbR^{2}$.
Moreover, it is clear from \eqref{eq:repul1} that $\psi'$ satisfies
$0\leq\psi'\leq1$ and $y\rd_{y}\psi'\aleq\langle y\rangle^{-1}$.

We turn to the proof of monotonicity \eqref{eq:Morawetz-repul}. Integrating
by parts, one has 
\[
(\td H_{Q}f,\Lmb_{\psi}f)_{r}=\int\psi''|\rd_{y}f|^{2}+\int\Big\{\frac{\psi'}{y}\Big(\td V_{Q}-\frac{y\rd_{y}\td V_{Q}}{2}\Big)-\frac{y^{2}\Delta^{2}\psi}{4}\Big\}\Big|\frac{f}{y}\Big|^{2}.
\]
Applying $\psi'\geq0$ and \eqref{eq:H-td-Q-repul}, we have 
\begin{equation}
(\td H_{Q}f,\Lmb_{\psi}f)_{r}\geq\int\psi''|\rd_{y}f|^{2}+\int\Big\{\frac{\psi'}{y}-\frac{y^{2}\Delta^{2}\psi}{4}\Big\}\Big|\frac{f}{y}\Big|^{2}.\label{eq:repul2}
\end{equation}
On one hand, we have 
\[
\wh{\psi}''=\frac{1}{\langle y\rangle^{3}}\quad\text{and}\quad\frac{\wh{\psi}'}{y}-\frac{y^{2}\Delta^{2}\wh{\psi}}{4}=\frac{3y^{6}+4y^{4}+20y^{2}+4}{4\langle y\rangle^{7}}\ageq\frac{1}{\langle y\rangle}.
\]
On the other hand, we have 
\[
\Big|\td{\psi}''-\frac{1}{\langle y\rangle^{2}}\Big|\aleq\frac{1}{\langle y\rangle^{4}}\quad\text{and}\quad\Big|\frac{\td{\psi}'}{y}-\frac{y^{2}\Delta^{2}\td{\psi}}{4}\Big|\aleq\frac{1}{\langle y\rangle^{2}}.
\]
Therefore, there exists some $0<\delta<1$ such that $\psi'=\wh{\psi}'+\delta\td{\psi}'$
satisfies 
\begin{equation}
\psi''\ageq\frac{1}{\langle y\rangle^{2}}\quad\text{and}\quad\frac{\psi'}{y}-\frac{y^{2}\Delta^{2}\psi}{4}\ageq\frac{1}{\langle y\rangle}.\label{eq:repul3}
\end{equation}
Substituting \eqref{eq:repul3} into \eqref{eq:repul2} gives \eqref{eq:Morawetz-repul}.
\end{proof}

\subsection{\label{subsec:Adapted-function-spaces}Adapted function spaces}

In this subsection, we quickly recall the definition of equivariant
Sobolev spaces and introduce adapted function spaces. For more explanations,
see \cite{KimKwon2023MAMS,KimKwon2023AnnPDE}.

Let $m\in\bbZ$. For $s\geq0$, we denote by $H_{m}^{s}$ and $\dot{H}_{m}^{s}$
the restriction of the usual Sobolev spaces $H^{s}(\bbR^{2})$ and
$\dot{H}^{s}(\bbR^{2})$ on $m$-equivariant functions. For high equivariance
indices, we have the generalized Hardy's inequality \cite[Lemma A.7]{KimKwon2023MAMS}:
whenever $0\leq k\leq|m|$, we have 
\begin{equation}
\||f|_{-k}\|_{L^{2}}\sim\|f\|_{\dot{H}_{m}^{k}}.\label{eq:GenHardySection2}
\end{equation}
Specializing this to $k=1$, we have the \emph{Hardy-Sobolev inequality}
\cite[Lemma A.6]{KimKwon2023MAMS}: whenever $|m|\geq1$, we have
\begin{equation}
\|r^{-1}f\|_{L^{2}}+\|f\|_{L^{\infty}}\aleq\|f\|_{\dot{H}_{m}^{1}}.\label{eq:HardySobolevSection2}
\end{equation}
Note in general that $H_{0}^{1}\hookrightarrow L^{\infty}$ is \emph{false}.
For $k\in\bbN$, the space $H_{m}^{k,k}$ is defined by the set of
$m$-equivariant $H^{k,k}(\bbR^{2})$ functions, where $f\in H^{k,k}(\bbR^{2})$
means that $\langle x\rangle^{\ell}f\in H^{k-\ell}(\bbR^{2})$ for
all $0\leq\ell\leq k$.

Let $m\geq1$. We turn to the definitions of adapted function spaces.
As mentioned before, we will need a priori bounds (such as \eqref{eq:size-control-1})
from mass/energy conservation and perform an energy estimate for $\eps_{2}$
to prove the $H^{3}$-boundedness of $\eps^{\sharp}$. For this purpose,
we need to transfer the controls on $\eps_{2}\approx\wh{\eps}_{2}=A_{Q}L_{Q}\eps$
and $\eps_{1}\approx\wh{\eps}_{1}=L_{Q}\eps$ to those for $\eps$.
The controls on $\eps$ and hence the norms of the adapted function
spaces will be captured by various (sub-)coercivity estimates for
the linear operators $A_{Q}$ and $L_{Q}$. Note that, as $A_{Q}$
and $L_{Q}$ have nontrivial kernels $\mathrm{span}_{\bbR}\{yQ,iyQ\}$
and $\mathrm{span}_{\bbR}\{\Lmb Q,iQ\}$, respectively, the coercivity
estimates (Lemma~\ref{lem:LinearCoercivity} below) hold after excluding
these kernel elements.

The following norms will be used in this paper. Recall from \eqref{eq:cov-conj-identification}
that $\eps,\eps_{1},\eps_{2}$ are naturally identified as $m$-,
$(m+1)$-, $(m+2)$-equivariant functions. For $\eps_{2}$, we will
use the norms 
\[
\|\eps_{2}\|_{L^{2}},\quad\|\eps_{2}\|_{\dot{H}_{m+2}^{1}},\quad\text{and}\quad\|\eps_{2}\|_{\dot{V}_{m+2}^{3/2}},
\]
where we recall the definition of the $\dot{V}_{m+2}^{3/2}$-norm
(see \eqref{eq:Morawetz-repul}): 
\begin{equation}
\|\eps_{2}\|_{\dot{V}_{m+2}^{3/2}}^{2}\coloneqq\Big\|\frac{\rd_{y}\eps_{2}}{\langle y\rangle}\Big\|_{L^{2}}^{2}+\Big\|\frac{\eps_{2}}{\langle y\rangle^{1/2}y}\Big\|_{L^{2}}^{2}\label{eq:def-V3/2}
\end{equation}
The first two norms are adapted to the coercivity of the energy functional
$\calF$ introduced in \eqref{eq:str-def-H3en}. The last norm captures
the spatial part of the Morawetz estimate in Lemma~\ref{lem:Morawetz}.

For $\eps_{1}$, we will use the norms 
\[
\|\eps_{1}\|_{L^{2}},\quad\|\eps_{1}\|_{\dot{H}_{m+1}^{1}},\quad\|\eps_{1}\|_{\dot{H}_{m+1}^{2}},\quad\text{and}\quad\|\eps_{1}\|_{\dot{V}_{m+1}^{5/2}},
\]
where the last $\dot{V}_{m+1}^{5/2}$-norm is defined by 
\begin{equation}
\|\eps_{1}\|_{\dot{V}_{m+1}^{5/2}}\coloneqq\Big\|\frac{\rd_{yy}\eps_{1}}{\langle y\rangle}\Big\|_{L^{2}}^{2}+\Big\|\frac{|\eps_{1}|_{-1}}{\langle y\rangle^{1/2}y}\Big\|_{L^{2}}^{2}\label{eq:def-V5/2}
\end{equation}
The first norm, i.e., $\|\eps_{1}\|_{L^{2}}$, is adapted to energy
conservation (see \eqref{eq:linearized-energy-expn}). The remaining
high Sobolev norms are adapted to the the $L^{2}$, $\dot{H}_{m+2}^{1}$,
and $\dot{V}_{m+2}^{3/2}$ controls on $\eps_{2}\approx A_{Q}\eps_{1}$;
see the second item of Lemma~\ref{lem:LinearCoercivity} below.

Finally, for $\eps$, we will use the norms 
\[
\|\eps\|_{L^{2}},\quad\|\eps\|_{\dot{H}_{m}^{1}},\quad\|\eps\|_{\dot{\calH}_{m}^{2}},\quad\text{and}\quad\|\eps\|_{\dot{\calH}_{m}^{3}},
\]
where the $\dot{\calH}_{m}^{2}$- and $\dot{\calH}_{m}^{3}$-norms
are defined by 
\begin{equation}
\|\eps\|_{\dot{\calH}_{m}^{2}}\coloneqq\begin{cases}
\||\eps|_{-2}\|_{L^{2}}\sim\|\eps\|_{\dot{H}_{m}^{2}} & \text{if }m\geq2,\\
{\displaystyle \|\rd_{+}\eps\|_{\dot{H}_{2}^{1}}+\|\rd_{yy}\eps\|_{L^{2}}+\Big\|\frac{|\eps|_{-1}}{\langle\log_{-}y\rangle y}\Big\|_{L^{2}}} & \text{if }m=1,
\end{cases}\label{eq:def-calH2}
\end{equation}
and 
\begin{equation}
\|\eps\|_{\dot{\calH}_{m}^{3}}\coloneqq\begin{cases}
\||\eps|_{-3}\|_{L^{2}}\sim\|\eps\|_{\dot{H}_{m}^{3}} & \text{if }m\geq3,\\
{\displaystyle \|\rd_{+}\eps\|_{\dot{H}_{3}^{2}}+\|\rd_{yyy}\eps\|_{L^{2}}+\Big\|\frac{|\eps|_{-2}}{\langle\log_{-}y\rangle y}\Big\|_{L^{2}}} & \text{if }m=2,\\
{\displaystyle \|\rd_{+}\eps\|_{\dot{H}_{2}^{2}}+\||\rd_{yy}\eps|_{-1}\|_{L^{2}}+\Big\|\frac{|\eps|_{-1}}{\langle\log_{-}y\rangle\langle y\rangle y}\Big\|_{L^{2}}} & \text{if }m=1.
\end{cases}\label{eq:def-calH3}
\end{equation}
Of course, the $L^{2}$-norm of $\eps$ is adapted to  mass conservation.
The remaining higher order norms are adapted to the the $L^{2}$,
$\dot{H}_{m+1}^{1}$, and $\dot{H}_{m+1}^{2}$ controls on $\eps_{1}\approx L_{Q}\eps$;
see the last item of Lemma~\ref{lem:LinearCoercivity} below. When
$m$ is small, the above $\dot{\calH}_{m}^{k}$-norms can be viewed
as slight modifications of the usual $\dot{H}_{m}^{k}$-norms. However,
these $\dot{\calH}_{m}^{k}$-norms control the scaling critical Hardy
norms in the region $y\geq1$: 
\[
\|\eps\|_{\dot{\calH}_{m}^{2}}\ageq\|\chf_{y\geq1}|\eps|_{-2}\|_{L^{2}}\qquad\text{and}\qquad\|\eps\|_{\dot{\calH}_{m}^{3}}\ageq\|\chf_{y\geq1}|\eps|_{-3}\|_{L^{2}}.
\]
In this sense, the $\dot{\calH}_{m}^{k}$-norms are stronger than
the $\dot{H}_{m}^{k}$-norms when $m$ is small.

The following lemma collects the coercivity estimates of $L_{Q}$
and $A_{Q}$.
\begin{lem}[Linear coercivity estimates]
\label{lem:LinearCoercivity}Let $m\geq1$.
\begin{enumerate}
\item (Unconditional coercivity of $\td H_{Q}$) We have 
\begin{equation}
(\eps_{2},\td H_{Q}\eps_{2})_{r}\geq\|\rd_{y}\eps_{2}\|_{L^{2}}^{2}+\|\tfrac{1}{y}\eps_{2}\|_{L^{2}}^{2}.\label{eq:H-td-coer}
\end{equation}
\item (Coercivity of $A_{Q}$) Suppose $\td{\calZ}_{3},\td{\calZ}_{4}\in C_{c,m+1}^{\infty}$\footnote{That is, $\td{\calZ}_{3}$ and $\td{\calZ}_{4}$ have smooth compactly
supported $(m+1)$-equivariant extensions on $\bbR^{2}$.} are such that the $2\times2$ matrix $\{(a_{kj})\}_{k,j\in\{3,4\}}$
defined by $a_{k3}=(\td{\calZ}_{k},iyQ)_{r}$ and $a_{k4}=(\td{\calZ}_{k},yQ)_{r}$
has nonzero determinant. Then, we have coercivity estimates 
\begin{align}
\|\eps_{1}\|_{\dot{H}_{m+1}^{1}} & \aleq_{\td{\calZ}_{3},\td{\calZ}_{4}}\|A_{Q}\eps_{1}\|_{L^{2}}\aleq\|\eps_{1}\|_{\dot{H}_{m+1}^{1}}, & \forall\eps_{1}\in\dot{H}_{m+1}^{1}\cap\{\td{\calZ}_{3},\td{\calZ}_{4}\}^{\perp},\label{eq:AQ-coer-L2-to-H1}\\
\|\eps_{1}\|_{\dot{H}_{m+1}^{2}} & \aleq_{\td{\calZ}_{3},\td{\calZ}_{4}}\|A_{Q}\eps_{1}\|_{\dot{H}_{m+2}^{1}}\aleq\|\eps_{1}\|_{\dot{H}_{m+1}^{2}}, & \forall\eps_{1}\in\dot{H}_{m+1}^{2}\cap\{\td{\calZ}_{3},\td{\calZ}_{4}\}^{\perp},\label{eq:AQ-coer-H1-to-H2}\\
\|\eps_{1}\|_{\dot{V}_{m+1}^{5/2}} & \aleq_{\td{\calZ}_{3},\td{\calZ}_{4}}\|A_{Q}\eps_{1}\|_{\dot{V}_{m+2}^{3/2}}\aleq\|\eps_{1}\|_{\dot{V}_{m+1}^{5/2}}, & \forall\eps_{1}\in\dot{V}_{m+1}^{5/2}\cap\{\td{\calZ}_{3},\td{\calZ}_{4}\}^{\perp},\label{eq:AQ-coer-V}
\end{align}
where $\perp$ is taken with respect to the real inner product $(\cdot,\cdot)_{r}$.
\item (Coercivity of $L_{Q}$) Suppose $\calZ_{1},\calZ_{2}\in C_{c,m}^{\infty}$
are such that the $2\times2$ matrix $\{(a_{kj})\}_{k,j\in\{1,2\}}$
defined by $a_{k1}=(\calZ_{k},\Lambda Q)_{r}$ and $a_{k2}=(\calZ_{k},iQ)_{r}$
has nonzero determinant. Then, we have coercivity estimates 
\begin{align}
\|\eps\|_{\dot{H}_{m}^{1}} & \aleq_{\calZ_{1},\calZ_{2}}\|L_{Q}\eps\|_{L^{2}}\aleq\|\eps\|_{\dot{H}_{m}^{1}}, & \forall\eps\in\dot{H}_{m}^{1}\cap\{\calZ_{1},\calZ_{2}\}^{\perp},\label{eq:LQ-coer-L2-to-H1}\\
\|\eps\|_{\dot{\calH}_{m}^{2}} & \aleq_{\calZ_{1},\calZ_{2}}\|L_{Q}\eps\|_{\dot{H}_{m+1}^{1}}\aleq\|\eps\|_{\dot{\calH}_{m}^{2}}, & \forall\eps\in\dot{\calH}_{m}^{2}\cap\{\calZ_{1},\calZ_{2}\}^{\perp},\label{eq:LQ-coer-H1-to-H2}\\
\|\eps\|_{\dot{\calH}_{m}^{3}} & \aleq_{\calZ_{1},\calZ_{2}}\|L_{Q}\eps\|_{\dot{H}_{m+1}^{2}}\aleq\|\eps\|_{\dot{\calH}_{m}^{3}}, & \forall\eps\in\dot{\calH}_{m}^{3}\cap\{\calZ_{1},\calZ_{2}\}^{\perp},\label{eq:LQ-coer-H2-to-H3}
\end{align}
where $\perp$ is taken with respect to the real inner product $(\cdot,\cdot)_{r}$.
\end{enumerate}
\end{lem}

See Appendix~\ref{sec:Linear-coercivity-estimates}. Some of these
estimates are already proved in \cite{KimKwon2023AnnPDE}.
\begin{rem}[Choice of $\calZ_{1},\calZ_{2},\td{\calZ}_{3},\td{\calZ}_{4}$]
\label{rem:Ortho-choice}In this work, we fix the orthogonality profiles
$\calZ_{1},\calZ_{2}\in C_{c,m}^{\infty}$ and $\td{\calZ}_{3},\td{\calZ}_{4}\in C_{c,m+1}^{\infty}$
satisfying the following properties. We require for $\calZ_{1},\calZ_{2}\in C_{c,m}^{\infty}$
that $\calZ_{1}$ is real, $\calZ_{2}$ is purely imaginary, and $\calZ_{1},\calZ_{2}$
satisfy the transversality condition 
\begin{equation}
(\Lambda Q,\calZ_{1})_{r},\ (iQ,\calZ_{2})_{r}\neq0\label{eq:Z1Z2-transv}
\end{equation}
and a gauge condition\emph{ }
\begin{equation}
(-i\tfrac{y^{2}}{4}Q,\calZ_{k})_{r}=(\rho,\calZ_{k})_{r}=0,\qquad\forall k\in\{1,2\}.\label{eq:Z1Z2-gauge-condition}
\end{equation}
Next, we fix $\td{\calZ}_{3},\td{\calZ}_{4}\in C_{c,m+1}^{\infty}$
such that 
\begin{equation}
\td{\calZ}_{3}\coloneqq-i\tfrac{1}{2}yQ\chi_{R_{\circ}}\qquad\text{and}\qquad\td{\calZ}_{4}\coloneqq-\tfrac{1}{2}yQ\chi_{R_{\circ}},\label{eq:def-td-Z3Z4}
\end{equation}
where $R_{\circ}>1$ is some radius such that 
\begin{equation}
\|\chf_{[R_{\circ},\infty)}yQ\|_{L^{2}}\leq\tfrac{1}{2}\|yQ\|_{L^{2}}\label{eq:def-R-circ}
\end{equation}
and the supports of $\calZ_{1}$ and $\calZ_{2}$ are contained in
$y\leq2R_{\circ}$. It is clear that Lemma~\ref{lem:LinearCoercivity}
is applicable with this choice of orthogonality profiles.
\end{rem}

\subsection{Weighted $L^{\infty}$ and integral estimates}

In this subsection, we record weighted $L^{\infty}$ and integral
estimates, which will be very useful in our analysis.
\begin{lem}[Weighted $L^{\infty}$ estimates]
\label{lem:weighted-Linfty-est}For $m$-equivariant functions $\eps$,
we have 
\begin{equation}
\|\eps\|_{L^{\infty}}\aleq\|\eps\|_{\dot{H}_{m}^{1}},\quad\|\langle y\rangle^{-1}\eps\|_{L^{\infty}}\aleq\|\eps\|_{\dot{\calH}_{m}^{2}},\quad\|\langle y\rangle^{-2}\eps\|_{L^{\infty}}\aleq\|\eps\|_{\dot{\calH}_{m}^{3}}.\label{eq:eps-Linfty}
\end{equation}
For $(m+1)$-equivariant functions $\eps_{1}$, we have 
\begin{equation}
\begin{aligned} & \|\eps_{1}\|_{L^{\infty}}\aleq\|\eps_{1}\|_{\dot{H}_{m+1}^{1}},\qquad\||\eps_{1}|_{-1}\|_{L^{\infty}}\aleq\|\eps_{1}\|_{\dot{H}_{m+1}^{2}},\\
 & \|\langle y\rangle^{-3/4}\rd_{y}\eps_{1}\|_{L^{\infty}}+\|\langle y\rangle^{-3/2}\tfrac{1}{y}\eps_{1}\|_{L^{\infty}}\aleq\|\eps_{1}\|_{\dot{V}_{m+1}^{5/2}}.
\end{aligned}
\label{eq:eps1-Linfty}
\end{equation}
For $(m+2)$-equivariant functions $\eps_{2}$, we have 
\begin{equation}
\|\eps_{2}\|_{L^{\infty}}\aleq\|\eps_{2}\|_{\dot{H}_{m+2}^{1}},\qquad\|\langle y\rangle^{-3/4}\eps_{2}\|_{L^{\infty}}\aleq\|\eps_{2}\|_{\dot{V}_{m+1}^{3/2}}.\label{eq:eps2-Linfty}
\end{equation}
\end{lem}

\begin{rem}
The controls in the region $y\aleq1$ are not sharp. It is also possible
to control higher derivatives of $\eps$. See for example \cite[Appendix A]{KimKwon2023AnnPDE}.
We will not need such sharp controls in this paper.
\end{rem}

\begin{proof}
All the estimates easily follow from the FTC argument. We only prove
the last inequality of \eqref{eq:eps1-Linfty}, which is the most
complicated, and omit the remaining proof. Integrating the inequalities
\begin{align*}
\rd_{y}(\langle y\rangle^{-3}y^{-2}|\eps_{1}|^{2}) & \aleq\langle y\rangle^{-3}y^{-2}\cdot|\eps_{1}|_{-1}|\eps_{1}|\qquad\text{and}\\
\rd_{y}(\langle y\rangle^{-3/2}|\rd_{y}\eps_{1}|^{2}) & \aleq\langle y\rangle^{-3/2}\cdot|\rd_{y}\eps_{1}|_{-1}|\rd_{y}\eps_{1}|
\end{align*}
backwards from the spatial infinity and using Cauchy--Schwarz, we
get 
\begin{align*}
 & \langle y\rangle^{-3}y^{-2}|\eps_{1}|^{2}\aleq\|\langle y\rangle^{-3/2}y^{-1}\cdot|\eps_{1}|_{-1}\|_{L^{2}(ydy)}\|\langle y\rangle^{-3/2}y^{-1}\cdot\tfrac{1}{y}\eps_{1}\|_{L^{2}(ydy)}\aleq\|\eps_{1}\|_{\dot{V}_{m+1}^{5/2}}^{2},\\
 & \langle y\rangle^{-3/2}|\rd_{y}\eps_{1}|^{2}\aleq\|\langle y\rangle^{-1}\cdot|\rd_{y}\eps_{1}|_{-1}\|_{L^{2}(ydy)}\|\langle y\rangle^{-1/2}\cdot\tfrac{1}{y}\rd_{y}\eps_{1}\|_{L^{2}(ydy)}\aleq\|\eps_{1}\|_{\dot{V}_{m+1}^{5/2}}^{2},
\end{align*}
which complete the proof of the last inequality of \eqref{eq:eps1-Linfty}.
\end{proof}
\begin{lem}[Mapping properties for integral operators \cite{KimKwonOh2022arXiv1}]
Let $1\leq p,q\leq\infty$ and $s\in[0,2]$ be such that $(p,q,s)=(1,\infty,0)$
or $\frac{1}{q}+1=\frac{1}{p}+\frac{s}{2}$ with $p>1$. Then, we
have 
\begin{equation}
\Big\|\frac{1}{y^{s}}\int_{0}^{y}f(y')y'dy'\Big\|_{L^{q}}\aleq_{p}\|f\|_{L^{p}}.\label{eq:integral-operator-bdd1}
\end{equation}
Denoting by $p'$ and $q'$ the Hölder conjugates of $p$ and $q$,
respectively, we have the dual estimate 
\begin{equation}
\Big\|\int_{y}^{\infty}g(y')(y')^{1-s}dy'\Big\|_{L^{p'}}\aleq_{p}\|f\|_{L^{q'}}.\label{eq:integral-operator-bdd2}
\end{equation}
\end{lem}

\begin{rem}
We will often use the above estimates when $p,q\in\{1,2,\infty\}$.
\end{rem}

\begin{rem}
If a nonnegative weight $w=w(y)$ is decreasing, then one can have
a weighted version of \eqref{eq:integral-operator-bdd1}. For example,
we can have 
\[
\Big\|\frac{w(y)}{y^{s}}\int_{0}^{y}f(y')\,y'dy'\Big\|_{L^{q}}\leq\Big\|\frac{1}{y^{s}}\int_{0}^{y}w(y')|f(y')|\,y'dy'\Big\|_{L^{q}}\aleq_{p}\|wf\|_{L^{p}}.
\]
Similar argument works for \eqref{eq:integral-operator-bdd2} if $w$
is increasing. We will implicitly use this simple argument in many
of the estimates.
\end{rem}

\subsection{Proximity to soliton for small energy profiles}

We recall a variational lemma from \cite{KimKwonOh2022arXiv1}, which
provides the starting point of our modulation analysis.
\begin{lem}[{Proximity to $Q$ for small energy profiles \cite[Lemma 4.2]{KimKwonOh2022arXiv1}}]
\label{lem:orbital-stability-small-energy}For any $M>0$ and $\alpha^{\ast}>0$,
there exists $\delta>0$ such that\footnote{Compared to \cite[Lemma 4.2]{KimKwonOh2022arXiv1}, we switched the
roles of $\alpha^{\ast}$ and $\delta$. This is because $\alpha^{\ast}$
in this paper will be used as a small parameter measuring the $\dot{H}_{m}^{1}$-distance
between the renormalized solution and the soliton $Q$.} the following holds. Let $u\in H_{m}^{1}$ be a nonzero profile satisfying
$\|u\|_{L^{2}}\leq M$ and the small energy condition $\sqrt{E[u]}\leq\delta\|u\|_{\dot{H}_{m}^{1}}$.
Then, there exists $\wh{\gmm}\in\bbR/2\pi\bbZ$ such that 
\begin{equation}
\|e^{-i\wh{\gmm}}\wh{\lmb}u(\wh{\lmb}\cdot)-Q\|_{\dot{H}_{m}^{1}}<\alpha^{\ast},\label{eq:orbital-small-en}
\end{equation}
where $\wh{\lmb}=\|Q\|_{\dot{H}_{m}^{1}}/\|u\|_{\dot{H}_{m}^{1}}$.
\end{lem}

\begin{rem}
Notice that we allow profiles having arbitrarily large mass and hence
the use of the $L^{2}$-norm is prohibited in \eqref{eq:orbital-small-en}.
The above lemma is a consequence of the self-duality \eqref{eq:energy-self-dual-form}.
More precisely, it is a consequence of the nonnegativity of energy
and the uniqueness of solutions to the Bogomol'nyi equation \eqref{eq:Bogomol'nyi-eq}
up to scaling and phase rotation invariances. For the mass-critical
NLS, the analogous statement only holds under additional restriction
on mass such as $\|\psi\|_{L^{2}}<\|R\|_{L^{2}}+\delta$ for the ground
state $R$ and in that case the $L^{2}$-distance can also be made
small.
\end{rem}

\begin{rem}
\label{rem:blow-up-vicinity}Any $H_{m}^{1}$ finite-time blow-up
solutions to \eqref{eq:CSS-m-equiv} eventually enters the vicinity
of modulated solitons. Indeed, if $u(t)$ is a $H_{m}^{1}$-solution
that blows up at $T<+\infty$, then $\|u(t)\|_{\dot{H}_{m}^{1}}\to\infty$
as $t\to T$ by the standard Cauchy theory. Therefore, for any $\delta>0$
the small energy condition $\sqrt{E[u]}\leq\delta\|u\|_{\dot{H}_{m}^{1}}$
eventually holds by energy conservation. Note that $\|u(t)\|_{L^{2}}$
remains bounded by mass conservation.
\end{rem}

\section{\label{sec:Modified-profiles}Modified profiles}

For the modulation analysis in the next section, we will decompose
the full solution $u(t)$ of the form 
\[
u(t,r)=\frac{e^{i\gmm(t)}}{\lmb(t)}[P(\cdot;b(t),\eta(t))+\eps(t,\cdot)]\Big(\frac{r}{\lmb(t)}\Big).
\]
(In fact, we decompose $u,u_{1},u_{2}$ as in \eqref{eq:str-decom}).
In this section, we aim to construct the profiles $P(\cdot;b,\eta)$,
which we call the \emph{modified profiles}, and derive formal modulation
equations for $b$ and $\eta$. Note that the four parameters $\lmb,\gmm,b,\eta$
correspond to the basis $\{\Lmb Q,iQ,-i\frac{y^{2}}{4}Q,\rho\}$ of
the formal generalized kernel of $i\calL_{Q}$. $P$ will be chosen
such that $P\approx Q$, $\rd_{b}P\approx-i\frac{y^{2}}{4}Q$, and
$\rd_{\eta}P\approx-(m+1)\rho$.

Basically, we will perform the tail computations in the spirit of
\cite{RaphaelRodnianski2012Publ.Math.,MerleRaphaelRodnianski2013InventMath,MerleRaphaelRodnianski2015CambJMath}
(and \cite{KimKwonOh2020arXiv} for \eqref{eq:CSS-m-equiv} when $m=0$).
We assume the \emph{adiabatic ansatz} 
\[
\frac{\lmb_{s}}{\lmb}+b=0\qquad\text{and}\qquad\gmm_{s}\approx(m+1)\eta
\]
and assume the expansions of $b_{s}$, $\eta_{s}$, and $P$, which
roughly take the form 
\begin{align*}
b_{s} & =\tsum{k+\ell\geq2}{}c_{k,\ell}^{(b)}b^{k}\eta^{\ell},\\
\eta_{s} & =\tsum{k+\ell\geq2}{}c_{k,\ell}^{(\eta)}b^{k}\eta^{\ell},\\
P & =Q-ib\tfrac{y^{2}}{4}Q-(m+1)\eta\rho+\tsum{k+\ell\geq2}{}b^{k}\eta^{\ell}T_{k,\ell}
\end{align*}
with some scalars $c_{k,\ell}^{(b)}$, $c_{k,\ell}^{(\eta)}$, and
profiles $T_{k,\ell}=T_{k,\ell}(y)$. By the tail computations, we
seek for the choices of $c_{k,\ell}^{(b)}$ and $c_{k,\ell}^{(\eta)}$
such that $T_{k,\ell}$ grows the slowest.

We mention some observations in the profile construction of this work.
\begin{itemize}
\item In practice, we need to truncate the profiles $P$. In this work,
we will truncate the profiles at the scale $y=B_{1}=(b^{2}+\eta^{2})^{-1/2}$,
which is much further than the self-similar scale $y=b^{-1/2}$, and
\emph{this scale turns out to be the only scale that our analysis
works for $m\in\{1,2\}$}. This can be  explained as follows. As mentioned
in Section~\ref{subsec:Strat-Proof}, our basic strategy for the
rigidity theorem is to prove that the remainder part $\eps^{\sharp}$
remains regular up to the blow-up time, or more specifically $\eps^{\sharp}$
remains bounded in the $H^{3}$-topology. Thus the blow-up part $P^{\sharp}$
should contain the singular part of the solution $u$, which turns
out to be nontrivial when $m\in\{1,2\}$ by Theorem~\ref{thm:blow-up-rigidity}.
To recognize this from the blow-up part $P^{\sharp}$, the cutoff
scale must be $r\sim1$, which corresponds to $y\sim B_{1}$ in the
modulation dynamics.
\item When $m=1$, there are nontrivial \emph{cubic corrections} to the
modulation equations of $b$ and $\eta$. The profile ansatz from
\cite{KimKwon2023AnnPDE} as well as the modulation equations $b_{s}+b^{2}+\eta^{2}=0$
and $\eta_{s}=0$ are no longer correct; the tail computations suggest
$b_{s}+b^{2}+\eta^{2}+p_{3}^{(b)}=0$ and $\eta_{s}+p_{3}^{(\eta)}=0$
instead, for some nonzero cubic polynomials $p_{3}^{(b)}$ and $p_{3}^{(\eta)}$.
We believe in general that there are $(m+2)$-order corrections, which
are tied to the spatial decay rate of $Q$ as in \cite{RaphaelRodnianski2012Publ.Math.},
when $m\geq1$.
\end{itemize}
Additionally, we will adopt the framework of \cite{KimKwonOh2020arXiv}.
Not only do we look at the variable $w$, we also look at the variables
$w_{1}$ and $w_{2}$; thus in this section we will construct the
profiles $P,P_{1},P_{2}$ adapted to the variables $w,w_{1},w_{2}$.
This strategy well separates the roles of each evolution equation.
An important advantage in the profile construction is that the derivation
of $b_{s}$- and $\eta_{s}$-modulation equations become much more
direct at the level of the $w_{1}$-equation. Another advantage is
that $P$- and $P_{1}$-ansatz do not need to be as precise as the
$P_{2}$-ansatz; it suffices to construct $P$, $P_{1}$, and $P_{2}$,
up to quadratic, cubic, and quartic orders, respectively. At the same
time, $P_{1}$ and $P_{2}$ enjoy degeneracies; the $P_{1}$- and
$P_{2}$-expansions start from the linear and quadratic terms, respectively.

\subsection{\label{subsec:Old-profile-ansatz}Old profile ansatz}

In this subsection, we discuss the modified profiles as well as the
modulation dynamics observed in \cite{KimKwon2023MAMS,KimKwon2023AnnPDE},
from which pseudoconformal blow-up and rotational instability manifest.

In \cite{KimKwon2023MAMS,KimKwon2023AnnPDE}, the authors studied
the dynamics near pseudoconformal blow-up solutions. They used the
profiles of the form
\begin{equation}
P_{\mathrm{old}}(y;b,\eta)=\chi_{B}(y)e^{-ib\frac{y^{2}}{4}}Q^{(\eta)}(y),\label{eq:old-P-1}
\end{equation}
where $B=\infty$ (i.e., no cutoff) and $\eta\geq0$ in \cite{KimKwon2023MAMS}
and $0<b\ll1$, $|\eta|\ll b$, and $B=b^{-1/2}$ in \cite{KimKwon2023AnnPDE}.
The profile $Q^{(\eta)}(y)$ is constructed in a way that it solves
the \emph{modified Bogomol'nyi equation} \cite{KimKwon2023MAMS}:
\[
\bfD_{Q^{(\eta)}}Q^{(\eta)}=-\eta\tfrac{y}{2}Q^{(\eta)}
\]
on $y\in(0,\infty)$ if $\eta\geq0$ and $0<y\ll|\eta|^{-1/2}$ if
$\eta<0$. This nonlinear construction of the profile effectively
avoided the difficulties from nonlocal nonlinearities and yielded
the formal modulation equations 
\begin{equation}
\frac{\lmb_{s}}{\lmb}+b=0,\quad\gmm_{s}\approx(m+1)\eta,\quad b_{s}+b^{2}+\eta^{2}=0,\quad\eta_{s}=0.\label{eq:old-mod-eqn}
\end{equation}
Typical solutions to \eqref{eq:old-mod-eqn}, parametrized by $\eta_{0}\in\bbR$,
are given by 
\begin{equation}
\begin{gathered}b(t)=-t,\quad\eta(t)=\eta_{0},\\
\lmb(t)=(t^{2}+\eta_{0}^{2})^{1/2},\quad\gmm(t)=\begin{cases}
0 & \text{if }\eta_{0}=0,\\
\mathrm{sgn}(\eta_{0})(m+1)\tan^{-1}(\frac{t}{|\eta_{0}|}) & \text{if }\eta_{0}\neq0.
\end{cases}
\end{gathered}
\label{eq:formal-sol}
\end{equation}
When $\eta_{0}=0$, the solution represents the pseudoconformal blow-up
and exhibits no phase rotation. When $\eta_{0}\neq0$, no matter how
much small $|\eta_{0}|$ is, the solution does not blow up and rather
exhibits a quick phase rotation (on the time interval of length $\sim|\eta_{0}|$)
by the fixed amount of angle $(m+1)\pi$ (\emph{rotational instability}).

Keeping the dynamics \eqref{eq:formal-sol} in mind, by varying $\eta_{0}\geq0$
and using the backward construction, the authors in \cite{KimKwon2023MAMS}
constructed one-parameter families of solutions exhibiting rotational
instability. Using the forward construction and the shooting method
for $\eta$, the authors in \cite{KimKwon2023AnnPDE} constructed
a codimension one set of smooth initial data leading to pseudoconformal
blow-up.

In the present context, we deal with arbitrary $H_{m}^{3}$ finite-time
blow-up solutions. Any hypotheses such as $b>0$ (contracting scale)
or $|\eta|\ll b$ (trapped regime) cannot be assumed; we need to allow
any (small) values of $b$ and $\eta$. A naive remedy would be to
slightly modify the definition \eqref{eq:old-P-1} of $P_{\mathrm{old}}$
and consider instead 
\begin{equation}
\td P_{\mathrm{old}}(y;b,\eta)=\chi_{\delta B_{0}}e^{-ib\frac{y^{2}}{4}}Q^{(\eta)}(y),\label{eq:old-P-2}
\end{equation}
where $0<\delta\ll1$ and $B_{0}=(b^{2}+\eta^{2})^{-1/4}\eqqcolon\beta^{-1/2}$.
When $m\geq3$, this profile indeed works in our framework. However,
it \emph{no longer works} when $m\in\{1,2\}$ because of the dangerous
profile equation error $\Psi$ (see \eqref{eq:P1-eqn} below for more
precise meaning of $\Psi$) of size $\|\Psi\|_{\dot{\calH}_{m}^{3}}\ageq\beta^{\frac{m}{2}+3}\ageq\beta^{4}$,
which is insufficient in view of Remark~\ref{rem:Bounds-for-Psi}
below. Moreover, $\Psi$ cannot be improved by adjusting the cutoff
radius because of the pseudoconformal phase $e^{-ib\frac{y^{2}}{4}}$
and the fact that $Q^{(\eta)}$ might not be defined for large $y$,
e.g., $y\gg\delta B_{0}$.

A similar difficulty was faced by the authors in \cite{KimKwonOh2020arXiv}
for their blow-up construction in the $m=0$ case. Motivated from
the connection between the linearized (CSS) and the linearized wave
maps or Schrödinger maps (see the discussion in Section~\ref{subsec:Linear-conjugation-identities})
and the fact that the zero resonance of $H_{Q}$ might lead to nontrivial
logarithmic corrections to the modulation equations (as observed for
example in \cite{RaphaelRodnianski2012Publ.Math.,MerleRaphaelRodnianski2013InventMath}
for energy-critical geometric equations), the authors were able to
carry out the tail computations at the level of the \emph{$w_{1}$-equation}
and observe a \emph{logarithmic correction} to the modulation equation
$b_{s}+b^{2}+\eta^{2}=0$. Note that several simplifications were
possible because the trapped regime $|\eta|\ll b$ was assumed in
\cite{KimKwonOh2020arXiv}.

In the present context, we further push this idea (i) in the $m\geq1$
case and (ii) without the assumption $|\eta|\ll b$. It turns out
that, when $m=1$, there are nontrivial \emph{cubic corrections} to
the modulation equations $b_{s}+b^{2}+\eta^{2}=0$ and $\eta_{s}=0$
(see e.g. \eqref{eq:formal-mod-cub-corr}) and the cubic terms of
$P_{1}$ and $P_{2}$ have polynomially improved spatial decays. Similarly,
\emph{quartic} terms have improved decays when $m=2$. As $b_{s}$
has the nontrivial quadratic term $b^{2}+\eta^{2}$, our profile construction
will involve very explicit computations. Moreover, the linearized
operator $H_{Q}$ for the $w_{1}$-equation has a zero eigenfunction
$yQ$, so the solvability condition must be taken into account when
we invert $H_{Q}$.

\subsection{Definition of unlocalized profiles}

In this subsection, we define the unlocalized profiles $\wh P$, $\wh P_{1}$,
$\wh P_{2}$ for $w$, $w_{1}$, $w_{2}$ and also derive the formal
modulation equations for $b$ and $\eta$. We denote 
\begin{equation}
\beta\coloneqq\sqrt{b^{2}+\eta^{2}}\quad\text{and}\quad\bbet\coloneqq ib+\eta.\label{eq:def-beta/bbet}
\end{equation}
The unlocalized profiles will take the form
\begin{equation}
\left\{ \begin{aligned}\wh P & =Q+T_{1}^{(0)}+T_{2}^{(0)}+T_{3}^{(0)},\\
\wh P_{1} & =T_{1}^{(1)}+T_{2}^{(1)}+T_{3}^{(1)},\\
\wh P_{2} & =T_{2}^{(2)}+T_{3}^{(2)}+T_{4}^{(2)},
\end{aligned}
\right.\label{eq:def-unloc-prof}
\end{equation}
where each $T_{j}^{(k)}$ can be further decomposed as 
\begin{equation}
T_{j}^{(k)}(y;b,\eta)=\sum_{j_{1}+j_{2}=j}b^{j_{1}}\eta^{j_{2}}T_{j_{1},j_{2}}^{(k)}(y).\label{eq:Tk_j-decomp}
\end{equation}
The quartic term $T_{4}^{(2)}$ will be chosen in the next subsection.
Moreover, we will set $T_{3}^{(0)}=0$ if $m=1$ or $m\geq3$, so
the $\wh P$-expansion is essentially quadratic. When $m=1$, the
formal modulation equations for $b$ and $\eta$ will be changed into
\begin{equation}
b_{s}+b^{2}+\eta^{2}+p_{3}^{(b)}=0\qquad\text{and}\qquad\eta_{s}+p_{3}^{(\eta)}=0,\label{eq:formal-mod-cub-corr}
\end{equation}
where $p_{3}^{(b)}$ and $p_{3}^{(\eta)}$ are real cubic polynomials
satisfying \eqref{eq:def-p_3}.

Up to quadratic orders, the profiles $T_{j}^{(k)}$ are straightforward
because the deviation from the old profile ansatz will appear (only
when $m=1$) from the cubic order. Thus the linear and quadratic profiles
can be easily determined from the old profile ansatz \eqref{eq:old-P-1},
which can be viewed as $\wh P_{1}=-\bbet\tfrac{y}{2}\wh P$ and similarly
$\wh P_{2}=-\bbet\tfrac{y}{2}\wh P_{1}$. Motivated from the generalized
kernel relations \eqref{eq:gen-kernel-rel} and \eqref{eq:gen-kernel-rel-LQ},
the linear terms are defined by
\begin{align}
T_{1}^{(1)} & \coloneqq-\bbet\tfrac{y}{2}Q,\label{eq:def-T1_1}\\
T_{1}^{(0)} & \coloneqq-ib\tfrac{y^{2}}{4}Q-(m+1)\eta\rho.\label{eq:def-T0_1}
\end{align}
The quadratic terms are given by 
\begin{align}
T_{2}^{(2)} & \coloneqq\bbet^{2}\tfrac{y^{2}}{4}Q,\label{eq:def-T2_2}\\
T_{2}^{(1)} & \coloneqq-\bbet\tfrac{y}{2}T_{1}^{(0)},\label{eq:def-T1_2}\\
T_{2}^{(0)} & \coloneqq\out L_{Q}^{-1}[T_{2}^{(1)}+\tfrac{2}{y}A_{\tht}[Q,T_{1}^{(0)}]T_{1}^{(0)}+\tfrac{1}{y}A_{\tht}[T_{1}^{(0)}]Q],\label{eq:def-T0_2}
\end{align}
where we recall the polarization $A_{\tht}[\psi_{1},\psi_{2}]=-\frac{1}{2}\int_{0}^{y}\Re(\br{\psi_{1}}\psi_{2})y'dy'$
and the operator $\out L_{Q}^{-1}$ from \eqref{eq:def-LQ-inv}. A
seemingly complicated definition of $T_{2}^{(0)}$ is motivated from
the compatibility $\bfD_{P}P\approx P_{1}$.

We turn to the cubic terms of the profile expansions. We first state
the definitions of these cubic profiles. First, we define $T_{3}^{(2)}$
by 
\begin{equation}
T_{3}^{(2)}\coloneqq\begin{cases}
\out(A_{Q}^{\ast})^{-1}[\bbet^{2}A_{Q}^{\ast}(\tfrac{y^{2}}{4}T_{1}^{(0)})-p_{3}^{(b)}\tfrac{y}{2}Q+p_{3}^{(\eta)}i\tfrac{y}{2}Q] & \text{if }m=1,\\
\bbet^{2}\tfrac{y^{2}}{4}T_{1}^{(0)} & \text{if }m\geq2,
\end{cases}\label{eq:def-T2_3}
\end{equation}
where $p_{3}^{(b)}$ and $p_{3}^{(\eta)}$ are cubic homogeneous real
polynomials of $b$ and $\eta$ such that 
\begin{equation}
p_{3}^{(b)}-ip_{3}^{(\eta)}=\begin{cases}
8\pi\bbet^{3}/\|yQ\|_{L^{2}}^{2} & \text{if }m=1,\\
0 & \text{if }m\geq2,
\end{cases}\label{eq:def-p_3}
\end{equation}
and $\out(A_{Q}^{\ast})^{-1}$ is the inverse of $A_{Q}^{\ast}$,
whose explicit formula is given in \eqref{eq:def-AQ-ast-inv}. Next,
we define $T_{3}^{(1)}$ using the compatibility $A_{P}P_{1}\approx P_{2}$
(see Step~3 of the proof of Proposition~\ref{prop:Modified-profiles}
for the derivation): 
\begin{align}
T_{3}^{(1)} & \coloneqq\out A_{Q}^{-1}[T_{3}^{(2)}-\bbet(A_{\tht}[Q,T_{2}^{(0)}]Q+A_{\tht}[Q,T_{1}^{(0)}]T_{1}^{(0)}+\tfrac{1}{2}A_{\tht}[T_{1}^{(0)}]Q)],\label{eq:def-T1_3}
\end{align}
where $\out A_{Q}^{-1}$ denotes a formal right inverse of $A_{Q}$
whose explicit formula is given in \eqref{eq:def-AQ-inv}. Finally,
we define $T_{3}^{(0)}$ by 
\begin{equation}
T_{3}^{(0)}\coloneqq\begin{cases}
0 & \text{if }m=1\text{ or }m\geq3,\\
-\bbet^{3}\tfrac{y^{6}}{384}Q & \text{if }m=2.
\end{cases}\label{eq:def-T0_3}
\end{equation}
In particular, the profile $T_{3}^{(0)}$ is nonzero only if $m=2$.
It is introduced to capture the singular part $z_{\sg}^{\ast}$ of
$z^{\ast}$ when $m=2$, not to improve the profile equation error
$\Psi$ for $P$.

The definition of $T_{3}^{(2)}$ can be derived from the tail computations.
Note that when $m\geq2$, the old profile ansatz \eqref{eq:old-P-1}
is still valid at the cubic order so we can easily expect that $T_{3}^{(2)}=\bbet^{2}\tfrac{y^{2}}{4}T_{1}^{(0)}$.
However, when $m=1$, nontrivial cubic order modifications are necessary
and the derivation of $T_{3}^{(2)}$ needs more involved computations.
In the spirit of the tail computations, we substitute the formal modulation
equations \eqref{eq:formal-mod-cub-corr} with $p_{3}^{(b)}$ and
$p_{3}^{(\eta)}$ unknown and $(w,w_{1},w_{2})=(P,P_{1},P_{2})$ into
\eqref{eq:w1-eqn} to have\footnote{More precisely, one needs to substitute $\gmm_{s}=-(m+1)\eta-\int_{0}^{\infty}\Re(\br PP_{1})dy$
instead of $\eta_{s}=(m+1)\eta$. See the definition \eqref{eq:def-mod-vec}
of $\Mod$ and Remark~\ref{rem:profile-phase-corr} below.} 
\[
\begin{aligned}0 & \approx(b^{2}+\eta^{2}+p_{3}^{(b)})i\rd_{b}P_{1}+p_{3}^{(\eta)}i\rd_{\eta}P_{1}-ib\Lmb_{-1}P_{1}\\
 & \peq-(m+1)\eta P_{1}-(\tint 0y\Re(\br PP_{1})dy')P_{1}+A_{P}^{\ast}P_{2}.
\end{aligned}
\]
After a series of explicit computations (see the proof of \eqref{eq:psi1-bd-10}
in Step~6 of the proof of Proposition~\ref{prop:Modified-profiles}),
we can extract out the cubic terms of the above display as 
\begin{equation}
0\approx p_{3}^{(b)}\tfrac{y}{2}Q-p_{3}^{(\eta)}i\tfrac{y}{2}Q+A_{Q}^{\ast}(T_{3}^{(2)}-\bbet^{2}\tfrac{y^{2}}{4}T_{1}^{(0)}).\label{eq:prof-cubic}
\end{equation}
We then require the following \emph{solvability condition} for $A_{Q}^{\ast}$
in order to have $T_{3}^{(2)}$ growing the slowest:
\begin{equation}
p_{3}^{(b)}\tfrac{y}{2}Q-p_{3}^{(\eta)}i\tfrac{y}{2}Q-A_{Q}^{\ast}(\bbet^{2}\tfrac{y^{2}}{4}T_{1}^{(0)})\perp yQ.\label{eq:solv-cond-T2_3}
\end{equation}
When $m\geq2$, fast spatial decay of $yQ$ and $A_{Q}(yQ)=0$ guarantee
$A_{Q}^{\ast}(\bbet^{2}\tfrac{y^{2}}{4}T_{1}^{(0)})\perp yQ$ and
thus we set $p_{3}^{(b)}=p_{3}^{(\eta)}=0$ and $T_{3}^{(2)}=\bbet^{2}\frac{y^{2}}{4}T_{1}^{(0)}$
as expected. When $m=1$, the spatial decay $yQ\sim y^{-2}$ is \emph{critical},
and the following inner product becomes \emph{nonzero}: 
\[
\int_{0}^{\infty}A_{Q}^{\ast}(\bbet^{2}\tfrac{y^{2}}{4}T_{1}^{(0)})\cdot yQ\ ydy=-\bbet^{2}\tfrac{y^{2}}{4}T_{1}^{(0)}\cdot yQ\cdot y\Big|_{y=0}^{\infty}=\bbet^{3}\lim_{y\to+\infty}\tfrac{1}{16}y^{6}Q^{2}=2\bbet^{3},
\]
where we used \eqref{eq:asymp-T-lin} in the middle equality. In order
to ensure \eqref{eq:solv-cond-T2_3}, this computation suggests us
to define $p_{3}^{(b)}$ and $p_{3}^{(\eta)}$ such that 
\[
p_{3}^{(b)}-ip_{3}^{(\eta)}=\frac{2\bbet^{3}}{\int_{0}^{\infty}\frac{y}{2}Q\cdot yQ\ ydy}=\frac{8\pi\bbet^{3}}{\|yQ\|_{L^{2}}^{2}},
\]
which is \eqref{eq:def-p_3}. With this choice of $p_{3}^{(b)}$ and
$p_{3}^{(\eta)}$, we construct $T_{3}^{(2)}$ according to \eqref{eq:prof-cubic},
giving the definition \eqref{eq:def-T2_3}.

We end this subsection with some useful pointwise estimates for the
profiles $T_{j}^{(k)}$.
\begin{lem}[Pointwise estimates for $T_{j}^{(k)}$]
\label{lem:TaylorExpands}Recall the definitions of $T_{j}^{(k)}$s
above.
\begin{itemize}
\item (Smooth profiles) All $T_{j}^{(k)}$s have smooth $(m+k)$-equivariant
extensions on $\bbR^{2}$.
\item (Large $y$-asymptotics) Let $y\geq2$ and $k\in\bbN$. We have 
\begin{align}
|T_{1}^{(0)}+\bbet\tfrac{y^{2}}{4}Q|_{k} & \aleq_{k}\beta Q\log y,\label{eq:asymp-T-lin}\\
|T_{2}^{(0)}-\bbet^{2}\tfrac{y^{4}}{32}Q|_{k}+y|T_{2}^{(1)}-\bbet^{2}\tfrac{y^{3}}{8}Q|_{k} & \aleq_{k}\beta^{2}y^{2}Q\log y.\label{eq:asymp-T-quad}
\end{align}
For the cubic profiles, we have 
\begin{align}
|T_{3}^{(1)}|_{k}+y|T_{3}^{(2)}|_{k} & \aleq_{k}\beta^{3}y^{3}Q\log y\quad\text{if }m=1,\label{eq:asym-T-cub-m1}\\
|T_{3}^{(1)}+\bbet^{3}\tfrac{y^{5}}{64}Q|_{k}+y|T_{3}^{(2)}+\bbet^{3}\tfrac{y^{4}}{16}Q|_{k} & \aleq_{k}\beta^{3}y^{3}Q\log y\quad\text{if }m\geq2.\label{eq:asym-T-cub-mgeq2}
\end{align}
In particular, we have 
\begin{equation}
|T_{3}^{(1)}|_{k}+y|T_{3}^{(2)}|_{k}\aleq_{k}\beta^{3}y.\label{eq:asymp-T-cub-rough}
\end{equation}
 The above asymptotics also hold after taking $\rd_{b}$ or $\rd_{\eta}$
to the LHS and removing one $\beta$ from the RHS.
\end{itemize}
\end{lem}

\begin{rem}
Typically the growth tail of $T_{j+1}^{(k)}$ is $\beta y^{2}$ times
that of $T_{j}^{(k)}$. However, when $m=1$, the growth tail of $T_{3}^{(k)}$
is only $\beta\log y$ times that of $T_{2}^{(k)}$. This polynomially
improved decay is a consequence of the tail computations and is crucial
to have the profile error estimates \eqref{eq:Psi2-H2-bound}--\eqref{eq:Psi2-H3-bound}.
Similarly, we will see in Proposition~\ref{prop:Modified-profiles}
that the same happens for the $T_{4}^{(2)}$ profile when $m=2$.
\end{rem}

\begin{proof}
The fact that all $T_{j}^{(k)}$ have smooth $(m+k)$-equivariant
extensions easily follows from the smoothness assertions in Appendix~\ref{sec:FormalRightInverses}.
We focus on the proof of large $y$-asymptotics from now on.

\emph{\uline{Proof of \mbox{\eqref{eq:asymp-T-lin}}.}} By \eqref{eq:def-T0_1}
and \eqref{eq:def-rho}, we have 
\[
T_{1}^{(0)}+\bbet\tfrac{y^{2}}{4}Q=\eta\{-\out L_{Q}^{-1}(\tfrac{y}{2}Q)+\tfrac{y^{2}}{4}Q\}.
\]
Using the formula \eqref{eq:def-LQ-inv} of $\out L_{Q}^{-1}$, we
have 
\[
\out L_{Q}^{-1}(\tfrac{y}{2}Q)=QJ_{1}\tint 0y[\rd_{y}J_{2}]\tfrac{(y')^{2}}{2}dy'-QJ_{2}\tint 0y[\rd_{y}J_{1}]\tfrac{(y')^{2}}{2}dy',
\]
where $J_{1}$ and $J_{2}$ are defined in \eqref{eq:def-J1J2}. We
apply the asymptotics \eqref{eq:asymp-J1J2} of $J_{1}$ and $J_{2}$
to the above. Note that the first term gives the the leading order
term $\frac{y^{2}}{4}Q$ with an error of size $O(Q)$, while the
second term is of size $O(Q\log y)$. This yields \eqref{eq:asymp-T-lin}.

\emph{\uline{Proof of \mbox{\eqref{eq:asymp-T-quad}}.}} Recall
from \eqref{eq:def-T1_2} that $T_{2}^{(1)}=-\bbet\frac{y}{2}T_{1}^{(0)}$
so \eqref{eq:asymp-T-quad} for $T_{2}^{(1)}$ follows from \eqref{eq:asymp-T-lin}.
For $T_{2}^{(0)}$ (recall the definition \eqref{eq:def-T0_2}), we
note the estimates for $y\geq2$:
\begin{align*}
\chf_{y\geq2}|T_{2}^{(1)}-\bbet^{2}\tfrac{y^{3}}{8}Q|_{k} & \aleq_{k}\chf_{y\geq2}\beta^{2}yQ\log y,\\
\chf_{y\geq2}|\tfrac{2}{y}A_{\tht}[Q,T_{1}^{(0)}]T_{1}^{(0)}+\tfrac{1}{y}A_{\tht}[T_{1}^{(0)}]Q|_{k} & \aleq_{k}\chf_{y\geq2}\beta^{2}yQ.
\end{align*}
Combining these estimates with \eqref{eq:def-T0_2} and \eqref{eq:def-LQ-inv}
gives \eqref{eq:asymp-T-quad} for $T_{2}^{(0)}$\emph{.}

\emph{\uline{Proof of \mbox{\eqref{eq:asym-T-cub-m1}}.}} We first
prove \eqref{eq:asym-T-cub-m1} for $T_{3}^{(2)}$. Recall the formula
\eqref{eq:def-T2_3}: 
\[
T_{3}^{(2)}=\out(A_{Q}^{\ast})^{-1}[\bbet^{2}A_{Q}^{\ast}(\tfrac{y^{2}}{4}T_{1}^{(0)})-p_{3}^{(b)}\tfrac{y}{2}Q+p_{3}^{(\eta)}i\tfrac{y}{2}Q].
\]
As $m=1$, we have $A_{Q}^{\ast}\approx-\rd_{y}+\frac{1}{y}$ and
$y^{2}T_{1}^{(0)}\sim\beta y^{4}Q\sim\beta y$, which result in the
cancellation $(-\rd_{y}+\frac{1}{y})\beta y=0$. Incorporating the
error terms using \eqref{eq:asymp-T-lin}, this cancellation is captured
by the following estimate 
\[
\chf_{y\geq2}|A_{Q}^{\ast}(\tfrac{y^{2}}{4}T_{1}^{(0)})|_{k}\aleq_{k}\chf_{y\geq2}\beta yQ\log y.
\]
Thus we have 
\begin{equation}
\chf_{y\geq2}|\bbet^{2}A_{Q}^{\ast}(\tfrac{y^{2}}{4}T_{1}^{(0)})-p_{3}^{(b)}\tfrac{y}{2}Q+p_{3}^{(\eta)}i\tfrac{y}{2}Q|_{k}\aleq_{k}\chf_{y\geq2}\beta^{3}yQ\log y.\label{eq:asymp-T-tmp1}
\end{equation}
We now recall that $p_{3}^{(b)}$ and $p_{3}^{(\eta)}$ are chosen
to satisfy \eqref{eq:solv-cond-T2_3}. Thus we can apply the formula
\eqref{eq:AQ-ast-inv-orthog}: 
\[
T_{3}^{(2)}=\tfrac{1}{y^{2}Q}\tint y{\infty}y'Q[\bbet^{2}A_{Q}^{\ast}(\tfrac{(y')^{2}}{4}T_{1}^{(0)})-p_{3}^{(b)}\tfrac{y'}{2}Q+p_{3}^{(\eta)}i\tfrac{y'}{2}Q]\,y'dy'.
\]
Substituting \eqref{eq:asymp-T-tmp1} into the above display gives
\eqref{eq:asym-T-cub-m1} for $T_{3}^{(2)}$.

Next, we prove \eqref{eq:asym-T-cub-m1} for $T_{3}^{(1)}$. Combining
the pointwise estimates 
\begin{align*}
|T_{3}^{(2)}|_{k} & \aleq_{k}\beta^{3}y^{2}Q\log y,\\
|\bbet(A_{\tht}[Q,T_{1}^{(0)}]T_{1}^{(0)}+\tfrac{1}{2}A_{\tht}[T_{1}^{(0)}]Q)|_{k} & \aleq_{k}\beta^{3}y^{2}Q
\end{align*}
with \eqref{eq:def-T1_3} and \eqref{eq:def-AQ-inv} gives \eqref{eq:asym-T-cub-m1}
for $T_{3}^{(1)}$.

\emph{\uline{Proof of \mbox{\eqref{eq:asym-T-cub-mgeq2}}.}} As
$m\geq2$, we have $p_{3}^{(b)}=p_{3}^{(\eta)}=0$ and $T_{3}^{(2)}=\bbet^{2}\frac{y^{2}}{4}T_{1}^{(0)}$.
Thus \eqref{eq:asym-T-cub-mgeq2} for $T_{3}^{(2)}$ easily follows
from \eqref{eq:asymp-T-lin}. Combining the pointwise estimates 
\begin{align*}
|T_{3}^{(2)}+\bbet^{3}\tfrac{y^{4}}{16}Q|_{k} & \aleq_{k}\beta^{3}y^{2}Q\log y,\\
|\bbet(A_{\tht}[Q,T_{1}^{(0)}]T_{1}^{(0)}+\tfrac{1}{2}A_{\tht}[T_{1}^{(0)}]Q)|_{k} & \aleq_{k}\beta^{3}y^{2}Q
\end{align*}
with \eqref{eq:def-T1_3} and \eqref{eq:def-AQ-inv} gives \eqref{eq:asym-T-cub-mgeq2}
for $T_{3}^{(1)}$.

\emph{\uline{Completion of the proof.}} The estimate \eqref{eq:asymp-T-cub-rough}
follows from \eqref{eq:asym-T-cub-m1} and \eqref{eq:asym-T-cub-mgeq2}.
As all $T_{j}^{(k)}$s are of the form \eqref{eq:Tk_j-decomp}, all
the estimates used in this proof holds after taking $\rd_{b}$ or
$\rd_{\eta}$ in the LHS and removing one $\beta$ from the RHS. This
completes the proof.
\end{proof}

\subsection{Localization and main proposition}

Having defined the profiles up to cubic terms, we can state the precise
form of the profiles $P$, $P_{1}$, and $P_{2}$ including the cutoffs.
Let 
\begin{equation}
B_{0}\coloneqq\beta^{-\frac{1}{2}}\qquad\text{and}\qquad B_{1}\coloneqq\beta^{-1}.\label{eq:def-cutoff-rad}
\end{equation}
We will use the modified profiles defined by (c.f. \eqref{eq:def-unloc-prof})
\begin{align}
P & =Q+\chi_{B_{1}}\{T_{1}^{(0)}+T_{2}^{(0)}+T_{3}^{(0)}\},\label{eq:def-P0}\\
P_{1} & =\chi_{B_{1}}\{T_{1}^{(1)}+T_{2}^{(1)}+T_{3}^{(1)}\},\label{eq:def-P1}\\
P_{2} & =\chi_{B_{1}}\{T_{2}^{(2)}+T_{3}^{(2)}\}+\chi_{B_{0}}T_{4}^{(2)},\label{eq:def-P2}
\end{align}
where the quartic term $T_{4}^{(2)}$ of $P_{2}$ will be defined
in the proof of Proposition~\ref{prop:Modified-profiles} below.
As mentioned at the beginning of Section~\ref{sec:Modified-profiles},
the cutoff radius $B_{1}$ is important in our analysis. See also
Remark~\ref{rem:cutoff-radius} below for more detailed explanations.

We now state the main proposition. Define the vectors 
\begin{equation}
\Mod\coloneqq\begin{pmatrix}\frac{\lmb_{s}}{\lmb}+b\\
\gmm_{s}+(m+1)\eta+\int_{0}^{\infty}\Re(\br PP_{1})dy\\
b_{s}+b^{2}+\eta^{2}+p_{3}^{(b)}\\
\eta_{s}+p_{3}^{(\eta)}
\end{pmatrix},\quad\bfv_{k}\coloneqq\begin{pmatrix}\Lmb P_{k}\\
-iP_{k}\\
-\rd_{b}P_{k}\\
-\rd_{\eta}P_{k}
\end{pmatrix},\label{eq:def-mod-vec}
\end{equation}
where we wrote $P_{0}=P$ and $\bfv_{0}=\bfv$.
\begin{prop}[Modified profiles]
\label{prop:Modified-profiles}Let $m\geq1$ and $\beta\in(-\frac{1}{10},\frac{1}{10})$.
There exists a smooth profile 
\[
T_{4}^{(2)}(y;b,\eta)=\sum_{i+j=4}b^{i}\eta^{j}T_{i,j}^{(2)}(y)
\]
that admits a smooth $(m+2)$-equivariant extension on $\bbR^{2}$
and satisfies the pointwise estimates 
\begin{equation}
|T_{4}^{(2)}|_{k}\aleq_{k}\begin{cases}
\beta^{4}y(\log y)^{2} & \text{if }m=1,\\
\beta^{4}\log y & \text{if }m=2,\\
\beta^{4}y^{4-m} & \text{if }m\geq3
\end{cases}\label{eq:T2_4-bound}
\end{equation}
for any $k\in\bbN$ in the region $y\geq2$, such that the following
bounds hold.
\begin{itemize}
\item (Estimates for modulation vectors) For $\bfv=\bfv_{0}$, we have for
any $R>0$
\begin{align}
\chf_{(0,R]}\{|(\bfv_{0})_{1}-\Lmb Q|_{k}+|(\bfv_{0})_{2}+iQ|_{k}\} & \aleq_{R}\beta,\label{eq:modvec0-1}\\
\chf_{(0,R]}\{|(\bfv_{0})_{3}-i\tfrac{y^{2}}{4}Q|_{k}+|(\bfv_{0})_{4}-(m+1)\rho|_{k}\} & \aleq_{R}\beta.\label{eq:modvec0-2}
\end{align}
For $\bfv_{1}$, we have for any $k\in\bbN$
\begin{gather}
|(\bfv_{1})_{1}|_{k}+|(\bfv_{1})_{2}|_{k}\aleq_{k}\chf_{(0,2]}\beta y^{2}+\chf_{[2,2B_{1}]}(\beta y^{-2}+\beta^{2}),\label{eq:modvec1-1}\\
|(\bfv_{1})_{3}-\chi_{B_{1}}i\tfrac{y}{2}Q|_{k}+|(\bfv_{1})_{4}-\chi_{B_{1}}\tfrac{y}{2}Q|_{k}\aleq_{k}\chf_{(0,2]}\beta y^{4}+\chf_{[2,2B_{1}]}\beta.\label{eq:modvec1-2}
\end{gather}
For $\bfv_{2}$, we have for any $k\in\bbN$ 
\begin{align}
|(\bfv_{2})_{1}|_{k}+|(\bfv_{2})_{2}|_{k} & \aleq_{k}\chf_{(0,2]}\beta^{2}y^{3}+\chf_{[2,2B_{1}]}\beta^{2}y^{-1},\label{eq:modvec2-1}\\
|(\bfv_{2})_{3}|_{k}+|(\bfv_{2})_{4}|_{k} & \aleq_{k}\chf_{(0,2]}\beta y^{3}+\chf_{[2,2B_{1}]}\beta y^{-1}.\label{eq:modvec2-2}
\end{align}
\item (Compatibility) For $\bfD_{P}P$ and $P_{1}$, we have 
\begin{align}
\|\bfD_{P}P-P_{1}\|_{L^{2}} & \aleq\beta,\label{eq:compat1-H1}\\
\|\bfD_{P}P-P_{1}\|_{\dot{H}_{m+1}^{1}} & \aleq\beta^{2},\label{eq:compat1-H2}\\
\|\bfD_{P}P-P_{1}\|_{\dot{H}_{m+1}^{2}} & \aleq\beta^{3}.\label{eq:compat1-H3}
\end{align}
For $A_{P}P_{1}$ and $P_{2}$, we have 
\begin{align}
\|A_{P}P_{1}-P_{2}\|_{L^{2}} & \aleq\beta^{2},\label{eq:compat2-H2}\\
\|A_{P}P_{1}-P_{2}\|_{\dot{H}_{m+2}^{1}} & \aleq\beta^{3},\label{eq:compat2-H3}\\
\|A_{P}P_{1}-P_{2}\|_{\dot{V}_{m+2}^{3/2}} & \aleq\beta^{7/2}.\label{eq:compat2-virial}
\end{align}
\item (Equation for $P$) We have 
\begin{equation}
\big(\rd_{s}-\frac{\lmb_{s}}{\lmb}\Lmb+\gmm_{s}i\big)P+i\bfD_{P}^{\ast}P_{1}+i\Big(\int_{y}^{\infty}\Re(\br PP_{1})dy'\Big)P=-\Mod\cdot\bfv+i\Psi,\label{eq:P-eqn}
\end{equation}
where $\Psi$ satisfies for any $R>0$ 
\begin{equation}
\chf_{(0,R]}|\Psi|\aleq_{R}\beta^{3}.\label{eq:Psi0-loc-bound}
\end{equation}
\item (Equation for $P_{1}$) We have 
\begin{equation}
\big(\rd_{s}-\frac{\lmb_{s}}{\lmb}\Lmb_{-1}+\gmm_{s}i\big)P_{1}+iA_{P}^{\ast}P_{2}+i\Big(\int_{y}^{\infty}\Re(\br PP_{1})dy'\Big)P_{1}=-\Mod\cdot\bfv_{1}+i\Psi_{1},\label{eq:P1-eqn}
\end{equation}
where $\Psi_{1}$ satisfies
\begin{equation}
\|\tfrac{1}{y^{2}}\Psi_{1}\|_{L^{1}}\aleq\beta^{4}|\log\beta|^{3}.\label{eq:Psi1-bound}
\end{equation}
\item (Equation for $P_{2}$) We have 
\begin{equation}
\begin{aligned}\big(\rd_{s}-\frac{\lmb_{s}}{\lmb}\Lmb_{-2}+\gmm_{s}i\big)P_{2}+i\td H_{P}P_{2}\qquad\qquad\\
+i\Big(\int_{y}^{\infty}\Re(\br PP_{1})dy'\Big)P_{2}-i\br PP_{1}^{2} & =-\Mod\cdot\bfv_{2}+i\Psi_{2}
\end{aligned}
\label{eq:P2-eqn}
\end{equation}
with the bounds 
\begin{align}
\|\Psi_{2}\|_{L^{2}} & \aleq\beta^{4}|\log\beta|^{2},\label{eq:Psi2-H2-bound}\\
\|\Psi_{2}\|_{\dot{H}_{m+2}^{1}} & \aleq\beta^{\frac{9}{2}}|\log\beta|^{2}.\label{eq:Psi2-H3-bound}
\end{align}
\end{itemize}
\end{prop}

\begin{rem}[Cutoff radius $B_{1}$]
\label{rem:cutoff-radius}As alluded to above, $B_{1}=\beta^{-1}$
is the only cutoff radius such that our analysis works when $m\in\{1,2\}$.
According to our strategy, $P^{\sharp}=P_{\lmb,\gmm}$ must contain
the singular part $z_{\sg}^{\ast}$ of the asymptotic profile $z^{\ast}$
and hence the cutoff at the original scale should be $r\sim1$. In
the formal modulation dynamics \eqref{eq:old-mod-eqn}, the quantity
\[
\frac{b^{2}+\eta^{2}}{\lmb^{2}}
\]
is \emph{conserved}, and hence $r\sim1$ corresponds to $y=\frac{r}{\lmb}\sim(b^{2}+\eta^{2})^{-1/2}=B_{1}$.

When we estimate the profile errors $\Psi_{k}$, we crucially use
the \emph{special cancellation} 
\begin{equation}
[(b^{2}+\eta^{2})\rd_{b}-by\rd_{y}]\chi_{B_{1}}=0,\label{eq:cutoff-cancellation}
\end{equation}
which corresponds to the obvious identity $\rd_{t}\chi_{R}=0$ for
fixed $R$ in the original $(t,r)$-coordinates. If we use other cutoff
radii, say $B=\beta^{-p}$ with $p\neq1$, then there are nontrivial
cutoff errors 
\begin{align*}
\|[(b^{2}+\eta^{2})\rd_{b}-by\rd_{y}]\chi_{B}T_{2}^{(2)}\|_{L^{2}} & \ageq\|\chf_{y\sim B}b\beta^{2}\tfrac{1}{y}\|_{L^{2}}\ageq|b|^{3}\text{ if }m=1,\\
\|[(b^{2}+\eta^{2})\rd_{b}-by\rd_{y}]\chi_{B}T_{3}^{(2)}\|_{\dot{H}_{m+2}^{1}} & \ageq\|\chf_{y\sim B}b\beta^{3}\|_{\dot{H}_{m+2}^{1}}\ageq|b|^{4}\text{ if }m=2,
\end{align*}
violating the necessary bounds \eqref{eq:necessary-bound} below.

When $m\geq3$, the cancellation \eqref{eq:cutoff-cancellation} is
not necessary and even the old profile ansatz \eqref{eq:old-P-2}
(where the cutoff is placed at the self-similar scale) works. This
is consistent with the statement of Theorem~\ref{thm:blow-up-rigidity}
that there is no $H^{3}$-singular part of the asymptotic profile.
\end{rem}

\begin{rem}[Bounds for $\Psi_{k}$]
\label{rem:Bounds-for-Psi}As we will see in Section~\ref{subsec:Energy-Morawetz},
it is necessary to obtain the bounds 
\begin{equation}
\chf_{y\aleq1}|\Psi|\aleq\beta^{2+},\quad\chf_{y\aleq1}|\Psi_{1}|\aleq\beta^{3+},\quad\|\Psi_{2}\|_{L^{2}}\aleq\beta^{3+},\quad\|\Psi_{2}\|_{\dot{H}_{m+2}^{1}}\aleq\beta^{4+}\label{eq:necessary-bound}
\end{equation}
for our main energy estimates (Proposition~\ref{prop:EnergyMorawetz}).
The first two bounds contribute to the $\lmb,\gmm$-modulation estimates
and $b,\eta$-modulation estimates, respectively. In the energy estimates,
we will use the $P_{2}$-equation and the degeneracy estimates \eqref{eq:modvec2-1}--\eqref{eq:modvec2-2}
for $\bfv_{2}$ allow the $\Psi$ and $\Psi_{1}$ bounds to be rougher
than the $\Psi_{2}$ bound by the factors $\beta^{2}$ and $\beta$,
respectively. This also explains why $P$ and $P_{1}$ ansatzes do
not need to be as precise as the $P_{2}$ ansatz.

By \eqref{eq:necessary-bound}, it is necessary to have at least quadratic,
cubic, and quartic expansions of $P$, $P_{1}$, and $P_{2}$, respectively.
This in principle guarantees 
\begin{equation}
\Psi=O(\beta^{3}),\ \Psi_{1}=O(\beta^{4}),\ \Psi_{2}=O(\beta^{5})\text{ in a compact region }y\aleq1.\label{eq:rem-Psi-cpt-bd}
\end{equation}
If $m$ is sufficiently large, \eqref{eq:rem-Psi-cpt-bd} remains
valid even in various Sobolev norms due to the fast spatial decay
of $Q$. Thus all the $\beta$-losses in \eqref{eq:Psi1-bound}, \eqref{eq:Psi2-H2-bound},
and \eqref{eq:Psi2-H3-bound} compared to \eqref{eq:rem-Psi-cpt-bd}
come from the lack of spatial decay of $Q$.

Finally, let us note that we did not attempt to make \eqref{eq:Psi1-bound},
\eqref{eq:Psi2-H2-bound}, and \eqref{eq:Psi2-H3-bound} sharp because
we only need the bound such as \eqref{eq:necessary-bound}. Some of
these bounds can be improved logarithmically or even polynomially
in $\beta$.
\end{rem}

\begin{rem}[Phase correction]
\label{rem:profile-phase-corr}The integral term $\int_{0}^{\infty}\Re(\br PP_{1})dy$
in the definition of $\Mod_{2}$ is employed to replace the $\int_{y}^{\infty}$-integrals
in the equations \eqref{eq:w-eqn}--\eqref{eq:w2-eqn} by the $-\int_{0}^{y}$-integrals
so that the resulting equations become more suited for the tail computations.
This integral term can be understood as a \emph{phase correction}
from the nonlocal nonlinearity. Note that it is equal to $-2(m+1)\eta$
at the linear order due to 
\begin{equation}
\tint 0{\infty}\Re(\br PP_{1})dy=-\eta\tfrac{1}{2}\tint 0{\infty}Q^{2}ydy+O(\beta^{2})=-2(m+1)\eta+O(\beta^{2}),\label{eq:phase-corr-1}
\end{equation}
so we have $\Mod_{2}\approx\gmm_{s}-(m+1)\eta$ as in \eqref{eq:old-mod-eqn}.
\end{rem}

\begin{rem}
In contrast to \cite{KimKwon2023AnnPDE}, we cannot simply bound $P$
by $Q$ because the higher order terms in the expansion of $P$ become
dominant in the region $B_{0}\ll y\aleq B_{1}$. Thus we need to take
care about the higher order terms in many places of the proof. In
several cases, the contribution of the highest order terms indeed
\emph{saturates} the error estimates.
\end{rem}

\begin{proof}
\textbf{Step 0. }Remarks and notation for the proof.

In many parts of the proof, we will only focus on the estimates in
the far region $y\geq2$. The contribution of error terms from the
region $y\leq2$ does not disrupt the validity of the proposition.
Note that we do not need to worry about singularity issues near the
origin $y=0$ because all the profiles $T_{j}^{(k)}$ have smooth
equivariant extensions on $\bbR^{2}$.

To simplify the presentation in various pointwise estimates, let us
introduce more notation. We denote by $f=\calO(g)$ to mean that 
\begin{equation}
f=\calO(g)\qquad\Longleftrightarrow\qquad\chf_{y\geq2}|f|_{k}\aleq_{k}g\qquad\forall k\in\bbN.\label{eq:def-calO}
\end{equation}
For a nonnegative integer $j\geq0$, a real number $\ell$, and a
function $f=f(y;b,\eta)$, we use the following notation: 
\begin{align}
f=\calI_{\ell}^{j}\quad & \Leftrightarrow\quad f=\sum_{j_{1}+j_{2}=j}b^{j_{1}}\eta^{j_{2}}f_{j_{1},j_{2}}(y)\text{ for some }f_{j_{1},j_{2}}=\calO(y^{\ell}),\label{eq:def-calI}\\
f=\calO(\calI_{\ell}^{j})\quad & \Leftrightarrow\quad f=\calO(\beta^{j}y^{\ell}).\label{eq:def-calOcalI}
\end{align}
We can also add logarithmic divergences to the above notation; for
a real number $\ell'$, we use the notation $f=\calI_{\ell,\ell'}^{j}$
if $f$ can be expressed as in \eqref{eq:def-calI} for some\footnote{We use the $\calO$-notation in the region $y\geq2$ so $(\log y)^{\ell'}$
is well-defined.} $f_{j_{1},j_{2}}=\calO(y^{\ell}(\log y)^{\ell'})$. We can similarly
define $f=\calO(\calI_{\ell,\ell'}^{j})$. We will also use this notation
combined with various operations such as addition, subtraction, and
integration. A useful fact is that 
\[
\int_{0}^{y}\calI_{\ell}^{j}\frac{dy'}{y'}=\begin{cases}
\calI_{\ell}^{j} & \text{if }\ell>0,\\
\calI_{0,1}^{j} & \text{if }\ell=0,\\
\calI_{0}^{j} & \text{if }\ell<0,
\end{cases}
\]
and similarly for $\int_{0}^{y}\calO(\calI_{\ell}^{j})\frac{dy'}{y'}$
by taking $\calO$. We will also frequently use the estimate $\calI_{\ell+1}^{j+1}\leq\calI_{\ell}^{j}$
in the region $y\leq2B_{1}$.

\textbf{Step 1.} Estimates for the modulation vectors.

The proof of \eqref{eq:modvec0-1}--\eqref{eq:modvec0-2} is clear
from the expansion \eqref{eq:def-P0} as we allow $R$-dependences.
The proof of \eqref{eq:modvec1-1}--\eqref{eq:modvec1-2} follows
from $T_{1}^{(1)}=-\bbet\tfrac{y}{2}Q$, \eqref{eq:asymp-T-quad},
and \eqref{eq:asymp-T-cub-rough}. Finally, the proof of \eqref{eq:modvec2-1}--\eqref{eq:modvec2-2}
follows from $T_{2}^{(2)}=\bbet^{2}\frac{y^{2}}{4}Q$, \eqref{eq:asymp-T-cub-rough},
and \eqref{eq:T2_4-bound} with $B_{0}=\beta^{-\frac{1}{2}}$.\footnote{Note that $T_{4}^{(2)}$ is not constructed yet, but it will be constructed
in Step~8 below satisfying $T_{4}^{(2)}=\calI_{1,2}^{4}$.}

\textbf{Step 2.} Proof of the compatibility estimates \eqref{eq:compat1-H1}--\eqref{eq:compat1-H3}.

We denote by $\calR_{c}^{(0)}$ any function satisfying the bounds
(c.f. RHS\eqref{eq:compat1-H1}--RHS\eqref{eq:compat1-H3})
\begin{equation}
\|\calR_{c}^{(0)}\|_{L^{2}}\aleq\beta,\quad\|\calR_{c}^{(0)}\|_{\dot{H}_{m+1}^{1}}\aleq\beta^{2},\quad\|\calR_{c}^{(0)}\|_{\dot{H}_{m+1}^{2}}\aleq\beta^{3}.\label{eq:def-Rc0}
\end{equation}
We denote $f=\calR_{c}^{(0)}$ if $f$ satisfies the estimates \eqref{eq:def-Rc0}.

\emph{\uline{First claim.}} We claim that we can approximate $\bfD_{P}P$
by $\chi_{B_{1}}\bfD_{\wh P}\wh P$:
\begin{equation}
\bfD_{P}P=\chi_{B_{1}}\bfD_{\wh P}\wh P+\calR_{c}^{(0)},\label{eq:compat1-tmp1}
\end{equation}
where $\wh P$ is the unlocalized profile \eqref{eq:def-unloc-prof}.
By the definition of $P$ and $\bfD_{Q}Q=0$, we have (recall $\chi_{\ageq B_{1}}=1-\chi_{B_{1}}$)
\begin{align*}
\bfD_{P}P & =\chi_{B_{1}}\bfD_{P}\wh P+\chi_{\ageq B_{1}}\bfD_{P}Q+(\rd_{y}\chi_{B_{1}})(\wh P-Q)\\
 & =\chi_{B_{1}}\bfD_{\wh P}\wh P+\chi_{B_{1}}(\bfD_{P}-\bfD_{\wh P})\wh P\\
 & \qquad+\chi_{\ageq B_{1}}\{(\bfD_{P}-\bfD_{Q})Q\}+(\rd_{y}\chi_{B_{1}})\{\wh P-Q\}.
\end{align*}
Thus it suffices to show that the last three terms of the above display
belong to $\calR_{c}^{(0)}$. Note that $\bfD_{P}-\bfD_{\wh P}$ and
$\bfD_{P}-\bfD_{Q}$ can be identified as the functions $\frac{1}{y}(A_{\tht}[\wh P]-A_{\tht}[P])$
and $\frac{1}{y}(A_{\tht}[Q]-A_{\tht}[P])$ and we have the estimates
\begin{alignat*}{1}
\chi_{B_{1}}(\bfD_{P}-\bfD_{\wh P}),\ \rd_{y}\chi_{B_{1}} & =\calO(\chf_{[B_{1},2B_{1}]}y^{-1}),\\
\chi_{\ageq B_{1}}(\bfD_{P}-\bfD_{Q}) & =\calO(\chf_{[B_{1},\infty)}y^{-1}).
\end{alignat*}
These estimates give 
\begin{align*}
\chi_{B_{1}}(\bfD_{P}-\bfD_{\wh P})\wh P+(\rd_{y}\chi_{B_{1}})\{\wh P-Q\} & =\calO(\chf_{[B_{1},2B_{1}]}\beta y^{-1})=\calR_{c}^{(0)},\\
\chi_{\ageq B_{1}}\{(\bfD_{P}-\bfD_{Q})Q\} & =\calO(\chf_{[B_{1},\infty)}y^{-4})=\calR_{c}^{(0)},
\end{align*}
as desired. This completes the proof of the claim \eqref{eq:compat1-tmp1}.

\emph{\uline{Second claim.}} We claim that 
\begin{equation}
\chi_{B_{1}}\bfD_{\wh P}\wh P=\chi_{B_{1}}\{\bfD_{\wh P-T_{3}^{(0)}}(\wh P-T_{3}^{(0)})+\bfD_{Q}T_{3}^{(0)}\}+\calR_{c}^{(0)}.\label{eq:compat1-tmp2}
\end{equation}
To show this, it suffices to show that 
\[
\chi_{B_{1}}\{(\bfD_{\wh P}-\bfD_{Q})T_{3}^{(0)}+(\bfD_{\wh P}-\bfD_{\wh P-T_{3}^{(0)}})(\wh P-T_{3}^{(0)})\}=\calR_{c}^{(0)}.
\]
As $T_{3}^{(0)}=0$ for $m\neq2$, we may assume $m=2$. Now the above
error estimate follows from 
\begin{align*}
\chi_{B_{1}}(\bfD_{\wh P}-\bfD_{Q})T_{3}^{(0)} & =\calO(\chf_{[2,2B_{1}]}\tfrac{1}{y}\tint 0y(\calI_{-6}^{1}+\calI_{4}^{6})y'dy'\cdot\calI_{2}^{3})\\
 & =\calO(\chf_{[2,2B_{1}]}\{\calI_{1}^{4}+\calI_{7}^{9}\})=\calR_{c}^{(0)}
\end{align*}
and 
\begin{align*}
\chi_{B_{1}}(\bfD_{\wh P}-\bfD_{\wh P-T_{3}^{(0)}})(\wh P-T_{3}^{(0)}) & =\calO(\chf_{[2,2B_{1}]}\tfrac{1}{y}\tint 0y(\calI_{-2}^{3}+\calI_{4}^{6})y'dy'\cdot(\calI_{-4}^{0}+\calI_{0}^{2}))\\
 & =\calO(\chf_{[2,2B_{1}]}\{\calI_{-5,1}^{3}+\calI_{-1,1}^{5}+\calI_{1}^{6}+\calI_{5}^{8}\})=\calR_{c}^{(0)}.
\end{align*}
Note that the contribution of the highest order term $\frac{1}{y}A_{\tht}[T_{3}^{(0)}]T_{3}^{(0)}$
saturates the bound \eqref{eq:def-Rc0}.

\emph{\uline{Last claim.}} We claim that 
\begin{equation}
\chi_{B_{1}}\bfD_{\wh P-T_{3}^{(0)}}(\wh P-T_{3}^{(0)})=P_{1}-\chi_{B_{1}}T_{3}^{(1)}+\calR_{c}^{(0)}.\label{eq:compat1-tmp3}
\end{equation}
To see this, we expand $\bfD_{\wh P-T_{3}^{(0)}}(\wh P-T_{3}^{(0)})$
using \eqref{eq:LinearizationBogomolnyi}: 
\begin{equation}
\begin{aligned}\chi_{B_{1}} & \bfD_{\wh P-T_{3}^{(0)}}(\wh P-T_{3}^{(0)})\\
 & =\chi_{B_{1}}\{L_{Q}(T_{1}^{(0)}+T_{2}^{(0)})+N_{Q}(T_{1}^{(0)}+T_{2}^{(0)})\}\\
 & =\chi_{B_{1}}\{L_{Q}T_{1}^{(0)}+(L_{Q}T_{2}^{(0)}-\tfrac{2}{y}A_{\tht}[Q,T_{1}^{(0)}]T_{1}^{(0)}-\tfrac{1}{y}A_{\tht}[T_{1}^{(0)}]Q)\}\\
 & \peq+\calO(\chf_{[2,2B_{1}]}\{\calI_{0}^{3}+\calI_{0,1}^{4}+\calI_{2}^{5}+\calI_{4}^{6}\}).
\end{aligned}
\label{eq:compat1-tmp}
\end{equation}
Applying \eqref{eq:gen-kernel-rel-LQ} for the linear term and \eqref{eq:def-T0_2}
for the quadratic term, the previous display continues as 
\begin{align*}
 & =\chi_{B_{1}}\{T_{1}^{(1)}+T_{2}^{(1)}\}+\calO(\chf_{[2,2B_{1}]}\{\calI_{0}^{3}+\calI_{0,1}^{4}+\calI_{2}^{5}+\calI_{4}^{6}\})\\
 & =P_{1}-\chi_{B_{1}}T_{3}^{(1)}+\calR_{c}^{(0)},
\end{align*}
completing the proof of the claim \eqref{eq:compat1-tmp3}. As in
the previous paragraph, we note that the highest order term $\frac{1}{y}A_{\tht}[T_{2}^{(0)}]T_{2}^{(0)}$
\emph{saturates} the bound \eqref{eq:def-Rc0} when $m=1$.

\emph{\uline{Completion of the proof of \mbox{\eqref{eq:compat1-H1}}--\mbox{\eqref{eq:compat1-H3}}.}}
Combining the claims \eqref{eq:compat1-tmp1}, \eqref{eq:compat1-tmp2},
and \eqref{eq:compat1-tmp3}, we have proved that 
\[
\bfD_{P}P-P_{1}=\chi_{B_{1}}\{\bfD_{Q}T_{3}^{(0)}-T_{3}^{(1)}\}+\calR_{c}^{(0)}.
\]
When $m=1$ or $m\geq3$, we have $T_{3}^{(0)}=0$ and $T_{3}^{(1)}=\calI_{0,1}^{3}$
(see \eqref{eq:asym-T-cub-m1}--\eqref{eq:asym-T-cub-mgeq2}) so
the RHS belongs to $\calR_{c}^{(0)}$. When $m=2$, we have $T_{3}^{(1)}=\calI_{1}^{3}$
so $\chi_{B_{1}}T_{3}^{(1)}$ \emph{does not} satisfy the last bound
of \eqref{eq:def-Rc0} logarithmically. Thus we introduced a nontrivial
$T_{3}^{(0)}=-\bbet^{3}\tfrac{y^{6}}{384}Q$ in this case so that
\[
\chi_{B_{1}}\{\bfD_{Q}T_{3}^{(0)}-T_{3}^{(1)}\}=\chi_{B_{1}}\{-\bbet^{3}\tfrac{y^{5}}{64}Q-T_{3}^{(1)}\}=\calO(\chf_{[2,2B_{1}]}\beta^{3}y^{-1}\log y)=\calR_{c}^{(0)},
\]
where we used \eqref{eq:asym-T-cub-mgeq2}. This completes the proof
of \eqref{eq:compat1-H1}--\eqref{eq:compat1-H3}.

\textbf{Step 3.} Proof of the compatibility estimates \eqref{eq:compat2-H2}--\eqref{eq:compat2-virial}.

We proceed similarly as in the previous step. Denote by $\calR_{c}^{(1)}$
any function satisfying the bounds (c.f. \eqref{eq:compat2-H2}--\eqref{eq:compat2-virial})
\begin{equation}
\|\calR_{c}^{(1)}\|_{L^{2}}\aleq\beta^{2},\quad\|\calR_{c}^{(1)}\|_{\dot{H}_{m+2}^{1}}\aleq\beta^{3},\quad\|\calR_{c}^{(1)}\|_{\dot{V}_{m+2}^{3/2}}\aleq\beta^{7/2}.\label{eq:def-Rc1}
\end{equation}

\emph{\uline{First claim.}} We claim that 
\begin{equation}
A_{P}P_{1}=\chi_{B_{1}}A_{\wh P}\wh P_{1}+\calR_{c}^{(1)}.\label{eq:compat2-cl1}
\end{equation}
To see this, we note that 
\[
A_{P}P_{1}=\chi_{B_{1}}A_{\wh P}\wh P_{1}+\chi_{B_{1}}(A_{P}-A_{\wh P})\wh P_{1}+(\rd_{y}\chi_{B_{1}})\wh P_{1}.
\]
Thus the claim follows from (using $\chf_{[B_{1},\infty)}|P_{1}|_{k}\aleq_{k}\chf_{[B_{1},2B_{1}]}\beta^{2}$)
\begin{align*}
\chi_{B_{1}}(A_{P}-A_{\wh P})\wh P_{1}+(\rd_{y}\chi_{B_{1}})\wh P_{1} & =\calO(\chf_{[B_{1},2B_{1}]}\beta^{2}y^{-1})=\calR_{c}^{(1)}.
\end{align*}

\emph{\uline{Second claim.}} We claim that 
\begin{equation}
\chi_{B_{1}}A_{\wh P}\wh P_{1}=\chi_{B_{1}}A_{\wh P-T_{3}^{(0)}}\wh P_{1}+\calR_{c}^{(1)}.\label{eq:compat2-cl2}
\end{equation}
To see this, as $T_{3}^{(0)}\neq0$ only if $m=2$, we may assume
$m=2$. Thus the claim follows from (using $\wh P_{1}=\calO(\calI_{-3}^{1}+\calI_{1}^{3})$)
\begin{align*}
\chi_{B_{1}} & (A_{\wh P}-A_{\wh P-T_{3}^{(0)}})\wh P_{1}\\
 & =\calO(\chf_{[2,2B_{1}]}\tfrac{1}{y}\tint 0y(\calI_{-2}^{3}+\calI_{4}^{6})y'dy'\cdot(\calI_{-3}^{1}+\calI_{1}^{3}))\\
 & =\calO(\chf_{[2,2B_{1}]}\{\calI_{-4,1}^{4}+\calI_{0,1}^{6}+\calI_{2}^{7}+\calI_{6}^{9}\})=\calR_{c}^{(1)}.
\end{align*}
Note that the highest order term $-\frac{1}{y}A_{\tht}[T_{3}^{(0)}]T_{3}^{(1)}$
saturates the bound \eqref{eq:def-Rc1} when $m=2$.

\emph{\uline{Third claim.}} We claim that 
\begin{equation}
\begin{aligned}\chi_{B_{1}}A_{\wh P-T_{3}^{(0)}}(T_{1}^{(1)}+T_{2}^{(1)}) & =\chi_{B_{1}}\big\{\bbet^{2}\tfrac{y^{2}}{4}Q+\bbet\big(A_{\tht}[Q,T_{2}^{(0)}]Q\\
 & \qquad+A_{\tht}[Q,T_{1}^{(0)}]T_{1}^{(0)}+\tfrac{1}{2}A_{\tht}[T_{1}^{(0)}]Q\big)\big\}+\calR_{c}^{(1)}.
\end{aligned}
\label{eq:compat2-cl3}
\end{equation}
Indeed, the claim follows from (c.f. \eqref{eq:compat1-tmp})
\begin{align*}
\chi_{B_{1}} & A_{\wh P-T_{3}^{(0)}}(T_{1}^{(1)}+T_{2}^{(1)})\\
 & =\chi_{B_{1}}\{-\bbet\tfrac{y}{2}\bfD_{\wh P-T_{3}^{(0)}}(Q+T_{1}^{(0)})\}\\
 & =\chi_{B_{1}}\big\{-\bbet\tfrac{y}{2}\big(L_{Q}T_{1}^{(0)}+(-\tfrac{2}{y}A_{\tht}[Q,T_{2}^{(0)}]Q\\
 & \qquad\quad-\tfrac{2}{y}A_{\tht}[Q,T_{1}^{(0)}]T_{1}^{(0)}-\tfrac{1}{y}A_{\tht}[T_{1}^{(0)}]Q)\big)\big\}+\calO(\chf_{[2,2B_{1}]}\{\calI_{-1,1}^{4}+\calI_{1}^{5}+\calI_{3}^{6}\})
\end{align*}
and applying $L_{Q}T_{1}^{(0)}=-\bbet\tfrac{y}{2}Q$. Note that the
term $-\tfrac{1}{y}A_{\tht}[T_{2}^{(0)}]T_{2}^{(1)}$ saturates the
bound \eqref{eq:def-Rc1} when $m=1$.

\emph{\uline{Last claim.}} We claim that 
\begin{equation}
\chi_{B_{1}}A_{\wh P-T_{3}^{(0)}}T_{3}^{(1)}=\chi_{B_{1}}A_{Q}T_{3}^{(1)}+\calR_{c}^{(1)}.\label{eq:compat2-cl4}
\end{equation}
This easily follows from (using $T_{3}^{(1)}=\calI_{1}^{3}$) 
\[
\chi_{B_{1}}A_{\wh P-T_{3}^{(0)}}T_{3}^{(1)}=\chi_{B_{1}}A_{Q}T_{3}^{(1)}+\calO(\chf_{[2,2B_{1}]}\{\calI_{0}^{4}+\calI_{0,1}^{5}+\calI_{2}^{6}+\calI_{4}^{7}\}).
\]

\emph{\uline{Completion of the proof of \mbox{\eqref{eq:compat2-H2}}--\mbox{\eqref{eq:compat2-virial}}.}}
By the claims \eqref{eq:compat2-cl1}, \eqref{eq:compat2-cl2}, \eqref{eq:compat2-cl3},
and \eqref{eq:compat2-cl4}, we have proved that 
\begin{align*}
A_{P}P_{1} & =\chi_{B_{1}}\{\bbet^{2}\tfrac{y^{2}}{4}Q\}+\chi_{B_{1}}\big\{ A_{Q}T_{3}^{(1)}\\
 & \peq+\bbet(A_{\tht}[Q,T_{2}^{(0)}]Q+A_{\tht}[Q,T_{1}^{(0)}]T_{1}^{(0)}+\tfrac{1}{2}A_{\tht}[T_{1}^{(0)}]Q)\big\}+\calR_{c}^{(1)}.
\end{align*}
By the definition of $T_{2}^{(2)},T_{3}^{(1)},P_{2}$, the above display
continues as 
\[
=\chi_{B_{1}}\{T_{2}^{(2)}+T_{3}^{(2)}\}+\calR_{c}^{(1)}=P_{2}-\chi_{B_{0}}T_{4}^{(2)}+\calR_{c}^{(2)}.
\]
Since $T_{4}^{(2)}=\calI_{1,2}^{4}$ and $B_{0}=\beta^{-1/2}$ (much
smaller than $B_{1}$), we have $\chi_{B_{0}}T_{4}^{(2)}=\calR_{c}^{(2)}$
and hence the proof of \eqref{eq:compat2-H2}--\eqref{eq:compat2-virial}
is completed.

\textbf{Step 4.} Proof of \eqref{eq:Psi0-loc-bound}.

We denote by $\calR_{0}$ any function satisfying the bound (c.f.
RHS\eqref{eq:Psi0-loc-bound})
\[
\chf_{(0,R]}|\calR_{0}|\aleq_{R}\beta^{3}.
\]
We need to estimate 
\begin{align*}
\Psi & =(b^{2}+\eta^{2}+p_{3}^{(b)})i\rd_{b}P+p_{3}^{(\eta)}i\rd_{\eta}P-ib\Lmb P\\
 & \peq-(m+1)\eta P-(\tint 0y\Re(\br PP_{1})dy')P+\bfD_{P}^{\ast}P_{1}.
\end{align*}
Since the bounds are allowed to depend on $R$, we only need to show
that $\Psi$ consists of cubic or higher order terms in $b$ and $\eta$.
In particular, we do not need to care about any growing tails in $y$
and cutoffs. The computations in this step are similar to those in
\cite{KimKwon2023AnnPDE}, where the modified profile \eqref{eq:old-P-1}
was used.

We begin the proof of \eqref{eq:Psi0-loc-bound}. Since $P_{1}=-\bbet\frac{y}{2}P+\calR_{0}$,
we have 
\begin{align*}
(\tint 0y\Re(\br PP_{1})dy')P & =\eta A_{\tht}[P]P+\calR_{0},\\
\bfD_{P}^{\ast}P_{1} & =-\bbet\tfrac{y}{2}(\bfD_{P}^{\ast}-\tfrac{1}{y})P+\calR_{0}.
\end{align*}
We also note that $p_{3}^{(b)}i\rd_{b}P+p_{3}^{(\eta)}i\rd_{\eta}P=\calR_{0}$.
Thus we have 
\[
\Psi=(b^{2}+\eta^{2})i\rd_{b}P-ib\Lmb P-\eta((m+1)+A_{\tht}[P])P-\bbet\tfrac{y}{2}(\bfD_{P}^{\ast}-\tfrac{1}{y})P+\calR_{0}.
\]
Applying to the above the operator identities 
\begin{align*}
\Lmb+\tfrac{y}{2}(\bfD_{P}^{\ast}-\tfrac{1}{y}) & =\tfrac{y}{2}\bfD_{P},\\
(m+1)+A_{\tht}[P]+\tfrac{y}{2}(\bfD_{P}^{\ast}-\tfrac{1}{y}) & =-\tfrac{y}{2}\bfD_{P},
\end{align*}
we have 
\[
\Psi=(b^{2}+\eta^{2})i\rd_{b}P-(ib-\eta)\tfrac{y}{2}\bfD_{P}P+\calR_{0}.
\]
Applying $(b^{2}+\eta^{2})i\rd_{b}P=(b^{2}+\eta^{2})\tfrac{y^{2}}{4}P+\calR_{0}$
and $\bfD_{P}P=P_{1}+\calR_{0}=-\bbet\tfrac{y}{2}P+\calR_{0}$ (see
\eqref{eq:compat1-H3}) to the above display completes the proof of
\eqref{eq:Psi0-loc-bound}.

\textbf{Step 5.} Some preliminary claims for the proof of \eqref{eq:Psi1-bound}
and \eqref{eq:Psi2-H2-bound}--\eqref{eq:Psi2-H3-bound}.

In contrast to the proof of \eqref{eq:Psi0-loc-bound}, we cannot
simply replace $P_{1}$ by $-\bbet\frac{y}{2}P$ when $m=1$ because
the cubic terms of $P_{1}$ and $-\bbet\frac{y}{2}P$, namely, $T_{3}^{(1)}$
and $-\bbet\tfrac{y}{2}T_{2}^{(0)}$, are not the same and $-\bbet\frac{y}{2}T_{2}^{(0)}$
rather has a dangerous growing tail. However, we would still like
to partially replace $P_{1}$ by $-\bbet\frac{y}{2}P$ in order to
handle the algebraic computations efficiently. For this purpose, it
is convenient to introduce the notation
\begin{equation}
\td P\coloneqq Q+T_{1}^{(0)}\quad\text{and}\quad\td P_{1}\coloneqq T_{1}^{(1)}+T_{2}^{(1)}\quad\text{so that}\quad\td P_{1}=-\bbet\tfrac{y}{2}\td P.\label{eq:def-tdP-tdP1}
\end{equation}
With this additional notation, we claim the following.

\emph{\uline{(i) \mbox{$A_{\tht}$}-like terms.}} We claim 
\begin{equation}
\begin{aligned}\chi_{B_{1}}A_{\tht}[\td P] & =\chi_{B_{1}}A_{\tht}[Q]+\chi_{B_{1}}\calI_{0}^{1}+\calO(\chf_{[2,2B_{1}]}\calI_{0,1}^{2})\\
 & =\chi_{B_{1}}A_{\tht}[Q]+\calO(\chf_{[2,2B_{1}]}\calI_{0}^{1})
\end{aligned}
\label{eq:ModProfPrem1}
\end{equation}
and 
\begin{equation}
\begin{aligned}\chi_{B_{1}} & \tint 0y\Re(\br PP_{1})dy'\\
 & =\chi_{B_{1}}\{\eta A_{\tht}[\td P]\}+\calO(\chf_{[2,2B_{1}]}\{\calI_{0,1}^{3}+\calI_{2}^{4}\})\\
 & =\chi_{B_{1}}\{\eta A_{\tht}[Q]\}+\chi_{B_{1}}\calI_{0}^{2}+\calO(\chf_{[2,2B_{1}]}\{\calI_{0,1}^{3}+\calI_{2}^{4}\})\\
 & =\chi_{B_{1}}\{-2(m+1)\eta\}+\chi_{B_{1}}\calI_{-2}^{1}+\calO(\chf_{[2,2B_{1}]}\calI_{0}^{2}).
\end{aligned}
\label{eq:ModProfPrem2}
\end{equation}
The estimate \eqref{eq:ModProfPrem1} is obvious from \eqref{eq:asymp-T-lin}.
The first equality of \eqref{eq:ModProfPrem2} follows from (using
$P_{1}=\calO(\calI_{-2}^{1}+\calI_{0}^{2})$ and $P_{1}+\bbet\tfrac{y}{2}\td P=\calO(\calI_{1}^{3})$
in the region $y\leq2B_{1}$) 
\begin{align*}
\chi_{B_{1}} & \tint 0y\Re(\br PP_{1}+\eta\tfrac{y'}{2}\td P^{2})dy'\\
 & =\chi_{B_{1}}\tint 0y\Re[(\br{P-\td P})P_{1}+\br{\td P}(P_{1}+\bbet\tfrac{y'}{2}\td P)]dy'\\
 & =\calO(\chf_{[2,2B_{1}]}\tint 0y\{\calI_{1}^{2}(\calI_{-2}^{1}+\calI_{0}^{2})+(\calI_{-3}^{0}+\calI_{-1}^{1})\calI_{1}^{3}\}dy')\\
 & =\calO(\chf_{[2,2B_{1}]}\{\calI_{0,1}^{3}+\calI_{2}^{4}\}).
\end{align*}
This estimate combined with \eqref{eq:ModProfPrem1} gives the second
equality of \eqref{eq:ModProfPrem2}. Further applying $A_{\tht}[Q]=-2(m+1)+\calI_{-4}^{0}$
gives the last equality of \eqref{eq:ModProfPrem2}.

\emph{\uline{(ii) Difference estimates for \mbox{$A_{P}^{\ast}$}}}
(similar bounds also hold for $A_{P}$). We claim
\begin{equation}
\begin{aligned}\chi_{B_{1}}A_{P}^{\ast} & =\chi_{B_{1}}A_{\td P}^{\ast}+\chi_{B_{1}}\calI_{-1,1}^{2}+\calO(\chf_{[2,2B_{1}]}\{\calI_{1}^{3}+\calI_{3}^{4}\})\\
 & =\chi_{B_{1}}A_{Q}^{\ast}+\chi_{B_{1}}\calI_{-1}^{1}+\calO(\chf_{[2,2B_{1}]}\{\calI_{-1,1}^{2}+\calI_{3}^{4}\}).
\end{aligned}
\label{eq:ModProfPrem3}
\end{equation}
This follows from 
\begin{align*}
\chi_{B_{1}}\{A_{P}^{\ast}-A_{\td P}^{\ast}\} & =\chi_{B_{1}}\{-\tfrac{1}{y}(A_{\tht}[P]-A_{\tht}[\td P])\}\\
 & =\chi_{B_{1}}\{\tfrac{1}{2y}\tint 0y\calI_{-2}^{2}y'dy'\}+\calO(\chf_{[2,2B_{1}]}\tfrac{1}{y}\tint 0y(\calI_{0}^{3}+\calI_{2}^{4})y'dy')\\
 & =\chi_{B_{1}}\calI_{-1,1}^{2}+\calO(\chf_{[2,2B_{1}]}\{\calI_{1}^{3}+\calI_{3}^{4}\})
\end{align*}
and \eqref{eq:ModProfPrem1}.

\emph{\uline{(iii) Difference estimates for \mbox{$\td H_{P}$}}}
(and the similar bounds hold for $H_{P}$). We claim
\begin{equation}
\begin{aligned}\chi_{B_{1}}\td H_{P} & =\chi_{B_{1}}\td H_{\td P}+\chi_{B_{1}}\calI_{-2,1}^{2}+\calO(\chf_{[2,2B_{1}]}\{\calI_{0}^{3}+\calI_{2}^{4}\})\\
 & =\chi_{B_{1}}\td H_{Q}+\chi_{B_{1}}\calI_{-2}^{1}+\calO(\chf_{[2,2B_{1}]}\{\calI_{-2,1}^{2}+\calI_{2}^{4}\}).
\end{aligned}
\label{eq:ModProfPrem4}
\end{equation}
The first equality follows from the expression
\[
\td H_{P}-\td H_{\td P}=\tfrac{1}{y}(2m+4+A_{\tht}[P]+A_{\tht}[\td P])\cdot\tfrac{1}{y}(A_{\tht}[P]-A_{\tht}[\td P])+\tfrac{1}{2}(|P|^{2}-|\td P|^{2}),
\]
the estimate $2m+4+A_{\tht}[P]+A_{\tht}[\td P]=\calI_{0}^{0}+\calO(\chf_{[2,2B_{1}]}\beta)$
in the region $y\leq2B_{1}$, and the first part of \eqref{eq:ModProfPrem3}.
The second equality can be proved similarly, using the second part
of \eqref{eq:ModProfPrem3} instead.

\textbf{Step 6.} Key claim for the proof of \eqref{eq:Psi1-bound}.

We denote by $\calR_{1}$ any function satisfying the bound (c.f.
RHS\eqref{eq:Psi1-bound})
\begin{equation}
\|\tfrac{1}{y^{2}}\calR_{1}\|_{L^{1}}\aleq\beta^{4}|\log\beta|^{3}.\label{eq:psi1-bd-00}
\end{equation}
The goal of this step is to prove the following: 
\begin{equation}
\begin{aligned}(b^{2}+\eta^{2}+p_{3}^{(b)})i\rd_{b}P_{1}+p_{3}^{(\eta)}i\rd_{\eta}P_{1}-ib\Lmb_{-1}P_{1}-(m+1)\eta P_{1}\qquad\\
-(\tint 0y\Re(\br PP_{1})dy')P_{1}+A_{P}^{\ast}(\chi_{B_{1}}\{T_{2}^{(2)}+T_{3}^{(2)}\}) & =\calR_{1}.
\end{aligned}
\label{eq:psi1-bd-01}
\end{equation}
Notice that LHS\eqref{eq:psi1-bd-01} is equal to $\Psi_{1}-A_{P}^{\ast}P_{2}$,
but the expression itself does not involve $T_{4}^{(2)}$ that has
not been chosen yet.

To show \eqref{eq:psi1-bd-01}, we write 
\begin{equation}
\begin{aligned} & \text{LHS\eqref{eq:psi1-bd-01}}\\
 & =A_{P}^{\ast}(\chi_{B_{1}}\{T_{2}^{(2)}+T_{3}^{(2)}\})\\
 & \peq+\chi_{B_{1}}\{[(b^{2}+\eta^{2})i\rd_{b}-ib\Lmb_{-1}-\eta(m+1)-(\tint 0y\Re(\br PP_{1})dy')]\td P_{1}\}\\
 & \peq+\chi_{B_{1}}\{[(b^{2}+\eta^{2})i\rd_{b}-ib\Lmb_{-1}-\eta(m+1)-(\tint 0y\Re(\br PP_{1})dy')]T_{3}^{(1)}\}\\
 & \peq+\{[(b^{2}+\eta^{2})i\rd_{b}-iby\rd_{y}]\chi_{B_{1}}\}\wh P_{1}\\
 & \peq+[p_{3}^{(b)}i\rd_{b}+p_{3}^{(\eta)}i\rd_{\eta}]P_{1}.
\end{aligned}
\label{eq:psi1-bd-03}
\end{equation}
In what follows, we show that each line of RHS\eqref{eq:psi1-bd-03}
belongs to $\calR_{1}$.

\emph{\uline{First line of RHS\mbox{\eqref{eq:psi1-bd-03}}.}}
We claim that 
\begin{equation}
\begin{aligned} & A_{P}^{\ast}(\chi_{B_{1}}\{T_{2}^{(2)}+T_{3}^{(2)}\})\\
 & =\chi_{B_{1}}\{-\bbet A_{\td P}^{\ast}(\tfrac{y}{2}\td P_{1})\}+\chi_{B_{1}}\{A_{Q}^{\ast}(T_{3}^{(2)}-\bbet^{2}\tfrac{y^{2}}{4}T_{1}^{(0)})\}+\calR_{1}.
\end{aligned}
\label{eq:psi1-bd-04}
\end{equation}
To see this, we start from writing 
\begin{align*}
 & A_{P}^{\ast}(\chi_{B_{1}}\{T_{2}^{(2)}+T_{3}^{(2)}\})\\
 & =\chi_{B_{1}}\{A_{\td P}^{\ast}T_{2}^{(2)}+A_{Q}^{\ast}T_{3}^{(2)}\}+\chi_{B_{1}}\{(A_{P}^{\ast}-A_{\td P}^{\ast})T_{2}^{(2)}+(A_{P}^{\ast}-A_{Q}^{\ast})T_{3}^{(2)})\}\\
 & \peq-(\rd_{y}\chi_{B_{1}})(T_{2}^{(2)}+T_{3}^{(2)}).
\end{align*}
We then substitute $T_{2}^{(2)}=\bbet^{2}\frac{y^{2}}{4}Q=-\bbet\tfrac{y}{2}\td P_{1}-\bbet^{2}\tfrac{y^{2}}{4}T_{1}^{(0)}$
into the above to obtain 
\begin{align*}
A_{P}^{\ast}P_{2} & =\chi_{B_{1}}\{-\bbet A_{\td P}^{\ast}(\tfrac{y}{2}\td P_{1})+A_{Q}^{\ast}(T_{3}^{(2)}-\bbet^{2}\tfrac{y^{2}}{4}T_{1}^{(0)})\}\\
 & \peq+\chi_{B_{1}}\{(A_{P}^{\ast}-A_{\td P}^{\ast})T_{2}^{(2)}+(A_{P}^{\ast}-A_{Q}^{\ast})T_{3}^{(2)}-(A_{\td P}^{\ast}-A_{Q}^{\ast})(\bbet^{2}\tfrac{y^{2}}{4}T_{1}^{(0)})\}\\
 & \peq-(\rd_{y}\chi_{B_{1}})(T_{2}^{(2)}+T_{3}^{(2)}).
\end{align*}
Hence it remains to show that all the terms in the last two lines
belong to $\calR_{1}$. Using \eqref{eq:ModProfPrem3}, we have 
\begin{align*}
\chi_{B_{1}}\{ & (A_{P}^{\ast}-A_{\td P}^{\ast})T_{2}^{(2)}\}-(\rd_{y}\chi_{B_{1}})T_{2}^{(2)}\\
 & =\calO((\chf_{[2,2B_{1}]}\{\calI_{-1,1}^{2}+\calI_{3}^{4}\}+\chf_{[B_{1},2B_{1}]}\tfrac{1}{y})\cdot\calI_{-1}^{2})\\
 & =\calO(\chf_{[2,2B_{1}]}\{\calI_{-2,1}^{4}+\calI_{2}^{6}\}+\chf_{[B_{1},2B_{1}]}\calI_{-2}^{2})=\calR_{1}.
\end{align*}
Using \eqref{eq:ModProfPrem3} and \eqref{eq:asymp-T-cub-rough},
we have 
\begin{align*}
\chi_{B_{1}}\{ & (A_{P}^{\ast}-A_{Q}^{\ast})T_{3}^{(2)})\}-(\rd_{y}\chi_{B_{1}})T_{3}^{(2)}\\
 & =\calO((\chf_{[2,2B_{1}]}\{\calI_{-1}^{1}+\calI_{3}^{4}\}+\chf_{[B_{1},2B_{1}]}\tfrac{1}{y})\cdot\calI_{0}^{3})\\
 & =\calO(\chf_{[2,2B_{1}]}\{\calI_{-1}^{4}+\calI_{3}^{7}\}+\chf_{[B_{1},2B_{1}]}\calI_{-1}^{3})=\calR_{1}.
\end{align*}
Using \eqref{eq:ModProfPrem1} and \eqref{eq:asymp-T-lin}, we have
\[
\chi_{B_{1}}\{(A_{\td P}^{\ast}-A_{Q}^{\ast})(\bbet^{2}\tfrac{y^{2}}{4}T_{1}^{(0)})\}=\calO(\chf_{[2,2B_{1}]}\calI_{-1}^{1}\cdot\calI_{1}^{3})=\calO(\chf_{[2,2B_{1}]}\calI_{0}^{4})=\calR_{1}.
\]
This completes the proof of the claim \eqref{eq:psi1-bd-04}.

\emph{\uline{Second line of RHS\mbox{\eqref{eq:psi1-bd-03}}.}}
We claim that
\begin{equation}
\chi_{B_{1}}(\tint 0y\Re(\br PP_{1})dy')\td P_{1}=\chi_{B_{1}}\{\eta A_{\tht}[\td P]\td P_{1}\}+\calR_{1}.\label{eq:psi1-bd-05}
\end{equation}
Indeed, this follows from \eqref{eq:ModProfPrem2}: 
\begin{align*}
\chi_{B_{1}}( & \tint 0y\Re(\br PP_{1})dy')\td P_{1}\\
 & =\chi_{B_{1}}\{\eta A_{\tht}[\td P]\td P_{1}\}+\calO(\chf_{[2,2B_{1}]}(\calI_{0,1}^{3}+\calI_{2}^{4})(\calI_{-2}^{1}+\calI_{0}^{2}))\\
 & =\chi_{B_{1}}\{\eta A_{\tht}[\td P]\td P_{1}\}+\calO(\chf_{[2,2B_{1}]}\{\calI_{-2,1}^{4}+\calI_{0,1}^{5}+\calI_{2}^{6}\})\\
 & =\chi_{B_{1}}\{\eta A_{\tht}[\td P]\td P_{1}\}+\calR_{1}.
\end{align*}

\emph{\uline{Third line of RHS\mbox{\eqref{eq:psi1-bd-03}}.}}
We claim that 
\begin{equation}
\chi_{B_{1}}\{[(b^{2}+\eta^{2})i\rd_{b}-ib\Lmb_{-1}-\eta(m+1)-(\tint 0y\Re(\br PP_{1})dy')]T_{3}^{(1)}\}=\calR_{1}.\label{eq:psi1-bd-06}
\end{equation}
We separate into two cases: (i) $m=1$ or $m\geq3$, and (ii) $m=2$.
In the first case, i.e., when $m=1$ or $m\geq3$, the square bracket
is of size $\calO(\beta)$ and we can use the spatial decay $T_{3}^{(1)}=\calI_{0,1}^{3}$
(see \eqref{eq:asym-T-cub-m1}--\eqref{eq:asym-T-cub-mgeq2}) to
have 
\[
\chi_{B_{1}}\{[(b^{2}+\eta^{2})i\rd_{b}-ib\Lmb_{-1}-\eta(m+1)-(\tint 0y\Re(\br PP_{1})dy')]T_{3}^{(1)}\}=\calO(\chf_{[2,2B_{1}]}\calI_{0,1}^{4}).
\]
In the latter case, i.e., when $m=2$, the bound $T_{3}^{(1)}=\calO(\calI_{1}^{3})$
is insufficient. However, we can exploit some spatial cancellation
using the spatial asymptotics \eqref{eq:asym-T-cub-mgeq2}. We handle
the integral term in the square bracket using \eqref{eq:ModProfPrem2}
\begin{align*}
\chi_{B_{1}}\{-\tint 0y\Re(\br PP_{1})dy'\} & =\chi_{B_{1}}\{\eta\tfrac{1}{2}\tint 0yQ^{2}y'dy'\}+\calO(\chf_{[2,2B_{1}]}\calI_{0}^{2})\\
 & =\chi_{B_{1}}\{\eta(2m+2)\}+\chi_{B_{1}}\calI_{-2}^{1}+\calO(\chf_{[2,2B_{1}]}\calI_{0}^{2}),
\end{align*}
from which we obtain 
\begin{align*}
\chi_{B_{1}}\{ & [(b^{2}+\eta^{2})i\rd_{b}-ib\Lmb_{-1}-\eta(m+1)-(\tint 0y\Re(\br PP_{1})dy')]T_{3}^{(1)}\}\\
 & =\chi_{B_{1}}\{[(b^{2}+\eta^{2})i\rd_{b}-ib\Lmb_{-1}+\eta(m+1)]T_{3}^{(1)}\}+\calO(\chf_{[2,2B_{1}]}\{\calI_{-1}^{4}+\calI_{1}^{5}\}).
\end{align*}
We apply \eqref{eq:asym-T-cub-mgeq2} to the above. Observe that \emph{the
contribution of the tail $\bbet^{3}y^{-(m-3)}=\bbet^{3}y$ of $T_{3}^{(1)}$
cancels when $m=2$}:
\begin{align*}
 & [(b^{2}+\eta^{2})i\rd_{b}-ib\Lmb_{-1}+\eta(m+1)]\bbet^{3}y\\
 & \qquad=[(b^{2}+\eta^{2})(-3\bbet^{2})-3ib\bbet^{3}+3\eta\bbet^{3})]y=0.
\end{align*}
Therefore, we have the \emph{improved spatial decay} 
\[
\chi_{B_{1}}\{[(b^{2}+\eta^{2})i\rd_{b}-ib\Lmb_{-1}+\eta(m+1)]T_{3}^{(1)}\}=\calO(\chf_{[2,2B_{1}]}\calI_{-1,1}^{4})=\calR_{1}.
\]
This completes the proof of the claim \eqref{eq:psi1-bd-06}.

\emph{\uline{Fourth line of RHS\mbox{\eqref{eq:psi1-bd-03}}.}}
Thanks to the \emph{special cancellation \eqref{eq:cutoff-cancellation}
of the cutoff}, we have 
\begin{equation}
\{[(b^{2}+\eta^{2})i\rd_{b}-iy\rd_{y}]\chi_{B_{1}}\}\wh P_{1}=0.\label{eq:psi1-bd-07}
\end{equation}

\emph{\uline{Last line of RHS\mbox{\eqref{eq:psi1-bd-03}}.}} We
claim that 
\begin{equation}
[p_{3}^{(b)}i\rd_{b}+p_{3}^{(\eta)}i\rd_{\eta}]P_{1}=\chi_{B_{1}}\{p_{3}^{(b)}\tfrac{y}{2}Q-p_{3}^{(\eta)}i\tfrac{y}{2}Q\}+\calR_{1}.\label{eq:psi1-bd-08}
\end{equation}
Indeed, this follows from 
\begin{align*}
[p_{3}^{(b)}i\rd_{b}+p_{3}^{(\eta)}i\rd_{\eta}]P_{1} & =\chi_{B_{1}}\{p_{3}^{(b)}\tfrac{y}{2}Q-p_{3}^{(\eta)}i\tfrac{y}{2}Q\}\\
 & \peq+\chi_{B_{1}}\{[p_{3}^{(b)}i\rd_{b}+p_{3}^{(\eta)}i\rd_{\eta}](\wh P_{1}+\bbet\tfrac{y}{2}Q)\}\\
 & \peq+\{[p_{3}^{(b)}i\rd_{b}+p_{3}^{(\eta)}i\rd_{\eta}]\chi_{B_{1}}\}\wh P_{1}\\
 & =\chi_{B_{1}}\{p_{3}^{(b)}\tfrac{y}{2}Q-p_{3}^{(\eta)}i\tfrac{y}{2}Q\}+\calO(\chf_{[2,2B_{1}]}\beta^{4}).
\end{align*}

\emph{\uline{Completion of the proof of \mbox{\eqref{eq:psi1-bd-01}}.}}
Substituting the claims \eqref{eq:psi1-bd-04}--\eqref{eq:psi1-bd-08}
into \eqref{eq:psi1-bd-03}, we arrive at 
\begin{align*}
\text{LHS\eqref{eq:psi1-bd-01}} & =\chi_{B_{1}}\{-\bbet A_{\td P}^{\ast}(\tfrac{y}{2}\td P_{1})\}+\chi_{B_{1}}\{A_{Q}^{\ast}(T_{3}^{(2)}-\bbet^{2}\tfrac{y^{2}}{4}T_{1}^{(0)})\}\\
 & \peq+\chi_{B_{1}}\{[(b^{2}+\eta^{2})i\rd_{b}-ib\Lmb_{-1}-\eta((m+1)+A_{\tht}[\td P])]\td P_{1}\}\\
 & \peq+\chi_{B_{1}}\{p_{3}^{(b)}\tfrac{y}{2}Q-p_{3}^{(\eta)}i\tfrac{y}{2}Q\}+\calR_{1}.
\end{align*}
Applying to the above the operator identities 
\begin{align*}
A_{\td P}^{\ast}(\tfrac{y}{2}\cdot)+\Lmb_{-1} & =\tfrac{y}{2}A_{\td P}+1,\\
A_{\td P}^{\ast}(\tfrac{y}{2}\cdot)+m+1+A_{\tht}[\td P] & =-(\tfrac{y}{2}A_{\td P}+1),
\end{align*}
we obtain 
\begin{equation}
\begin{aligned}\text{LHS\eqref{eq:psi1-bd-01}} & =\chi_{B_{1}}\{[(b^{2}+\eta^{2})i\rd_{b}-(ib-\eta)(\tfrac{y}{2}A_{\td P}+1)]\td P_{1}\}\\
 & \peq+\chi_{B_{1}}\{p_{3}^{(b)}\tfrac{y}{2}Q-p_{3}^{(\eta)}i\tfrac{y}{2}Q+A_{Q}^{\ast}(T_{3}^{(2)}-\bbet^{2}\tfrac{y^{2}}{4}T_{1}^{(0)})\}+\calR_{1}.
\end{aligned}
\label{eq:psi1-bd-09}
\end{equation}
A direct computation using $\td P_{1}=-\bbet\tfrac{y}{2}\td P$, $T_{1,0}^{(0)}=-i\tfrac{y^{2}}{4}Q$,
and $L_{Q}T_{1}^{(0)}+\bbet\tfrac{y}{2}Q=0$ gives 
\begin{align*}
[(b^{2} & +\eta^{2})i\rd_{b}-(ib-\eta)(\tfrac{y}{2}A_{\td P}+1)]\td P_{1}\\
 & =(b^{2}+\eta^{2})(\tfrac{y}{2}\td P-i\bbet\tfrac{y}{2}T_{1,0}^{(0)}-(\tfrac{y}{2}A_{\td P}+1)(\tfrac{y}{2}\td P))\\
 & =(b^{2}+\eta^{2})(-i\bbet\tfrac{y}{2}T_{1,0}^{(0)}-\tfrac{y^{2}}{4}\bfD_{\td P}\td P)\\
 & =-(b^{2}+\eta^{2})\tfrac{y^{2}}{4}(\bbet\tfrac{y}{2}Q+L_{Q}T_{1}^{(0)}+N_{Q}(T_{1}^{(0)}))\\
 & =-(b^{2}+\eta^{2})\tfrac{y^{2}}{4}N_{Q}(T_{1}^{(0)}),
\end{align*}
which belongs to $\calR_{1}$: 
\begin{align*}
\chi_{B_{1}}\{ & -(b^{2}+\eta^{2})\tfrac{y^{2}}{4}N_{Q}(T_{1}^{(0)})\}\\
 & =\calO(\chf_{[2,2B_{1}]}\calI_{2}^{2}\cdot(\tfrac{1}{y}\tint 0y(\calI_{-4}^{1}+\calI_{-2}^{2})y'dy'\cdot\calI_{-1}^{1}+\tfrac{1}{y}\tint 0y\calI_{-2}^{2}y'dy'\cdot\calI_{-3}^{0}))\\
 & =\calO(\chf_{[2,2B_{1}]}\calI_{0}^{4})=\calR_{1}.
\end{align*}
Therefore, \eqref{eq:psi1-bd-09} becomes 
\begin{equation}
\text{LHS\eqref{eq:psi1-bd-01}}=\chi_{B_{1}}\{p_{3}^{(b)}\tfrac{y}{2}Q-p_{3}^{(\eta)}i\tfrac{y}{2}Q+A_{Q}^{\ast}(T_{3}^{(2)}-\bbet^{2}\tfrac{y^{2}}{4}T_{1}^{(0)})\}+\calR_{1}.\label{eq:psi1-bd-10}
\end{equation}
By the definition of $p_{3}^{(b)}$, $p_{3}^{(\eta)}$, and $T_{3}^{(2)}$
(see \eqref{eq:def-T2_3} and \eqref{eq:def-p_3}), the curly bracket
of \eqref{eq:psi1-bd-10} \emph{vanishes}. This completes the proof
of \eqref{eq:psi1-bd-01}.

\textbf{Step 7.} Key claim for the construction of $T_{4}^{(2)}$
and the proof of \eqref{eq:Psi2-H2-bound}--\eqref{eq:Psi2-H3-bound}.

We denote by $\calR_{2}$ any function satisfying the bounds (c.f.
RHS\eqref{eq:Psi2-H2-bound}--RHS\eqref{eq:Psi2-H3-bound})
\begin{equation}
\|\calR_{2}\|_{L_{m+2}^{2}}\aleq\beta^{4}|\log\beta|^{2}\quad\text{and}\quad\|\calR_{2}\|_{\dot{H}_{m+2}^{1}}\aleq\beta^{9/2}|\log\beta|^{2}.\label{eq:psi2-bd-21}
\end{equation}
Denote by $\calA_{4}^{(2)}$ any function $f(y;b,\eta)$ of the form
$f=\sum_{i+j=4}b^{i}\eta^{j}f_{i,j}(y)$ such that each $f_{i,j}$
satisfies the pointwise bounds 
\begin{equation}
|f_{i,j}|_{k}\aleq_{k}\begin{cases}
y^{-m}\log y & \text{if }m\in\{1,2\},\\
y^{-m+2} & \text{if }m\geq3.
\end{cases}\label{eq:psi2-bd-22}
\end{equation}
The goal of this step is to prove the following: 
\begin{equation}
\begin{aligned}\big[(b^{2}+\eta^{2}+p_{3}^{(b)})i\rd_{b}+p_{3}^{(\eta)}i\rd_{\eta}-ib\Lmb_{-2}-(m+1)\eta\qquad\\
-(\tint 0y\Re(\br PP_{1})dy')+\td H_{P}\big](\chi_{B_{1}}\{T_{2}^{(2)}+T_{3}^{(2)}\})-\br PP_{1}^{2} & =\chi_{B_{1}}F_{4}^{(2)}+\calR_{2}
\end{aligned}
\label{eq:psi2-bd-00}
\end{equation}
for some function $F_{4}^{(2)}$ of class $\calA_{4}^{(2)}$. Notice
that LHS\eqref{eq:psi2-bd-00} does not involve $T_{4}^{(2)}$ that
has not been chosen yet.

To show \eqref{eq:psi2-bd-00}, we write 
\begin{equation}
\begin{aligned} & \text{LHS\eqref{eq:psi2-bd-00}}\\
 & =\td H_{P}(\chi_{B_{1}}\{T_{2}^{(2)}+T_{3}^{(2)}\})-\br PP_{1}^{2}\\
 & \peq+\chi_{B_{1}}\{[(b^{2}+\eta^{2})i\rd_{b}-ib\Lmb_{-2}-(m+1)\eta-(\tint 0y\Re(\br PP_{1})dy')]T_{2}^{(2)}\}\\
 & \peq+\chi_{B_{1}}\{[(b^{2}+\eta^{2})i\rd_{b}-ib\Lmb_{-2}-(m+1)\eta-(\tint 0y\Re(\br PP_{1})dy')]T_{3}^{(2)}\}\\
 & \peq+\{[(b^{2}+\eta^{2})i\rd_{b}-iby\rd_{y}]\chi_{B_{1}}\}(T_{2}^{(2)}+T_{3}^{(2)})\\
 & \peq+[p_{3}^{(b)}i\rd_{b}+p_{3}^{(\eta)}i\rd_{\eta}](\chi_{B_{1}}\{T_{2}^{(2)}+T_{3}^{(2)}\})
\end{aligned}
\label{eq:psi2-bd-01}
\end{equation}
In the following, we show that each line of RHS\eqref{eq:psi2-bd-01}
takes the form $\chi_{B_{1}}\calA_{4}^{(2)}+\calR_{2}$.

\emph{\uline{First line of RHS\mbox{\eqref{eq:psi2-bd-01}}.}}
We claim that 
\begin{equation}
\begin{aligned} & \td H_{P}(\chi_{B_{1}}\{T_{2}^{(2)}+T_{3}^{(2)}\})-\br PP_{1}^{2}\\
 & \quad=\chi_{B_{1}}\{-\bbet^{3}(A_{Q}^{\ast}-\tfrac{1}{y})(\tfrac{y^{3}}{8}Q)\}+\chi_{B_{1}}\calA_{4}^{(2)}+\calR_{2}.
\end{aligned}
\label{eq:psi2-bd-16}
\end{equation}
To see this, we first treat the cutoff errors: 
\begin{align*}
 & \td H_{P}(\chi_{B_{1}}\{T_{2}^{(2)}+T_{3}^{(2)}\})-\br PP_{1}^{2}\\
 & \quad=\chi_{B_{1}}\td H_{P}(T_{2}^{(2)}+T_{3}^{(2)})-\chi_{B_{1}}\br{\wh P}\wh P_{1}^{2}+\calO(\chf_{[B_{1},2B_{1}]}\calI_{-3}^{2}).
\end{align*}
The term $\chi_{B_{1}}\br{\wh P}\wh P_{1}^{2}$ can be written as
\begin{align*}
 & \chi_{B_{1}}\br{\wh P}\wh P_{1}^{2}\\
 & =\chi_{B_{1}}\br{\td P}\td P_{1}^{2}+\chi_{B_{1}}\calI_{-3m}^{4}+\calO(\chf_{[2,2B_{1}]}\{\calI_{-1}^{5}+\calI_{1}^{6}\})\\
 & =\chi_{B_{1}}\{|\td P|^{2}(\bbet^{2}\tfrac{y^{2}}{4}(Q+T_{1}^{(0)}))\}+\chi_{B_{1}}\calI_{-3m}^{4}+\calO(\chf_{[2,2B_{1}]}\{\calI_{-1}^{5}+\calI_{1}^{6}\})\\
 & =\chi_{B_{1}}\{|\td P|^{2}\cdot T_{2}^{(2)}+Q^{2}\cdot\bbet^{2}\tfrac{y^{2}}{4}T_{1}^{(0)}\}+\chi_{B_{1}}\calA_{4}^{(2)}+\calR_{2}.
\end{align*}
Next, we use \eqref{eq:ModProfPrem4}, $T_{2}^{(2)}=\calI_{-m}^{2}$,
$T_{3}^{(2)}=\calI_{-m+2}^{3}$ (when it contributes to the quartic
term) and $T_{3}^{(2)}=\calI_{0}^{3}$ (when it contributes to the
remainder, see \eqref{eq:asymp-T-cub-rough}) to have 
\begin{align*}
\chi_{B_{1}}\td H_{P}T_{2}^{(2)} & =\chi_{B_{1}}\td H_{\td P}T_{2}^{(2)}+\chi_{B_{1}}\calI_{-(m+2),1}^{4}+\calO(\chf_{[2,2B_{1}]}\{\calI_{-1}^{5}+\calI_{1}^{6}\}),\\
\chi_{B_{1}}\td H_{P}T_{3}^{(2)} & =\chi_{B_{1}}\td H_{Q}T_{3}^{(2)}+\chi_{B_{1}}\calI_{-m}^{4}+\calO(\chf_{[2,2B_{1}]}\{\calI_{-2,1}^{5}+\calI_{2}^{7}\}),
\end{align*}
By the previous three displays, we have proved that 
\begin{equation}
\begin{aligned} & \td H_{P}(\chi_{B_{1}}\{T_{2}^{(2)}+T_{3}^{(2)}\})-\br PP_{1}^{2}\\
 & =\chi_{B_{1}}\{(\td H_{\td P}-|\td P|^{2})T_{2}^{(2)}+(\td H_{Q}T_{3}^{(2)}-Q^{2}\cdot\bbet^{2}\tfrac{y^{2}}{4}T_{1}^{(0)})\}+\chi_{B_{1}}\calA_{4}^{(2)}+\calR_{2}.
\end{aligned}
\label{eq:psi2-bd-17}
\end{equation}
For the first term of RHS\eqref{eq:psi2-bd-17}, using the operator
identity 
\[
\td H_{\td P}-|\td P|^{2}=(A_{\td P}^{\ast}-\tfrac{1}{y})(A_{\td P}-\tfrac{1}{y})
\]
and the cancellation $(A_{Q}-\tfrac{1}{y})T_{2}^{(2)}=0$, we have
\[
\chi_{B_{1}}\{(\td H_{\td P}-|\td P|^{2})T_{2}^{(2)}\}=\chi_{B_{1}}\{(A_{\td P}^{\ast}-\tfrac{1}{y})(A_{\td P}-A_{Q})T_{2}^{(2)}\}.
\]
We decompose this into the cubic, quartic, and the remainder terms:
\begin{align*}
 & \chi_{B_{1}}\{(A_{\td P}^{\ast}-\tfrac{1}{y})(A_{\td P}-A_{Q})T_{2}^{(2)}\}\\
 & =\chi_{B_{1}}\{[A_{Q}^{\ast}-\tfrac{1}{y}](-\tfrac{2}{y}A_{\tht}[Q,T_{1}^{(0)}]T_{2}^{(2)})\}\\
 & \peq+\chi_{B_{1}}\{[A_{Q}^{\ast}-\tfrac{1}{y}](\tfrac{1}{y}\tint 0y\calI_{-2}^{2}y'dy'\cdot\calI_{-m}^{2})+\tfrac{1}{y^{2}}\tint 0y\calI_{-4}^{1}y'dy'\tint 0y\calI_{-4}^{1}y'dy'\cdot\calI_{-m}^{2}\}\\
 & \peq+\calO(\chf_{[2,2B_{1}]}(\tfrac{1}{y^{2}}\tint 0y(\calI_{-4}^{1}+\calI_{-2}^{2})y'dy'\tint 0y\calI_{-2}^{2}y'dy')\cdot\calI_{-1}^{2})\\
 & =\chi_{B_{1}}\{\bbet^{2}[A_{Q}^{\ast}-\tfrac{1}{y}](-\tfrac{y}{2}A_{\tht}[Q,T_{1}^{(0)}]Q)\}\\
 & \peq+\chi_{B_{1}}\{\calI_{-m-2,1}^{4}\}+\calO(\chf_{[2,2B_{1}]}(\calI_{-3,1}^{5}+\calI_{-3,2}^{6})).
\end{align*}
For the second term of RHS\eqref{eq:psi2-bd-17}, we use $\td H_{Q}T_{3}^{(2)}=\bbet^{2}\td H_{Q}(\tfrac{y^{2}}{4}T_{1}^{(0)})$
(see \eqref{eq:def-T2_3}) to have 
\begin{align*}
 & \chi_{B_{1}}\{\td H_{Q}T_{3}^{(2)}-\bbet^{2}\tfrac{y^{2}}{4}Q^{2}T_{1}^{(0)}\}\\
 & =\chi_{B_{1}}\{\bbet^{2}[\td H_{Q}-Q^{2}](\tfrac{y^{2}}{4}T_{1}^{(0)})\}\\
 & =\chi_{B_{1}}\{\bbet^{2}(A_{Q}^{\ast}-\tfrac{1}{y})(A_{Q}-\tfrac{1}{y})(\tfrac{y^{2}}{4}T_{1}^{(0)})\}\\
 & =\chi_{B_{1}}\{\bbet^{2}[A_{Q}^{\ast}-\tfrac{1}{y}](\tfrac{y^{2}}{4}\bfD_{Q}T_{1}^{(0)})\}.
\end{align*}
Substituting the previous three displays into \eqref{eq:psi2-bd-17}
gives 
\begin{align*}
 & \td H_{P}(\chi_{B_{1}}\{T_{2}^{(2)}+T_{3}^{(2)}\})-\br PP_{1}^{2}\\
 & =\chi_{B_{1}}\{\bbet^{2}[A_{Q}^{\ast}-\tfrac{1}{y}](\tfrac{y^{2}}{4}\bfD_{Q}T_{1}^{(0)}-\tfrac{y}{2}A_{\tht}[Q,T_{1}^{(0)}]Q)\}+\chi_{B_{1}}\calA_{4}^{(2)}+\calR_{2}.
\end{align*}
Applying the identity 
\[
\bfD_{Q}T_{1}^{(0)}-\tfrac{2}{y}A_{\tht}[Q,T_{1}^{(0)}]Q=L_{Q}T_{1}^{(0)}=-\bbet\tfrac{y}{2}Q
\]
completes the proof of the claim \eqref{eq:psi2-bd-16}.

\emph{\uline{Second line of RHS\mbox{\eqref{eq:psi2-bd-01}}.}}
We claim that 
\begin{equation}
\begin{aligned}\chi_{B_{1}}\{[(b^{2}+\eta^{2})i\rd_{b}-ib\Lmb_{-2}-(m+1)\eta-(\tint 0y\Re(\br PP_{1})dy')]T_{2}^{(2)}\}\qquad\\
=\chi_{B_{1}}\{-\bbet^{3}\rd_{y}(\tfrac{y^{3}}{4}Q)\}+\chi_{B_{1}}\calA_{4}^{(2)}+\calR_{2}.
\end{aligned}
\label{eq:psi2-bd-11}
\end{equation}
To see this, we rewrite LHS\eqref{eq:psi2-bd-11} as 
\begin{align*}
 & \chi_{B_{1}}\{[(b^{2}+\eta^{2})i\rd_{b}-ib\Lmb_{-2}-(m+1)\eta-(\tint 0y\Re(\br PP_{1})dy')]T_{2}^{(2)}\}\\
 & =\chi_{B_{1}}\{[(b^{2}+\eta^{2})i\rd_{b}-ib\Lmb_{-2}-\eta((m+1)+A_{\tht}[Q])]T_{2}^{(2)}\}\\
 & \peq-\chi_{B_{1}}\{(\tint 0y\Re(\br PP_{1})dy'-\eta A_{\tht}[Q])T_{2}^{(2)}\}.
\end{align*}
For the first term, we use $T_{2}^{(2)}=\bbet^{2}\tfrac{y^{2}}{4}Q$
and $(m+A_{\tht}[Q])Q=y\rd_{y}Q$ (by \eqref{eq:Bogomol'nyi-eq})
to have the algebraic identity 
\begin{align*}
[(b^{2} & +\eta^{2})i\rd_{b}-ib\Lmb_{-2}-\eta((m+1)+A_{\tht}[Q])]T_{2}^{(2)}\\
 & =[2(ib-\eta)-ib\Lmb_{-2}-\eta(y\rd_{y}-1)](\bbet^{2}\tfrac{y^{2}}{4}Q)\\
 & =[-\bbet(y\rd_{y}+1)](\bbet^{2}\tfrac{y^{2}}{4}Q)=-\bbet^{3}\rd_{y}(\tfrac{y^{3}}{4}Q).
\end{align*}
For the second term, we use \eqref{eq:ModProfPrem2} and $T_{2}^{(2)}=\calI_{-m}^{2}$
to have 
\begin{align*}
\chi_{B_{1}}\{ & (\tint 0y\Re(\br PP_{1})dy'-\eta A_{\tht}[Q])T_{2}^{(2)}\}\\
 & =\chi_{B_{1}}\calI_{-m}^{4}+\calO(\chf_{[2,2B_{1}]}\{\calI_{-m,1}^{5}+\calI_{-m+2}^{6}\})=\chi_{B_{1}}\calA_{4}^{(2)}+\calR_{2}.
\end{align*}
This completes the proof of \eqref{eq:psi2-bd-11}.

\emph{\uline{Third line of RHS\mbox{\eqref{eq:psi2-bd-01}}.}}
We claim that 
\begin{equation}
\begin{aligned}\chi_{B_{1}}\{[(b^{2}+\eta^{2})i\rd_{b}-ib\Lmb_{-2}-(m+1)\eta-(\tint 0y\Re(\br PP_{1})dy')]T_{3}^{(2)}\}\\
=\chi_{B_{1}}\calA_{4}^{(2)} & +\calR_{2}.
\end{aligned}
\label{eq:psi2-bd-12}
\end{equation}
Indeed, the proof of this claim is very similar to that of \eqref{eq:psi1-bd-06}.
Using $T_{3}^{(2)}=\calI_{0}^{3}$ and \eqref{eq:ModProfPrem2} and
proceeding as in the proof of \eqref{eq:psi1-bd-06}, we have 
\begin{align*}
 & \chi_{B_{1}}\{[(b^{2}+\eta^{2})i\rd_{b}-ib\Lmb_{-2}-\eta(m+1)-(\tint 0y\Re(\br PP_{1})dy')]T_{3}^{(2)}\}\\
 & =\chi_{B_{1}}\{[(b^{2}+\eta^{2})i\rd_{b}-ib\Lmb_{-2}+\eta(m+1)+\calI_{-2}^{1}+\calO(\calI_{0}^{2})]T_{3}^{(2)}\}.
\end{align*}
When $m=1$, we apply the improved decay $T_{3}^{(2)}=\calI_{-1,1}^{3}$
from \eqref{eq:asym-T-cub-m1}. When $m\geq3$, we apply the decay
$T_{3}^{(2)}=\calI_{-m+2}^{3}$. When $m=2$, as in the proof of \eqref{eq:psi1-bd-06},
there is a full cancellation of the contribution of the tail $\bbet^{3}$
from $T_{3}^{(2)}$. Therefore, we have 
\begin{align*}
 & \chi_{B_{1}}\{[(b^{2}+\eta^{2})i\rd_{b}-ib\Lmb_{-1}-\eta(m+1)-(\tint 0y\Re(\br PP_{1})dy')]T_{3}^{(2)}\}\\
 & =\begin{cases}
\chi_{B_{1}}\calI_{0}^{1}\cdot\calI_{-1,1}^{3}+\calO(\chf_{[2,2B_{1}]}\calI_{0}^{2}\cdot\calI_{-1,1}^{3}) & \text{if }m=1,\\
\chi_{B_{1}}(\calI_{0}^{1}\cdot\calI_{-2,1}^{3}+\calI_{-2}^{1}\cdot\calI_{0}^{3})+\calO(\chf_{[2,2B_{1}]}\calI_{0}^{2}\cdot\calI_{0}^{3}) & \text{if }m=2,\\
\chi_{B_{1}}\calI_{0}^{1}\cdot\calI_{-m+2}^{3}+\calO(\chf_{[2,2B_{1}]}\calI_{0}^{2}\cdot\calI_{-(m-2)}^{3}) & \text{if }m\geq3
\end{cases}\\
 & =\chi_{B_{1}}\calA_{4}^{(2)}+\calR_{2}.
\end{align*}
This completes the proof of \eqref{eq:psi2-bd-12}.

\emph{\uline{Fourth line of RHS\mbox{\eqref{eq:psi2-bd-01}}.}}
Thanks to the cancellation \eqref{eq:cutoff-cancellation}, we have
\begin{equation}
\{[(b^{2}+\eta^{2})i\rd_{b}-iby\rd_{y}]\chi_{B_{1}}\}(T_{2}^{(2)}+T_{3}^{(2)})=0.\label{eq:psi2-bd-13}
\end{equation}

\emph{\uline{Last line of RHS\mbox{\eqref{eq:psi2-bd-01}}.}} We
claim that 
\begin{equation}
[p_{3}^{(b)}i\rd_{b}+p_{3}^{(\eta)}i\rd_{\eta}](\chi_{B_{1}}\{T_{2}^{(2)}+T_{3}^{(2)}\})=\chi_{B_{1}}\calA_{4}^{(2)}+\calR_{2}.\label{eq:psi2-bd-14}
\end{equation}
Indeed, 
\begin{align*}
 & [p_{3}^{(b)}i\rd_{b}+p_{3}^{(\eta)}i\rd_{\eta}](\chi_{B_{1}}\{T_{2}^{(2)}+T_{3}^{(2)}\})\\
 & =\chi_{B_{1}}\{[p_{3}^{(b)}i\rd_{b}+p_{3}^{(\eta)}i\rd_{\eta}]T_{2}^{(2)}\}\\
 & \peq+\{[p_{3}^{(b)}i\rd_{b}+p_{3}^{(\eta)}i\rd_{\eta}]\chi_{B_{1}}\}T_{2}^{(2)}+[p_{3}^{(b)}i\rd_{b}+p_{3}^{(\eta)}i\rd_{\eta}](\chi_{B_{1}}T_{3}^{(2)})\\
 & =\chi_{B_{1}}\calI_{-m}^{4}+\calO(\chf_{[2,2B_{1}]}\calI_{0}^{5}+\chf_{[B_{1},2B_{1}]}\calI_{-m}^{4})\\
 & =\chi_{B_{1}}\calA_{4}^{(2)}+\calR_{2}.
\end{align*}

\emph{\uline{Completion of the proof of \mbox{\eqref{eq:psi2-bd-00}}.}}
Substituting the claims \eqref{eq:psi2-bd-16} and \eqref{eq:psi2-bd-11}--\eqref{eq:psi2-bd-14}
into \eqref{eq:psi2-bd-01}, we obtain 
\[
\text{LHS\eqref{eq:psi2-bd-00}}=\chi_{B_{1}}\{-\bbet^{3}\rd_{y}(\tfrac{y^{3}}{4}Q)-\bbet^{3}(A_{Q}^{\ast}-\tfrac{1}{y})(\tfrac{y^{3}}{8}Q)\}+\chi_{B_{1}}\calA_{4}^{(2)}+\calR_{2}.
\]
Since $\rd_{y}(\tfrac{y^{3}}{4}Q)+(A_{Q}^{\ast}-\tfrac{1}{y})(\tfrac{y^{3}}{8}Q)=(A_{Q}-\tfrac{2}{y})(\tfrac{y^{3}}{8}Q)=0$,
the curly bracket of the above display vanishes. This completes the
proof of \eqref{eq:psi2-bd-00}.

\textbf{Step 8.} Construction of $T_{4}^{(2)}$ and the proof of \eqref{eq:Psi1-bound}
and \eqref{eq:Psi2-H2-bound}--\eqref{eq:Psi2-H3-bound}.

\emph{\uline{Construction of \mbox{$T_{4}^{(2)}$}.}} Recall \eqref{eq:psi2-bd-00}.
From the expression of LHS\eqref{eq:psi2-bd-00}, the function $F_{4}^{(2)}$
depends only on already defined functions $T_{j}^{(k)}$, $j\leq3$.
Now we define 
\begin{equation}
T_{4}^{(2)}\coloneqq\begin{cases}
-\out\td H_{Q}^{-1}F_{4}^{(2)} & \text{if }m=1,\\
-\inn\td H_{Q}^{-1}F_{4}^{(2)} & \text{if }m\geq2,
\end{cases}\label{eq:def-T2_4}
\end{equation}
where $\out\td H_{Q}^{-1}$ and $\inn\td H_{Q}^{-1}$ are defined
in Lemma~\ref{lem:HtdQ-inv}. One can easily check the bound \eqref{eq:T2_4-bound}
for $T_{4}^{(2)}$ using Lemma~\ref{lem:HtdQ-inv}.

\emph{\uline{Proof of \mbox{\eqref{eq:Psi1-bound}}.}} It suffices
to estimate 
\[
\Psi_{1}=A_{P}^{\ast}(\chi_{B_{0}}T_{4}^{(2)})+\text{LHS\eqref{eq:psi1-bd-01}}.
\]
By the claim \eqref{eq:psi1-bd-01} and recalling $\calR_{1}$ from
\eqref{eq:psi1-bd-00}, it suffices to show $A_{P}^{\ast}(\chi_{B_{0}}T_{4}^{(2)})=\calR_{1}$.
Using \eqref{eq:T2_4-bound}, this follows from 
\[
A_{P}^{\ast}(\chi_{B_{0}}T_{4}^{(2)})=\calO(\chf_{[2,2B_{0}]}\calI_{0,2}^{4})=\calR_{1}.
\]
Note that this term saturates the bound \eqref{eq:Psi1-bound}. This
completes the proof of \eqref{eq:Psi1-bound}.

\emph{\uline{Proof of \mbox{\eqref{eq:Psi2-H2-bound}}--\mbox{\eqref{eq:Psi2-H3-bound}}.}}
It suffices to estimate 
\begin{align*}
\Psi_{2}=\big[(b^{2}+\eta^{2}+p_{3}^{(b)})i\rd_{b}+p_{3}^{(\eta)}i\rd_{\eta}-ib\Lmb_{-2}-(m+1)\eta\quad\\
-(\tint 0y\Re(\br PP_{1})dy')+\td H_{P}\big] & (\chi_{B_{0}}T_{4}^{(2)})+\text{LHS\eqref{eq:psi2-bd-00}}.
\end{align*}
By the claim \eqref{eq:psi2-bd-00} and the definition \eqref{eq:def-T2_4}
of $T_{4}^{(2)}$, we see that 
\begin{align*}
\Psi_{2} & =[(b^{2}+\eta^{2}+p_{3}^{(b)})i\rd_{b}+p_{3}^{(\eta)}i\rd_{\eta}-ib\Lmb_{-2}\\
 & \qquad\qquad-(m+1)\eta-(\tint 0y\Re(\br PP_{1})dy')](\chi_{B_{0}}T_{4}^{(2)})\\
 & \peq+\chi_{B_{0}}(\td H_{P}-\td H_{Q})T_{4}^{(2)}+[\chi_{B_{0}},\rd_{yy}+\tfrac{1}{y}\rd_{y}]T_{4}^{(2)}+(\chi_{B_{1}}-\chi_{B_{0}})F_{4}^{(2)}+\calR_{2},
\end{align*}
where we recall $\calR_{2}$ from \eqref{eq:psi2-bd-21}. It suffices
to show that all terms of the above display belong to $\calR_{2}$.
For the first and second lines, as the terms in the square bracket
are of size $\calO(\beta)$, these lines are of size $\calO(\chf_{[2,2B_{0}]}\calI_{1,2}^{5})$,
which belongs to $\calR_{2}$ due to $B_{0}=\beta^{-1/2}$. For the
last line of the above display, we use \eqref{eq:ModProfPrem4}, \eqref{eq:T2_4-bound},
and \eqref{eq:psi2-bd-22} to have 
\begin{align*}
\chi_{B_{0}}(\td H_{P}-\td H_{Q})T_{4}^{(2)} & =\calO(\chf_{[2,2B_{0}]}\{\calI_{-1,2}^{5}+\calI_{3,2}^{8}\})=\calR_{2},\\{}
[\chi_{B_{0}},\rd_{yy}+\tfrac{1}{y}\rd_{y}]T_{4}^{(2)} & =\calO(\chf_{[B_{0},2B_{0}]}\calI_{-1,2}^{4})=\calR_{2},\\
(\chi_{B_{1}}-\chi_{B_{0}})F_{4}^{(2)} & =\calR_{2}.
\end{align*}
This completes the proof of \eqref{eq:Psi2-H2-bound}--\eqref{eq:Psi2-H3-bound}.
This finishes the proof.
\end{proof}

\section{\label{sec:Modulation-analysis}Modulation analysis}

In this section, we study the dynamics of solutions near soliton using
modulation analysis. Near-soliton dynamics is relevant to our ultimate
goal (Theorem~\ref{thm:blow-up-rigidity}) because any $H^{1}$ finite-time
blow-up solutions eventually enters a small neighborhood of $Q$ after
renormalization by Remark \ref{rem:blow-up-vicinity}.

Using modulation analysis, our main goal of this section is the \emph{(local-in-time)
uniform $H^{3}$-bound of the remainder part of the solution}. More
precisely, using the modified profile $P(\cdot;b,\eta)$ constructed
in the previous section, we show that a $H_{m}^{3}$-solution $u(t)$
lying in a soliton tube (i.e., in a neighborhood of solitons $Q_{\lmb,\gmm}$)
admits a decomposition 
\[
u(t,r)=\frac{e^{i\gmm(t)}}{\lmb(t)}[P(\cdot;b(t),\eta(t))+\eps(t,\cdot)]\Big(\frac{r}{\lmb(t)}\Big)=[P+\eps]^{\sharp}
\]
with \emph{the remainder $\eps^{\sharp}(t)$ being $H_{m}^{3}$-bounded
on $I$}. It should be remarked that we obtain this $H^{3}$-boundedness
without any dynamical hypotheses on the modulation parameters such
as $b>0$ ($\lmb$ is contracting) or $|\eta|\ll b$ (rotational instability
is turned off), which are typically assumed in the blow-up construction.

\subsection{Decomposition of solutions}

In this subsection, we explain how we fix the decomposition of $u(t)$.
The decomposition will be fixed for each time $t$, by imposing four
(nonlinear) orthogonality conditions on the remainder $\eps(t)$.
The following is a fixed-time analysis and hence we omit the time
variable $t$ in this subsection.

Recall that we view \eqref{eq:CSS-m-equiv} as a system \eqref{eq:u-eqn}--\eqref{eq:u2-eqn}
of equations for $u$, $u_{1}=\bfD_{u}u$, and $u_{2}=A_{u}\bfD_{u}u$.
As in the previous section, where we constructed the modified profiles
$P$, $P_{1}$, and $P_{2}$ for the renormalized variables $w$,
$w_{1}$, and $w_{2}$, the decomposition here will take the form
\begin{align*}
w=\lmb e^{-i\gmm}u(\lmb\cdot) & =P(\cdot;b,\eta)+\eps,\\
w_{1}=\bfD_{w}w & =P_{1}(\cdot;b,\eta)+\eps_{1},\\
w_{2}=A_{w}w_{1} & =P_{2}(\cdot;b,\eta)+\eps_{2}.
\end{align*}

To fix the decomposition, as in \cite{KimKwonOh2020arXiv}, we impose
two orthogonality conditions on $\eps$ and the remaining two on $\eps_{1}$.
More precisely, we use the orthogonality conditions 
\begin{equation}
(\eps,\calZ_{1})_{r}=(\eps,\calZ_{2})_{r}=(\eps_{1},\td{\calZ}_{3})_{r}=(\eps_{1},\td{\calZ}_{4})_{r}=0,\label{eq:DecomOrtho}
\end{equation}
where $\calZ_{1},\calZ_{2}\in C_{c,m}^{\infty}$ and $\td{\calZ}_{3},\td{\calZ}_{4}\in C_{c,m+1}^{\infty}$
are chosen in Remark~\ref{rem:Ortho-choice}. The reason for using
two $\eps_{1}$-orthogonality conditions is that the $b$- and $\eta$-modulation
equations become more apparent in the $w_{1}$-equation compared to
the $w$-equation. We note that because $\eps_{1}\approx L_{Q}\eps$
at the leading order, the last two orthogonality conditions of \eqref{eq:DecomOrtho}
can be viewed as slight modifications of $\eps$-orthogonality conditions
$(\eps,L_{Q}^{\ast}\td{\calZ}_{3})_{r}=(\eps,L_{Q}^{\ast}\td{\calZ}_{4})_{r}=0$.

We now state the decomposition lemma. Denote a $\dot{H}_{m}^{1}$-soliton
tube by 
\begin{equation}
\dot{\calT}_{\alpha^{\ast}}\coloneqq\{u\in\dot{H}_{m}^{1}:\inf_{(\lmb,\gmm)\in\bbR_{+}\times(\bbR/2\pi\bbZ)}\|u_{1/\lmb,-\gmm}-Q\|_{\dot{H}_{m}^{1}}<\alpha^{\ast}\}.\label{eq:def-Hdot1-soliton-tube}
\end{equation}
As a standard application of the implicit function theorem, we obtain
the following.
\begin{lem}[Decomposition]
\label{lem:decomp}There exist $\delta_{0},\alpha_{0}>0$ such that
any $u\in\dot{\calT}_{\alpha_{0}}$ admits the unique decomposition
\begin{equation}
\left\{ \begin{aligned}u & =\frac{e^{i\gmm}}{\lmb}[P(\cdot;b,\eta)+\eps]\Big(\frac{\cdot}{\lmb}\Big),\\
u_{1}=\bfD_{u}u & =\frac{e^{i\gmm}}{\lmb^{2}}[P_{1}(\cdot;b,\eta)+\eps_{1}]\Big(\frac{\cdot}{\lmb}\Big),
\end{aligned}
\right.\label{eq:DecomDec}
\end{equation}
satisfying the orthogonality conditions 
\[
(\eps,\calZ_{1})_{r}=(\eps,\calZ_{2})_{r}=(\eps_{1},\td{\calZ}_{3})_{r}=(\eps_{1},\td{\calZ}_{4})_{r}=0.\tag{\ref{eq:DecomOrtho}}
\]
and smallness
\begin{equation}
\beta+\|\eps\|_{\dot{H}_{m}^{1}}<\delta_{0}.\label{eq:DecomSmall1}
\end{equation}
Moreover, if $u\in\dot{\calT}_{\alpha^{\ast}}$ with $\alpha^{\ast}<\alpha_{0}$,
we have a qualitative smallness 
\begin{equation}
\beta+\|\eps\|_{\dot{H}_{m}^{1}}=o_{\alpha^{\ast}\to0}(1).\label{eq:DecomSmall2}
\end{equation}
\end{lem}

\begin{proof}[Proof sketch]
By the standard argument using the implicit function theorem, see
\cite[Lemma 5.2]{KimKwonOh2020arXiv} for a closely related proof,
it suffices to show that 
\begin{equation}
\begin{aligned}\begin{pmatrix}((\bfv)_{1},\calZ_{1})_{r} & ((\bfv)_{2},\calZ_{1})_{r} & ((\bfv)_{3},\calZ_{1})_{r} & ((\bfv)_{4},\calZ_{1})_{r}\\
((\bfv)_{1},\calZ_{2})_{r} & ((\bfv)_{2},\calZ_{2})_{r} & ((\bfv)_{3},\calZ_{2})_{r} & ((\bfv)_{4},\calZ_{2})_{r}\\
((\bfv_{1})_{1},\td{\calZ}_{3})_{r} & ((\bfv_{1})_{2},\td{\calZ}_{3})_{r} & ((\bfv_{1})_{3},\td{\calZ}_{3})_{r} & ((\bfv_{1})_{4},\td{\calZ}_{3})_{r}\\
((\bfv_{1})_{1},\td{\calZ}_{4})_{r} & ((\bfv_{1})_{2},\td{\calZ}_{4})_{r} & ((\bfv_{1})_{3},\td{\calZ}_{4})_{r} & ((\bfv_{1})_{4},\td{\calZ}_{4})_{r}
\end{pmatrix}\\
=\mathrm{diag}\big((\Lmb Q,\calZ_{1})_{r},(-iQ,\calZ_{2})_{r},(i\tfrac{y}{2}Q,\td{\calZ}_{3})_{r},(\tfrac{y}{2}Q,\td{\calZ}_{4})_{r}\big) & +O(\beta).
\end{aligned}
\label{eq:matrix-transv}
\end{equation}
Note that the entries in the RHS are nonzero due to \eqref{eq:Z1Z2-transv}
and \eqref{eq:def-td-Z3Z4}.

For the proof of \eqref{eq:matrix-transv}, by \eqref{eq:modvec0-1}--\eqref{eq:modvec1-2}
and the fact that $\calZ_{1},\calZ_{2},\td{\calZ}_{3},\td{\calZ}_{4}$
are compactly supported, we have 
\[
\text{LHS\eqref{eq:matrix-transv}}=\begin{pmatrix}(\Lmb Q,\calZ_{1})_{r} & (-iQ,\calZ_{1})_{r} & (i\tfrac{y^{2}}{4}Q,\calZ_{1})_{r} & ((m+1)\rho,\calZ_{1})_{r}\\
(\Lmb Q,\calZ_{2})_{r} & (-iQ,\calZ_{2})_{r} & (i\tfrac{y^{2}}{4}Q,\calZ_{2})_{r} & ((m+1)\rho,\calZ_{2})_{r}\\
0 & 0 & (i\tfrac{y}{2}Q,\td{\calZ}_{3})_{r} & (\tfrac{y}{2}Q,\td{\calZ}_{3})_{r}\\
0 & 0 & (i\tfrac{y}{2}Q,\td{\calZ}_{4})_{r} & (\tfrac{y}{2}Q,\td{\calZ}_{4})_{r}
\end{pmatrix}+O(\beta).
\]
Here, the upper right $2\times2$ matrix is in fact zero, thanks to
\eqref{eq:Z1Z2-gauge-condition}. Moreover, as $\calZ_{1},\td{\calZ}_{4}$
are real and $\calZ_{2},\td{\calZ}_{3}$ are imaginary, the remaining
off-diagonal entries vanish. This completes the proof of \eqref{eq:matrix-transv}.
For the remaining arguments for the proof of this lemma, we refer
to \cite[Lemma 5.2]{KimKwonOh2020arXiv}.
\end{proof}

\subsection{Nonlinear coercivity of energy}

For $H_{m}^{1}$-solutions, a remarkable fact observed in \cite{KimKwonOh2022arXiv1}
(in the context where $b$- and $\eta$-modulations are absent and
only the first two orthogonality conditions of \eqref{eq:DecomOrtho}
are assumed) is that the qualitative smallness $\|\eps\|_{\dot{H}_{m}^{1}}=o_{\alpha^{\ast}\to0}(1)$
can be improved to the \emph{quantitative smallness} 
\begin{equation}
\|\eps\|_{\dot{H}_{m}^{1}}\aleq_{M[u]}\lmb\sqrt{E[u]},\label{eq:quant-small-no-b-eta}
\end{equation}
which we call the\emph{ nonlinear coercivity of energy}. Note that
\eqref{eq:quant-small-no-b-eta} is a significant improvement over
the qualitative smallness in the sense that \eqref{eq:quant-small-no-b-eta}
gives the \emph{uniform $H^{1}$-control} $\|\eps^{\sharp}\|_{\dot{H}_{m}^{1}}\aleq_{M[u]}\sqrt{E[u]}$.
In \cite{KimKwonOh2022arXiv1}, this estimate was at the heart of
the proof of soliton resolution and in particular of the proof of
the upper bound \eqref{eq:thm1-lmb-upper-bound}. We remark that \eqref{eq:quant-small-no-b-eta}
is a non-perturbative estimate. It is rather a consequence of the
\emph{self-duality} and \emph{non-locality} of our problem, which
are two distinguished features of \eqref{eq:CSS-m-equiv}; see the
introduction of \cite{KimKwonOh2022arXiv1} or Remark~\ref{rem:on-quant-small}
below for more details.

In this subsection, we will obtain such a quantitative smallness estimate
in the presence of $b$- and $\eta$-modulations, namely, the bound
\begin{equation}
|b|+|\eta|+\|\eps\|_{\dot{H}_{m}^{1}}\aleq_{M[u]}\lmb\sqrt{E[u]}.\label{eq:quant-small}
\end{equation}
This bound will be the key \emph{a priori} bound allowing the energy
estimates in Section~\ref{subsec:Energy-Morawetz}.

Before we state the full result, let us introduce more pieces of notation.
For $u\in H_{m}^{1}\cap\dot{\calT}_{\alpha_{0}}$, hence $u$ having
finite energy and having a decomposition as in Lemma~\ref{lem:decomp},
let us define a \emph{scaling invariant} quantity 
\begin{equation}
\mu\coloneqq\lmb\sqrt{E[u]}.\label{eq:def-mu}
\end{equation}
For $M,\alpha^{\ast}>0$, we denote a $H_{m}^{1}$-soliton tube by
\begin{equation}
\calT_{\alpha^{\ast},M}\coloneqq\{u\in H_{m}^{1}:u\in\dot{\calT}_{\alpha^{\ast}}\text{ and }M[u]\leq M\}.\label{eq:def-H1-soliton-tube}
\end{equation}

\begin{lem}[Nonlinear coercivity of energy]
\label{lem:NonlinCoerEnergy}For any $M>0$, the following holds
if $\alpha^{\ast}>0$ is sufficiently small (depending on $M$). For
any $u\in\calT_{\alpha^{\ast},M}$, consider the decomposition of
$u$ as in Lemma~\ref{lem:decomp}.
\begin{itemize}
\item (Rough bound from mass conservation) We have 
\begin{equation}
\|\eps\|_{L^{2}}\aleq_{M}1.\label{eq:rough-bd-mass}
\end{equation}
\item (Small parameter $\mu$) We have\footnote{Recall the notation $\delta_{M}(\alpha^{\ast})$ from Section \ref{subsec:Notation}.}
\begin{equation}
\mu=\delta_{M}(\alpha^{\ast}).\label{eq:mu-small-par}
\end{equation}
\item (Nonlinear coercivity of energy) We have 
\begin{align}
\beta+\|\eps\|_{\dot{H}_{m}^{1}}+\|\eps_{1}\|_{L^{2}} & \sim_{M}\mu,\label{eq:H1-nonlin-coer}
\end{align}
\end{itemize}
\end{lem}

\begin{rem}
In view of the uniqueness of zero energy solutions, we have $\mu=0$
if and only if $u=Q_{\lmb,\gmm}$. Thus it is natural to expect that
$\mu$ is a small parameter.
\end{rem}

\begin{rem}[Self-duality, non-locality, and nonlinear coercivity of energy]
\label{rem:on-quant-small}Following the discussion in the introduction
of \cite{KimKwonOh2022arXiv1}, we explain why \eqref{eq:quant-small-no-b-eta}
is a non-perturbative estimate and how the self-duality and non-locality
give \eqref{eq:quant-small-no-b-eta}.

By the \emph{self-duality} \eqref{eq:energy-self-dual-form}, \eqref{eq:quant-small-no-b-eta}
is implied by 
\begin{equation}
\|\bfD_{Q+\eps}(Q+\eps)\|\ageq_{M}\|\eps\|_{\dot{H}_{m}^{1}}.\label{eq:non-lin-coer-1}
\end{equation}
Note that we have $\|\eps\|_{\dot{H}_{m}^{1}}\ll1$ from the qualitative
smallness but $\|\eps\|_{L^{2}}$ might be large. In the compact region
$y\aleq1$, the $\dot{H}_{m}^{1}$-smallness is enough to guarantee
\[
\bfD_{Q+\eps}(Q+\eps)\approx L_{Q}\eps\qquad\text{for }y\aleq1
\]
in view of \eqref{eq:LinearizationBogomolnyi}. Applying the linear
coercivity \eqref{eq:LQ-coer-L2-to-H1} gives \eqref{eq:non-lin-coer-1}
in the region $y\aleq1$. However, in the far region $y\gg1$, the
pure $\eps$-term $\frac{1}{y}A_{\tht}[\eps]\eps$ of $\bfD_{Q+\eps}(Q+\eps)$
is \emph{not perturbative} due to the optimality of the scaling-invariant
estimate $\|\frac{1}{y}A_{\tht}[\eps]\eps\|_{L^{2}}\aleq\|\eps\|_{L^{2}}^{2}\|\eps\|_{\dot{H}_{m}^{1}}$
and the fact that $\|\eps\|_{L^{2}}$ is not necessarily small. In
this region, we only have 
\[
\bfD_{Q+\eps}(Q+\eps)\approx\big(\rd_{y}-\tfrac{1}{y}(m+A_{\tht}[Q]+A_{\tht}[\eps])\big)\eps\eqqcolon\bfD_{Q,\eps}\eps\qquad\text{for }y\gg1.
\]
Now we use the \emph{non-locality} of the problem, particularly the
fact that $m+A_{\tht}[Q]\approx-(m+2)$ is negative. In view of \eqref{eq:energy-self-dual-form},
this says that the exterior energy of $Q+\eps$ is similar to the
exterior energy of $\eps$ for the $-(m+2)$-equivariant \eqref{eq:CSS-m-equiv}.
Thanks to the negativity of both $m+A_{\tht}[Q]$ and $A_{\tht}[\eps]$
as well as the uniform bound $|A_{\tht}[\eps]|\aleq\|\eps\|_{L^{2}}^{2}\aleq_{M}1$,
we can expect a Hardy-type inequality for the operator $\bfD_{Q,\eps}$
as 
\[
\|\chf_{y\sim1}|f|_{-1}\|_{L^{2}}+\|\chf_{y\gg1}\bfD_{Q,\eps}f\|_{L^{2}}\ageq_{M}\|\chf_{y\gg1}|f|_{-1}\|_{L^{2}}.
\]
Substituting $f=\eps$ into the above and using the control of $\eps$
in the region $y\sim1$, we obtain the control of $\eps$ even in
the region $y\gg1$, completing (the sketch of) the proof of \eqref{eq:non-lin-coer-1}.
\end{rem}

\begin{proof}[Proof of Lemma~\ref{lem:NonlinCoerEnergy}]
In the proof of \eqref{eq:H1-nonlin-coer}, we will use the inequality
\begin{equation}
\|L_{Q}\eps-\tfrac{1}{y}A_{\tht}[\eps]\eps\|_{L^{2}}\ageq_{M}\|\eps\|_{\dot{H}_{m}^{1}},\label{eq:nonlin-coer-to-use}
\end{equation}
which is essentially proved in Remark~\ref{rem:on-quant-small}.
For a rigorous proof, the readers may refer to \cite[Lemma 4.4]{KimKwonOh2022arXiv1}.
Note that the proof of $\|\eps\|_{L^{2}}\aleq_{M}1$ is obvious from
mass conservation. Henceforth, we show \eqref{eq:mu-small-par} and
\eqref{eq:H1-nonlin-coer}.

\textbf{Step 1. }A preliminary claim.

In this step, we show that 
\begin{equation}
\mu\sim\beta+\|\eps_{1}\|_{L^{2}}.\label{eq:nonlin-coer-01}
\end{equation}
To see this, using the self-dual form \eqref{eq:energy-self-dual-form}
of energy, we have 
\[
\mu\sim\|P_{1}+\eps_{1}\|_{L^{2}}.
\]
Thus the $\aleq$-inequality of \eqref{eq:nonlin-coer-01} easily
follows from $\|P_{1}\|_{L^{2}}\aleq\beta$ and the triangle inequality.

We turn to the more delicate proof of the $\ageq$-inequality of \eqref{eq:nonlin-coer-01}.
As before, we will consider the expression $\|P_{1}+\eps_{1}\|_{L^{2}}$.
The higher order terms of $P_{1}$ are \emph{non-perturbative} in
the region $y\sim B_{1}$, so we first use the inequality
\begin{equation}
\|P_{1}+\eps_{1}\|_{L^{2}}\geq\|\chf_{(0,B_{0}]}(P_{1}+\eps_{1})\|_{L^{2}}.\label{eq:nonlin-coer-02}
\end{equation}
In this region, the higher order terms of $P_{1}$ become perturbative
(using \eqref{eq:DecomSmall2}): 
\begin{equation}
\|\chf_{(0,B_{0}]}(P_{1}+\bbet\tfrac{y}{2}Q)\|_{L^{2}}\aleq\|\chf_{(0,B_{0}]}\beta^{2}\|_{L^{2}}\aleq\beta^{3/2}=\delta(\alpha^{\ast})\cdot\beta.\label{eq:nonlin-coer-03}
\end{equation}
On the other hand, by the choice of $\td{\calZ}_{3},\td{\calZ}_{4}$
(see \eqref{eq:def-R-circ}), we have 
\begin{align*}
 & |(\chf_{(0,B_{0}]}\bbet\tfrac{y}{2}Q,\chf_{(0,B_{0}]}\eps_{1})_{r}|\\
 & \quad=|(\bbet(\chf_{(0,B_{0}]}-\chi_{R_{\circ}})\tfrac{y}{2}Q,\chf_{(0,B_{0}]}\eps_{1})_{r}|\leq\tfrac{1}{2}\|\chf_{(0,B_{0}]}\beta\tfrac{y}{2}Q\|_{L^{2}}\|\chf_{(0,B_{0}]}\eps_{1}\|_{L^{2}}
\end{align*}
and hence 
\begin{equation}
\begin{aligned} & \|\chf_{(0,B_{0}]}(\bbet\tfrac{y}{2}Q+\eps_{1})\|_{L^{2}}^{2}\\
 & \quad=\|\chf_{(0,B_{0}]}\beta\tfrac{y}{2}Q\|_{L^{2}}^{2}+2(\chf_{(0,B_{0}]}\bbet\tfrac{y}{2}Q,\chf_{(0,B_{0}]}\eps_{1})_{r}+\|\chf_{(0,B_{0}]}\eps_{1}\|_{L^{2}}^{2}\\
 & \quad\geq\|\chf_{(0,B_{0}]}\beta\tfrac{y}{2}Q\|_{L^{2}}^{2}-\|\chf_{(0,B_{0}]}\beta\tfrac{y}{2}Q\|_{L^{2}}\|\chf_{(0,B_{0}]}\eps_{1}\|_{L^{2}}+\|\chf_{(0,B_{0}]}\eps_{1}\|_{L^{2}}^{2}\\
 & \quad\geq\tfrac{1}{2}(\|\chf_{(0,B_{0}]}\beta\tfrac{y}{2}Q\|_{L^{2}}^{2}+\|\chf_{(0,B_{0}]}\eps_{1}\|_{L^{2}}^{2})\\
 & \quad\ageq\beta^{2}+\|\chf_{(0,B_{0}]}\eps_{1}\|_{L^{2}}^{2}.
\end{aligned}
\label{eq:nonlin-coer-04}
\end{equation}
By \eqref{eq:nonlin-coer-02}, \eqref{eq:nonlin-coer-03}, and \eqref{eq:nonlin-coer-04},
we have 
\begin{equation}
\mu\sim\|P_{1}+\eps_{1}\|_{L^{2}}\ageq\beta+\|\chf_{(0,B_{0}]}\eps_{1}\|_{L^{2}}.\label{eq:nonlin-coer-05}
\end{equation}
Since $\|\eps_{1}\|_{L^{2}}\aleq\|P_{1}\|_{L^{2}}+\mu\aleq\beta+\mu\aleq\mu$
(using $\beta\aleq\mu$ from \eqref{eq:nonlin-coer-05}), we can improve
\eqref{eq:nonlin-coer-05} to 
\[
\mu\ageq\beta+\|\eps_{1}\|_{L^{2}}.
\]
This completes the proof of the claim \eqref{eq:nonlin-coer-01}.

\textbf{Step 2.} $\mu$ is a small parameter: proof of \eqref{eq:mu-small-par}.

By \eqref{eq:DecomSmall2} and \eqref{eq:nonlin-coer-01}, it suffices
to show that 
\begin{equation}
\|\eps_{1}\|_{L^{2}}\aleq_{M}\beta+\|\eps\|_{\dot{H}_{m}^{1}}.\label{eq:nonlin-coer-06}
\end{equation}
For the proof of \eqref{eq:nonlin-coer-06}, we start from the relation
\begin{equation}
\eps_{1}=w_{1}-P_{1}=(\bfD_{P}P-P_{1})+(\bfD_{w}w-\bfD_{P}P),\label{eq:nonlin-coer-07}
\end{equation}
take the $L^{2}$-norm, and apply the bound \eqref{eq:compat1-H1}
to have 
\[
\|\eps_{1}\|_{L^{2}}\aleq\beta+\|\bfD_{w}w-\bfD_{P}P\|_{L^{2}}.
\]
We then notice that $\bfD_{w}w-\bfD_{P}P$ is the sum of $(\rd_{y}-\tfrac{m}{y})\eps$
and a linear combination of $\tfrac{1}{y}A_{\tht}[\psi_{1},\psi_{2}]\psi_{3}$
with at least one $\psi_{j}=\eps$. Using the estimate 
\[
\|\tfrac{1}{y}A_{\tht}[\psi_{1},\psi_{2}]\psi_{3}\|_{L^{2}}\aleq\|\psi_{j_{1}}\|_{L^{2}}\|\psi_{j_{2}}\|_{L^{2}}\|\tfrac{1}{y}\psi_{j_{3}}\|_{L^{2}}
\]
for any choice of $j_{1},j_{2},j_{3}$, putting $\psi_{j_{3}}=\eps$,
and using $\|P\|_{L^{2}}+\|\eps\|_{L^{2}}\aleq_{M}1$, we have 
\[
\|\tfrac{1}{y}A_{\tht}[\psi_{1},\psi_{2}]\psi_{3}\|_{L^{2}}\aleq\|(\rd_{y}-\tfrac{m}{y})\eps\|_{L^{2}}+O_{M}(\|\tfrac{1}{y}\eps\|_{L^{2}})\aleq_{M}\|\eps\|_{\dot{H}_{m}^{1}},
\]
completing the proof of the claim \eqref{eq:nonlin-coer-06}.

\textbf{Step 3.} Nonlinear coercivity of energy: proof of \eqref{eq:H1-nonlin-coer}.

By \eqref{eq:nonlin-coer-01} and \eqref{eq:nonlin-coer-06}, it suffices
to show that
\begin{equation}
\|\eps\|_{\dot{H}_{m}^{1}}\aleq_{M}\beta+\|\eps_{1}\|_{L^{2}}.\label{eq:nonlin-coer-08}
\end{equation}
For the proof of \eqref{eq:nonlin-coer-08}, we rewrite the relation
\eqref{eq:nonlin-coer-07} as 
\[
L_{Q}\eps-\tfrac{1}{y}A_{\tht}[\eps]\eps=\eps_{1}-(\bfD_{P}P-P_{1})-\{(L_{P}-L_{Q})\eps+N_{P}(\eps)+\tfrac{1}{y}A_{\tht}[\eps]\eps\}.
\]
In the above display, we take the $L^{2}$-norm, apply the nonlinear
Hardy inequality \eqref{eq:nonlin-coer-to-use} for the LHS, and \eqref{eq:compat1-H1}
for $\bfD_{P}P-P_{1}$ to have 
\begin{equation}
\begin{aligned}\|\eps\|_{\dot{H}_{m}^{1}} & \aleq_{M}\|L_{Q}\eps-\tfrac{1}{y}A_{\tht}[\eps]\eps\|_{L^{2}}\\
 & \aleq_{M}\|\eps_{1}\|_{L^{2}}+\beta+\|(L_{P}-L_{Q})\eps+N_{P}(\eps)+\tfrac{1}{y}A_{\tht}[\eps]\eps\|_{L^{2}}.
\end{aligned}
\label{eq:nonlin-coer-09}
\end{equation}
The following estimate will be shown in the next paragraph: 
\begin{equation}
\|(L_{P}-L_{Q})\eps+N_{P}(\eps)+\tfrac{1}{y}A_{\tht}[\eps]\eps\|_{L^{2}}\aleq O_{M}(\beta)+\delta_{M}(\alpha^{\ast})\cdot\|\eps\|_{\dot{H}_{m}^{1}}.\label{eq:nonlin-coer-10}
\end{equation}
Substituting this estimate into \eqref{eq:nonlin-coer-09}, we have
\[
\|\eps\|_{\dot{H}_{m}^{1}}\aleq O_{M}(\beta+\|\eps_{1}\|_{L^{2}})+\delta_{M}(\alpha^{\ast})\cdot\|\eps\|_{\dot{H}_{m}^{1}}.
\]
Using the smallness of $\delta_{M}(\alpha^{\ast})$, the last term
of the RHS can be absorbed into the LHS, yielding the claim \eqref{eq:nonlin-coer-08}
and hence \eqref{eq:H1-nonlin-coer}.

It remains to show the estimate \eqref{eq:nonlin-coer-10}, which
is used in the proof of \eqref{eq:nonlin-coer-08}. We note that $(L_{P}-L_{Q})\eps+N_{P}(\eps)+\tfrac{1}{y}A_{\tht}[\eps]\eps$
is a linear combination of $\tfrac{1}{y}A_{\tht}[\psi_{1},\psi_{2}]\psi_{3}$,
where either (for $(L_{P}-L_{Q})\eps$) exactly one $\psi_{j_{1}}=P-Q$,
exactly one $\psi_{j_{2}}=\eps$, and the other $\psi_{j_{3}}\in\{P,Q\}$
or (for $N_{P}(\eps)+\frac{1}{y}A_{\tht}[\eps]\eps$) exactly two
$\psi_{j_{1}}=\psi_{j_{2}}=\eps$ and exactly one $\psi_{j_{3}}=P$.
In what follows, we only estimate $\|\tfrac{1}{y}A_{\tht}[\psi_{1},\psi_{2}]\psi_{3}\|_{L^{2}}$
that can contribute to either $(L_{P}-L_{Q})\eps$ or $N_{P}(\eps)+\tfrac{1}{y}A_{\tht}[\eps]\eps$.

(i) If $\psi_{3}=\eps$, then we use 
\begin{align*}
|\psi_{1}\psi_{2}| & \aleq|P-Q|(|P|+Q)+|P\eps|\\
 & \aleq|P-Q|(|P|+Q+|\eps|)+|Q\eps|
\end{align*}
to obtain 
\begin{align*}
\|\tfrac{1}{y} & A_{\tht}[\psi_{1},\psi_{2}]\eps\|_{L^{2}}\\
 & \aleq\|P-Q\|_{L^{\infty}}\||P|+Q+|\eps|\|_{L^{2}}\|\eps\|_{L^{2}}+\|Q\eps\|_{L^{1}}\|\tfrac{1}{y}\eps\|_{L^{2}}\\
 & \aleq\beta(1+\|\eps\|_{L^{2}})\|\eps\|_{L^{2}}+\|\tfrac{1}{y}\eps\|_{L^{2}}^{2}\\
 & \aleq O_{M}(\beta)+\delta_{M}(\alpha^{\ast})\cdot\|\eps\|_{\dot{H}_{m}^{1}}.
\end{align*}

(ii) If $\psi_{3}\in\{P,Q\}$, then we use 
\[
|\psi_{1}\psi_{2}|\aleq(|P-Q|+|\eps|)|\eps|\quad\text{and}\quad|\psi_{3}|\aleq Q+\chf_{(0,2B_{1}]}\beta^{2}y^{4}Q
\]
to obtain (using \eqref{eq:HardySobolevSection2} for $\|\eps\|_{L^{\infty}}$)
\begin{align*}
\|\tfrac{1}{y} & A_{\tht}[\psi_{1},\psi_{2}]\psi_{3}\|_{L^{2}}\\
 & \aleq\|\tfrac{1}{y}\psi_{1}\psi_{2}\|_{L^{2}}\|yQ\|_{L^{2}}+\|\psi_{1}\psi_{2}\|_{L^{1}}\|\chf_{(0,2B_{1}]}\beta^{2}y^{3}Q\|_{L^{2}}\\
 & \aleq(\beta+\|\eps\|_{L^{\infty}})\|\tfrac{1}{y}\eps\|_{L^{2}}+(1+\|\eps\|_{L^{2}})\|\eps\|_{L^{2}}\beta\\
 & \aleq\delta_{M}(\alpha^{\ast})\cdot\|\eps\|_{\dot{H}_{m}^{1}}+O_{M}(\beta).
\end{align*}

(iii) If $\psi_{3}=P-Q$, then we use 
\[
|\psi_{1}\psi_{2}|\aleq(|P|+Q)|\eps|
\]
to obtain 
\[
\|\tfrac{1}{y}A_{\tht}[\psi_{1},\psi_{2}](P-Q)\|_{L^{2}}\aleq\|\tfrac{1}{y}(P-Q)\|_{L^{2}}\|\psi_{1}\psi_{2}\|_{L^{1}}\aleq\beta\|\eps\|_{L^{2}}\aleq_{M}\beta.
\]
This completes the proof of \eqref{eq:nonlin-coer-10}.
\end{proof}
\begin{rem}[\textbf{Assumptions in the rest of this section}]
Until the end of this section, we always assume that our solution
$u$ has $H_{m}^{3}$-regularity and belongs to $\calT_{\alpha^{\ast},M}$
for sufficiently small $\alpha^{\ast}>0$, depending on $M$. The
required smallness of $\alpha^{\ast}>0$ may shrink in each of the
claims below.
\end{rem}

\subsection{Hardy bounds for higher Sobolev and local norms}

In later sections, we will ultimately do the energy estimates for
the variable $\eps_{2}$. To close the energy estimates for $\eps_{2}$
and to obtain useful information on the original solution $u=[P+\eps]_{\lmb,\gmm}$
from the control of $\eps_{2}$, it is necessary to convert various
Hardy-type bounds for $\eps_{2}$ into those for $\eps$ and $\eps_{1}$.
For this purpose, let us define 
\begin{align}
X_{2} & \coloneqq\|\eps_{2}\|_{L^{2}}+\mu^{2},\label{eq:def-X2}\\
X_{3} & \coloneqq\|\eps_{2}\|_{\dot{H}_{m+2}^{1}}+\mu X_{2},\label{eq:def-X3}\\
V_{7/2} & \coloneqq\|\eps_{2}\|_{\dot{V}_{m+2}^{3/2}}+\mu^{1/2}X_{3}.\label{eq:def-V7/2}
\end{align}
Note that the $\dot{V}_{m+2}^{3/2}$-norm defined in \eqref{eq:def-V3/2}
was motivated by the lower bound of the Morawetz monotonicity \eqref{eq:Morawetz-repul}.
The goal of this subsection is to show that these quantities bound
the Sobolev norms of $\eps$ and $\eps_{1}$ with correct scalings.
\begin{lem}[Hardy bounds for higher Sobolev and local norms]
\label{lem:HigherHardyBounds}\ 
\begin{enumerate}
\item For $\eps_{1}$, we have 
\begin{align}
\|\eps_{1}\|_{\dot{H}_{m+1}^{1}} & \aleq_{M}X_{2},\label{eq:e1-H1}\\
\|\eps_{1}\|_{\dot{H}_{m+1}^{2}} & \aleq_{M}X_{3},\label{eq:e1-H2}\\
\|\eps_{1}\|_{\dot{V}_{m+1}^{5/2}} & \aleq_{M}V_{7/2}.\label{eq:e1-V}
\end{align}
\item For $\eps$, we have 
\begin{align}
\|\eps\|_{\dot{\calH}_{m}^{2}} & \aleq_{M}X_{2},\label{eq:e-H2}\\
\|\eps\|_{\dot{\calH}_{m}^{3}} & \aleq_{M}X_{3}.\label{eq:e-H3}
\end{align}
\end{enumerate}
\end{lem}

\begin{rem}[Weighted $L^{\infty}$ controls]
\label{rem:Weighted-Linfty}One can also control various weighted
$L^{\infty}$-norms on $\eps,\eps_{1},\eps_{2}$ by combining Lemma
\ref{lem:weighted-Linfty-est} with the above lemma.
\end{rem}

\begin{proof}
\textbf{Step 0.} Remark on the order of the proof.

To be more rigorous, we need to prove this lemma in the following
order: \eqref{eq:e1-H1}, \eqref{eq:e-H2}, \eqref{eq:e1-H2}, \eqref{eq:e-H3},
\eqref{eq:e1-V}. However, we will give the proof of \eqref{eq:e1-H1}--\eqref{eq:e1-V}
first and then give the proof of \eqref{eq:e-H2}--\eqref{eq:e-H3}
in order to group similar proofs as possible.

\textbf{Step 1. }From $\eps_{2}$ to $\eps_{1}$: proof of \eqref{eq:e1-H1}--\eqref{eq:e1-V}.

Rewrite $A_{w}w_{1}=w_{2}$ as 
\begin{equation}
A_{Q}\eps_{1}=\eps_{2}-(A_{P}P_{1}-P_{2})-\{(A_{w}-A_{Q})\eps_{1}-(A_{w}-A_{P})P_{1}\}.\label{eq:e1-coer-01}
\end{equation}
We claim that 
\begin{align}
\|(A_{w}-A_{Q})\eps_{1}-(A_{w}-A_{P})P_{1}\|_{L^{2}} & \aleq_{M}\mu^{2},\label{eq:e1-coer-02}\\
\|(A_{w}-A_{Q})\eps_{1}-(A_{w}-A_{P})P_{1}\|_{\dot{H}_{m+2}^{1}} & \aleq_{M}\mu X_{2},\label{eq:e1-coer-03}\\
\|(A_{w}-A_{Q})\eps_{1}-(A_{w}-A_{P})P_{1}\|_{\dot{V}_{m+2}^{3/2}} & \aleq_{M}\mu^{1/2}X_{3}.\label{eq:e1-coer-04}
\end{align}
Let us first assume these claims and finish the proof of \eqref{eq:e1-H1}--\eqref{eq:e1-V}.
Applying to \eqref{eq:e1-coer-01} the linear coercivity \eqref{eq:AQ-coer-L2-to-H1}--\eqref{eq:AQ-coer-V}
for $A_{Q}\eps_{1}$, the bounds \eqref{eq:compat2-H2}--\eqref{eq:compat2-virial}
for $A_{P}P_{1}-P_{2}$, the claims \eqref{eq:e1-coer-02}--\eqref{eq:e1-coer-04}
for $(A_{w}-A_{Q})\eps_{1}-(A_{w}-A_{P})P_{1}$, and $\beta\aleq_{M}\mu$,
we obtain 
\begin{align*}
\|\eps_{1}\|_{\dot{H}_{m+1}^{1}} & \aleq\|\eps_{2}\|_{L^{2}}+\beta^{2}+O_{M}(\mu^{2})\aleq_{M}X_{2},\\
\|\eps_{1}\|_{\dot{H}_{m+1}^{2}} & \aleq\|\eps_{2}\|_{\dot{H}_{m+2}^{1}}+\beta^{3}+O_{M}(\mu X_{2})\aleq_{M}X_{3},\\
\|\eps_{1}\|_{\dot{V}_{m+1}^{5/2}} & \aleq\|\eps_{2}\|_{\dot{V}_{m+2}^{3/2}}+\beta^{7/2}+O_{M}(\mu^{1/2}X_{3})\aleq_{M}V_{7/2}.
\end{align*}
This completes the proof of \eqref{eq:e1-H1}--\eqref{eq:e1-V}.

It remains to show the claims \eqref{eq:e1-coer-02}--\eqref{eq:e1-coer-04}.
For the term $(A_{w}-A_{Q})\eps_{1}$, we note the pointwise estimates
\begin{align*}
(A_{w}-A_{Q})\eps_{1} & \aleq\tfrac{1}{y}\tint 0y||w|^{2}-Q^{2}|y'dy'\cdot|\eps_{1}|\\
|(A_{w}-A_{Q})\eps_{1}|_{-1} & \aleq\tfrac{1}{y}\tint 0y||w|^{2}-Q^{2}|y'dy'\cdot|\eps_{1}|_{-1}+||w|^{2}-Q^{2}||\eps_{1}|.
\end{align*}
Using \eqref{eq:eps1-Linfty} and $\|\langle y\rangle^{-\frac{1}{2}}\eps_{1}\|_{L^{\infty}}\aleq\|\langle y\rangle^{-\frac{1}{2}}|\eps_{1}|_{-1}\|_{L^{2}}$,
we have 
\begin{align*}
\|(A_{w}-A_{Q})\eps_{1}\|_{L^{2}} & \aleq\||w|^{2}-Q^{2}\|_{L^{2}}\|\eps_{1}\|_{L^{2}},\\
\|(A_{w}-A_{Q})\eps_{1}\|_{\dot{H}_{m+1}^{1}} & \aleq\||w|^{2}-Q^{2}\|_{L^{2}}\|\eps_{1}\|_{\dot{H}_{m+1}^{1}},\\
\|(A_{w}-A_{Q})\eps_{1}\|_{\dot{V}_{m+1}^{3/2}} & \aleq\||w|^{2}-Q^{2}\|_{L^{2}}\|\langle y\rangle^{-\frac{1}{2}}|\eps_{1}|_{-1}\|_{L^{2}}.
\end{align*}
Applying to the above the estimates
\begin{align*}
\||w|^{2}-Q^{2}\|_{L^{2}} & \aleq_{M}\|w-Q\|_{L^{\infty}}\aleq_{M}\mu,\\
\|y^{-\frac{1}{2}}|\eps_{1}|_{-1}\|_{L^{2}} & \aleq\|\eps_{1}\|_{\dot{H}_{m+1}^{1}}^{1/2}\|\eps_{1}\|_{\dot{H}_{m+1}^{2}}^{1/2},
\end{align*}
we have 
\begin{align*}
\|(A_{w}-A_{Q})\eps_{1}\|_{L^{2}} & \aleq_{M}\mu\|\eps_{1}\|_{L^{2}}\aleq_{M}\mu^{2},\\
\|(A_{w}-A_{Q})\eps_{1}\|_{\dot{H}_{m+1}^{1}} & \aleq_{M}\mu\|\eps_{1}\|_{\dot{H}_{m+1}^{1}}\aleq_{M}\mu X_{2},\\
\|(A_{w}-A_{Q})\eps_{1}\|_{\dot{V}_{m+1}^{3/2}} & \aleq_{M}\mu\|\eps_{1}\|_{\dot{H}_{m+1}^{1}}^{1/2}\|\eps_{1}\|_{\dot{H}_{m+1}^{2}}^{1/2}\aleq_{M}\mu^{1/2}X_{3}.
\end{align*}
For the term $(A_{w}-A_{P})P_{1}$, we use 
\begin{align*}
\||w|^{2}-|P|^{2}\|_{L^{2}} & \aleq_{M}\|\eps\|_{L^{\infty}}\aleq_{M}\mu,\\
\|\tfrac{1}{\langle y\rangle}(|w|^{2}-|P|^{2})\|_{L^{2}} & \aleq_{M}\|\tfrac{1}{\langle y\rangle}Q\eps\|_{L^{2}}+\||P-Q|+|\eps|\|_{L^{\infty}}\|\tfrac{1}{\langle y\rangle}\eps\|_{L^{2}}\aleq_{M}X_{2},\\
\|\tfrac{1}{\langle y\rangle^{3/2}}(|w|^{2}-|P|^{2})\|_{L^{2}} & \aleq_{M}\|\tfrac{1}{\langle y\rangle^{3/2}}Q\eps\|_{L^{2}}+\||P-Q|+|\eps|\|_{L^{\infty}}\|\tfrac{1}{\langle y\rangle^{3/2}}\eps\|_{L^{2}}\\
 & \aleq_{M}\|\eps\|_{\dot{\calH}_{m}^{3}}+\mu\|\eps\|_{\dot{H}_{m}^{1}}^{1/2}\|\eps\|_{\dot{\calH}_{m}^{2}}^{1/2}\aleq_{M}X_{3}+\mu^{1/2}X_{2},
\end{align*}
and
\[
\|P_{1}\|_{L^{2}}+\|\langle y\rangle|P_{1}|_{1}\|_{L^{\infty}}\aleq\beta\aleq_{M}\mu
\]
to have 
\begin{align*}
\|(A_{w}-A_{P})P_{1}\|_{L^{2}} & \aleq\||w|^{2}-|P|^{2}\|_{L^{2}}\|P_{1}\|_{L^{2}}\aleq_{M}X_{2},\\
\|(A_{w}-A_{P})P_{1}\|_{\dot{H}_{m+1}^{1}} & \aleq\|\tfrac{1}{\langle y\rangle}(|w|^{2}-|P|^{2})\|_{L^{2}}\|\langle y\rangle|P_{1}|_{1}\|_{L^{\infty}}\aleq_{M}X_{3},\\
\|(A_{w}-A_{P})P_{1}\|_{\dot{V}_{m+1}^{3/2}} & \aleq\|\tfrac{1}{\langle y\rangle^{3/2}}(|w|^{2}-|P|^{2})\|_{L^{2}}\|\langle y\rangle|P_{1}|_{1}\|_{L^{\infty}}\aleq_{M}V^{7/2}.
\end{align*}
This completes the proof of the claims \eqref{eq:e1-coer-02}--\eqref{eq:e1-coer-04}.

\textbf{Step 2.} From $\eps_{1}$ to $\eps$: proof of \eqref{eq:e-H2}--\eqref{eq:e-H3}.

Rewrite $\bfD_{w}w=w_{1}$ as 
\begin{equation}
L_{Q}\eps=\eps_{1}-(\bfD_{P}P-P_{1})-\{(L_{P}-L_{Q})\eps+N_{P}(\eps)\}.\label{eq:e-coer-01}
\end{equation}
We claim that 
\begin{align}
\|(L_{P}-L_{Q})\eps+N_{P}(\eps)\|_{\dot{H}_{m+1}^{1}} & \aleq_{M}\mu^{2},\label{eq:e-coer-02}\\
\|(L_{P}-L_{Q})\eps+N_{P}(\eps)\|_{\dot{H}_{m+1}^{2}} & \aleq_{M}\mu X_{2}.\label{eq:e-coer-03}
\end{align}
Let us first assume these claims and finish the proof of \eqref{eq:e-H2}--\eqref{eq:e-H3}.
Applying to \eqref{eq:e-coer-01} the linear coercivity estimates
\eqref{eq:LQ-coer-H1-to-H2}--\eqref{eq:LQ-coer-H2-to-H3} for $L_{Q}\eps$,
the bounds \eqref{eq:e1-H1}--\eqref{eq:e1-H2} for $\eps_{1}$,
the estimates \eqref{eq:compat1-H2}--\eqref{eq:compat1-H3} for
$\bfD_{P}P-P_{1}$, and the claims \eqref{eq:e-coer-02}--\eqref{eq:e-coer-03}
for $(L_{P}-L_{Q})\eps+N_{P}(\eps)$, we obtain 
\begin{align*}
\|\eps\|_{\dot{\calH}_{m}^{2}} & \aleq X_{2}+\beta^{2}+O_{M}(\mu^{2})\aleq_{M}X_{2},\\
\|\eps\|_{\dot{\calH}_{m}^{3}} & \aleq X_{3}+\beta^{3}+O_{M}(\mu X_{2})\aleq_{M}X_{3}.
\end{align*}
This completes the proof of \eqref{eq:e-H2} and \eqref{eq:e-H3}.

It only remains to prove the claims \eqref{eq:e-coer-02}--\eqref{eq:e-coer-03}.
Recall that $(L_{P}-L_{Q})\eps+N_{P}(\eps)$ is a linear combination
of $\tfrac{1}{y}A_{\tht}[\psi_{1},\psi_{2}]\psi_{3}$ for suitable
choices of $\psi_{j}$'s in $\{P,Q,P-Q,\eps\}$. Note the pointwise
estimates (recall $\rd_{+}=\rd_{y}-\frac{m}{y}$ from the notation
section)
\begin{align*}
|\tfrac{1}{y}A_{\tht}[\psi_{1},\psi_{2}]\psi_{3}|_{-1} & \aleq\tfrac{1}{y}|A_{\tht}[\psi_{1},\psi_{2}]||\psi_{3}|_{-1}+|\psi_{1}\psi_{2}\psi_{3}|.\\
|\tfrac{1}{y}A_{\tht}[\psi_{1},\psi_{2}]\psi_{3}|_{-2} & \aleq\tfrac{1}{y}|A_{\tht}[\psi_{1},\psi_{2}]||\rd_{+}\psi_{3}|_{-1}+\tfrac{1}{y^{3}}|A_{\tht}[\psi_{1},\psi_{2}]||\psi_{3}|+|\psi_{1}\psi_{2}\psi_{3}|_{-1}.
\end{align*}
By symmetry, we may replace the last term $|\psi_{1}\psi_{2}\psi_{3}|_{-1}$
by $|\psi_{1}\psi_{2}||\psi_{3}|_{-1}$. Therefore, it suffices to
show the estimates 
\begin{align}
\|\tfrac{1}{y}|A_{\tht}[\psi_{1},\psi_{2}]||\psi_{3}|_{-1}+|\psi_{1}\psi_{2}\psi_{3}|\|_{L^{2}} & \aleq_{M}\mu^{2},\label{eq:high-order-11}\\
\|\tfrac{1}{y}|A_{\tht}[\psi_{1},\psi_{2}]|(|\rd_{+}\psi_{3}|_{-1}+\tfrac{1}{y^{2}}|\psi_{3}|)+|\psi_{1}\psi_{2}||\psi_{3}|_{-1}\|_{L^{2}} & \aleq_{M}\mu X_{2}\label{eq:high-order-12}
\end{align}
for the choices of $\psi_{1},\psi_{2},\psi_{3}$ that can contribute
to either $(L_{P}-L_{Q})\eps$ or $N_{P}(\eps)$.

(i) $\psi_{3}=\eps$: we have 
\begin{align*}
\text{LHS}\eqref{eq:high-order-11} & \aleq\|\psi_{1}\psi_{2}\|_{L^{2}}(\||\eps|_{-1}\|_{L^{2}}+\|\eps\|_{L^{\infty}})\aleq\|\psi_{1}\psi_{2}\|_{L^{2}}\|\eps\|_{\dot{H}_{m}^{1}},\\
\text{LHS}\eqref{eq:high-order-12} & \aleq\|\tfrac{\langle y\rangle}{y}\psi_{1}\psi_{2}\|_{L^{2}}(\||\rd_{+}\eps|_{-1}\|_{L^{2}}+\|\tfrac{y}{\langle y\rangle}|\eps|_{-1}\|_{L^{\infty}})\\
 & \aleq\|\tfrac{\langle y\rangle}{y}\psi_{1}\psi_{2}\|_{L^{2}}\|\eps\|_{\dot{\calH}_{m}^{2}}.
\end{align*}
Since 
\begin{align*}
\|\tfrac{\langle y\rangle}{y}\psi_{1}\psi_{2}\|_{L^{2}} & \aleq\|\tfrac{\langle y\rangle}{y}\{|P-Q|(|P|+Q)+(|P|+|\eps|)|\eps|\}\|_{L^{2}}\\
 & \aleq\|P-Q\|_{L^{\infty}}+\|\eps\|_{L^{\infty}}(1+\|\tfrac{\langle y\rangle}{y}\eps\|_{L^{2}})\aleq_{M}\mu,
\end{align*}
we have 
\begin{align*}
\text{LHS}\eqref{eq:high-order-11} & \aleq_{M}\mu\|\eps\|_{\dot{H}_{m}^{1}}\aleq_{M}\mu^{2},\\
\text{LHS}\eqref{eq:high-order-12} & \aleq_{M}\mu\|\eps\|_{\dot{\calH}_{m}^{2}}\aleq_{M}\mu X_{2}.
\end{align*}

(ii) $\psi_{3}\in\{P,Q\}$: we use $|P|_{2}+|Q|_{2}\aleq\langle y\rangle^{-1}$
to have 
\begin{align*}
\text{LHS}\eqref{eq:high-order-11} & \aleq\|\tfrac{1}{\langle y\rangle}\psi_{1}\psi_{2}\|_{L^{2}},\\
\text{LHS}\eqref{eq:high-order-12} & \aleq\|\tfrac{1}{y\langle y\rangle}\psi_{1}\psi_{2}\|_{L^{2}}.
\end{align*}
Applying to the above the estimates 
\begin{align*}
|\psi_{1}\psi_{2}| & \aleq|P-Q||\eps|+|\eps^{2}|\aleq(\beta+|\eps|)|\eps|\aleq_{M}\mu|\eps|,
\end{align*}
we obtain 
\begin{align*}
\text{LHS}\eqref{eq:high-order-11} & \aleq_{M}\mu\|\tfrac{1}{\langle y\rangle}\eps\|_{L^{2}}\aleq_{M}\mu\|\eps\|_{\dot{H}_{m}^{1}}\aleq_{M}\mu^{2},\\
\text{LHS}\eqref{eq:high-order-12} & \aleq_{M}\mu\|\tfrac{1}{y\langle y\rangle}\eps\|_{L^{2}}\aleq_{M}\mu\|\eps\|_{\dot{\calH}_{m}^{2}}\aleq_{M}\mu X_{2}.
\end{align*}

(iii) $\psi_{3}=P-Q$: we use 
\[
\|\tfrac{1}{y}|P-Q|_{2}\|_{L^{2}}+\||P-Q|_{1}\|_{L^{\infty}}\aleq\beta
\]
to have 
\begin{align*}
\text{LHS}\eqref{eq:high-order-11} & \aleq\beta\|\psi_{1}\psi_{2}\|_{L^{2}},\\
\text{LHS}\eqref{eq:high-order-12} & \aleq\beta\|\tfrac{1}{y}\psi_{1}\psi_{2}\|_{L^{2}}.
\end{align*}
We then apply $|\psi_{1}\psi_{2}|\aleq(|P|+Q+|\eps|)|\eps|\aleq(\langle y\rangle^{-1}+|\eps|)|\eps|$
to have 
\begin{align*}
\text{LHS}\eqref{eq:high-order-11} & \aleq\beta(\|\tfrac{1}{\langle y\rangle}\eps\|_{L^{2}}+\|\eps\|_{L^{2}}\|\eps\|_{L^{\infty}})\aleq_{M}\mu\|\eps\|_{\dot{H}_{m}^{1}}\aleq_{M}\mu^{2},\\
\text{LHS}\eqref{eq:high-order-12} & \aleq\beta(\|\tfrac{1}{y\langle y\rangle}\eps\|_{L^{2}}+\|\tfrac{1}{y}\eps\|_{L^{2}}\|\eps\|_{L^{\infty}})\aleq_{M}\mu(\|\eps\|_{\dot{\calH}_{m}^{2}}+\mu^{2})\aleq_{M}\mu X_{2}.
\end{align*}
This completes the proof of \eqref{eq:e-coer-02}--\eqref{eq:e-coer-03}.
\end{proof}

\subsection{\label{subsec:First-modulation-estimates}First modulation estimates}

As is standard, we differentiate the orthogonality conditions \eqref{eq:DecomOrtho}
in time to obtain preliminary controls on the dynamics of the modulation
parameters. As mentioned before, the $\eps$-orthogonality conditions
of \eqref{eq:DecomOrtho} (i.e., the first two of \eqref{eq:DecomOrtho})
give controls on the $\lmb$- and $\gmm$-dynamics. The $b$- and
$\eta$-modulation estimates are obtained from differentiating the
$\eps_{1}$-orthogonality conditions. Note that the evolution equation
for $\eps_{1}$ (see \eqref{eq:eps1-eqn} below) effectively detects
the dynamics of $b$ and $\eta$, thanks to the degeneracy $\Lmb P_{1},iP_{1}=O(\beta)$
for the first two components of $\bfv_{1}$.

Now we derive the evolution equations for $\eps$ and $\eps_{1}$
in the renormalized variables $(s,y)$. To derive the equation for
$\eps$, we first recall the equations for $w$ and $P$: (see \eqref{eq:def-mod-vec}
for the definition of $\Mod$ and $\bfv$) 
\begin{align*}
(\rd_{s}-\tfrac{\lmb_{s}}{\lmb}\Lmb+\gmm_{s}i)w+i\bfD_{w}^{\ast}w_{1}+(\tint y{\infty}\Re(\br ww_{1})dy')iw & =0,\tag{\ref{eq:w-eqn}}\\
(\rd_{s}-\tfrac{\lmb_{s}}{\lmb}\Lmb+\gmm_{s}i)P+i\bfD_{P}^{\ast}P_{1}+(\tint y{\infty}\Re(\br PP_{1})dy')iP & =-\Mod\cdot\bfv+i\Psi.\tag{\ref{eq:P-eqn}}
\end{align*}
Subtracting the second equation from the first, we obtain a preliminary
equation for $\eps$:
\begin{align*}
(\rd_{s} & -\tfrac{\lmb_{s}}{\lmb}\Lmb+\gmm_{s}i)\eps+i\bfD_{w}^{\ast}w_{1}-i\bfD_{P}^{\ast}P_{1}\\
 & +(\tint y{\infty}\Re(\br ww_{1}-\br PP_{1})dy')iP+(\tint y{\infty}\Re(\br ww_{1})dy')i\eps=\Mod\cdot\bfv-i\Psi.
\end{align*}
The above form is however not suitable for our analysis because we
need to incorporate the\emph{ phase correction} from the nonlocal
nonlinearity. For this purpose, we introduce a new modulation vector
$\td{\Mod}$, whose components are defined by 
\begin{equation}
\td{\Mod}_{j}\coloneqq\begin{cases}
\gmm_{s}+(m+1)\eta+\int_{0}^{\infty}\Re(\br ww_{1})dy & \text{if }j=2,\\
\Mod_{j} & \text{otherwise.}
\end{cases}\label{eq:def-td-Mod}
\end{equation}
Using this modulation vector, we obtain the equation for $\eps$:
\begin{equation}
\begin{aligned}(\rd_{s} & -\tfrac{\lmb_{s}}{\lmb}\Lmb+\gmm_{s}i)\eps+i\bfD_{w}^{\ast}w_{1}-i\bfD_{P}^{\ast}P_{1}\\
 & -(\tint 0y\Re(\br ww_{1}-\br PP_{1})dy')iP+(\tint y{\infty}\Re(\br ww_{1})dy')i\eps=\td{\Mod}\cdot\bfv-i\Psi.
\end{aligned}
\label{eq:eps-eqn}
\end{equation}
In a similar fashion, using \eqref{eq:w1-eqn}, \eqref{eq:P1-eqn},
and \eqref{eq:def-td-Mod}, we obtain the equation for $\eps_{1}$:
\begin{equation}
\begin{aligned} & (\rd_{s}-\tfrac{\lmb_{s}}{\lmb}\Lmb_{-1}+\gmm_{s}i)\eps_{1}+iA_{w}^{\ast}w_{2}-iA_{P}^{\ast}P_{2}\\
 & \quad-(\tint 0y\Re(\br ww_{1}-\br PP_{1})dy')iP_{1}+(\tint y{\infty}\Re(\br ww_{1})dy')i\eps_{1}=\td{\Mod}\cdot\bfv_{1}-i\Psi_{1}.
\end{aligned}
\label{eq:eps1-eqn}
\end{equation}

The following is the main proposition of this subsection.
\begin{prop}[First modulation estimates]
\label{prop:FirstModEst}We have 
\begin{align}
|\Mod_{1}|+|\td{\Mod}_{2}| & \aleq_{M}X_{3},\label{eq:ModEst-3}\\
|\Mod_{3}|+|\Mod_{4}| & \aleq_{M}V_{7/2}.\label{eq:ModEst-4}
\end{align}
In particular, we have rough estimates 
\begin{align}
\Big|\frac{\lmb_{s}}{\lmb}\Big|+|\gmm_{s}| & \aleq_{M}\mu,\label{eq:ModEst-1}\\
|b_{s}+b^{2}+\eta^{2}|+|\eta_{s}| & \aleq_{M}X_{3}.\label{eq:ModEst-2}
\end{align}
\end{prop}

\begin{rem}
We mention once again that we need to incorporate the phase correction
to the blow-up part of the solution as expressed in the definition
of $\td{\Mod}_{2}$. Such a phase correction arises from the $A_{t}$-potential
and this fact was already observed in \cite{KimKwon2023MAMS,KimKwon2023AnnPDE}.
Note that $\td{\Mod}_{2}\approx\gmm_{s}-(m+1)\eta$, as can be seen
from the estimate 
\begin{equation}
\begin{aligned}\tint 0{\infty}\Re(\br ww_{1})dy & =\tint 0{\infty}\Re(\br PP_{1})dy\\
 & \peq+O(\|\tfrac{1}{y}\eps\|_{L^{2}}\|w_{1}\|_{L^{2}}+\|P\|_{L^{2}}\|\tfrac{1}{y}\eps_{1}\|_{L^{2}})\\
 & =-2(m+1)\eta+O_{M}(\mu^{2}),
\end{aligned}
\label{eq:mod-est-phase-cor}
\end{equation}
where we used \eqref{eq:phase-corr-1}, \eqref{eq:H1-nonlin-coer},
and \eqref{eq:e1-H1}.

Note that we require different bounds for the components of $\td{\Mod}$;
the bound \eqref{eq:ModEst-3} is rougher than \eqref{eq:ModEst-4},
but it is sufficient for our modulation analysis. This fact allows
us to construct $P$ only up to quadratic terms.
\end{rem}

To obtain these modulation estimates, we differentiate the orthogonality
conditions \eqref{eq:DecomOrtho} in time and use the evolution equations
for $\eps$ and $\eps_{1}$. In other words, we test \eqref{eq:eps-eqn}
and \eqref{eq:eps1-eqn} against $\calZ_{1},\calZ_{2}$ and $\td{\calZ}_{3},\td{\calZ}_{4}$,
respectively. It will be convenient to rewrite \eqref{eq:eps-eqn}
and \eqref{eq:eps1-eqn} as 
\begin{align}
\td{\Mod}\cdot\bfv & =(\rd_{s}-\tfrac{\lmb_{s}}{\lmb}\Lmb+\gmm_{s}i)\eps+\check{F},\label{eq:def-check-F}\\
\td{\Mod}\cdot\bfv_{1} & =(\rd_{s}-\tfrac{\lmb_{s}}{\lmb}\Lmb_{-1}+\gmm_{s}i)\eps_{1}+\check{F}_{1},\label{eq:def-check-F1}
\end{align}
where $\check{F}$ and $\check{F}_{1}$ are defined accordingly. The
following lemma collects the estimates for $\check{F}$ and $\check{F}_{1}$
for the proof of Proposition~\ref{prop:FirstModEst}. We relegate
its proof to Section~\ref{subsec:proof-struc-mod}.
\begin{lem}[Structural lemma for the modulation estimates]
\label{lem:StructFirstModEst}\ 
\begin{itemize}
\item (Rough low regularity estimates in compact region) We have 
\begin{equation}
|(\check{F},\calZ_{1})_{r}|+|(\check{F},\calZ_{2})_{r}|+|(\check{F}_{1},\td{\calZ}_{3})_{r}|+|(\check{F}_{1},\td{\calZ}_{4})_{r}|\aleq_{M}\mu.\label{eq:mod-struc1}
\end{equation}
\item (Higher Sobolev estimates in compact region) We have 
\begin{align}
|(\check{F},\calZ_{1})_{r}|+|(\check{F},\calZ_{2})_{r}| & \aleq_{M}X_{3},\label{eq:mod-struc2}\\
|(\check{F}_{1},\td{\calZ}_{3})_{r}|+|(\check{F}_{1},\td{\calZ}_{4})_{r}| & \aleq_{M}V_{7/2}.\label{eq:mod-struc3}
\end{align}
\item (Testing against almost nullspace elements) We have
\begin{equation}
|(\check{F}_{1},\chi_{\mu^{-1}}yQ)_{r}|+|(\check{F}_{1},\chi_{\mu^{-1}}iyQ)_{r}|\aleq_{M}\mu X_{3}+\beta^{4}|\log\beta|^{3}.\label{eq:mod-struc4}
\end{equation}
\end{itemize}
\end{lem}

\begin{rem}
The rough low regularity estimates \eqref{eq:mod-struc1} will be
used for the proof of \eqref{eq:ModEst-1}. The higher Sobolev estimates
\eqref{eq:mod-struc2} and \eqref{eq:mod-struc3} will be used for
the proof of \eqref{eq:ModEst-3} and \eqref{eq:ModEst-4}, respectively.
The last estimate \eqref{eq:mod-struc4} will be used to obtain second
modulation estimates (Lemma~\ref{lem:SecondModEst}) in Section~\ref{subsec:Asymptotics-of-modulation}.
\end{rem}

\begin{proof}[Proof of Proposition~\ref{prop:FirstModEst} assuming Lemma~\ref{lem:StructFirstModEst}]
Recall from Remark \ref{rem:Ortho-choice} that the orthogonality
profiles $\calZ_{1},\calZ_{2},\td{\calZ}_{3},\td{\calZ}_{4}$ are
supported in $y\leq2R_{\circ}$.

\textbf{Step 1.} Proof of \eqref{eq:ModEst-3} and \eqref{eq:ModEst-4}.

On one hand, we rearrange the equation for $\eps$ as 
\[
\td{\Mod}\cdot\{\bfv+(\Lmb\eps,-i\eps,0,0)^{t}\}=\rd_{s}\eps+b\Lmb\eps-(m+1)\eta i\eps+F.
\]
We take the inner product of the above display with $\calZ_{k}$ for
$k\in\{1,2\}$. Note that we have 
\begin{align*}
(\Lmb\eps,\calZ_{k})_{r} & =-(\eps,\Lmb\calZ_{k})_{r}\aleq\|\chf_{(0,2R_{\circ}]}\eps\|_{L^{\infty}}\aleq_{M}\min\{\mu,X_{3}\},\\
(-i\eps,\calZ_{k})_{r} & =(\eps,i\calZ_{k})_{r}\aleq\|\chf_{(0,2R_{\circ}]}\eps\|_{L^{\infty}}\aleq_{M}\min\{\mu,X_{3}\},\\
(\rd_{s}\eps,\calZ_{k})_{r} & =\rd_{s}(\eps,\calZ_{k})_{r}=0,
\end{align*}
where in the last equality we used the orthogonality condition \eqref{eq:DecomOrtho}.
Combining the above display with the estimate \eqref{eq:mod-struc2}
for $F$, we obtain
\begin{equation}
\Big|\sum_{j=1}^{4}\{((\bfv)_{j},\calZ_{k})_{r}+O_{M}(\mu)\}\td{\Mod}_{j}\Big|\aleq_{M}X_{3}.\label{eq:mod-est-04}
\end{equation}
On the other hand, we rearrange the equation for $\eps_{1}$ as 
\[
\td{\Mod}\cdot\{\bfv+(\Lmb_{-1}\eps_{1},-i\eps_{1},0,0)^{t}\}=\rd_{s}\eps_{1}+b\Lmb_{-1}\eps_{1}-(m+1)\eta i\eps_{1}+F_{1}.
\]
We take the inner product of the above display with $\td{\calZ}_{k}$
for $k\in\{3,4\}$. Proceeding similarly as in the previous paragraph,
but this time using \eqref{eq:mod-struc3} for $F_{1}$, we obtain
\begin{equation}
\Big|\sum_{j=1}^{4}\{((\bfv_{1})_{j},\td{\calZ}_{k})_{r}+O_{M}(\mu)\}\td{\Mod}_{j}\Big|\aleq_{M}V_{7/2}.\label{eq:mod-est-05}
\end{equation}
Thanks to the almost diagonal matrix structure \eqref{eq:matrix-transv},
we have 
\begin{align*}
|\td{\Mod}_{1}|+|\td{\Mod}_{2}| & \aleq_{M}\text{RHS\eqref{eq:mod-est-04}}+\mu\cdot\text{RHS\eqref{eq:mod-est-05}},\\
|\td{\Mod}_{3}|+|\td{\Mod}_{4}| & \aleq_{M}\text{RHS\eqref{eq:mod-est-05}}+\mu\cdot\text{RHS\eqref{eq:mod-est-04}}.
\end{align*}
Substituting the estimates \eqref{eq:mod-est-04} and \eqref{eq:mod-est-05}
into the above completes the proof of \eqref{eq:ModEst-3} and \eqref{eq:ModEst-4}.

\textbf{Step 2.} Proof of \eqref{eq:ModEst-1} and \eqref{eq:ModEst-2}.

For the proof of \eqref{eq:ModEst-1}, one follows the argument of
the previous step, but uses the rougher estimate \eqref{eq:mod-struc1}
instead of \eqref{eq:mod-struc2} and \eqref{eq:mod-struc3}. This
gives the estimate $|\td{\Mod}|\aleq_{M}\mu$. Combining this with
$\beta\aleq_{M}\mu$ and \eqref{eq:mod-est-phase-cor}, we obtain
\eqref{eq:ModEst-1}.

The estimate \eqref{eq:ModEst-2} easily follows from \eqref{eq:ModEst-4},
$V_{7/2}\leq X_{3}$, and $|p_{3}^{(b)}|+|p_{3}^{(\eta)}|\aleq\beta^{3}\aleq\mu^{3}$.
This completes the proof.
\end{proof}

\subsection{\label{subsec:Energy-Morawetz}Energy-Morawetz functional and uniform
$H^{3}$-bound}

This subsection is devoted to the proof of the key assertion of this
paper, namely, the uniform $H^{3}$-bound of $\eps^{\sharp}$. As
mentioned in Section~\ref{subsec:Strat-Proof}, the basic strategy
here is the \emph{energy method with repulsivity}, which has been
successfully employed in the blow-up constructions such as in \cite{RodnianskiSterbenz2010Ann.Math.,RaphaelRodnianski2012Publ.Math.,MerleRaphaelRodnianski2013InventMath,MerleRaphaelRodnianski2015CambJMath,Collot2018MemAMS}
and for \eqref{eq:CSS-m-equiv} in \cite{KimKwon2023AnnPDE,KimKwonOh2020arXiv}.
However, as we are considering arbitrary $H_{m}^{3}$-solutions near
modulated solitons, we need to control them without dynamical controls
such as \eqref{eq:old-dyn-control-1}. Note that controlling the solutions
without such dynamical assumptions are addressed in \cite{MerleRaphael2005AnnMath,Raphael2005MathAnn,RaphaelSzeftel2011JAMS}
for the mass-critical NLS and are given in a more complete picture
in \cite{MartelMerleRaphael2014Acta} for the critical gKdV. See also
\cite{JendrejLawrie2018Invent,JendrejLawrieRodriguez2022ASENS} for
energy-critical equivariant wave maps.

We introduce the energy functional 
\begin{equation}
\calF\coloneqq\frac{(\eps_{2},\td H_{Q}\eps_{2})_{r}}{\mu^{6}}-A\frac{(i\eps_{2},\psi'\rd_{y}\eps_{2})_{r}}{\mu^{5}}+A^{2}\frac{\|\eps_{2}\|_{L^{2}}^{2}}{\mu^{4}}\label{eq:def-EnergyMorawetz}
\end{equation}
with some large constant $A=A(M)\gg1$ and a function $\psi'$ given
in Lemma~\ref{lem:Morawetz}, and claim that $\calF\aleq1$ can be
propagated on any finite time intervals only under the size control
\eqref{eq:H1-nonlin-coer}. Since $0\leq\psi'\leq1$, it is clear
that $\calF$ is coercive in the $\eps_{2}$-variable, and hence it
controls $\|\eps\|_{\dot{\calH}_{m}^{3}}/\mu^{3}$ and $\|\eps\|_{\dot{\calH}_{m}^{2}}/\mu^{2}$
by Lemma~\ref{lem:HigherHardyBounds}. Note that the powers of $\mu$
in \eqref{eq:def-EnergyMorawetz} cannot be changed; see Remark~\ref{rem:scales}.

Note that $\mu=0$ if and only if $u$ is the modulated soliton $Q_{\lmb,\gmm}$
so we will not consider this trivial case. The key results \eqref{eq:unif-H3-bound}--\eqref{eq:unif-H2-bound}
of this subsection are nevertheless well-defined even for this soliton
solution.

In what follows, we explain how $\calF\aleq1$ can be propagated.
Our discussion will be sketchy and non-rigorous. We look at the structure
of the $\eps_{2}$-equation, whose full expression can be found in
\eqref{eq:e2-eqn}--\eqref{eq:def-F2} below. Here we only discuss
a toy model: 
\begin{equation}
\big(\rd_{s}-\frac{\lmb_{s}}{\lmb}\Lmb_{-2}+\gmm_{s}i\big)\eps_{2}+i\td H_{Q}\eps_{2}=i\beta V_{\mathrm{loc}}\eps_{2}+i(\beta^{2}+|\eps|^{2})\eps_{2}-i\Psi_{2},\label{eq:e2-toy-model}
\end{equation}
where $V_{\mathrm{loc}}$ is a smooth fast decaying potential and
$\Psi_{2}$ is the inhomogeneous error term (see \eqref{eq:P2-eqn}).
We will call the first term of RHS\eqref{eq:e2-toy-model} a \emph{local
term}, and call the second term a \emph{perturbative term}, whose
meaning will be made clear below.

We first consider the contribution of the perturbative term, whose
$\dot{H}^{1}$-norm can be roughly estimated by 
\[
\|i(\beta^{2}+|\eps|^{2})\eps_{2}\|_{\dot{H}^{1}}\aleq(\beta^{2}+\|\eps\|_{L^{\infty}}^{2})\|\eps_{2}\|_{\dot{H}^{1}}\aleq\mu^{2}\|\eps_{2}\|_{\dot{H}^{1}}
\]
and similarly for its $L^{2}$-norm, where in the last inequality
we applied the \emph{size control} \eqref{eq:H1-nonlin-coer}. In
this spirit, the contribution of the perturbative terms in the energy
estimates can be considered as 
\[
\rd_{s}\sqrt{\calF}\aleq\mu^{2}\sqrt{\calF},\qquad\text{or}\qquad\rd_{t}\sqrt{\calF}\aleq\sqrt{\calF}.
\]
By Grönwall's inequality, this inequality propagates the control $\calF\aleq1$
on any finite time intervals.

In a similar fashion, the contribution of the $\Psi_{2}$-term appears
as 
\[
\rd_{s}\sqrt{\calF}\aleq\mu^{-3}\|\Psi_{2}\|_{\dot{H}_{m}^{1}}+\mu^{-2}\|\Psi_{2}\|_{L^{2}}=\mu^{2}\big(\mu^{-5}\|\Psi_{2}\|_{\dot{H}_{m}^{1}}+\mu^{-4}\|\Psi_{2}\|_{L^{2}}\big).
\]
To have $\sqrt{\calF}$ bounded, in view of $\rd_{t}=\lmb^{-2}\rd_{s}$
and $\lmb\sim\mu$, we require that 
\begin{equation}
\frac{\|\Psi_{2}\|_{\dot{H}_{m}^{1}}}{\lmb^{5}}+\frac{\|\Psi_{2}\|_{L^{2}}}{\lmb^{4}}\text{ is integrable in }t.\label{eq:Psi2-integrable}
\end{equation}
Considering the formal ODE solutions \eqref{eq:formal-sol}, this
suggests us to have at least $\|\Psi_{2}\|_{\dot{H}_{m}^{1}}=O(\beta^{4+})$
and $\|\Psi_{2}\|_{L^{2}}=O(\beta^{3+})$ as stated in \eqref{eq:necessary-bound}.

We turn to the contribution of the local term, whose treatment requires
the \emph{Morawetz repulsivity} in a similar spirit to \cite{RodnianskiSterbenz2010Ann.Math.,MerleRaphaelRodnianski2015CambJMath}.
(The contribution of $-\frac{\lmb_{s}}{\lmb}\Lmb_{-2}\eps_{2}$ in
the LHS\eqref{eq:e2-toy-model} will also require the same treatment.)
Recall that we do not have a dynamical control such as $b>0$, but
our claim was that the size control \eqref{eq:H1-nonlin-coer} is
sufficient to run the energy method with repulsivity. We observe in
the computation of $\rd_{s}\{\mu^{-6}(\eps_{2},\td H_{Q}\eps_{2})_{r}\}$
that the following local terms appear:
\begin{equation}
\frac{1}{\mu^{6}}\Big\{-\frac{\lmb_{s}}{\lmb}\int\frac{\rd_{y}\td V_{Q}}{2y}|\eps_{2}|^{2}+\big(A_{Q}^{\ast}\eps_{2},A_{Q}^{\ast}(i\beta V_{\mathrm{loc}}\eps_{2})\big)_{r}\Big\},\label{eq:toy-model-nonpert}
\end{equation}
where the factorization $\td H_{Q}=A_{Q}A_{Q}^{\ast}$ is used for
the second term. In view of $-\frac{\lmb_{s}}{\lmb}\approx b$, \eqref{eq:toy-model-nonpert}
might be estimated by $\mu^{-6}\beta\|\eps_{2}\|_{\dot{H}^{1}}^{2}\aleq\beta\calF$,
and hence $\rd_{t}\calF\aleq\frac{\beta}{\lmb^{2}}\calF$, but this
estimate is not sufficient because $\frac{\beta}{\lmb^{2}}$ is not
integrable in $t$. Thus the term \eqref{eq:toy-model-nonpert} is
non-perturbative in this sense. However, if we recall that the potentials
involved in \eqref{eq:toy-model-nonpert} are \emph{well-localized}
and the \emph{size control} $\beta\aleq\mu$ is at hand, \eqref{eq:toy-model-nonpert}
can be estimated by a more local expression
\[
\mu^{-6}\beta\|\eps_{2}\|_{\dot{V}_{m+2}^{3/2}}^{2}\aleq\mu^{-5}\|\eps_{2}\|_{\dot{V}_{m+2}^{3/2}}^{2},
\]
which can be dominated by the monotonicity from the Morawetz term:
\[
\tfrac{1}{2}\rd_{s}\{\mu^{-5}(i\eps_{2},\psi'\rd_{y}\eps_{2})_{r}\}\geq\frkc\mu^{-5}\|\eps_{2}\|_{\dot{V}_{m+2}^{3/2}}^{2}+\text{(lower order)}.
\]
This motivates the Morawetz term in \eqref{eq:def-EnergyMorawetz}.
Note that the last lowest order term of \eqref{eq:def-EnergyMorawetz}
is introduced to make $\calF$ coercive.

We now state the key monotonicity estimate obtained from $\calF$,
which is the main result of this subsection. We will integrate this
estimate in Corollary~\ref{cor:integ-monot} below to obtain the
uniform $H^{3}$-bound of $\eps^{\sharp}$.
\begin{prop}[Energy-Morawetz functional]
\label{prop:EnergyMorawetz}There exist large constants $A=A(M)>1$
and $K=K(M)>1$ such that the functional 
\[
\calF\coloneqq\frac{(\eps_{2},\td H_{Q}\eps_{2})_{r}}{\mu^{6}}-A\frac{(i\eps_{2},\psi'\rd_{y}\eps_{2})_{r}}{\mu^{5}}+A^{2}\frac{\|\eps_{2}\|_{L^{2}}^{2}}{\mu^{4}}\tag{\ref{eq:def-EnergyMorawetz}}
\]
satisfies the following properties:
\begin{itemize}
\item (Coercivity) 
\begin{equation}
\calF\sim\frac{\|\eps_{2}\|_{\dot{H}_{m+2}^{1}}^{2}}{\mu^{6}}+A^{2}\frac{\|\eps_{2}\|_{L^{2}}^{2}}{\mu^{4}}\qquad\text{and hence}\qquad X_{3}\sim_{M}\mu^{3}\sqrt{1+\calF}\label{eq:EnergyMorawetzCoer}
\end{equation}
\item (Monotonicity) 
\begin{equation}
\rd_{s}\calF+\frac{\|\eps_{2}\|_{\dot{V}_{m+2}^{3/2}}^{2}}{\mu^{5}}\leq K\mu^{2}\Big\{(1+\calF)+\sqrt{\calF}\cdot\frac{\beta^{9/2}|\log\beta|^{2}}{\mu^{5}}\Big\}.\label{eq:EnergyMorawetzMonot}
\end{equation}
\end{itemize}
\end{prop}

\begin{rem}[The scales]
\label{rem:scales}The scales in \eqref{eq:def-EnergyMorawetz},
i.e., the powers of $\mu$, cannot be changed. The use of other powers
will contradict the \emph{$H^{3}$-boundedness of $\eps^{\sharp}$},
either in the shrinking regime or in the expanding regime. In this
sense, we are using crucially the fact that the linearized Hamiltonian
$\td H_{Q}$ defines a coercive bilinear form on a \emph{homogeneous}
Sobolev space $\dot{H}_{m+2}^{1}$.
\end{rem}

\begin{rem}[On $m\geq1$]
In the following proof, we heavily use the Morawetz monotonicity
(Lemma~\ref{lem:Morawetz}) to bound the local terms. Hence the assumption
$m\geq1$ is crucially used in the proof. See also Remark~\ref{rem:Morawetz}.
\end{rem}

We now record the full equation for $\eps_{2}$. As in Section~\ref{subsec:First-modulation-estimates},
we subtract the $P_{2}$-equation \eqref{eq:P2-eqn} from the $w_{2}$-equation
\eqref{eq:w2-eqn} and incorporate the phase correction \eqref{eq:def-td-Mod}
to obtain 
\begin{equation}
(\rd_{s}-\frac{\lmb_{s}}{\lmb}\Lmb_{-2}+\gmm_{s}i)\eps_{2}+i\td H_{Q}\eps_{2}=F_{2},\label{eq:e2-eqn}
\end{equation}
where 
\begin{equation}
\begin{aligned}F_{2} & =\td{\Mod}\cdot\bfv_{2}-i\Psi_{2}+(i\br ww_{1}^{2}-i\br PP_{1}^{2})\\
 & \peq-(i\td H_{w}-i\td H_{Q})\eps_{2}-(i\td H_{w}-i\td H_{P})P_{2}\\
 & \peq-(\tint y{\infty}\Re(\br ww_{1})dy')i\eps_{2}+(\tint 0y\Re(\br ww_{1}-\br PP_{1})dy')iP_{2}.
\end{aligned}
\label{eq:def-F2}
\end{equation}
The following lemma collects several estimates for $F_{2}$ for the
proof of Proposition~\ref{prop:EnergyMorawetz}. As illustrated before,
we are basically decomposing the terms of $F_{2}$ into the sum of
local and perturbative terms. The proof is relegated to Section~\ref{subsec:proof-struc-energy}.
\begin{lem}[Structural lemma for the energy estimate]
\label{lem:StrucEnergyEst}We can decompose 
\begin{equation}
F_{2}=F_{2,1}^{\mathrm{loc}}+F_{2,2}^{\mathrm{loc}}+F_{2}^{\mathrm{pert}},\quad\text{where}\quad F_{2,2}^{\loc}\coloneqq\sum_{j=3}^{4}\Mod_{j}(\bfv_{2})_{j},\label{eq:lyapunov-struc0}
\end{equation}
such that 
\begin{align}
\|F_{2,1}^{\loc}+F_{2}^{\pert}\|_{L^{2}} & \aleq_{M}\mu X_{3}+\beta^{4}|\log\beta|^{2},\label{eq:lyapunov-struc1}\\
\|\langle y\rangle|F_{2,1}^{\loc}|_{-1}\|_{L^{2}} & \aleq_{M}\mu V_{7/2},\label{eq:lyapunov-struc2}\\
\||F_{2}^{\pert}|_{-1}\|_{L^{2}} & \aleq_{M}\mu^{2}X_{3}+\beta^{9/2}|\log\beta|^{2},\label{eq:lyapunov-struc3}
\end{align}
and 
\begin{align}
\|yF_{2,2}^{\mathrm{loc}}\|_{L^{2}} & \aleq_{M}X_{3},\label{eq:lyapunov-struc4}\\
\|\langle y\rangle^{1/2}|F_{2,2}^{\mathrm{loc}}|_{1}\|_{L^{2}} & \aleq_{M}\mu^{1/2}V_{7/2},\label{eq:lyapunov-struc5}\\
\|\langle y\rangle^{1/2}y^{-1}|F_{2,2}^{\mathrm{loc}}|_{2}\|_{L^{2}} & \aleq_{M}\mu V_{7/2}.\label{eq:lyapunov-struc6}
\end{align}
\end{lem}

\begin{rem}
When $m=1$, the modulation term $F_{2,2}^{\loc}=\sum_{j=3}^{4}\td{\Mod}_{j}\cdot(\bfv_{2})_{j}$
does not satisfy the bound \eqref{eq:lyapunov-struc2} due to the
lack of spatial decay of $(\bfv_{2})_{j}$. Thus we will directly
estimate this contribution in the energy estimates.
\end{rem}

\begin{proof}[Proof of Proposition~\ref{prop:EnergyMorawetz} assuming Lemma~\ref{lem:StrucEnergyEst}]
\ 

\textbf{Step 1.} Proof of the coercivity \eqref{eq:EnergyMorawetzCoer}.

This is obvious from \eqref{eq:H-td-coer}, $|\psi'|\leq1$, Cauchy--Schwarz,
and \eqref{eq:def-X2}--\eqref{eq:def-X3}.

\textbf{Step 2.} Algebraic computation of $\rd_{s}\calF$.

In Steps 2--3, we prove the monotonicity \eqref{eq:EnergyMorawetzMonot}.
The goal of this step is to derive the algebraic identity 
\begin{equation}
\begin{aligned}\frac{1}{2}\rd_{s}\calF & =\frac{1}{\mu^{6}}\Big\{-\frac{\lmb_{s}}{\lmb}\int\frac{\rd_{y}\td V_{Q}}{2y}|\eps_{2}|^{2}+(A_{Q}^{\ast}\eps_{2},A_{Q}^{\ast}F_{2})_{r}\Big\}\\
 & \peq-\frac{A}{\mu^{5}}\Big\{-\frac{\lmb_{s}}{2\lmb}\big(i\eps_{2},(y\rd_{y}\psi')\rd_{y}\eps_{2}\big)_{r}+(\td H_{Q}\eps_{2},\Lmb_{\psi}\eps_{2})_{r}+(iF_{2},\Lmb_{\psi}\eps_{2})_{r}\Big\}\\
 & \peq+\frac{A^{2}}{\mu^{4}}(\eps_{2},F_{2})_{r}.
\end{aligned}
\label{eq:Lyapunov-2-1}
\end{equation}
First, we use $\frac{\mu_{s}}{\mu}=\frac{\lmb_{s}}{\lmb}$ and the
equation \eqref{eq:e2-eqn} for $\eps_{2}$ to have 
\begin{align*}
\frac{1}{2}\rd_{s}\Big\{\frac{(\eps_{2},\td H_{Q}\eps_{2})_{r}}{\mu^{6}}\Big\} & =\frac{1}{\mu^{6}}\big((\rd_{s}-3\frac{\lmb_{s}}{\lmb})\eps_{2},\td H_{Q}\eps_{2}\big)_{r}\\
 & =\frac{1}{\mu^{6}}\big(\frac{\lmb_{s}}{\lmb}y\rd_{y}\eps_{2}-\gmm_{s}i\eps_{2}-i\td H_{Q}\eps_{2}+F_{2},\td H_{Q}\eps_{2}\big)_{r}.
\end{align*}
Using $\td H_{Q}=A_{Q}A_{Q}^{\ast}=-\rd_{yy}-\tfrac{1}{y}\rd_{y}+\tfrac{1}{y^{2}}\td V_{Q}$
(see \eqref{eq:H-td-factor}), we obtain 
\begin{equation}
\frac{1}{2}\rd_{s}\Big\{\frac{(\eps_{2},\td H_{Q}\eps_{2})_{r}}{\mu^{6}}\Big\}=\frac{1}{\mu^{6}}\Big\{-\frac{\lmb_{s}}{\lmb}\int\frac{\rd_{y}\td V_{Q}}{2y}|\eps_{2}|^{2}+(A_{Q}^{\ast}\eps_{2},A_{Q}^{\ast}F_{2})_{r}\Big\}.\label{eq:Lyapunov-2-2}
\end{equation}
Next, we similarly compute 
\begin{align*}
\frac{1}{2}\rd_{s}\Big\{ & -A\frac{(i\eps_{2},\psi'\rd_{y}\eps_{2})_{r}}{\mu^{5}}\Big\}\\
 & =-\frac{A}{2\mu^{5}}\Big[\frac{\lmb_{s}}{\lmb}\Big\{(i\Lmb_{-2}\eps_{2},\psi'\rd_{y}\eps_{2})_{r}+(i\eps_{2},\psi'\rd_{y}\Lmb_{-2}\eps_{2})_{r}-5(i\eps_{2},\psi'\rd_{y}\eps_{2})_{r}\Big\}\\
 & \peq\qquad+\big\{(\td H_{Q}\eps_{2},\psi'\rd_{y}\eps_{2})_{r}-(\eps_{2},\psi'\rd_{y}\td H_{Q}\eps_{2})_{r}\big\}\\
 & \peq\qquad+\big\{(iF_{2},\psi'\rd_{y}\eps_{2})_{r}+(i\eps_{2},\psi'\rd_{y}F_{2})_{r}\big\}\Big].
\end{align*}
In the above display, we integrate by parts the scaling term
\begin{align*}
(i\eps_{2},\psi'\rd_{y}\Lmb_{-2}\eps_{2})_{r} & =\big(i\eps_{2},\Lmb_{-3}(\psi'\rd_{y}\eps_{2})-(y\rd_{y}\psi')\rd_{y}\eps_{2}\big)_{r}\\
 & =-\big(i\Lmb_{3}\eps_{2},\psi'\rd_{y}\eps_{2})_{r}-(i\eps_{2},(y\rd_{y}\psi')\rd_{y}\eps_{2}\big)_{r}
\end{align*}
and apply the operator identity $\psi'\rd_{y}-\rd_{y}^{\ast}\psi'=2\Lmb_{\psi}$
for the terms involving $\td H_{Q}$ and $F_{2}$. This gives 
\begin{equation}
\begin{aligned}\frac{1}{2}\rd_{s} & \Big\{-A\frac{(i\eps_{2},\psi'\rd_{y}\eps_{2})_{r}}{\mu^{5}}\Big\}\\
 & =-\frac{A}{\mu^{5}}\Big\{-\frac{\lmb_{s}}{2\lmb}\big(i\eps_{2},(y\rd_{y}\psi')\rd_{y}\eps_{2}\big)_{r}+(\td H_{Q}\eps_{2},\Lmb_{\psi}\eps_{2})_{r}+(iF_{2},\Lmb_{\psi}\eps_{2})_{r}\Big\}.
\end{aligned}
\label{eq:Lyapunov-2-3}
\end{equation}
Finally, we have 
\begin{equation}
\frac{1}{2}\rd_{s}\Big\{ A^{2}\frac{\|\eps_{2}\|_{L^{2}}^{2}}{\mu^{4}}\Big\}=\frac{A^{2}}{\mu^{4}}(\eps_{2},F_{2})_{r}.\label{eq:Lyapunov-2-4}
\end{equation}
Summing up the displays \eqref{eq:Lyapunov-2-2}--\eqref{eq:Lyapunov-2-4}
completes the proof of the claim \eqref{eq:Lyapunov-2-1}.

\textbf{Step 3.} Proof of the monotonicity \eqref{eq:EnergyMorawetzMonot}.

In this step, we prove \eqref{eq:EnergyMorawetzMonot} using the computation
\eqref{eq:Lyapunov-2-1}.

\emph{\uline{First line of RHS\mbox{\eqref{eq:Lyapunov-2-1}}.}}
We claim that 
\begin{equation}
\begin{aligned}\frac{1}{\mu^{6}}\Big\{ & -\frac{\lmb_{s}}{\lmb}\int\frac{\rd_{y}\td V_{Q}}{2y}|\eps_{2}|^{2}+(A_{Q}^{\ast}\eps_{2},A_{Q}^{\ast}F_{2})_{r}\Big\}\\
 & \leq C_{1}(M)\Big\{\frac{\|\eps_{2}\|_{\dot{V}_{m+2}^{3/2}}^{2}}{\mu^{5}}+\mu^{2}\Big(1+\calF+\sqrt{\calF}\cdot\frac{\beta^{9/2}|\log\beta|^{2}}{\mu^{5}}\Big)\Big\},
\end{aligned}
\label{eq:Lyapunov-04}
\end{equation}
where $C_{1}(M)$ is some constant depending on $M$. We rewrite LHS\eqref{eq:Lyapunov-04}
using the decomposition \eqref{eq:lyapunov-struc0} as 
\begin{align*}
\text{LHS\eqref{eq:Lyapunov-04}}=\frac{1}{\mu^{6}}\Big\{ & -\frac{\lmb_{s}}{\lmb}\int\frac{\rd_{y}\td V_{Q}}{2y}|\eps_{2}|^{2}+(\eps_{2},\td H_{Q}F_{2,2}^{\loc})_{r}\\
 & +(A_{Q}^{\ast}\eps_{2},A_{Q}^{\ast}(F_{2,1}^{\loc}+F_{2}^{\pert}))_{r}\Big\}.
\end{align*}
For the first term (which is a local term), we use \eqref{eq:ModEst-1}
to have 
\[
-\frac{\lmb_{s}}{\lmb}\int\frac{\rd_{y}\td V_{Q}}{2y}|\eps_{2}|^{2}=\frac{\lmb_{s}}{2\lmb}\int Q^{2}|\eps_{2}|^{2}\aleq_{M}\mu\|\eps_{2}\|_{\dot{V}_{m+2}^{3/2}}^{2}.
\]
For the second term, we use \eqref{eq:lyapunov-struc6}, \eqref{eq:def-V7/2},
and \eqref{eq:EnergyMorawetzCoer} to have 
\begin{align*}
\big(\eps_{2},\td H_{Q}F_{2,2}^{\loc}\big)_{r} & \aleq\|\langle y\rangle^{-1/2}y^{-1}\eps_{2}\|_{L^{2}}\|\langle y\rangle^{1/2}y^{-1}|F_{2,2}^{\loc}|_{2}\|_{L^{2}}\\
 & \aleq_{M}\|\eps_{2}\|_{\dot{V}_{m+2}^{3/2}}\cdot\mu V_{7/2}\\
 & \aleq_{M}\mu\|\eps_{2}\|_{\dot{V}_{m+2}^{3/2}}^{2}+\mu^{2}(X_{3})^{2}\aleq_{M}\mu\|\eps_{2}\|_{\dot{V}_{m+2}^{3/2}}^{2}+\mu^{8}(1+\calF).
\end{align*}
For the third term, we use \eqref{eq:lyapunov-struc2}--\eqref{eq:lyapunov-struc3}
to have 
\begin{align*}
 & \big(A_{Q}^{\ast}\eps_{2},A_{Q}^{\ast}(F_{2,1}^{\loc}+F_{2}^{\pert})\big)_{r}\\
 & \aleq\|\langle y\rangle^{-1}A_{Q}^{\ast}\eps_{2}\|_{L^{2}}\|\langle y\rangle|F_{2,1}^{\loc}|_{-1}\|_{L^{2}}+\|A_{Q}^{\ast}\eps_{2}\|_{L^{2}}\||F_{2}^{\pert}|_{-1}\|_{L^{2}}\\
 & \aleq_{M}\|\eps_{2}\|_{\dot{V}_{m+2}^{3/2}}\cdot\mu V_{7/2}+\|\eps_{2}\|_{\dot{H}_{m+2}^{1}}\cdot(\mu^{2}X_{3}+\beta^{9/2}|\log\beta|^{2})\\
 & \aleq_{M}\mu(V_{7/2})^{2}+\mu^{2}(X_{3})^{2}+\mu^{3}\sqrt{\calF}\cdot\beta^{9/2}|\log\beta|^{2}\\
 & \aleq_{M}\mu\|\eps_{2}\|_{\dot{V}_{m+2}^{3/2}}^{2}+\mu^{8}(1+\calF)+\mu^{3}\sqrt{\calF}\cdot\beta^{9/2}|\log\beta|^{2}
\end{align*}
This completes the proof of the claim \eqref{eq:Lyapunov-04}.

\emph{\uline{Second line of RHS\mbox{\eqref{eq:Lyapunov-2-1}}.}}
We claim 
\begin{equation}
\begin{aligned}- & \frac{A}{\mu^{5}}\Big\{-\frac{\lmb_{s}}{2\lmb}\big(i\eps_{2},(y\rd_{y}\psi')\rd_{y}\eps_{2}\big)_{r}+(\td H_{Q}\eps_{2},\Lmb_{\psi}\eps_{2})_{r}+(iF_{2},\Lmb_{\psi}\eps_{2})_{r}\Big\}\\
 & \leq-\frac{\frkc A}{2\mu^{5}}\|\eps_{2}\|_{\dot{V}_{m+2}^{3/2}}^{2}+C_{2}(M)A\mu^{2}\Big\{1+\calF+\sqrt{\calF}\cdot\frac{\beta^{4}|\log\beta|^{\frac{5}{2}}}{\mu^{4}}\Big\},
\end{aligned}
\label{eq:Lyapunov-05}
\end{equation}
where $\frkc$ is the constant in \eqref{eq:Morawetz-repul} and $C_{2}(M)$
is some constant depending on $M$. To see this, similarly as in the
previous paragraph, we rewrite 
\begin{align*}
\text{LHS}\ref{eq:Lyapunov-05} & =-\frac{A}{\mu^{5}}\Big\{(\td H_{Q}\eps_{2},\Lmb_{\psi}\eps_{2})_{r}-\frac{\lmb_{s}}{2\lmb}\big(i\eps_{2},(y\rd_{y}\psi')\rd_{y}\eps_{2}\big)_{r}\\
 & \qquad\qquad-(\eps_{2},i\Lmb_{\psi}F_{2,2}^{\loc})_{r}-(i\Lmb_{\psi}\eps_{2},F_{2,1}^{\loc}+F_{2}^{\pert})_{r}\Big\}.
\end{align*}
We apply \eqref{eq:Morawetz-repul} to the first term of LHS\eqref{eq:Lyapunov-05}
to obtain the \emph{key monotonicity} 
\[
-\frac{A}{\mu^{5}}(\td H_{Q}\eps_{2},\Lmb_{\psi}\eps_{2})_{r}\leq-\frac{\frkc A}{\mu^{5}}\|\eps_{2}\|_{\dot{V}_{m+2}^{3/2}}^{2}.
\]
The other terms of LHS\eqref{eq:Lyapunov-05} are error terms. Using
\eqref{eq:ModEst-1} and $y\rd_{y}\psi'\aleq\langle y\rangle^{-1}$
from \eqref{eq:ydy-psi'}, we have 
\[
-\frac{\lmb_{s}}{2\lmb}\big(i\eps_{2},(y\rd_{y}\psi')\rd_{y}\eps_{2}\big)_{r}\aleq_{M}\mu\|\eps_{2}\|_{\dot{H}_{m+2}^{1}}^{2}\aleq_{M}\mu^{7}\calF.
\]
Using \eqref{eq:lyapunov-struc5} and \eqref{eq:def-V7/2}, we have
\begin{align*}
(\eps_{2},i\Lmb_{\psi}F_{2,2}^{\loc})_{r} & \aleq\|\langle y\rangle^{-1/2}y^{-1}\eps_{2}\|_{L^{2}}\|\langle y\rangle^{1/2}|F_{2,2}^{\loc}|_{1}\|_{L^{2}}\\
 & \aleq_{M}\|\eps_{2}\|_{\dot{V}_{m+2}^{3/2}}\cdot\mu^{1/2}V_{7/2}\\
 & \aleq_{M}\mu^{1/2}(\|\eps_{2}\|_{\dot{V}_{m+2}^{3/2}}^{2}+\mu X_{3}^{2})\\
 & \aleq\delta_{M}(\alpha^{\ast})(\|\eps_{2}\|_{\dot{V}_{m+2}^{3/2}}^{2}+\mu^{7}(1+\calF)).
\end{align*}
Using \eqref{eq:lyapunov-struc1}, we have 
\begin{align*}
(i\Lmb_{\psi}\eps_{2},F_{2,1}^{\loc}+F_{2}^{\pert})_{r} & \aleq_{M}\||\eps_{2}|_{-1}\|_{L^{2}}\cdot(\mu X_{3}+\beta^{4}|\log\beta|^{2})\\
 & \aleq_{M}\mu^{3}\sqrt{\calF}\cdot(\mu^{4}(\sqrt{\calF}+1)+\beta^{4}|\log\beta|^{2}).
\end{align*}
Substituting the previous bounds into LHS\eqref{eq:Lyapunov-05} and
using $\frkc-\delta_{M}(\alpha^{\ast})>\frac{1}{2}\frkc$ complete
the proof of the claim \eqref{eq:Lyapunov-05}.

\emph{\uline{Last line of RHS\mbox{\eqref{eq:Lyapunov-2-1}}.}}
We claim 
\begin{equation}
\frac{A^{2}}{\mu^{4}}(\eps_{2},F_{2})_{r}\leq C_{3}(M)A\mu^{2}\Big\{\calF+\sqrt{\calF}\Big(\frac{\beta^{4}|\log\beta|^{2}}{\mu^{4}}+1\Big)\Big\},\label{eq:Lyapunov-06}
\end{equation}
where $C_{3}(M)$ is some constant depending on $M$. Indeed, this
follows from \eqref{eq:lyapunov-struc1} and \eqref{eq:EnergyMorawetzCoer}:
\begin{align*}
\frac{A^{2}}{\mu^{4}}(\eps_{2},F_{2})_{r} & \aleq_{M}\frac{A^{2}}{\mu^{4}}\|\eps_{2}\|_{L^{2}}\cdot(\mu X_{3}+\beta^{4}|\log\beta|^{2})\\
 & \aleq_{M}A\mu^{2}\sqrt{\calF}\Big\{(\sqrt{\calF}+1)+\frac{\beta^{4}|\log\beta|^{2}}{\mu^{4}}\Big\}.
\end{align*}

\emph{\uline{Completion of the proof of \mbox{\eqref{eq:EnergyMorawetzMonot}}.}}
Choose $A=A(M)>0$ such that $\frac{1}{2}\frkc A>C_{1}(M)+1$. Substituting
the claims \eqref{eq:Lyapunov-04}--\eqref{eq:Lyapunov-06} into
\eqref{eq:Lyapunov-2-1}, we have 
\[
\frac{1}{2}\rd_{s}\calF+\frac{\|\eps_{2}\|_{\dot{V}_{m+2}^{3/2}}}{\mu^{5}}\leq(C_{2}(M)+C_{3}(M))A\mu^{2}\Big\{1+\calF+\sqrt{\calF}\Big(\frac{\beta^{9/2}|\log\beta|^{2}}{\mu^{5}}+\frac{\beta^{4}|\log\beta|^{2}}{\mu^{4}}\Big)\Big\}.
\]
Using $\beta\aleq_{M}\mu$ and $\mu\ll1$, the above display can be
written more compactly as 
\[
\rd_{s}\calF+\frac{\|\eps_{2}\|_{\dot{V}_{m+2}^{3/2}}}{\mu^{5}}\aleq_{M}\mu^{2}\Big\{1+\calF+\sqrt{\calF}\cdot\frac{\beta^{9/2}|\log\beta|^{2}}{\mu^{5}}\Big\},
\]
which is \eqref{eq:EnergyMorawetzMonot}. This completes the proof.
\end{proof}
Integrating the monotonicity estimate \eqref{eq:EnergyMorawetzMonot},
we obtain the key (local-in-time) \emph{uniform $H^{3}$ bound} as
well as a spacetime bound.
\begin{cor}[Integration of the monotonicity]
\label{cor:integ-monot}Let $I$ be a bounded time interval either
of the form $[t_{0},t_{1})$ or $[t_{0},t_{1}]$ such that $u(t)\in\calT_{\alpha^{\ast},M}$
for all $t\in I$. Let $L>0$ be such that $E[u](t_{1}-t_{0})\leq L$.
Then, we have 
\begin{equation}
\begin{aligned}\sup_{t\in I}\Big(\frac{\|\eps_{2}\|_{\dot{H}_{m+2}^{1}}^{2}}{\mu^{6}}+\frac{\|\eps_{2}\|_{L^{2}}^{2}}{\mu^{4}}\Big)(t) & +\int_{I}\frac{\|\eps_{2}\|_{\dot{V}_{m+2}^{3/2}}^{2}}{\mu^{5}}\frac{dt}{\lmb^{2}}\\
 & \aleq_{M,L}1+\Big(\frac{\|\eps_{2}\|_{\dot{H}_{m+2}^{1}}^{2}}{\mu^{6}}+\frac{\|\eps_{2}\|_{L^{2}}^{2}}{\mu^{4}}\Big)(t_{0}).
\end{aligned}
\label{eq:integ-monot}
\end{equation}
In particular, we have local-in-time uniform $H^{3}$-bounds:
\begin{align}
\Big(\|\eps\|_{\dot{\calH}_{m}^{3}}+\|\eps_{1}\|_{\dot{H}_{m+1}^{2}}+\|\eps_{2}\|_{\dot{H}_{m+2}^{1}}\Big)(t) & \aleq_{M,L}\Big(\frac{\lmb(t)}{\lmb(t_{0})}\Big)^{3}X_{3}(t_{0}),\label{eq:unif-H3-bound}\\
\Big(\|\eps\|_{\dot{\calH}_{m}^{2}}+\|\eps_{1}\|_{\dot{H}_{m+1}^{1}}+\|\eps_{2}\|_{L^{2}}\Big)(t) & \aleq_{M,L}\Big(\frac{\lmb(t)}{\lmb(t_{0})}\Big)^{2}\sqrt{\mu X_{3}}(t_{0}).\label{eq:unif-H2-bound}
\end{align}
\end{cor}

\begin{rem}[Uniform $H^{3}$-bound]
As is clear from the implicit constants involved in the bounds \eqref{eq:unif-H3-bound}--\eqref{eq:unif-H2-bound},
we see that $\|\eps^{\sharp}(t)\|_{H_{m}^{3}}$ is \emph{uniformly
bounded} for all $H^{3}$ nearby solutions (due to the presence of
$M$ and $X_{3}$) on a \emph{bounded time interval} (due to the presence
of $L$). The uniform bounds \eqref{eq:unif-H3-bound}--\eqref{eq:unif-H2-bound}
can also be viewed as persistence of regularity estimates, but the
point is that the $H^{3}$-regularity of $\eps^{\sharp}$ persists
regardless whether $u$ blows up or not. We also note that, as can
be seen in the proof below, the $L$-dependences of the implicit constants
in \eqref{eq:integ-monot} and \eqref{eq:unif-H3-bound} are exponential
in view of Grönwall's inequality.
\end{rem}

\begin{rem}
All the bounds in the above corollary are scaling invariant.
\end{rem}

\begin{rem}
We will not use the spacetime bound (i.e., the second term of LHS\eqref{eq:integ-monot})
in our proof of Theorem~\ref{thm:blow-up-rigidity}. However, we
include it to emphasize that \eqref{eq:integ-monot} is indeed a monotonicity
estimate.
\end{rem}

\begin{proof}
\textbf{Step 1.} Fixed-time bound of \eqref{eq:integ-monot}.

In this step, we prove for any $t\in I$ 
\begin{equation}
\calF(t)\aleq_{M,L}\calF(t_{0})+1.\label{eq:integ-monot-3}
\end{equation}
To show this, we first ignore the spacetime term of \eqref{eq:EnergyMorawetzMonot}
and divide the both sides by $\sqrt{1+\calF}$ to obtain 
\begin{equation}
\rd_{s}\sqrt{1+\calF}\leq K\mu^{2}\Big\{\sqrt{1+\calF}+\frac{\beta^{9/2}|\log\beta|^{2}}{\mu^{5}}\Big\}.\label{eq:integ-monot-1}
\end{equation}
To integrate the inhomogeneous term of the RHS, we use \eqref{eq:ModEst-3}
and \eqref{eq:ModEst-2} to note the following inequality
\begin{equation}
\begin{aligned}\rd_{s}\Big(\frac{b}{\mu^{3/4}}\Big) & =\frac{1}{\mu^{3/4}}\Big\{-\Big(\frac{1}{4}b^{2}+\eta^{2}\Big)+(b_{s}+b^{2}+\eta^{2})-\frac{3}{4}b\Big(\frac{\mu_{s}}{\mu}+b\Big)\Big\}\\
 & \leq-\frac{\beta^{2}}{4\mu^{3/4}}+K'\mu^{3-\frac{3}{4}}\sqrt{1+\calF},
\end{aligned}
\label{eq:integ-monot-4}
\end{equation}
for some constant $K'=K'(M)$. Summing the above display with \eqref{eq:integ-monot-1},
we obtain 
\[
\rd_{s}\Big(\sqrt{1+\calF}+\frac{b}{\mu^{3/4}}\Big)\leq\Big(K\frac{\beta^{9/2}|\log\beta|^{2}}{\mu^{3}}-\frac{\beta^{2}}{4\mu^{3/4}}\Big)+\Big(K\mu^{2}+K'\mu^{3-\frac{3}{4}}\Big)\sqrt{1+\calF}.
\]
Since $\beta\aleq_{M}\mu\ll1$, we have $|b|/\mu^{3/4}\ll1$ and $\beta^{9/2}|\log\beta|^{2}/\mu^{3}\ll\beta^{2}/\mu^{3/4}$.
Thus the above display can be simplified as 
\begin{equation}
\rd_{s}\Big(\sqrt{1+\calF}+\frac{b}{\mu^{3/4}}\Big)\leq K''\mu^{2}\Big(\sqrt{1+\calF}+\frac{b}{\mu^{3/4}}\Big),\label{eq:integ-monot-2}
\end{equation}
for some constant $K''=K''(M)$. Now \eqref{eq:integ-monot-3} follows
from integrating \eqref{eq:integ-monot-2} using $ds=\lmb^{2}dt$
and $\mu=\lmb\sqrt{E[u]}$: 
\begin{align*}
\sqrt{\calF(t)}\leq\Big(\sqrt{1+\calF}+\frac{b}{\mu^{3/4}}\Big)(t)+1 & \leq e^{K''E[u](t-t_{0})}\Big(\sqrt{1+\calF}+\frac{b}{\mu^{3/4}}\Big)(t_{0})+1\\
 & \aleq_{M,L}\sqrt{1+\calF(t_{0})}.
\end{align*}

\textbf{Step 2.} Spacetime bound of \eqref{eq:integ-monot}.

In this step, we show the spacetime bound of \eqref{eq:integ-monot}.
We integrate \eqref{eq:EnergyMorawetzMonot}, apply \eqref{eq:integ-monot-3}
to pointwisely bound $\calF(t)$, and use $\mu=\lmb\sqrt{E[u]}$ to
have 
\begin{equation}
\begin{aligned} & \int_{I}\frac{\|\eps_{2}\|_{\dot{V}_{m+2}^{3/2}}^{2}}{\mu^{5}}\frac{dt}{\lmb^{2}}\\
 & \quad\aleq\sup_{t\in I}\calF(t)+\int_{I}K\mu^{2}\Big\{1+\calF+\sqrt{\calF}\cdot\frac{\beta^{9/2}|\log\beta|^{2}}{\mu^{5}}\Big\}\frac{dt}{\lmb^{2}}\\
 & \quad\aleq_{M,L}\Big(\calF(t_{0})+1\Big)+\sqrt{1+\calF(t_{0})}\int_{s(I)}\frac{\beta^{9/2}|\log\beta|^{2}}{\mu^{3}}ds.
\end{aligned}
\label{eq:integ-monot-5}
\end{equation}
For the last integral, we note as a consequence of \eqref{eq:integ-monot-4},
$\mu^{1/4}=\delta_{M}(\alpha^{\ast})$, and \eqref{eq:integ-monot-3}
that 
\begin{align*}
\int_{s(I)}\frac{\beta^{9/2}|\log\beta|^{2}}{\mu^{3}}ds & \leq\int_{s(I)}\frac{\beta^{2}}{\mu^{3/4}}ds\\
 & \aleq\sup_{t\in I}\frac{|b(t)|}{\mu^{3/4}(t)}+\int_{s(I)}K'\mu^{3-\frac{3}{4}}\sqrt{1+\calF}ds\\
 & \aleq\delta_{M}(\alpha^{\ast})+\delta_{M}(\alpha^{\ast})\Big(\sup_{t\in I}\sqrt{1+\calF(t)}\Big)\int_{s(I)}\mu^{2}ds\\
 & \aleq_{M,L}\sqrt{1+\calF(t_{0})}.
\end{align*}
Substituting this bound into \eqref{eq:integ-monot-5} gives 
\[
\int_{I}\frac{\|\eps_{2}\|_{\dot{V}_{m+2}^{3/2}}^{2}}{\mu^{5}}\frac{dt}{\lmb^{2}}\aleq_{M,L}\Big(\calF(t_{0})+1\Big),
\]
which is the spacetime bound of \eqref{eq:integ-monot}. This completes
the proof of \eqref{eq:integ-monot}.

\textbf{Step 3.} Proof of the uniform bounds \eqref{eq:unif-H3-bound}--\eqref{eq:unif-H2-bound}.

We first show \eqref{eq:unif-H3-bound}. By \eqref{eq:EnergyMorawetzCoer}
and \eqref{eq:integ-monot-3}, we have 
\[
X_{3}(t)\aleq_{M,L}\mu^{3}(t)(1+\sqrt{\calF(t_{0})})\sim_{M,L}\frac{\mu^{3}(t)}{\mu^{3}(t_{0})}X_{3}(t_{0})=\frac{\lmb^{3}(t)}{\lmb^{3}(t_{0})}X_{3}(t_{0}).
\]
Combining the above with Lemma \ref{lem:HigherHardyBounds} gives
\eqref{eq:unif-H3-bound}.

We turn to the proof of \eqref{eq:unif-H2-bound}. Interpolating \eqref{eq:unif-H3-bound}
and \eqref{eq:H1-nonlin-coer}, we have 
\[
\|\eps_{1}\|_{\dot{H}_{m+1}^{1}}\aleq\|\eps_{1}\|_{L^{2}}^{1/2}\|\eps_{1}\|_{\dot{H}_{m+1}^{2}}^{1/2}\aleq_{M,L}\mu^{1/2}(t)\frac{\lmb^{3/2}(t)}{\lmb^{3/2}(t_{0})}\sqrt{X_{3}(t_{0})}=\frac{\lmb^{2}(t)}{\lmb^{2}(t_{0})}\sqrt{\mu X_{3}}(t_{0}).
\]
This bound, combined with \eqref{eq:e1-coer-01}, \eqref{eq:e1-coer-02},
and \eqref{eq:compat2-H2}, gives 
\[
\|\eps_{2}\|_{L^{2}}\aleq\|\eps_{1}\|_{\dot{H}_{m+1}^{1}}+\beta^{2}+O_{M}(\mu^{2})\aleq_{M,L}\frac{\lmb^{2}(t)}{\lmb^{2}(t_{0})}\sqrt{\mu X_{3}}(t_{0}).
\]
This bound propagates for $\eps$ via \eqref{eq:e-H2}. This completes
the proof of \eqref{eq:unif-H2-bound}.
\end{proof}

\subsection{\label{subsec:proof-struc-mod}Proof of the structural Lemma~\ref{lem:StructFirstModEst}}
\begin{proof}[Proof of Lemma~\ref{lem:StructFirstModEst}]
\ 

\textbf{Step 1.} Estimates for $\check{F}$.

We show the estimates \eqref{eq:mod-struc1} and \eqref{eq:mod-struc2}
for $\check{F}$. In particular, we only need estimates restricted
to the compact region $y\leq2R_{\circ}$, where we recall $R_{\circ}$
from Remark~\ref{rem:Ortho-choice}. Recall that 
\[
\check{F}=i(\bfD_{w}^{\ast}w_{1}-\bfD_{P}^{\ast}P_{1})-(\tint 0y\Re(\br ww_{1}-\br PP_{1})dy')iP+(\tint y{\infty}\Re(\br ww_{1})dy')i\eps+i\Psi.
\]

\uline{The term \mbox{$i\bfD_{w}^{\ast}\eps_{1}$} of \mbox{$\check{F}$}.}
This term has a derivative, so we need to integrate by parts to have
low regularity estimate. Since $\calZ_{k}\in C_{c,m}^{\infty}$, we
have 
\[
(i\bfD_{w}^{\ast}\eps_{1},\calZ_{k})_{r}=-(\eps_{1},i\bfD_{w}\calZ_{k})_{r}\aleq_{M}\|\chf_{(0,2R_{\circ}]}\eps_{1}\|_{L^{2}}\aleq_{M}\min\{\mu,X_{3}\},
\]
where we applied either \eqref{eq:H1-nonlin-coer} or \eqref{eq:e1-H2}
in the last inequality. This yields \eqref{eq:mod-struc1} and \eqref{eq:mod-struc2}
for the term $-i\bfD_{w}^{\ast}\eps_{1}$ of $\check{F}$.

\uline{Remaining terms of \mbox{$\check{F}$}.} The remaining terms
are given by 
\begin{align*}
\check{F}-i\bfD_{w}^{\ast}\eps_{1} & =i(\bfD_{w}^{\ast}-\bfD_{P}^{\ast})P_{1}\\
 & \qquad-(\tint 0y\Re(\br ww_{1}-\br PP_{1})dy')iP+(\tint y{\infty}\Re(\br ww_{1})dy')i\eps+i\Psi.
\end{align*}
These terms are merely estimated in local $L^{2}$-norms. We claim
that 
\begin{equation}
\|\chf_{(0,2R_{\circ}]}(\check{F}+i\bfD_{w}^{\ast}\eps_{1})\|_{L^{2}}\aleq_{M}\min\{\mu,X_{3}\}.\label{eq:mod-struc-cl1}
\end{equation}
This claim immediately implies \eqref{eq:mod-struc1} and \eqref{eq:mod-struc2}
for $\check{F}$ since $\calZ_{k}\in L^{2}$. We turn to the proof
of \eqref{eq:mod-struc-cl1}. First, we have 
\begin{align*}
\|\chf_{(0,2R_{\circ}]}(\bfD_{w}^{\ast}-\bfD_{P}^{\ast})P_{1}\|_{L^{2}} & \aleq\|\chf_{(0,2R_{\circ}]}(|w|^{2}-|P|^{2})\|_{L^{2}}\|\chf_{(0,2R_{\circ}]}P_{1}\|_{L^{2}}\\
 & \quad\aleq\beta\|\chf_{(0,2R_{\circ}]}\eps\|_{L^{\infty}}\|w+P\|_{L^{2}}\aleq_{M}\mu\|\chf_{(0,2R_{\circ}]}\eps\|_{L^{\infty}}.
\end{align*}
Next, we have 
\begin{align*}
 & \|\chf_{(0,2R_{\circ}]}(\tint 0y\Re(\br ww_{1}-\br PP_{1})dy')iP\|_{L^{2}}\\
 & \quad\aleq\|\chf_{(0,2R_{\circ}]}\tfrac{1}{y}|\br ww_{1}-\br PP_{1}|\|_{L^{1}}\|\chf_{(0,2R_{\circ}]}P\|_{L^{2}}\\
 & \quad\aleq_{M}\|\chf_{(0,2R_{\circ}]}\tfrac{1}{y}\eps\|_{L^{2}}+\|\chf_{(0,2R_{\circ}]}\eps_{1}\|_{L^{2}}.
\end{align*}
Next, we have 
\begin{align*}
 & \|\chf_{(0,2R_{\circ}]}(\tint y{\infty}\Re(\br ww_{1})dy')i\eps\|_{L^{2}}\\
 & \quad\aleq\|\tfrac{1}{y}w\|_{L^{2}}\|w_{1}\|_{L^{2}}\|\chf_{(0,2R_{\circ}]}\eps\|_{L^{2}}\aleq_{M}\mu\|\chf_{(0,2R_{\circ}]}\eps\|_{L^{2}}.
\end{align*}
Finally, by \eqref{eq:Psi0-loc-bound}, we have 
\[
\|\chf_{(0,2R_{\circ}]}\Psi\|_{L^{2}}\aleq\beta^{3}\aleq_{M}\min\{\mu,X_{3}\}.
\]
As all the norms in the above error estimates are restricted to the
compact region $y\leq2R_{\circ}$, we can apply either \eqref{eq:H1-nonlin-coer}
(for the bound $\mu$) or Lemma \ref{lem:HigherHardyBounds} combined
with Remark \ref{rem:Weighted-Linfty} (for the bound $X_{3}$). This
completes the proof of the claim \eqref{eq:mod-struc-cl1} and hence
the proof of \eqref{eq:mod-struc1} and \eqref{eq:mod-struc2} for
$\check{F}$.

\textbf{Step 2.} Estimates for $\check{F}_{1}$.

We prove \eqref{eq:mod-struc1} for $\check{F}_{1}$, \eqref{eq:mod-struc3},
and \eqref{eq:mod-struc4}. Recall that
\[
\check{F}_{1}=i(A_{w}^{\ast}w_{2}-A_{P}^{\ast}P_{2})-(\tint 0y\Re(\br ww_{1}-\br PP_{1})dy')iP_{1}+(\tint y{\infty}\Re(\br ww_{1})dy')i\eps_{1}+i\Psi_{1}.
\]

\uline{The term \mbox{$i(A_{w}^{\ast}w_{2}-A_{P}^{\ast}P_{2})$}.}
This term has derivatives and contains the main linear part $iA_{Q}^{\ast}\eps_{2}$
of the $\eps_{1}$-evolution.

\emph{(i) Proof of \eqref{eq:mod-struc1}.}

For the low regularity estimate \eqref{eq:mod-struc1}, we rewrite
\begin{align*}
A_{w}^{\ast}w_{2}-A_{P}^{\ast}P_{2} & =H_{w}w_{1}-H_{P}P_{1}+A_{P}^{\ast}(A_{P}P_{1}-P_{2})\\
 & =(H_{w}-H_{P})P_{1}+H_{w}\eps_{1}+A_{P}^{\ast}(A_{P}P_{1}-P_{2})
\end{align*}
and integrate by parts to avoid derivatives falling on $\eps_{1}$:
\begin{align*}
(i & (A_{w}^{\ast}w_{2}-A_{P}^{\ast}P_{2}),\td{\calZ}_{k})_{r}\\
 & =-(\eps_{1},iH_{w}\td{\calZ}_{k})_{r}-((H_{w}-H_{P})P_{1},i\td{\calZ}_{k})_{r}-(A_{P}P_{1}-P_{2},iA_{P}\td{\calZ}_{k})_{r}.
\end{align*}
Since $\td{\calZ}_{k}\in C_{c,m+1}^{\infty}$ is compactly supported,
we can simply use $P\aleq1$, $\eps\aleq1$, $P_{1}\aleq\beta$, $A_{P}P_{1}-P_{2}\aleq\beta^{3}$
in the region $y\leq2R_{\circ}$ (c.f. \eqref{eq:compat2-H3}). As
a result, we have 
\begin{align*}
(\eps_{1},iH_{w}\td{\calZ}_{k})_{r} & \aleq_{M}\|\chf_{(0,2R_{\circ}]}\eps_{1}\|_{L^{2}}\aleq_{M}\mu,\\
((H_{w}-H_{P})P_{1},i\td{\calZ}_{k})_{r} & \aleq_{M}\beta\|\chf_{(0,2R_{\circ}]}\eps\|_{L^{\infty}}\aleq_{M}\mu,\\
(A_{P}P_{1}-P_{2},iA_{P}\td{\calZ}_{k})_{r} & \aleq\beta^{3}\aleq_{M}\mu,
\end{align*}
which yield \eqref{eq:mod-struc1} for $i(A_{w}^{\ast}w_{2}-A_{P}^{\ast}P_{2})$.

\emph{(ii) Proof of \eqref{eq:mod-struc3} and \eqref{eq:mod-struc4}.}

For the higher Sobolev estimates \eqref{eq:mod-struc3} and \eqref{eq:mod-struc4},
we begin by rewriting 
\[
A_{w}^{\ast}w_{2}-A_{P}^{\ast}P_{2}=A_{Q}^{\ast}\eps_{2}+\big\{(A_{w}^{\ast}-A_{Q}^{\ast})\eps_{2}+(A_{w}^{\ast}-A_{P}^{\ast})P_{2}\big\}.
\]
For the first term $A_{Q}^{\ast}\eps_{2}$, using $A_{Q}(yQ)=0$,
we have 
\begin{align*}
(iA_{Q}^{\ast}\eps_{2},\td{\calZ}_{k})_{r} & =-(\eps_{2},iA_{Q}\td{\calZ}_{k})_{r}\aleq\|\chf_{(0,2R_{\circ}]}\eps_{2}\|_{L^{2}}\aleq\|\eps_{2}\|_{\dot{V}_{m+2}^{3/2}},\\
(iA_{Q}^{\ast}\eps_{2},\chi_{\mu^{-1}}yQ)_{r} & =-(\eps_{2},(\rd_{y}\chi_{\mu^{-1}})iyQ)_{r}\aleq\|\chf_{[\mu^{-1},2\mu^{-1}]}y^{-3}\eps_{2}\|_{L^{1}}\aleq\mu^{2}\|\eps_{2}\|_{L^{2}},
\end{align*}
and similarly for the inner product $(iA_{Q}^{\ast}\eps_{2},\chi_{\mu^{-1}}iyQ)_{r}$.
This shows \eqref{eq:mod-struc3} and \eqref{eq:mod-struc4} for the
term $iA_{Q}^{\ast}\eps_{2}$.

To show \eqref{eq:mod-struc3} and \eqref{eq:mod-struc4} for the
remaining term, namely, $i\{(A_{w}^{\ast}-A_{Q}^{\ast})\eps_{2}+(A_{w}^{\ast}-A_{P}^{\ast})P_{2}\}$,
we claim the following $L^{1}$-bound:
\begin{equation}
\|\tfrac{1}{y^{2}}\{(A_{w}^{\ast}-A_{Q}^{\ast})\eps_{2}+(A_{w}^{\ast}-A_{P}^{\ast})P_{2}\}\|_{L^{1}}\aleq_{M}\mu X_{3}.\label{eq:mod-struc-cl2}
\end{equation}
This claim follows from the estimates 
\begin{align*}
\|\tfrac{1}{y^{2}} & (A_{w}^{\ast}-A_{Q}^{\ast})\eps_{2}\|_{L^{1}}\\
 & \aleq\|\tfrac{1}{y^{2}}\tint 0y||w|^{2}-Q^{2}|y'dy'\|_{L^{2}}\|\tfrac{1}{y}\eps_{2}\|_{L^{2}}\\
 & \aleq\||w|^{2}-Q^{2}\|_{L^{2}}\|\tfrac{1}{y}\eps_{2}\|_{L^{2}}\aleq_{M}\|w-Q\|_{L^{\infty}}\|\eps_{2}\|_{\dot{H}_{m+2}^{1}}\aleq_{M}\mu X_{3}
\end{align*}
and 
\begin{align*}
\|\tfrac{1}{y^{2}} & (A_{w}^{\ast}-A_{P}^{\ast})P_{2}\|_{L^{1}}\\
 & \aleq\|\tfrac{1}{y^{3}\langle y\rangle}\tint 0y||w|^{2}-|P|^{2}|y'dy'\|_{L^{1}}\|\langle y\rangle P_{2}\|_{L^{\infty}}\\
 & \aleq\|\tfrac{1}{y\langle y\rangle}(|w|^{2}-|P|^{2})\|_{L^{1}}\beta^{2}\aleq_{M}\mu^{2}\|\tfrac{1}{y\langle y\rangle}\eps\|_{L^{2}}\aleq_{M}\mu^{2}X_{2},
\end{align*}
where we used Fubini for the inequality $\|\tfrac{1}{y^{3}\langle y\rangle}\tint 0yfy'dy'\|_{L^{1}}\aleq\|\tfrac{1}{y\langle y\rangle}f\|_{L^{1}}$.
Now, the claim \eqref{eq:mod-struc-cl2} implies \eqref{eq:mod-struc3}
and \eqref{eq:mod-struc4} since $\|y^{2}\td{\calZ}_{k}\|_{L^{\infty}}+\|y^{2}\cdot\chi_{\mu^{-1}}yQ\|_{L^{\infty}}\aleq1$.
Thus we have proved \eqref{eq:mod-struc3} and \eqref{eq:mod-struc4}
for the term $i(A_{w}^{\ast}w_{2}-A_{P}^{\ast}P_{2})$.

\uline{Remaining terms of \mbox{$\check{F}_{1}$}.} Recall that
the remaining terms are given by 
\begin{equation}
\begin{aligned} & \check{F}_{1}-i(A_{w}^{\ast}w_{2}-A_{P}^{\ast}P_{2})\\
 & \quad=-(\tint 0y\Re(\br ww_{1}-\br PP_{1})dy')iP_{1}+(\tint y{\infty}\Re(\br ww_{1})dy')i\eps_{1}+i\Psi_{1}.
\end{aligned}
\label{eq:mod-struc-cl4}
\end{equation}

\emph{(i) Proof of \eqref{eq:mod-struc1}.}

Since $\td{\calZ}_{k}\in C_{c,m+1}^{\infty}$, the estimate \eqref{eq:mod-struc1}
easily follows from
\begin{align*}
 & \|\chf_{(0,2R_{\circ}]}\text{RHS\eqref{eq:mod-struc-cl4}}\|_{L^{2}}\\
 & \aleq\|\tfrac{1}{y}(|\br ww_{1}|+|\br PP_{1}|)\|_{L^{1}}\||P_{1}|+|\eps_{1}|\|_{L^{2}}+\|\chf_{(0,2R_{\circ}]}\Psi_{1}\|_{L^{2}}\\
 & \aleq_{M}\mu^{2}+\beta^{4}|\log\beta|^{3}\aleq\mu,
\end{align*}
where we used \eqref{eq:Psi1-bound} for the term $\Psi_{1}$.

\emph{(ii) Proof of \eqref{eq:mod-struc3} and \eqref{eq:mod-struc4}.}

Since $\|y^{2}\td{\calZ}_{k}\|_{L^{\infty}}+\|y^{2}\cdot\chi_{\mu^{-1}}yQ\|_{L^{\infty}}\aleq1$
and $\beta^{4}|\log\beta|^{3}\aleq\mu^{7/2}$, it suffices to show
(c.f. \eqref{eq:mod-struc-cl2})
\begin{equation}
\|\tfrac{1}{y^{2}}\text{RHS\eqref{eq:mod-struc-cl4}}\|_{L^{1}}\aleq_{M}\mu X_{3}.\label{eq:mod-struc-cl3}
\end{equation}
For the first term of RHS\eqref{eq:mod-struc-cl4}, we use $\|P_{1}\|_{L^{2}}+\|\langle y\rangle P_{1}\|_{L^{\infty}}\aleq\beta$
to have 
\begin{align*}
\|\tfrac{1}{y^{2}} & \tint 0y\Re(\br ww_{1}-\br PP_{1})dy'\cdot P_{1}\|_{L^{1}}\\
 & \aleq\|\tfrac{1}{y}(\br ww_{1}-\br PP_{1})\|_{L^{2}}\|P_{1}\|_{L^{2}}\\
 & \aleq(\|w\|_{L^{2}}\|\tfrac{1}{y}\eps_{1}\|_{L^{\infty}}+\|\tfrac{1}{y\langle y\rangle}\eps\|_{L^{2}}\|\langle y\rangle P_{1}\|_{L^{\infty}})\cdot\beta\\
 & \aleq_{M}\mu\|\eps_{1}\|_{\dot{H}_{m+1}^{2}}+\mu^{2}\|\eps\|_{\dot{\calH}_{m}^{2}}\aleq_{M}\mu X_{3}.
\end{align*}
For the second term, we have 
\begin{align*}
\|\tfrac{1}{y^{2}} & \tint y{\infty}\Re(\br ww_{1})dy'\cdot\eps_{1}\|_{L^{1}}\\
 & \aleq\|ww_{1}\|_{L^{1}}\|\tfrac{1}{y^{2}}\eps_{1}\|_{L^{2}}\aleq_{M}\|w_{1}\|_{L^{2}}\|\tfrac{1}{y^{2}}\eps_{1}\|_{L^{2}}\aleq_{M}\mu X_{3}.
\end{align*}
For the last term, by \eqref{eq:Psi1-bound}, we have 
\[
\|\tfrac{1}{y^{2}}\Psi_{1}\|_{L^{1}}\aleq\beta^{4}|\log\beta|^{3},
\]
completing the proof of \eqref{eq:mod-struc-cl3}, and hence the proof
of \eqref{eq:mod-struc3} and \eqref{eq:mod-struc4}.
\end{proof}

\subsection{\label{subsec:proof-struc-energy}Proof of the structural Lemma~\ref{lem:StrucEnergyEst}}
\begin{proof}[Proof of Lemma~\ref{lem:StrucEnergyEst}]
We estimate each term of $F_{2}$. Recall that
\[
\begin{aligned}F_{2} & =\td{\Mod}\cdot\bfv_{2}-i\Psi_{2}+(i\br ww_{1}^{2}-i\br PP_{1}^{2})\\
 & \peq-(i\td H_{w}-i\td H_{Q})\eps_{2}-(i\td H_{w}-i\td H_{P})P_{2}\\
 & \peq-(\tint y{\infty}\Re(\br ww_{1})dy')i\eps_{2}+(\tint 0y\Re(\br ww_{1}-\br PP_{1})dy')iP_{2}.
\end{aligned}
\tag{\ref{eq:def-F2}}
\]

\uline{Modulation terms \mbox{$\td{\Mod}_{j}(\bfv_{2})_{j}$}}.
For $j\in\{1,2\}$, we show that $\td{\Mod}_{j}(\bfv_{2})_{j}$ is
a perturbative term. Thus we need to prove \eqref{eq:lyapunov-struc1}
and \eqref{eq:lyapunov-struc3}. These estimates easily follow from
\eqref{eq:ModEst-3} and \eqref{eq:modvec2-1}: 
\begin{align*}
|\td{\Mod}_{j}|\|(\bfv_{2})_{j}\|_{L^{2}} & \aleq_{M}X_{3}\cdot\beta^{2}|\log\beta|^{\frac{1}{2}}\aleq_{M}\mu X_{3},\\
|\td{\Mod}_{j}|\|A_{Q}^{\ast}(\bfv_{2})_{j}\|_{L^{2}} & \aleq_{M}X_{3}\cdot\beta^{2}\aleq_{M}\mu^{2}X_{3}.
\end{align*}
For $j\in\{3,4\}$, we need to prove \eqref{eq:lyapunov-struc4}--\eqref{eq:lyapunov-struc6}.
These estimates follow from \eqref{eq:ModEst-4}, $V_{7/2}\leq X_{3}$,
and \eqref{eq:modvec2-2}: 
\begin{align*}
|\td{\Mod}_{j}|\|y|(\bfv_{2})_{j}|_{1}\|_{L^{2}} & \aleq_{M}V_{7/2}\aleq_{M}X_{3},\\
|\td{\Mod}_{j}|\|\langle y\rangle^{\frac{1}{2}}|(\bfv_{2})_{j}|_{1}\|_{L^{2}} & \aleq_{M}V_{7/2}\cdot\beta^{\frac{1}{2}}\aleq_{M}\mu^{1/2}V_{7/2},\\
|\td{\Mod}_{j}|\|\langle y\rangle^{\frac{1}{2}}y(\td H_{Q}\bfv_{2})_{j}\|_{L^{2}} & \aleq_{M}V_{7/2}\cdot\beta\aleq_{M}\mu V_{7/2}.
\end{align*}

\uline{Inhomogeneous term \mbox{$-i\Psi_{2}$}}. This term is perturbative;
the estimates \eqref{eq:lyapunov-struc1} and \eqref{eq:lyapunov-struc3}
are immediate from \eqref{eq:Psi2-H2-bound} and \eqref{eq:Psi2-H3-bound}.

\uline{Nonlinear term \mbox{$-i(\td H_{w}-\td H_{P})P_{2}$}}.
Let 
\[
h_{1}\coloneqq|w|^{2}-|P|^{2}\quad\text{and}\quad h_{2}\coloneqq|w|^{2}+|P|^{2}
\]
so that 
\[
\td H_{w}-\td H_{P}=\tfrac{1}{y^{2}}(2m+4-\tfrac{1}{2}\tint 0yh_{2}y'dy')\cdot(-\tfrac{1}{2}\tint 0yh_{1}y'dy')+\tfrac{1}{2}h_{1}.
\]
We show that $-i(\td H_{w}-\td H_{P})P_{2}$ is a perturbative term;
thus $F_{2}^{\pert}$ collects this term. Taking $|\cdot|_{-1}$ to
the above display and using $\|h_{2}\|_{L^{1}}\aleq_{M}1$ and $|h_{2}|\aleq|P|^{2}+|\eps|^{2}\aleq y^{-2}+|\eps|^{2}$,
we have 
\begin{align*}
i(\td H_{w}-\td H_{P})P_{2} & \aleq_{M}(\tfrac{1}{y^{2}}\tint 0y|h_{1}|y'dy'+|h_{1}|)\cdot|P_{2}|,\\
|i(\td H_{w}-\td H_{P})P_{2}|_{-1} & \aleq_{M}(\tfrac{1}{y^{2}}\tint 0y|h_{1}|y'dy'+|h_{1}|)\cdot|P_{2}|_{-1}+|(\rd_{y}h_{1})P_{2}|\\
 & \qquad\qquad\qquad\qquad\qquad\qquad+(\tfrac{1}{y}\tint 0y|h_{1}|y'dy')\cdot|\eps^{2}P_{2}|.
\end{align*}
We now estimate the $L^{2}$-norms. First, using $|P_{2}|_{-1}\aleq\beta^{2}\langle y\rangle^{-2}$
and $\|P\|_{L^{2}}\aleq1$, we have for $\ell\in\{0,1\}$ 
\begin{align*}
 & \|(\tfrac{1}{y^{2}}\tint 0y|h_{1}|y'dy'+|h_{1}|)\cdot|P_{2}|_{-\ell}\|_{L^{2}}\\
 & \quad\aleq\beta^{2}\|\langle y\rangle^{-(1+\ell)}h_{1}\|_{L^{2}}\aleq_{M}\mu^{2}\|\langle y\rangle^{-(1+\ell)}\eps\|_{L^{\infty}}\aleq_{M}\mu^{2}X_{2+\ell},
\end{align*}
where we used \eqref{eq:eps-Linfty} and \eqref{eq:e-H2}--\eqref{eq:e-H3}.
The case $\ell=0$ gives \eqref{eq:lyapunov-struc1}. Next, we similarly
estimate using $|P|_{1}\aleq\langle y\rangle^{-1}$ and $|P_{2}|\aleq\beta^{2}\langle y\rangle^{-1}$:
\begin{align*}
\|(\rd_{y}h_{1})P_{2}\|_{L^{2}} & \aleq\||\eps|_{-1}\cdot(|P|_{1}+|\eps|)P_{2}\|_{L^{2}}\\
 & \aleq\beta^{2}(\|\langle y\rangle^{-1}|\eps|_{-1}\|_{L^{2}}\|\eps\|_{L^{\infty}}+\|\langle y\rangle^{-2}|\eps|_{-1}\|_{L^{2}})\aleq_{M}\mu^{2}X_{3}.
\end{align*}
Finally, we have 
\begin{align*}
\|\tint 0y|h_{1}|y'dy'\cdot\tfrac{1}{y}\eps^{2}P_{2}\|_{L^{2}} & \aleq\|\tfrac{1}{y^{2}}\tint 0y|h_{1}|y'dy'\|_{L^{2}}\|\eps\|_{L^{\infty}}^{2}\beta^{2}\\
 & \aleq_{M}\beta^{2}\|h_{1}\|_{L^{2}}\mu^{2}\aleq_{M}\mu^{4}\|\eps\|_{L^{\infty}}\aleq_{M}\mu^{5}\aleq_{M}\mu^{2}X_{3}.
\end{align*}
This completes the proof of \eqref{eq:lyapunov-struc1}--\eqref{eq:lyapunov-struc3}
for the term $-i(\td H_{w}-\td H_{P})P_{2}$.

\uline{Nonlinear term \mbox{$-i(\td H_{w}-\td H_{Q})\eps_{2}$}}.
Let 
\[
h_{1}\coloneqq|w|^{2}-Q{}^{2}\quad\text{and}\quad h_{2}\coloneqq|w|^{2}+Q^{2}
\]
so that 
\[
\td H_{w}-\td H_{Q}=\tfrac{1}{y^{2}}(2m+4-\tfrac{1}{2}\tint 0yh_{2}y'dy')\cdot(-\tfrac{1}{2}\tint 0yh_{1}y'dy')+\tfrac{1}{2}h_{1}.
\]
We further decompose 
\[
h_{1}=2\Re(\br{(w-Q)}Q)+|w-Q|^{2}\eqqcolon h_{1}^{\loc}+h_{1}^{\pert}\qquad\text{and}\qquad h_{2}=2Q^{2}+h_{1}
\]
so that
\begin{align*}
\td H_{w}-\td H_{Q} & =\{\tfrac{1}{y^{2}}(2m+4-\tint 0yQ^{2}y'dy')\cdot(-\tfrac{1}{2}\tint 0yh_{1}^{\loc}y'dy')+\tfrac{1}{2}h_{1}^{\loc}\}\\
 & \peq+\{\tfrac{1}{y^{2}}(2m+4-\tint 0yQ^{2}y'dy')\cdot(-\tfrac{1}{2}\tint 0yh_{1}^{\pert}y'dy')+\tfrac{1}{2}h_{1}^{\pert}\}\\
 & \peq+\tfrac{1}{4y^{2}}(\tint 0yh_{1}y'dy')^{2}\\
 & \eqqcolon V_{\loc}+V_{\pert,1}+V_{\pert,2}.
\end{align*}
We claim that the contribution of $V_{\loc}$ is local and those of
$V_{\pert,1},V_{\pert,2}$ are perturbative. In other words, $F_{2,1}^{\loc}$
collects $-iV_{\loc}\eps_{2}$ and $F_{2}^{\pert}$ collects $-i(V_{\pert,1}+V_{\pert,2})\eps_{2}$.

For $-iV_{\loc}\eps_{2}$, we use $\|w-Q\|_{\dot{H}_{m}^{1}}\aleq_{M}\mu$
to have 
\begin{align*}
\|iV_{\loc}\eps_{2}\|_{L^{2}} & \aleq\|(\tfrac{1}{y^{2}}\tint 0y|h_{1}^{\loc}|y'dy'+|h_{1}^{\loc}|)\cdot\eps_{2}\|_{L^{2}}\\
 & \aleq\|h_{1}^{\loc}\|_{L^{2}}\|\eps_{2}\|_{L^{\infty}}\aleq_{M}\|w-Q\|_{L^{\infty}}\|\eps_{2}\|_{\dot{H}_{m+2}^{1}}\aleq_{M}\mu X_{3}
\end{align*}
and (further using \eqref{eq:eps2-Linfty} and $Q\aleq\langle y\rangle^{-3}$)
\begin{align*}
 & \|\langle y\rangle|iV_{\loc}\eps_{2}|_{-1}\|_{L^{2}}\\
 & \quad\aleq\|\langle y\rangle(\rd_{y}h_{1}^{\loc})\eps_{2}\|_{L^{2}}+\|\langle y\rangle(\tfrac{1}{y^{2}}\tint 0y|h_{1}^{\loc}|y'dy'+|h_{1}^{\loc}|)\cdot|\eps_{2}|_{-1}\|_{L^{2}}\\
 & \quad\aleq\||w-Q|_{-1}\|_{L^{2}}\|\langle y\rangle^{-2}\eps_{2}\|_{L^{\infty}}+\|w-Q\|_{L^{\infty}}\|\langle y\rangle^{-2}|\eps_{2}|_{-1}\|_{L^{2}}\\
 & \quad\aleq_{M}\mu\|\eps_{2}\|_{\dot{V}_{m+2}^{3/2}}\aleq_{M}\mu V_{7/2}.
\end{align*}
For $-iV_{\pert,1}\eps_{2}$, we have 
\begin{align*}
 & \|iV_{\pert,1}\eps_{2}\|_{L^{2}}\\
 & \quad\aleq\|(\tfrac{1}{y^{2}}\tint 0y|h_{1}^{\pert}|y'dy'+|h_{1}^{\pert}|)\cdot\eps_{2}\|_{L^{2}}\aleq\|w-Q\|_{L^{\infty}}^{2}\|\eps_{2}\|_{L^{2}}\aleq_{M}\mu X_{3}
\end{align*}
and similarly 
\begin{align*}
 & \||iV_{\pert,1}\eps_{2}|_{-1}\|_{L^{2}}\\
 & \quad\aleq\|(\rd_{y}h_{1}^{\pert})\eps_{2}\|_{L^{2}}+\|(\tfrac{1}{y^{2}}\tint 0y|h_{1}^{\pert}|y'dy'+|h_{1}^{\pert}|)\cdot|\eps_{2}|_{-1}\|_{L^{2}}\\
 & \quad\aleq\|\rd_{y}h_{1}^{\pert}\|_{L^{2}}\|\eps_{2}\|_{L^{\infty}}+\|h_{1}^{\pert}\|_{L^{\infty}}\||\eps_{2}|_{-1}\|_{L^{2}}\\
 & \quad\aleq\|w-Q\|_{\dot{H}_{m}^{1}}^{2}\|\eps_{2}\|_{\dot{H}_{m+2}^{1}}\aleq_{M}\mu^{2}X_{3}.
\end{align*}
For $-iV_{\pert,2}\eps_{2}$, we have 
\begin{align*}
 & \|iV_{\pert,2}\eps_{2}\|_{L^{2}}\\
 & \quad\aleq\|(\tfrac{1}{y}\tint 0y|h_{1}|y'dy')^{2}\cdot\eps_{2}\|_{L^{2}}\aleq\|h_{1}\|_{L^{2}}^{2}\|\eps_{2}\|_{L^{2}}\aleq_{M}\|w-Q\|_{L^{\infty}}^{2}\|\eps_{2}\|_{L^{2}}\aleq_{M}\mu X_{3}
\end{align*}
and similarly 
\begin{align*}
 & \||iV_{\pert,2}\eps_{2}|_{-1}\|_{L^{2}}\\
 & \quad\aleq\|(\tfrac{1}{y}\tint 0y|h_{1}|y'dy')^{2}\cdot|\eps_{2}|_{-1}\|_{L^{2}}+\|\tfrac{1}{y}\tint 0y|h_{1}|y'dy'\cdot h_{1}\eps_{2}\|_{L^{2}}\\
 & \quad\aleq\|h_{1}\|_{L^{2}}^{2}(\||\eps_{2}|_{-1}\|_{L^{2}}+\|\eps_{2}\|_{L^{\infty}})\\
 & \quad\aleq_{M}\|w-Q\|_{L^{\infty}}^{2}\|\eps_{2}\|_{\dot{H}_{m+1}^{1}}\aleq_{M}\mu^{2}X_{3}.
\end{align*}
This completes the proof of \eqref{eq:lyapunov-struc1}--\eqref{eq:lyapunov-struc3}
for the term $-i(\td H_{w}-\td H_{Q})\eps_{2}$.

\uline{Nonlinear term \mbox{$(\tint 0y\Re(\br ww_{1}-\br PP_{1})dy')iP_{2}$}}.
We show that this term is perturbative. Using $P\aleq\langle y\rangle^{-1}$,
$P_{1}\aleq\beta\langle y\rangle^{-1}$, $|P_{2}|_{1}\aleq\chf_{(0,2B_{1}]}\beta^{2}\langle y\rangle^{-1}$,
and Lemma~\ref{lem:HigherHardyBounds}, we have 
\begin{align*}
 & \|(\tint 0y\Re(\br ww_{1}-\br PP_{1})dy')iP_{2}\|_{L^{2}}\\
 & \quad\aleq\|\tfrac{1}{y}(\br ww_{1}-\br PP_{1})\|_{L^{2}}\|y^{2}P_{2}\|_{L^{\infty}}\\
 & \quad\aleq(\|\tfrac{1}{y\langle y\rangle}\eps_{1}\|_{L^{2}}+\|\eps\|_{L^{\infty}}\|\tfrac{1}{y}\eps_{1}\|_{L^{2}}+\beta\|\tfrac{1}{y\langle y\rangle}\eps\|_{L^{2}})\cdot\beta\aleq_{M}\mu X_{3}.
\end{align*}
This gives \eqref{eq:lyapunov-struc1}. Similarly, we have 
\[
\||(\tint 0y\Re(\br ww_{1}-\br PP_{1})dy')iP_{2}|_{-1}\|_{L^{2}}\aleq\|\tfrac{1}{y}(\br ww_{1}-\br PP_{1})\|_{L^{2}}\|y|P_{2}|_{1}\|_{L^{\infty}}\aleq_{M}\mu^{2}X_{3},
\]
which is \eqref{eq:lyapunov-struc3}. This completes the proof of
\eqref{eq:lyapunov-struc1}--\eqref{eq:lyapunov-struc3} for the
term $(\tint 0y\Re(\br ww_{1}-\br PP_{1})dy')iP_{2}$.

\uline{Nonlinear term \mbox{$-(\tint y{\infty}\Re(\br ww_{1})dy')i\eps_{2}$}}.
The estimate \eqref{eq:lyapunov-struc1} easily follows from 
\[
\|(\tint y{\infty}\Re(\br ww_{1})dy')i\eps_{2}\|_{L^{2}}\aleq\|\br ww_{1}\|_{L^{1}}\|\eps_{2}\|_{L^{\infty}}\aleq_{M}\|w_{1}\|_{L^{2}}\|\eps_{2}\|_{\dot{H}_{m+2}^{1}}\aleq_{M}\mu X_{3}.
\]
For \eqref{eq:lyapunov-struc3}, we decompose 
\[
\tint y{\infty}\Re(\br ww_{1})dy'=\tint y{\infty}\Re(Qw_{1})dy'+\tint y{\infty}\Re(\br{(w-Q)}w_{1})dy'\eqqcolon V_{\loc}+V_{\pert}.
\]
The term $-iV_{\loc}\eps_{2}$ is local; we use $Q\aleq\langle y\rangle^{-3}$
to have 
\begin{align*}
\|\langle y\rangle|iV_{\loc}\eps_{2}|_{-1}\|_{L^{2}} & \aleq\|\langle y\rangle\tint y{\infty}|Qw_{1}|dy'\cdot|\eps_{2}|_{-1}\|_{L^{2}}+\|\langle y\rangle Qw_{1}\eps_{2}\|_{L^{2}}\\
 & \aleq\big\|\|w_{1}\|_{L^{2}}\cdot\langle y\rangle^{-2}|\eps_{2}|_{-1}\big\|_{L^{2}}+\|w_{1}\|_{L^{2}}\|\langle y\rangle^{-2}\eps_{2}\|_{L^{\infty}}\\
 & \aleq_{M}\mu\|\langle y\rangle^{-2}|\eps_{2}|_{-1}\|_{L^{2}}\aleq_{M}\mu V_{7/2}.
\end{align*}
The term $-iV_{\pert}\eps_{2}$ is perturbative; we use $\|w-Q\|_{\dot{H}_{m}^{1}}+\|w_{1}\|_{L^{2}}\aleq_{M}\mu$
to have 
\begin{align*}
\||iV_{\pert}\eps_{2}|_{-1}\|_{L^{2}} & \aleq\|\tint y{\infty}|(w-Q)w_{1}|dy'\cdot|\eps_{2}|_{-1}\|_{L^{2}}+\|(w-Q)w_{1}\eps_{2}\|_{L^{2}}\\
 & \aleq\|\tfrac{1}{y}(w-Q)w_{1}\|_{L^{1}}\||\eps_{2}|_{-1}\|_{L^{2}}+\|(w-Q)w_{1}\|_{L^{2}}\|\eps_{2}\|_{L^{\infty}}\\
 & \aleq\|w-Q\|_{\dot{H}_{m}^{1}}\|w_{1}\|_{L^{2}}\|\eps_{2}\|_{\dot{H}_{m+2}^{1}}\aleq_{M}\mu^{2}X_{3}.
\end{align*}
This completes the proof of \eqref{eq:lyapunov-struc1}--\eqref{eq:lyapunov-struc3}
for the term $-(\tint y{\infty}\Re(\br ww_{1})dy')i\eps_{2}$.

\uline{Nonlinear term \mbox{$i(\br ww_{1}^{2}-\br PP_{1}^{2})$}}.
We begin with the decomposition 
\begin{align*}
\br ww_{1}^{2}-\br PP_{1}^{2} & =Q\eps_{1}(2P_{1}+\eps_{1})+\br{(w-Q)}\eps_{1}(2P_{1}+\eps_{1})+\br{\eps}P_{1}^{2}\\
 & \eqqcolon R_{\loc}+R_{\pert,1}+R_{\pert,2}.
\end{align*}
First, $R_{\loc}$ is a local term; we use \eqref{eq:eps1-Linfty}
and \eqref{eq:H1-nonlin-coer} to have 
\[
\|R_{\loc}\|_{L^{2}}\aleq\|\langle y\rangle^{-3}\eps_{1}\|_{L^{\infty}}\|P_{1}+2\eps_{1}\|_{L^{2}}\aleq_{M}\mu\|\eps_{1}\|_{\dot{H}_{m+1}^{2}}\aleq_{M}\mu X_{3}
\]
and we use \eqref{eq:eps1-Linfty} and \eqref{eq:e1-V} to have 
\begin{align*}
\|\langle y\rangle|R_{\loc}|_{-1}\|_{L^{2}} & \aleq\|\langle y\rangle^{-2}|\eps_{1}|_{-1}\cdot(|P_{1}|_{1}+|\eps_{1}|)\|_{L^{2}}\\
 & \aleq_{M}\mu\|\langle y\rangle^{-2}|\eps_{1}|_{-1}\|_{L^{\infty}}\aleq_{M}\mu\|\eps_{1}\|_{\dot{V}_{m+1}^{5/2}}\aleq_{M}\mu V_{7/2}.
\end{align*}
Next, $R_{\pert,1}$ is a perturbative term; we use $\|w-Q\|_{\dot{H}_{m}^{1}}+\||P_{1}|_{1}\|_{L^{2}\cap y^{-1}L^{\infty}}\aleq_{M}\mu$
to have 
\begin{align*}
 & \|(w-Q)\eps_{1}(P_{1}+2\eps_{1})\|_{L^{2}}\\
 & \quad\aleq\|w-Q\|_{L^{\infty}}\|\eps_{1}\|_{L^{\infty}}\||P_{1}|+|\eps_{1}|\|_{L^{2}}\aleq_{M}\mu^{2}\|\eps_{1}\|_{\dot{H}_{m}^{1}}\aleq_{M}\mu X_{3}
\end{align*}
and (also using $\|\eps_{1}\|_{L^{\infty}}^{2}\aleq\|\eps_{1}\|_{L^{2}}\|\eps_{1}\|_{\dot{H}_{m+1}^{2}}$)
\begin{align*}
\||R_{\pert,1}|_{-1}\|_{L^{2}} & \aleq\||w-Q|_{-1}\|_{L^{2}}(\|\tfrac{1}{y}\eps_{1}\|_{L^{\infty}}\|yP_{1}\|_{L^{\infty}}+\|\eps_{1}\|_{L^{\infty}}^{2})\\
 & \peq+\|w-Q\|_{L^{\infty}}\||\eps_{1}|_{-1}\|_{L^{\infty}}(\||P_{1}|_{1}\|_{L^{2}}+\|\eps_{1}\|_{L^{2}})\\
 & \aleq_{M}\mu^{2}\|\eps_{1}\|_{\dot{H}_{m+1}^{2}}\aleq_{M}\mu^{2}X_{3}.
\end{align*}
Finally, $R_{\pert,2}$ is a perturbative term; we use the estimate
$|P_{1}^{2}|_{1}\aleq\beta^{2}\langle y\rangle^{-2}$ to have for
$\ell\in\{0,1\}$:
\begin{align*}
\||\br{\eps}P_{1}^{2}|_{-\ell}\|_{L^{2}} & \aleq\|\langle y\rangle^{2}|P_{1}^{2}|_{1}\|_{L^{\infty}}\|\langle y\rangle^{-2}|\eps|_{-\ell}\|_{L^{2}}\\
 & \aleq\begin{cases}
\beta^{2}\|\eps\|_{\dot{\calH}_{m}^{2}}\aleq_{M}\mu X_{3} & \text{if }\ell=0,\\
\beta^{2}\|\eps\|_{\dot{\calH}_{m}^{3}}\aleq_{M}\mu^{2}X_{3} & \text{if }\ell=1.
\end{cases}
\end{align*}
This completes the proof of \eqref{eq:lyapunov-struc1}--\eqref{eq:lyapunov-struc3}
for the term $i(\br ww_{1}^{2}-\br PP_{1}^{2})$. The proof of Lemma~\ref{lem:StrucEnergyEst}
is completed.
\end{proof}

\section{\label{sec:Proof-of-the-main-thm}Proof of the main theorem}

In this section, we prove Theorem~\ref{thm:blow-up-rigidity}. Let
$u(t)$ be a $H_{m}^{3}$-solution to \eqref{eq:CSS-m-equiv} that
blows up at $T<+\infty$. Denote by $M=M[u]$ the $L^{2}$-mass of
$u$. Fix a small parameter $\alpha^{\ast}>0$ such that the results
in Section~\ref{sec:Modulation-analysis} are valid. By Remark~\ref{rem:blow-up-vicinity},
the solution $u(t)$ eventually belongs to the soliton tube $\calT_{\alpha^{\ast},M}$
(recall \eqref{eq:def-H1-soliton-tube}). Note that $u(t)$ admits
the decomposition as in Lemma~\ref{lem:decomp}.
\begin{rem}[On implicit constants]
In this section, the solution $u(t)$ is fixed and we allow the implicit
constants to depend on $u$. In other words, we omit any dependences
on $M[u]$, $E[u]$, $\|u(t_{0})\|_{H_{m}^{3}}$, and so on, in the
following inequalities. For example, the nonlinear coercivity estimate
in Lemma~\ref{lem:NonlinCoerEnergy} and the uniform bound in Corollary~\ref{cor:integ-monot}
can be stated as 
\begin{equation}
\beta+\|\eps\|_{\dot{H}_{m}^{1}}+\|\eps_{1}\|_{L^{2}}\aleq\mu\qquad\text{and}\qquad X_{3}\aleq\mu^{3},\label{eq:unif-bound-no-imp-const}
\end{equation}
respectively.
\end{rem}

\subsection{\label{subsec:Asymptotics-of-modulation}Asymptotics of modulation
parameters}

The goal of this subsection is to derive sharp asymptotics of the
modulation parameters. In particular, we show that the blow-up rate
is pseudoconformal. We integrate the modulation equations. However,
our $b_{s}$- and $\eta_{s}$-modulation estimates \eqref{eq:ModEst-4}
are not sufficient\footnote{Assuming $T=0$, consider the regime $\lmb(t)\sim|t|^{\nu}$ and $b(t)\sim|t|^{2\nu-1}$
with $\nu\gg1$.}, and we need to obtain more refined estimates. We achieve this by
testing the $\eps_{1}$-equation \eqref{eq:eps1-eqn} against the
almost nullspace elements $\chi_{\mu^{-1}}iyQ$ and $\chi_{\mu^{-1}}yQ$.
\begin{lem}[Second modulation estimates]
\label{lem:SecondModEst}Let 
\begin{equation}
\wh b\coloneqq b-\frac{(\eps_{1},\chi_{1/\mu}2iyQ)_{r}}{\|yQ\|_{L^{2}}^{2}}\quad\text{and}\quad\wh{\eta}\coloneqq\eta-\frac{(\eps_{1},\chi_{1/\mu}2yQ)_{r}}{\|yQ\|_{L^{2}}^{2}}.\label{eq:def-b-hat-eta-hat}
\end{equation}
Then, we have 
\begin{align}
|\wh b-b|+|\wh{\eta}-\eta| & \aleq\mu^{2},\label{eq:b-hat-diff}\\
\Big|\frac{\lmb_{s}}{\lmb}+\wh b\Big|+|\gmm_{s}-(m+1)\wh{\eta}| & \aleq\mu^{2},\label{eq:b-hat-mod1}\\
|\wh b_{s}+\wh b^{2}+\wh{\eta}^{2}|+|\wh{\eta}_{s}| & \aleq\beta^{3}+\beta\mu^{2}+\mu^{4}.\label{eq:b-hat-mod2}
\end{align}
\end{lem}

\begin{proof}
In the proof, we denote $R=\mu^{-1}$, $\psi_{3}=-2iyQ$, $\psi_{4}=-2yQ$.

\textbf{Step 1.} Proof of \eqref{eq:b-hat-diff} and \eqref{eq:b-hat-mod1}.

\eqref{eq:b-hat-diff} follows from \eqref{eq:unif-bound-no-imp-const}
and \eqref{eq:e1-H2}: 
\[
(\eps_{1},\chi_{R}\psi_{k})_{r}\aleq\|y^{-2}\eps_{1}\|_{L^{2}}\|\chi_{R}y^{2}\psi_{k}\|_{L^{2}}\aleq X_{3}\cdot R\aleq\mu^{2}.
\]
Now \eqref{eq:b-hat-mod1} follows from \eqref{eq:ModEst-3}, \eqref{eq:b-hat-diff},
and \eqref{eq:mod-est-phase-cor}: 
\begin{align*}
 & \Big|\tfrac{\lmb_{s}}{\lmb}+\wh b|+|\gmm_{s}-(m+1)\wh{\eta}|\\
 & \quad\aleq|\Mod_{1}|+|b-\wh b|+|\td{\Mod}_{2}|+|\tint 0{\infty}\Re(\br ww_{1})dy+2(m+1)\eta|+|\eta-\wh{\eta}|\aleq\mu^{2}.
\end{align*}

\textbf{Step 2.} Proof of \eqref{eq:b-hat-mod2}.

First, we claim 
\begin{equation}
|\rd_{s}(\eps_{1},\chi_{R}\psi_{k})_{r}-(\td{\Mod}\cdot\bfv_{1},\chi_{R}\psi_{k})_{r}|\aleq(\beta+\mu^{2})\mu^{2},\qquad\forall k\in\{3,4\}.\label{eq:ref-mod-1}
\end{equation}
To see this, we rewrite \eqref{eq:def-check-F1} as 
\[
\rd_{s}\eps_{1}=\td{\Mod}\cdot\bfv_{1}+(\tfrac{\lmb_{s}}{\lmb}\Lmb_{-1}-\gmm_{s}i)\eps_{1}-\check{F}_{1}
\]
and test against $\chi_{R}\psi_{k}$ to obtain 
\begin{align*}
\rd_{s}(\eps_{1},\chi_{R}\psi_{k})_{r} & =(\rd_{s}\eps_{1},\chi_{R}\psi_{k})_{r}+(\eps_{1},\rd_{s}\chi_{R}\psi_{k})_{r}\\
 & =(\td{\Mod}\cdot\bfv_{1},\chi_{R}\psi_{k})_{r}-(\check{F}_{1},\chi_{R}\psi_{k})_{r}\\
 & \peq+\big(\eps_{1},\chi_{R}(-\tfrac{\lmb_{s}}{\lmb}y\rd_{y}\psi_{k}+\gmm_{s}i\psi_{k})\big)_{r}+\big(\eps_{1},\{(\rd_{s}-\tfrac{\lmb_{s}}{\lmb}y\rd_{y})\chi_{R}\}\psi_{k}\big)_{r}.
\end{align*}
We keep the first term and estimate the remainders. By \eqref{eq:mod-struc4},
we have 
\[
(\check{F}_{1},\chi_{R}\psi_{k})_{r}\aleq\mu^{4}+\beta^{4}|\log\beta|^{3}\aleq\beta\mu^{2}+\mu^{4}.
\]
Next, we use \eqref{eq:b-hat-mod1} to have 
\begin{align*}
\big(\eps_{1},\chi_{R}(-\tfrac{\lmb_{s}}{\lmb}y\rd_{y}\psi_{k}+\gmm_{s}i\psi_{k})\big)_{r} & \aleq(|\tfrac{\lmb_{s}}{\lmb}|+|\gmm_{s}|)\|y^{-2}\eps_{1}\|_{L^{2}}\|y^{2}\chf_{(0,2R]}|\psi_{k}|_{1}\|_{L^{2}}\\
 & \aleq(\beta+\mu^{2})\cdot\mu^{3}\cdot\mu^{-1}\aleq(\beta+\mu^{2})\cdot\mu^{2}.
\end{align*}
Finally, since $\frac{\lmb_{s}}{\lmb}=\frac{\mu_{s}}{\mu}$, we have
\[
\big(\eps_{1},\{(\rd_{s}-\tfrac{\lmb_{s}}{\lmb}y\rd_{y})\chi_{R}\}\psi_{k}\big)_{r}=0.
\]
The proof of the claim \eqref{eq:ref-mod-1} is completed.

Next, we claim that 
\begin{align}
\td{\Mod}_{j}\cdot((\bfv_{1})_{j},\chi_{R}\psi_{k})_{r} & \aleq\mu^{4}\text{ if }j\in\{1,2\},\label{eq:ref-mod-2}\\
\td{\Mod}_{j}\cdot((\bfv_{1})_{j},\chi_{R}\psi_{k})_{r} & \aleq\beta\mu^{2}\text{ if }j\in\{3,4\}\text{ and }j\neq k,\label{eq:ref-mod-3}\\
\td{\Mod}_{3}\cdot((\bfv_{1})_{3},\chi_{R}\psi_{3})_{r} & =-(b_{s}+b^{2}+\eta^{2})\|yQ\|_{L^{2}}^{2}+O(\beta^{3}+\beta\mu^{2}+\mu^{4}),\label{eq:ref-mod-4}\\
\td{\Mod}_{4}\cdot((\bfv_{1})_{4},\chi_{R}\psi_{4})_{r} & =-\eta_{s}\|yQ\|_{L^{2}}^{2}+O(\beta^{3}+\beta\mu^{2}+\mu^{4}).\label{eq:ref-mod-5}
\end{align}
Indeed, \eqref{eq:ref-mod-2} follows from \eqref{eq:ModEst-3} and
$|P_{1}|_{1}\aleq\beta\langle y\rangle^{-2}+\beta^{2}$: 
\[
\td{\Mod}_{j}\cdot((\bfv_{1})_{j},\chi_{R}\psi_{k})_{r}\aleq X_{3}\cdot\|\chf_{(0,2R]}|P_{1}|_{1}\cdot yQ\|_{L^{1}}\aleq\mu^{3}(\beta+\beta^{2}|\log\mu|)\aleq\mu^{4}.
\]
We turn to the proof of \eqref{eq:ref-mod-3}--\eqref{eq:ref-mod-5}.
Let $j\in\{3,4\}$. Using \eqref{eq:modvec1-2}, we have 
\begin{align*}
((\bfv_{1})_{j},\chi_{R}\psi_{k})_{r} & =\begin{cases}
0 & \text{if }j\neq k\\
-(\chi_{B_{1}}yQ,\chi_{R}yQ)_{r} & \text{if }j=k
\end{cases}\ +O(\|\chf_{(0,2R]}\beta\cdot yQ\|_{L^{1}})\\
 & =\begin{cases}
0 & \text{if }j\neq k\\
-\|yQ\|_{L^{2}}^{2}+O(\mu^{2}) & \text{if }j=k
\end{cases}\ +O(\beta|\log\mu|).
\end{align*}
Multiplying the above by $\td{\Mod}_{j}$ and applying \eqref{eq:ModEst-4}
with $V_{7/2}\leq X_{3}\aleq\mu^{3}$, we have 
\begin{align*}
 & \td{\Mod}_{j}\cdot((\bfv_{1})_{j},\chi_{R}\psi_{k})_{r}\\
 & \quad=\begin{cases}
0 & \text{if }j\neq k\\
-\td{\Mod}_{j}\|yQ\|_{L^{2}}^{2}+O(\mu^{5}) & \text{if }j=k
\end{cases}\ +O(\beta\mu^{3}|\log\mu|)\\
 & \quad=\begin{cases}
O(\beta\mu^{2}) & \text{if }j\neq k,\\
-\td{\Mod}_{j}\|yQ\|_{L^{2}}^{2}+O(\beta\mu^{2}+\mu^{4}) & \text{if }j=k.
\end{cases}
\end{align*}
This completes the proof of \eqref{eq:ref-mod-3}. When $j=k$, further
applying $|p_{3}^{(b)}|+|p_{3}^{(\eta)}|\aleq\beta^{3}$ to the above
completes the proof of \eqref{eq:ref-mod-4}--\eqref{eq:ref-mod-5}.

We are now ready to finish the proof. Substituting the claims \eqref{eq:ref-mod-2}--\eqref{eq:ref-mod-5}
into \eqref{eq:ref-mod-1} gives 
\[
\Big|b_{s}+b^{2}+\eta^{2}+\frac{\rd_{s}(\eps_{1},\chi_{R}\psi_{3})_{r}}{\|yQ\|_{L^{2}}^{2}}\Big|+\Big|\eta_{s}+\frac{\rd_{s}(\eps_{1},\chi_{R}\psi_{4})_{r}}{\|yQ\|_{L^{2}}^{2}}\Big|\aleq\beta^{3}+\beta\mu^{2}+\mu^{4}.
\]
By \eqref{eq:def-b-hat-eta-hat} and \eqref{eq:b-hat-diff}, this
completes the proof of \eqref{eq:b-hat-mod2}.
\end{proof}
We now prove the main result of this subsection.
\begin{prop}[Asymptotics of modulation parameters]
\label{prop:AsympModPar}There exist $\ell=\ell(u)\in(0,\infty)$
and $\gmm^{\ast}\in\bbR/2\pi\bbZ$ such that 
\begin{align}
\lmb(t) & =(\ell+o_{t\to T}(1))\cdot(T-t),\label{eq:lmb-asymp}\\
\gmm(t) & \to\gmm^{\ast},\label{eq:gmm-asymp}\\
b(t) & =(\ell^{2}+o_{t\to T}(1))\cdot(T-t),\label{eq:b-asymp}\\
\eta(t) & \aleq(T-t)^{2}.\label{eq:eta-asymp}
\end{align}
In particular, the blow-up rate is pseudoconformal.
\end{prop}

\begin{proof}
\textbf{Step 1.} Preliminary controls.

Let $\wh{\beta}\coloneqq(\wh b^{2}+\wh{\eta}^{2})^{1/2}$. We claim
the following bound for all $t$ close to $T$ (c.f. \eqref{eq:thm1-lmb-upper-bound}):
\begin{equation}
\beta+\wh{\beta}+\lmb\aleq T-t.\label{eq:bl-rate-cl0}
\end{equation}
To see this, we write \eqref{eq:ModEst-1} as $\lmb_{t}\aleq\mu/\lmb\aleq1$
and integrate it backwards from the blow-up time to get $\lmb\aleq T-t$.
The bounds for $\beta$ and $\wh{\beta}$ then follow from \eqref{eq:unif-bound-no-imp-const}
and \eqref{eq:b-hat-diff}.

\textbf{Step 2.} Nondegeneracy of $\wh{\beta}/\mu$.

We claim that there exists $\ell>0$ such that 
\begin{equation}
\frac{\wh{\beta}}{\lmb}\to\ell>0\qquad\text{as }t\to T.\label{eq:bl-rate-cl1}
\end{equation}
To see this, we begin with the identity 
\[
\Big(\frac{\wh{\beta}^{2}}{\lmb^{2}}\Big)_{s}=-2\Big(\frac{\wh{\beta}^{2}}{\lmb^{2}}\Big)\Big(\frac{\lmb_{s}}{\lmb}+\wh b\Big)+\frac{2}{\lmb^{2}}\Big\{\wh b(\wh b_{s}+\wh b^{2}+\wh{\eta}^{2})+\wh{\eta}\wh{\eta}_{s}\Big\}.
\]
Dividing both sides by $\wh{\beta}/\lmb$ and changing to the original
time variable $t$, we have 
\[
\Big(\frac{\wh{\beta}}{\lmb}\Big)_{t}\aleq\frac{\wh{\beta}}{\lmb}\cdot\frac{1}{\lmb^{2}}\Big|\frac{\lmb_{s}}{\lmb}+\wh b\Big|+\frac{1}{\lmb^{3}}\Big(|\wh b_{s}+\wh b^{2}+\wh{\eta}^{2}|+|\wh{\eta}_{s}|\Big).
\]
Substituting the estimates \eqref{eq:b-hat-mod1} and \eqref{eq:b-hat-mod2}
into the above gives 
\[
\Big(\frac{\wh{\beta}}{\lmb}\Big)_{t}\aleq\frac{\wh{\beta}}{\lmb}\Big(\frac{\mu^{2}}{\lmb^{2}}+\frac{\wh{\beta}^{2}}{\lmb^{2}}\Big)+\frac{\mu^{4}}{\lmb^{3}}\aleq\frac{\wh{\beta}}{\lmb}+\lmb.
\]
Combining this with the modulation estimate \eqref{eq:b-hat-mod1},
we have 
\[
\Big(\frac{\wh{\beta}}{\lmb}+\lmb\Big)_{t}\aleq\Big(\frac{\wh{\beta}}{\lmb}+\lmb\Big)+\frac{1}{\lmb}(|\wh b|+\mu^{2})\aleq\Big(\frac{\wh{\beta}}{\lmb}+\lmb\Big).
\]
This inequality says that $(\wh{\beta}/\lmb)+\lmb$ cannot go to zero
in finite time and there exists $\ell>0$ such that 
\[
\frac{\wh{\beta}}{\lmb}+\lmb\to\ell.
\]
Since $\lmb\to0$ by the blow-up hypothesis, the claim \eqref{eq:bl-rate-cl1}
follows.

\textbf{Step 3.} Proof of \eqref{eq:b-asymp} and almost decreasing
property of $\lmb$.

In this step, we claim that 
\begin{align}
b(t) & =(\ell^{2}+o_{t\to T}(1))\cdot(T-t),\tag{\ref{eq:b-asymp}}\nonumber \\
\lmb(t_{2}) & \aleq\lmb(t_{1})\qquad\text{for all }t_{1}\leq t_{2}<T.\label{eq:bl-rate-cl3}
\end{align}
For the proof of \eqref{eq:b-asymp}, we use \eqref{eq:b-hat-mod2}
to have 
\[
\Big|\wh b_{t}+\frac{\wh{\beta}^{2}}{\lmb^{2}}\Big|\aleq\frac{\wh{\beta}^{3}+\wh{\beta}\mu^{2}+\mu^{4}}{\lmb^{2}}\aleq\frac{\mu^{3}}{\lmb^{2}}\aleq\lmb.
\]
By \eqref{eq:bl-rate-cl1}, the above display says 
\[
\wh b_{t}+\ell^{2}=o_{t\to T}(1).
\]
Integrating backwards in time using \eqref{eq:bl-rate-cl0} and \eqref{eq:b-hat-diff}
gives 
\[
b(t)=\{b(t)-\wh b(t)\}-\int_{t}^{T}\wh b_{t}dt'=(\ell^{2}+o_{t\to T}(1))\cdot(T-t).
\]
For the proof of \eqref{eq:bl-rate-cl3}, we use $\wh b>0$ and \eqref{eq:b-hat-mod1}
to have 
\[
\lmb_{t}\leq\lmb_{t}+\frac{\wh b}{\lmb}\leq C\lmb.
\]
Applying Grönwall's inequality gives the almost decreasing property
\eqref{eq:bl-rate-cl3}.

\textbf{Step 4.} Smallness of $\eta$ and proof of \eqref{eq:eta-asymp}.

In this step, we claim that 
\begin{equation}
|\eta|\aleq\lmb(t)\cdot(T-t).\label{eq:bl-rate-cl4}
\end{equation}
Note that \eqref{eq:bl-rate-cl4} gives \eqref{eq:eta-asymp} in view
of $\lmb(t)\aleq T-t$. To show \eqref{eq:bl-rate-cl4}, we write
\eqref{eq:b-hat-mod2} as $|\wh{\eta}_{t}|\aleq\lmb$ and integrate
it backwards in time using \eqref{eq:bl-rate-cl0}, \eqref{eq:b-hat-diff},
and  \eqref{eq:bl-rate-cl3}: 
\[
\eta(t)\leq|\eta(t)-\wh{\eta}(t)|+\int_{t}^{T}|\wh{\eta}_{t}|dt'\aleq\lmb^{2}+\int_{t}^{T}\lmb dt'\aleq\lmb\cdot(T-t)+\lmb\int_{t}^{T}dt'\aleq\lmb(t)\cdot(T-t).
\]

\textbf{Step 5.} Completion of the proof.

It suffices to show \eqref{eq:lmb-asymp} and \eqref{eq:gmm-asymp}.
For the proof of \eqref{eq:lmb-asymp}, we combine \eqref{eq:bl-rate-cl1},
\eqref{eq:bl-rate-cl4}, and $b>0$ to have 
\[
\frac{b}{\lmb}\to\ell.
\]
Substituting the $b$-asymptotics \eqref{eq:b-asymp} into the above
gives \eqref{eq:lmb-asymp}:
\[
\lmb(t)=(\ell+o_{t\to T}(1))\cdot(T-t).
\]
For the proof of \eqref{eq:gmm-asymp}, we substitute \eqref{eq:lmb-asymp}
and \eqref{eq:bl-rate-cl4} into \eqref{eq:b-hat-mod1} to have 
\[
|\gmm_{t}|\aleq\frac{1}{\lmb^{2}}(|\eta|+\mu^{2})\aleq\frac{T-t}{\lmb}+1\aleq1.
\]
Thus $\gmm_{t}$ is integrable near the blow-up time and \eqref{eq:gmm-asymp}
follows. This completes the proof.
\end{proof}
\begin{rem}
\label{rem:H3/2-rigidity}If one has modulation estimates of the form
\begin{equation}
\begin{aligned}\Big|\frac{\lmb_{s}}{\lmb}+b\Big|+|\gmm_{s}-(m+1)\eta| & \aleq\|\eps\|_{\loc}+\beta^{2},\\
|b_{s}+b^{2}+\eta^{2}|+|\eta_{s}| & \aleq\beta\|\eps\|_{\loc}+\|\eps\|_{\loc}^{2}+\beta^{3}
\end{aligned}
\label{eq:rmk-formal-mod-est}
\end{equation}
with some $\|\eps\|_{\loc}$ satisfying
\begin{equation}
\int_{t^{\ast}}^{T}\frac{\|\eps\|_{\loc}^{2}}{\lmb^{4}}dt<+\infty,\label{eq:rmk-spacetime-bd}
\end{equation}
then the above proof still works with simple modifications. Indeed,
one has 
\begin{equation}
\Big(\frac{\beta}{\lmb}+\lmb\Big)_{t}\aleq\Big(1+\frac{\|\eps\|_{\loc}^{2}}{\lmb^{4}}\Big)\Big(\frac{\beta}{\lmb}+\lmb\Big)\label{eq:rmk-diff-ineq}
\end{equation}
and hence \eqref{eq:rmk-spacetime-bd} still guarantees $\frac{\beta}{\lmb}\to\ell>0$
in view of Grönwall's inequality, which is the most important part
of the proof. The previous proof is presented in the case where $b,\eta,\|\eps\|_{\loc}$
are replaced by $\wh b,\wh{\eta},\mu^{2}$ and \eqref{eq:rmk-spacetime-bd}
is then obviously satisfied. This says that only the spacetime control
of the form \eqref{eq:rmk-spacetime-bd} is important and the $\lmb^{7}$-factor
of the spacetime control \eqref{eq:integ-monot} (note that $\mu^{-5}ds\sim\lmb^{-7}dt$)
might be too strong. As the exponent $7$ is obtained from the numerology
$7=2s+1$ with $s=3$, the spacetime control of the form \eqref{eq:rmk-spacetime-bd}
might be available for any $H_{m}^{3/2}$ finite-time blow-up solutions.
Thus Theorem~\ref{thm:blow-up-rigidity} might be true for $H_{m}^{3/2}$-solutions.
\end{rem}

\subsection{Asymptotic decomposition and structure of asymptotic profile}
\begin{proof}[Completion of the proof of Theorem~\ref{thm:blow-up-rigidity}]
By the rigidity Proposition~\ref{prop:AsympModPar}, it suffices
to show the statements regarding the asymptotic profile. We note that,
by Theorem~\ref{thm:asymptotic-description}, there exist parameters
$\wh{\lmb}(t)\to0$ and $\wh{\gmm}(t)$ (possibly not equal to $\lmb(t)$
and $\gmm(t)$) and the asymptotic profile $z^{\ast}$ such that $u(t)-Q_{\wh{\lmb}(t),\wh{\gmm}(t)}\to z^{\ast}$
in $L^{2}$. Therefore, it suffices to show the following:
\begin{enumerate}
\item \label{enu:1}$u(t)-Q_{\lmb(t),\gmm(t)}\to z^{\ast}$ in $L^{2}$.
\item \label{enu:2}$z^{\ast}$ satisfies the properties as in Theorem~\ref{thm:blow-up-rigidity}.
\end{enumerate}
\textbf{Step 1.} Proof of (\ref{enu:1}).

In this step, we show the first item (\ref{enu:1}), namely, 
\begin{equation}
u(t)-Q_{\lmb(t),\gmm(t)}\to z^{\ast}\text{ in }L^{2}.\label{eq:asymp-decomp-1}
\end{equation}
For this purpose, we claim the following: 
\begin{equation}
\begin{aligned}\text{If \ensuremath{u(t_{n})-Q_{\lmb(t_{n}),\gmm(t_{n})}\weakto\td z^{\ast}} in \ensuremath{H_{m}^{1}}}\\
\text{for some \ensuremath{t_{n}\to T} and \ensuremath{\td z^{\ast}\in H_{m}^{1}},} & \text{ then \ensuremath{\td z^{\ast}=z^{\ast}}.}
\end{aligned}
\label{eq:asymp-decomp-2}
\end{equation}
Assuming this claim, we finish the proof of \eqref{eq:asymp-decomp-1}.
Note that $u(t)-Q_{\lmb(t),\gmm(t)}$ is $H_{m}^{1}$-bounded: 
\[
\|u-Q_{\lmb,\gmm}\|_{H_{m}^{1}}\aleq\|(P-Q)+\eps\|_{L^{2}}+\lmb^{-1}\|(P-Q)+\eps\|_{\dot{H}_{m}^{1}}\aleq1+(\mu/\lmb)\aleq1.
\]
Combining this fact with the further subsequence argument and the
claim \eqref{eq:asymp-decomp-2}, we have $u(t)-Q_{\lmb(t),\gmm(t)}\weakto z^{\ast}$
in $H_{m}^{1}$. By the Rellich--Kondrachov theorem, we have $u(t)-Q_{\lmb(t),\gmm(t)}\to z^{\ast}$
in $L_{\loc}^{2}$. On the other hand, since $\lmb(t),\wh{\lmb}(t)\to0$
and $u(t)-Q_{\wh{\lmb}(t),\wh{\gmm}(t)}\to z^{\ast}$, we have $\chf_{r\geq R}\{u(t)-Q_{\lmb(t),\gmm(t)}\}\to\chf_{r\geq R}z^{\ast}$
for any $R>0$. Therefore, we have \eqref{eq:asymp-decomp-1}.

It remains to show \eqref{eq:asymp-decomp-2}. This again follows
from the Rellich--Kondrachov theorem and $\lmb(t),\wh{\lmb}(t)\to0$;
for any $R>1$ we have
\begin{align*}
\chf_{[R^{-1},R]}\td z^{\ast} & =\lim_{n\to\infty}\chf_{[R^{-1},R]}\{u(t_{n})-Q_{\lmb(t_{n}),\gmm(t_{n})}\}\\
 & =\lim_{n\to\infty}\chf_{[R^{-1},R]}\{u(t_{n})-Q_{\wh{\lmb}(t_{n}),\wh{\gmm}(t_{n})}\}=\chf_{[R^{-1},R]}z^{\ast},
\end{align*}
where the limits are taken in the $L^{2}$-topology.

\textbf{Step 2.} Proof of (\ref{enu:2}).

In this step, we show the second item of Theorem~\ref{thm:blow-up-rigidity}.
Recall from \eqref{eq:asymp-decomp-1} that the asymptotic profile
$z^{\ast}$ is the $L^{2}$-limit of $u-Q^{\sharp}=[P-Q]^{\sharp}+\eps^{\sharp}$.
The limit of $[P-Q]^{\sharp}$ gives the singular part $z_{\sg}^{\ast}$
of $z^{\ast}$ and the limit of $\eps^{\sharp}$ gives the regular
part $z_{\rg}^{\ast}$. Thus it suffices to show the following two
claims: 
\begin{align}
 & [P-Q]^{\sharp}(t,r)\to\begin{cases}
-e^{i\gmm^{\ast}}\ell^{2}\frac{\sqrt{2}}{8}\chi(\ell r)\cdot r\text{ in }L^{2} & \text{if }m=1,\\
e^{i\gmm^{\ast}}\ell^{3}\frac{\sqrt{2}}{64}i\chi(\ell r)\cdot r^{2}\text{ in }L^{2} & \text{if }m=2,\\
0\text{ in }L^{2} & \text{if }m\geq3,
\end{cases}\label{eq:struc-z-1}\\
 & \|\chf_{r\geq\lmb(t)}|\eps^{\sharp}(t)|_{-3}\|_{L^{2}}\aleq1.\label{eq:struc-z-2}
\end{align}

\emph{Proof of \eqref{eq:struc-z-1}.} When $m=1$, we show that only
the $T_{2,0}^{(0)}$-part\footnote{Recall the notation \eqref{eq:Tk_j-decomp}.}
of $P-Q$ contributes to the singular part of $z^{\ast}$. Indeed,
using the pointwise bounds in Lemma~\ref{lem:TaylorExpands} and
the asymptotics of the modulation parameters in Proposition~\ref{prop:AsympModPar},
we have 
\begin{align*}
\|\chi_{B_{1}}T_{1}^{(0)}\|_{L^{2}} & \aleq\|\chf_{(0,2B_{1}]}\beta\langle y\rangle^{-1}\|_{L^{2}}\aleq\beta|\log\beta|^{1/2}\to0,\\
\|\chi_{B_{1}}T_{3}^{(0)}\|_{L^{2}} & =0,\\
\|\chi_{B_{1}}(T_{1,1}^{(0)}+T_{0,2}^{(0)})\|_{L^{2}} & \aleq\|\chf_{(0,2B_{1}]}|\eta|\beta y\|_{L^{2}}\aleq|\eta|\beta^{-1}\aleq(T-t)\to0.
\end{align*}
For the $T_{2,0}^{(0)}$-part of $P$, using \eqref{eq:asymp-T-quad}
and \eqref{eq:Q-formula}, we have 
\begin{align*}
 & \|\chi_{B_{1}}(T_{2,0}^{(0)}+b^{2}\tfrac{\sqrt{2}}{8}y)\|_{L^{2}}\aleq\|\chf_{(0,2B_{1}]}b^{2}\langle y\rangle^{-1}\log y\|_{L^{2}}\aleq b^{2}|\log\beta|^{3/2}\to0
\end{align*}
and using Proposition~\ref{prop:AsympModPar}, we have 
\begin{align*}
 & \frac{e^{i\gmm(t)}}{\lmb(t)}\chi_{B_{1}}\Big(\frac{r}{\lmb(t)}\Big)\cdot\Big(-(b(t))^{2}\frac{\sqrt{2}}{8}\frac{r}{\lmb(t)}\Big)\\
 & \qquad=-e^{i\gmm(t)}\Big(\frac{b(t)}{\lmb(t)}\Big)^{2}\frac{\sqrt{2}}{8}\cdot\chi\Big(\frac{\beta(t)}{\lmb(t)}r\Big)\cdot r\to-e^{i\gmm^{\ast}}\ell^{2}\frac{\sqrt{2}}{8}\chi(\ell r)\cdot r\text{ in }L^{2}.
\end{align*}

When $m=2$, we show that only the $T_{3,0}^{(0)}$-part of $P-Q$
contributes to the singular part of $z^{\ast}$. Indeed, proceeding
as in the previous paragraph but using the pointwise bounds for the
$m=2$ case, we have 
\begin{align*}
\|\chi_{B_{1}}T_{1}^{(0)}\|_{L^{2}} & \aleq\|\chf_{(0,2B_{1}]}\beta\langle y\rangle^{-2}\|_{L^{2}}\aleq\beta\to0,\\
\|\chi_{B_{1}}T_{2}^{(0)}\|_{L^{2}} & \aleq\|\chf_{(0,2B_{1}]}\beta^{2}\|_{L^{2}}\aleq\beta\to0,\\
\|\chi_{B_{1}}(T_{2,1}^{(0)}+T_{1,2}^{(0)}+T_{0,3}^{(0)})\|_{L^{2}} & \aleq\|\chf_{(0,2B_{1}]}|\eta|\beta^{2}y^{2}\|_{L^{2}}\aleq|\eta|\beta^{-1}\aleq(T-t)\to0.
\end{align*}
For the $T_{3,0}^{(0)}$-part of $P$, using \eqref{eq:def-T0_3}
instead, we have 
\[
\|\chi_{B_{1}}(T_{3,0}^{(0)}-ib^{3}\tfrac{\sqrt{2}}{64}y^{2})\|_{L^{2}}\aleq\|\chf_{(0,2B_{1}]}b^{3}\langle y\rangle^{-4}\|_{L^{2}}\aleq b^{3}\to0
\]
and 
\begin{align*}
 & \frac{e^{i\gmm(t)}}{\lmb(t)}\chi_{B_{1}}\Big(\frac{r}{\lmb(t)}\Big)\cdot\Big(i(b(t))^{3}\frac{\sqrt{2}}{64}\Big(\frac{r}{\lmb(t)}\Big)^{2}\Big)\\
 & \qquad=e^{i\gmm(t)}\Big(\frac{b(t)}{\lmb(t)}\Big)^{3}\frac{\sqrt{2}}{64}i\cdot\chi\Big(\frac{\beta(t)}{\lmb(t)}r\Big)\cdot r^{2}\to e^{i\gmm^{\ast}}\ell^{3}\frac{\sqrt{2}}{64}i\chi(\ell r)\cdot r^{2}\text{ in }L^{2}.
\end{align*}

When $m\geq3$, there is no singular part of $z^{\ast}$: 
\[
\|P-Q\|_{L^{2}}\aleq\|\chf_{(0,2B_{1}]}(\beta\langle y\rangle^{-3}+\beta^{2}\langle y\rangle^{-1})\|_{L^{2}}\aleq\beta\to0.
\]
This completes the proof of \eqref{eq:struc-z-1}.

\emph{Proof of \eqref{eq:struc-z-2}.} This easily follows from scaling:
\[
\|\chf_{r\geq\lmb(t)}|\eps^{\sharp}(t)|_{-3}\|_{L^{2}}=\lmb^{-3}\|\chf_{y\geq1}|\eps(t)|_{-3}\|_{L^{2}}\aleq\lmb^{-3}\|\eps\|_{\dot{\calH}_{m}^{3}}\aleq1.
\]
This completes the proof of Theorem~\ref{thm:blow-up-rigidity}.
\end{proof}
\begin{rem}[On the proof of Corollary~\ref{cor:H33-soliton-resolution}]
\label{rem:cor-H33-soliton-res}Corollary~\ref{cor:H33-soliton-resolution}
is a standard consequence of Theorem~\ref{thm:blow-up-rigidity}
and pseudoconformal symmetry. Indeed, if a $H_{m}^{3,3}$ global solution
$u$ does not scatter, then $v\coloneqq\calC u$ becomes a $H_{m}^{3}$
finite-time blow-up solution. Thus $v$ satisfies the properties as
in Theorem~\ref{thm:blow-up-rigidity}. Applying the pseudoconformal
transform in reverse, we see that $u$ scatters to some $Q_{\ell,\gmm^{\ast}}$,
for some $\ell\in(0,\infty)$ and $\gmm^{\ast}\in\bbR/2\pi\bbZ$.
We omit the details and refer the reader to \cite[Section 5.2]{KimKwonOh2020arXiv}
or \cite[Section 4]{KimKwonOh2022arXiv1}.
\end{rem}

\appendix

\section{\label{sec:FormalRightInverses}Formal right inverses of linearized
operators}

In this appendix, we record explicit formulas of formal right inverses
for the linear operators $L_{Q}$, $A_{Q}$, $A_{Q}^{\ast}$, $H_{Q}$,
and $\td H_{Q}$ used in the profile construction (Section~\ref{sec:Modified-profiles}).
As $L_{Q}$, $A_{Q}$, and $H_{Q}$ have nontrivial kernels, their
right inverses are not necessarily unique. As in \cite{KimKwonOh2020arXiv},
we fix the choice of formal right inverses, denoted by $\out T^{-1}$
for a linear operator $T$, by requiring the \emph{outgoing property};
its Green's function $^{(T)}G(y,y')$ is determined by the properties
\begin{align*}
T\Big({}^{(T)}G(y,y')\Big) & =\delta_{y'}(y),\\
^{(T)}G(y,y') & =0\qquad\text{for }0<y<y'.
\end{align*}
Roughly speaking, this means that if $f(y)=0$ for $y\leq A$, then
$[\out T^{-1}f](y)=0$ also for $y\leq A$.

We begin with the explicit formula for the formal right inverse of
$L_{Q}$.
\begin{lem}[Formal right inverse of $L_{Q}$]
\label{lem:LQ-inv}Let $m\geq0$. Let 
\begin{equation}
J_{1}(y)\coloneqq\frac{1-y^{2m+2}}{1+y^{2m+2}},\qquad J_{2}(y)\coloneqq\frac{1-y^{2m+2}}{1+y^{2m+2}}\log y+\frac{2}{m+1}\frac{1}{1+y^{2m+2}}.\label{eq:def-J1J2}
\end{equation}
\begin{itemize}
\item (Other expressions for $J_{k}$) We have 
\begin{equation}
J_{1}=\frac{m+1+A_{\tht}[Q]}{m+1}=\frac{\Lmb Q}{(m+1)Q},\qquad J_{2}=J_{1}\log y+\frac{2m+2+A_{\tht}[Q]}{(m+1)^{2}}.\label{eq:Jk-other-exp}
\end{equation}
\item (Asymptotics of $J_{1}$ and $J_{2}$ and their derivatives) For $y\geq2$
and $k\in\bbN$, 
\begin{equation}
|J_{1}+1|_{k}+(\log y)^{-1}|J_{2}+\log y|_{k}\aleq_{k}y^{-(2m+2)}.\label{eq:asymp-J1J2}
\end{equation}
\item (Formal right inverse of $L_{Q}$) The operator $\out L_{Q}^{-1}$
defined by the formula 
\begin{equation}
\begin{aligned}[][\out L_{Q}^{-1}f](y) & \coloneqq Q(y)\Big[J_{1}(y)\int_{0}^{y}[\rd_{y}J_{2}](y')\frac{\Re f(y')}{Q(y')}y'dy'\\
 & \qquad-J_{2}(y)\int_{0}^{y}[\rd_{y}J_{1}](y')\frac{\Re f(y')}{Q(y')}y'dy'+\int_{0}^{y}\frac{i\Im f(y')}{Q(y')}dy'\Big].
\end{aligned}
\label{eq:def-LQ-inv}
\end{equation}
is a formal right inverse of $L_{Q}$, i.e., for any function $f:(0,\infty)\to\bbC$
such that the integrals of \eqref{eq:def-LQ-inv} are finite, we have
\begin{equation}
L_{Q}(\out L_{Q}^{-1}f)=f.\label{eq:apdx-1-1}
\end{equation}
\item (Smooth extension) If $f$ has a smooth (resp., analytic) $(m+1)$-equivariant
extension on $\bbR^{2}$, then $\out L_{Q}^{-1}f$ has a smooth (resp.,
analytic) $m$-equivariant extension on $\bbR^{2}$.
\end{itemize}
\end{lem}

\begin{rem}
The formula \eqref{eq:def-LQ-inv} for the real part can be derived
from the observation that 
\[
(\bfD_{Q}+\frac{1}{y})L_{Q}=(\bfD_{Q}+\frac{1}{y})\bfD_{Q}+Q^{2}
\]
is a \emph{local} second-order differential operator, the fact that
$\Lmb Q$ belongs to the kernel, and Wronskian relations of this operator.
Note that this shares a similar spirit with the observation in Jackiw--Pi
\cite{JackiwPi1990PRL} that the Bogomol'nyi equation \eqref{eq:Bogomol'nyi-eq}
can be linked to the Liouville equation, which is local.
\end{rem}

\begin{proof}
We omit the obvious proof of \eqref{eq:Jk-other-exp}--\eqref{eq:asymp-J1J2}.

We turn to the proof of \eqref{eq:apdx-1-1}. For the imaginary part
of $f$, we have $L_{Q}=\bfD_{Q}$ and hence \eqref{eq:apdx-1-1}
is immediate from the relation $\bfD_{Q}Q=0$. Henceforth, we assume
$f$ is real. Note that $J_{1}$ and $J_{2}$ solve 
\begin{equation}
\rd_{yy}J_{k}+\frac{1}{y}\rd_{y}J_{k}+Q^{2}J_{k}=0,\qquad\forall k\in\{1,2\},\label{eq:apdx-1-2}
\end{equation}
which easily follows from the expressions \eqref{eq:Jk-other-exp}
and the computation
\[
\Big(\rd_{yy}+\frac{1}{y}\rd_{y}\Big)A_{\tht}[Q]=\frac{1}{y}\rd_{y}\Big(-\frac{1}{2}y^{2}Q^{2}\Big)=-Q\Lmb Q.
\]
We also note the Wronskian relation 
\begin{equation}
J_{1}(\rd_{y}J_{2})-(\rd_{y}J_{1})J_{2}=\frac{1}{y}.\label{eq:apdx-1-3}
\end{equation}
Letting $\td f(y)\coloneqq f(y)/Q(y)$ and using $\bfD_{Q}Q=0$, we
have 
\begin{align*}
L_{Q}(\out L_{Q}^{-1}f) & =Q\Big\{\rd_{y}\Big[J_{1}\int_{0}^{y}(\rd_{y}J_{2})\td fy'dy'-J_{2}\int_{0}^{y}(\rd_{y}J_{1})\td fy'dy'\Big]\\
 & +\frac{1}{y}\int_{0}^{y}Q^{2}\Big[J_{1}\int_{0}^{y'}(\rd_{y}J_{2})\td fy''dy''-J_{2}\int_{0}^{y'}(\rd_{y}J_{1})\td fy''dy''\Big]y'dy'\Big)\Big\}.
\end{align*}
On one hand, we use the Wronskian relation \eqref{eq:apdx-1-3} to
obtain 
\begin{align*}
 & \rd_{y}\Big[J_{1}\int_{0}^{y}(\rd_{y}J_{2})\td fy'dy'-J_{2}\int_{0}^{y}(\rd_{y}J_{1})\td fy'dy'\Big]\\
 & =\td f(y')+\Big[(\rd_{y}J_{1})\int_{0}^{y}(\rd_{y}J_{2})\td fy'dy'-(\rd_{y}J_{2})\int_{0}^{y}(\rd_{y}J_{1})\td fy'dy'\Big].
\end{align*}
On the other hand, using the relation \eqref{eq:apdx-1-2} in the
form $yQ^{2}J_{k}=-\rd_{y}(y\rd_{y}J_{k})$ and integrating by parts,
we obtain 
\begin{align*}
 & \frac{1}{y}\int_{0}^{y}Q^{2}\Big[J_{1}\int_{0}^{y'}(\rd_{y}J_{2})\td fy''dy''-J_{2}\int_{0}^{y'}(\rd_{y}J_{1})\td fy''dy''\Big]y'dy'\Big)\\
 & =-\Big[(\rd_{y}J_{1})\int_{0}^{y}(\rd_{y}J_{2})\td fy'dy'-(\rd_{y}J_{2})\int_{0}^{y}(\rd_{y}J_{1})\td fy'dy'\Big].
\end{align*}
Summing up the previous two displays completes the proof of \eqref{eq:apdx-1-1}.

We turn to the proof of the smoothness assertion. We recall a basic
fact (see \cite[Lemmas A.2 and A.3]{KimKwon2023MAMS} for the proof
of particularly similar statements) that a smooth function $f:(0,\infty)\to\bbC$
admits a smooth $m$-equivariant extension on $\bbR^{2}$ if and only
if its even or odd extension on $\bbR$ (i.e., $\td f(x)=(-1)^{m}f(|x|)$)
is also smooth and has vanishing Taylor coefficients at $0$ up to
$(|m|-1)$-th order. A similar statement holds for analytic extensions.
With this fact and the expression \eqref{eq:def-LQ-inv}, the assertion
is obvious for the imaginary part of $f$.

Henceforth, we show the assertion for the real part of $f$. We may
assume $f$ is real. Note that $J_{2}$ can be written in the form
\[
J_{2}=J_{1}\log y+\td J_{2},
\]
where $\td J_{2}(y)=\frac{2}{m+1}(1+y^{2m+2})^{-1}$. Since $\td J_{2}$
is analytic, its contribution to the expression \eqref{eq:def-LQ-inv}
is safe. By a similar reason, the contribution of $J_{1}(y)/y$ of
$\rd_{y}(J_{1}\log y)$ is safe. Thus it remains to show that (writing
$\td f(y)\coloneqq f(y)/Q(y)$ as before) 
\begin{equation}
Q\Big[J_{1}\int_{0}^{y}(\rd_{y}J_{1})\log y'\cdot\td fy'dy'-J_{1}\log y\int_{0}^{y}(\rd_{y}J_{1})\td fy'dy'\Big]\label{eq:apdx-1-4}
\end{equation}
has a smooth expression, in particular at $y=0$. This can be seen
by Fubini:
\[
-\eqref{eq:apdx-1-4}=QJ_{1}\int_{0}^{y}(\rd_{y}J_{1})\Big(\int_{y'}^{y}\frac{dy''}{y''}\Big)\td fy'dy'=QJ_{1}\int_{0}^{y}\Big\{\int_{0}^{y'}(\rd_{y}J_{1})\td fy''dy''\Big\}\frac{dy'}{y'}.
\]
This completes the proof of the smoothness assertion.
\end{proof}
It is easy to find formal right inverses of $A_{Q}$, $A_{Q}^{\ast}$,
$H_{Q}$, and $\td H_{Q}$ using $A_{Q}(yQ)=0$, $A_{Q}^{\ast}(\frac{1}{y^{2}Q})=0$,
$H_{Q}=A_{Q}^{\ast}A_{Q}$, and $\td H_{Q}=A_{Q}A_{Q}^{\ast}$. In
the following lemmas, we record explicit formulas of right inverses
of these operators. We omit the proofs, which are clear from their
explicit formulas. See also \cite[Section 3.2]{KimKwonOh2020arXiv}.
\begin{lem}[Formal inverse of $A_{Q}$]
\label{lem:AQ-inv}Let $m\geq0$.
\begin{itemize}
\item The operator $\out A_{Q}^{-1}$ defined by 
\begin{equation}
[\out A_{Q}^{-1}f](y)\coloneqq yQ(y)\int_{0}^{y}\frac{f(y')}{y'Q(y')}dy'\label{eq:def-AQ-inv}
\end{equation}
is a formal right inverse of $A_{Q}$.
\item If $f$ has a smooth (resp., analytic) $(m+2)$-equivariant extension
on $\bbR^{2}$, then $\out A_{Q}^{-1}f$ has a smooth (resp., analytic)
$(m+1)$-equivariant extension.
\end{itemize}
\end{lem}

\begin{lem}[Formal inverse of $A_{Q}^{\ast}$]
\label{lem:AQ-ast-inv}Let $m\geq0$.
\begin{itemize}
\item The operator $\out(A_{Q}^{\ast})^{-1}$ defined by 
\begin{equation}
[\out(A_{Q}^{\ast})^{-1}f](y)\coloneqq-\frac{1}{y^{2}Q(y)}\int_{0}^{y}y'Q(y')f(y')y'dy'\label{eq:def-AQ-ast-inv}
\end{equation}
is a formal right inverse of $A_{Q}$.
\item If in addition $f\perp yQ$ with respect to the $L^{2}(ydy)$-inner
product, then 
\begin{equation}
[\out(A_{Q}^{\ast})^{-1}f](y)=\frac{1}{y^{2}Q(y)}\int_{y}^{\infty}y'Q(y')f(y')y'dy'.\label{eq:AQ-ast-inv-orthog}
\end{equation}
\item If $f$ has a smooth (resp., analytic) $(m+1)$-equivariant extension
on $\bbR^{2}$, then $\out(A_{Q}^{\ast})^{-1}f$ has a smooth (resp.,
analytic) $(m+2)$-equivariant extension.
\end{itemize}
\end{lem}

\begin{lem}[Formal right inverse of $H_{Q}$]
\label{lem:HQ-inv}Let $m\geq0$. Let 
\[
h_{1}(y)\coloneqq yQ(y),\qquad h_{2}(y)\coloneqq yQ(y)\int_{1}^{y}\frac{1}{(y')^{3}Q^{2}(y')}dy'.
\]
\begin{itemize}
\item The operator $\out(H_{Q})^{-1}$ defined by 
\begin{equation}
[\out H_{Q}^{-1}f](y)\coloneqq h_{1}(y)\int_{0}^{y}h_{2}(y')f(y')y'dy'-h_{2}(y)\int_{0}^{y}h_{1}(y')f(y')y'dy'\label{eq:def-HQ-inv}
\end{equation}
is a formal right inverse of $H_{Q}$.
\item If in addition $f\perp yQ$ with respect to the $L^{2}(ydy)$-inner
product, then
\begin{equation}
[\out H_{Q}^{-1}f](y)=h_{1}(y)\int_{0}^{y}h_{2}(y')f(y')y'dy'+h_{2}(y)\int_{y}^{\infty}h_{1}(y')f(y')y'dy'.\label{eq:HQ-inv-orthog}
\end{equation}
\item If $f$ has a smooth (resp., analytic) $(m+1)$-equivariant extension
on $\bbR^{2}$, then $\out H_{Q}^{-1}f$ has a smooth (resp., analytic)
$(m+1)$-equivariant extension.
\end{itemize}
\end{lem}

\begin{lem}[Formal right inverses of $\td H_{Q}$]
\label{lem:HtdQ-inv}Let $m\geq0$. Let 
\begin{align*}
\td h_{1}(y) & \coloneqq\frac{1}{y^{2}Q(y)}\int_{0}^{y}(y'Q(y'))^{2}\,y'dy'\\
\td h_{2}(y) & \coloneqq\begin{cases}
\frac{1}{y^{2}Q(y)}\cdot\frac{1}{\frkp}\int_{y}^{\infty}(y'Q(y'))^{2}\,y'dy' & \text{if }m\geq1,\\
\frac{1}{y^{2}Q(y)} & \text{if }m=0,
\end{cases}
\end{align*}
where $\frkp=\int_{0}^{\infty}(yQ(y))^{2}\,ydy$.
\begin{itemize}
\item The operators $\out\td H_{Q}^{-1}$ and $\inn\td H_{Q}^{-1}$ defined
by 
\begin{align}
[\out\td H_{Q}^{-1}f](y) & \coloneqq\td h_{2}(y)\int_{0}^{y}\td h_{1}(y')f(y')y'dy'-\td h_{1}(y)\int_{0}^{y}\td h_{2}(y')f(y')y'dy',\label{eq:def-HtdQ-inv-out}\\{}
[\inn\td H_{Q}^{-1}f](y) & \coloneqq\td h_{2}(y)\int_{0}^{y}\td h_{1}(y')f(y')y'dy'+\td h_{1}(y)\int_{y}^{\infty}\td h_{2}(y')f(y')y'dy'\label{eq:def-HtdQ-inv-in}
\end{align}
are formal right inverses of $\td H_{Q}$. Note that the definition
of $\inn\td H_{Q}^{-1}f$ requires a suitable integrability condition
for $f$ in the region $y\geq1$.
\item If $f$ has a smooth (resp., analytic) $(m+2)$-equivariant extension
on $\bbR^{2}$, then $\out\td H_{Q}^{-1}f$ has a smooth (resp., analytic)
$(m+2)$-equivariant extension. The analogous statement holds for
$\inn\td H_{Q}^{-1}f$ when it is well-defined.
\end{itemize}
\end{lem}

\section{\label{sec:Linear-coercivity-estimates}Linear coercivity estimates}

In this appendix, we give a sketch of the proof of Lemma~\ref{lem:LinearCoercivity}.
The estimates \eqref{eq:AQ-coer-H1-to-H2}, \eqref{eq:LQ-coer-L2-to-H1},
and \eqref{eq:LQ-coer-H2-to-H3} are already proved in \cite[Appendix A]{KimKwon2023AnnPDE}.
The reader may fill the omitted details by following the proofs therein.

We collect some notation. Recall the Cauchy--Riemann operator $\rd_{+}=\rd_{y}-\frac{m}{y}$
acting on $m$-equivarirant functions \eqref{eq:def-euc-CR}. We denote
$\dot{\calH}_{m}^{1}=\dot{H}_{m}^{1}$. We use the $\dot{\calH}_{m,\leq}^{k}$-
and $\dot{\calH}_{m,\geq}^{k}$-norms to denote the $\dot{\calH}_{m}^{k}$-norms
restricted on $(0,\frac{1}{2}]$ and $[\frac{1}{2},\infty)$, respectively.
More precisely, one adds $\chf_{(0,\frac{1}{2}]}$ or $\chf_{[\frac{1}{2},\infty)}$
to the norms appearing in the RHS of the definitions of $\dot{\calH}_{m}^{k}$.
Note that $\|f\|_{\dot{\calH}_{m,\geq}^{k}}\sim\|\chf_{[\frac{1}{2},\infty)}|f|_{-k}\|_{L^{2}}$.

For the proof of Lemma~\ref{lem:LinearCoercivity}, we first note
that \eqref{eq:H-td-coer} is immediate from $\td V_{Q}\geq1$. The
remaining coercivity estimates for $A_{Q}$ and $L_{Q}$ follow from
combining the corresponding subcoercivity estimates (Lemmas \ref{lem:subcoer-LQ}
and \ref{lem:subcoer-AQ}) below with a soft functional analytic argument;
see for instance \cite[Lemma A.6]{KimKwon2023AnnPDE}.
\begin{lem}[Subcoercivity estimates for $L_{Q}$]
\label{lem:subcoer-LQ}For any $k\in\{0,1,2\}$ and $f\in C_{c,m}^{\infty}$,
we have 
\begin{equation}
\|L_{Q}f\|_{\dot{H}_{m+1}^{k}}+\|\chf_{y\sim1}f\|_{L^{2}}\sim\|f\|_{\dot{\calH}_{m}^{k+1}}.\label{eq:subcoer-LQ}
\end{equation}
Moreover, the kernel of $L_{Q}:H_{m,\loc}^{1}\to L_{\loc}^{2}$ is
$\mathrm{span}_{\bbR}\{\Lmb Q,iQ\}$.
\end{lem}

\begin{proof}
As mentioned above, here we only sketch the proof. The details for
the $k\in\{0,2\}$ cases can be found in \cite[Appendix A]{KimKwon2023AnnPDE}.

\textbf{Step 1.} Estimates in the region $y\aleq1$.

In the region $y\aleq1$, $L_{Q}$ is well approximated by $\rd_{y}-\frac{m}{y}=\rd_{+}$
and we decompose 
\begin{equation}
L_{Q}f=\rd_{+}f-\tfrac{1}{y}A_{\tht}[Q]f-\tfrac{2}{y}A_{\tht}[Q,f]Q\eqqcolon\rd_{+}f+\td L_{Q,\leq}f.\label{eq:LQ-small-y-1}
\end{equation}
The following estimates hold for $k\in\{0,1,2\}$: 
\begin{align}
\|f\|_{\dot{\calH}_{m,\leq}^{k+1}} & \aleq\|\chf_{(0,1]}|\rd_{+}f|_{-k}\|_{L^{2}}+\|\chf_{[\frac{1}{2},1]}f\|_{L^{2}},\label{eq:LQ-small-y-2}\\
\|\chf_{(0,1]}|\td L_{Q,\leq}f|_{-k}\|_{L^{2}} & \aleq\|\chf_{(0,1]}y|f|_{-k}\|_{L^{2}}.\label{eq:LQ-small-y-3}
\end{align}
The first estimate is a consequence of (logarithmic) Hardy's inequality
and the operator identities $\rd_{yyy}=(\rd_{y}+\frac{1}{y})^{2}\rd_{+}$
for $m=2$ and $\rd_{yy}=(\rd_{y}+\frac{1}{y})\rd_{+}$ for $m=1$.
The second estimate can be proved using $|Q|_{2}\aleq y$ in the region
$y\aleq1$.

\textbf{Step 2.} Estimates in the region $y\ageq1$.

In the region $y\ageq1$, $L_{Q}$ is well approximated by $\rd_{y}+\frac{m+2}{y}$
and we decompose 
\begin{equation}
\begin{aligned}L_{Q}f & =(\rd_{y}+\tfrac{m+2}{y})f-\tfrac{1}{y}(2m+2+A_{\tht}[Q])f-\tfrac{2}{y}A_{\tht}[Q,f]Q\\
 & \eqqcolon(\rd_{y}+\tfrac{m+2}{y})f+\td L_{Q,\geq}f.
\end{aligned}
\label{eq:LQ-small-y-4}
\end{equation}
The following estimates hold for $k\in\{0,1,2\}$: 
\begin{align}
\|f\|_{\dot{\calH}_{m,\geq}^{k+1}} & \aleq\|\chf_{[\frac{1}{4},\infty)}|(\rd_{y}+\tfrac{m+2}{y})f|_{-k}\|_{L^{2}}+\|\chf_{[\frac{1}{4},\frac{1}{2}]}f\|_{L^{2}},\label{eq:LQ-small-y-5}\\
\|\chf_{[\frac{1}{4},\infty)}|\td L_{Q,\geq}f|_{-k}\|_{L^{2}} & \aleq\|\chf_{[\frac{1}{4},\infty)}y^{-3}|f|_{-k}\|_{L^{2}}+\|\chf_{(0,\frac{1}{4}]}y^{1-k}f\|_{L^{2}}.\label{eq:LQ-small-y-6}
\end{align}
The first estimate is a consequence of Hardy's inequality (using $m+2>0$).
The second estimate can be proved using \eqref{eq:LQ-small-y-3} and
$|Q|_{2}\aleq y^{-3}$ in the region $y\ageq1$. Note that the second
term of RHS\eqref{eq:LQ-small-y-6} appears due to the nonlocal term
$-\frac{2}{y}A_{\tht}[Q,f]Q$.

\textbf{Step 3.} Further treatment of localized lower order errors.

We further estimate the localized lower order terms (i.e., \eqref{eq:LQ-small-y-3}
and \eqref{eq:LQ-small-y-6}) as 
\begin{equation}
\|(\chf_{(0,1]}y+\chf_{[1,\infty)}y^{-3})\cdot|f|_{-k}\|_{L^{2}}\leq o_{r_{0}\to0}(\|f\|_{\dot{\calH}_{m}^{k+1}})+O_{r_{0}}(\|\chf_{[r_{0},1/r_{0}]}f\|_{L^{2}}).\label{eq:LQ-small-y-7}
\end{equation}
For the proof of \eqref{eq:LQ-small-y-7}, we first note that the
contributions from the regions $y\leq2r_{0}$ and $y\geq1/(2r_{0})$
are bounded by $o_{r_{0}\to0}(\|f\|_{\dot{\calH}_{m}^{k+1}})$. The
contribution from the region $2r_{0}\leq y\leq1/(2r_{0})$ can be
estimated using interpolation, for example
\begin{align*}
 & \|\chf_{[2r_{0},1/(2r_{0})]}|f|_{-2}\|_{L^{2}}\\
 & \quad\aleq_{r_{0}}\|\varphi f\|_{H_{m}^{2}}\aleq_{r_{0}}\|\varphi f\|_{L^{2}}^{1/3}\|\varphi f\|_{H_{m}^{3}}^{2/3}\aleq_{r_{0}}\|\chf_{[r_{0},1/r_{0}]}f\|_{L^{2}}^{1/3}\|f\|_{\dot{\calH}_{m}^{k+1}}^{2/3},
\end{align*}
where $\varphi=\varphi(y)$ is some smooth compactly supported radial
function such that $\varphi\equiv1$ on $[2r_{0},1/(2r_{0})]$ and
$\varphi\equiv0$ outside $[r_{0},1/r_{0}]$. This gives \eqref{eq:LQ-small-y-7}.

\textbf{Step 4.} Proof of the subcoercivity estimate \eqref{eq:subcoer-LQ}.

The proof of the $(\aleq)$-direction follows from first applying
the decompositions \eqref{eq:LQ-small-y-1} and \eqref{eq:LQ-small-y-4},
and then applying the estimates \eqref{eq:LQ-small-y-3} and \eqref{eq:LQ-small-y-6}:
\begin{align*}
 & \|L_{Q}f\|_{\dot{H}_{m+1}^{k}}+\|\chf_{y\sim1}f\|_{L^{2}}\\
 & \aleq\big\|\chf_{(0,\frac{1}{2}]}|\rd_{+}f|_{-k}\|_{L^{2}}+\|\chf_{(0,\frac{1}{2}]}y|f|_{-k}\|_{L^{2}}+\|\chf_{[\frac{1}{2},\infty)}|f|_{-(k+1)}\big\|_{L^{2}}\aleq\|f\|_{\dot{\calH}_{m}^{k+1}}.
\end{align*}
The proof of the $(\ageq)$-direction follows from applying \eqref{eq:LQ-small-y-2}
and \eqref{eq:LQ-small-y-5} first, then applying \eqref{eq:LQ-small-y-1}
and \eqref{eq:LQ-small-y-4}, \eqref{eq:LQ-small-y-3} and \eqref{eq:LQ-small-y-6},
and finally \eqref{eq:LQ-small-y-7}: 
\begin{align*}
\|f\|_{\dot{\calH}_{m}^{k+1}} & \aleq\|\chf_{(0,1]}|\rd_{+}f|_{-k}\|_{L^{2}}+\|\chf_{[\frac{1}{4},\infty)}|(\rd_{y}+\tfrac{m+2}{y})f|_{-k}\|_{L^{2}}+\|\chf_{[\frac{1}{4},1]}f\|_{L^{2}}\\
 & \aleq\||L_{Q}f|_{-k}\|_{L^{2}}+\|\chf_{(0,1]}y|f|_{-k}\|_{L^{2}}+\|\chf_{[\frac{1}{4},\infty)}y^{-3}|f|_{-k}\|_{L^{2}}\\
 & \aleq\|L_{Q}f\|_{\dot{H}_{m+1}^{k}}+o_{r_{0}\to0}(\|f\|_{\dot{\calH}_{m}^{k+1}})+O_{r_{0}}(\|\chf_{[r_{0},1/r_{0}]}f\|_{L^{2}}).
\end{align*}
Choosing $r_{0}>0$ sufficiently small, the second term of RHS can
be absorbed into the LHS. This gives the $(\ageq)$-inequality of
\eqref{eq:subcoer-LQ}.

\textbf{Step 5.} Kernel characterization.

If $f\in H_{m,\loc}^{1}$ belongs to the kernel of $L_{Q}$, then
$f$ is smooth by elliptic regularity; see for example \cite[Lemma A.5]{KimKwon2023AnnPDE}.
Then, the result follows by the characterization of the smooth kernel
of $L_{Q}$ proved in \cite[Lemma 3.3]{KimKwon2023MAMS}.
\end{proof}
\begin{lem}[Subcoercivity estimates for $A_{Q}$]
\label{lem:subcoer-AQ}For any $k\in\{0,1\}$ and $f\in C_{c,m+1}^{\infty}$,
we have 
\begin{align*}
\|A_{Q}f\|_{\dot{H}_{m+1}^{k}}+\|\chf_{y\sim1}f\|_{L^{2}} & \sim\|f\|_{\dot{H}_{m+1}^{k+1}},\\
\|A_{Q}f\|_{\dot{V}_{m+1}^{3/2}}+\|\chf_{y\sim1}f\|_{L^{2}} & \sim\|f\|_{\dot{V}_{m+1}^{5/2}}.
\end{align*}
Moreover, the kernel of $A_{Q}:H_{m+1,\loc}^{1}\to L_{\loc}^{2}$
is $\mathrm{span}_{\bbC}\{yQ\}$.
\end{lem}

\begin{proof}
We proceed similarly as in the proof of Lemma~\ref{lem:subcoer-LQ}
and omit the details. In the region $y\aleq1$, we approximate $A_{Q}$
by $\rd_{y}-\frac{m+1}{y}=\rd_{+}$ and decompose 
\[
A_{Q}f=\rd_{+}f-\tfrac{1}{y}A_{\tht}[Q]f\eqqcolon\rd_{+}f+\td A_{Q,\leq}f.
\]
We have the following estimates for $k\in\{0,1\}$: 
\begin{align*}
\|\chf_{(0,\frac{1}{2}]}|f|_{-(k+1)}\|_{L^{2}} & \aleq\|\chf_{(0,1]}|\rd_{+}f|_{-k}\|_{L^{2}}+\|\chf_{[\frac{1}{2},1]}f\|_{L^{2}},\\
\|\chf_{(0,1]}|\td A_{Q,\leq}f|_{-k}\|_{L^{2}} & \aleq\|\chf_{(0,1]}y|f|_{-k}\|_{L^{2}}.
\end{align*}
On the other hand, in the region $y\ageq1$, we approximate $A_{Q}$
by $\rd_{y}+\frac{m+1}{y}$ and decompose 
\[
A_{Q}f=\rd_{+}f-\tfrac{1}{y}(2m+2+A_{\tht}[Q])f\eqqcolon(\rd_{y}+\tfrac{m+1}{y})f+\td A_{Q,\geq}f.
\]
We have the following estimates for $k\in\{0,1\}$: 
\begin{align*}
\|\chf_{[\frac{1}{2},\infty)}|f|_{-(k+1)}\|_{L^{2}} & \aleq\|\chf_{[\frac{1}{4},\infty)}|(\rd_{y}+\tfrac{m+1}{y})f|_{-k}\|_{L^{2}}+\|\chf_{[\frac{1}{4},\frac{1}{2}]}f\|_{L^{2}},\\
\|\chf_{[\frac{1}{2},\infty)}\tfrac{1}{y^{3/2}}|f|_{-1}\|_{L^{2}} & \aleq\|\chf_{[\frac{1}{4},\infty)}\tfrac{1}{y^{3/2}}(\rd_{y}+\tfrac{m+1}{y})f\|_{L^{2}}+\|\chf_{[\frac{1}{4},\frac{1}{2}]}f\|_{L^{2}},\\
\|\chf_{[\frac{1}{2},\infty)}\tfrac{1}{y}\rd_{yy}f\|_{L^{2}} & \aleq\|\chf_{[\frac{1}{2},\infty)}\tfrac{1}{y}\rd_{y}(\rd_{y}+\tfrac{m+1}{y})f\|_{L^{2}}+\|\chf_{[\frac{1}{2},\infty)}\tfrac{1}{y^{2}}|f|_{-1}\|_{L^{2}},\\
\|\chf_{[\frac{1}{4},\infty)}|\td A_{Q,\geq}f|_{-k}\|_{L^{2}} & \aleq\|\chf_{[\frac{1}{4},\infty)}y^{-3}|f|_{-k}\|_{L^{2}}.
\end{align*}
Here, the first and second estimates are consequences of Hardy's inequality
(using $m+1>0$). The second and third estimates are adapted to the
$\dot{V}_{m+1}^{5/2}$- and $\dot{V}_{m+2}^{3/2}$-norms. Finally,
we have for $k\in\{0,1\}$ 
\[
\|(\chf_{(0,1]}y+\chf_{[1,\infty)}y^{-3})\cdot|f|_{-k}\|_{L^{2}}\leq o_{r_{0}\to0}(\|f\|_{\dot{H}_{m+1}^{k+1}})+O_{r_{0}}(\|\chf_{[r_{0},1/r_{0}]}f\|_{L^{2}})
\]
and in fact we have a better estimate
\[
\|(\chf_{(0,1]}y+\chf_{[1,\infty)}y^{-3})\cdot|f|_{-1}\|_{L^{2}}\leq o_{r_{0}\to0}(\|f\|_{\dot{V}_{m+1}^{5/2}})+O_{r_{0}}(\|\chf_{[r_{0},1/r_{0}]}f\|_{L^{2}}).
\]

Proceeding as in Step~4 of the proof of Lemma~\ref{lem:subcoer-LQ},
we obtain the subcoercivity estimates. The kernel characterization
of $A_{Q}$ is in fact easier than that of $L_{Q}$ because $A_{Q}(yQ)=0$
($A_{Q}$ is $\bbC$-linear) and we can use the uniqueness of ODE
solutions.
\end{proof}
\bibliographystyle{abbrv}
\bibliography{References}

\begin{thebibliography}{10}

\bibitem{BergeDeBouardSaut1995Nonlinearity}
L.~Berg\'{e}, A.~De~Bouard, and J.-C. Saut.
\newblock Blowing up time-dependent solutions of the planar, {C}hern-{S}imons
  gauged nonlinear {S}chr\"{o}dinger equation.
\newblock {\em Nonlinearity}, 8(2):235--253, 1995.

\bibitem{BergeDeBouardSaut1995PRL}
L.~Berg\'{e}, A.~De~Bouard, and J.-C. Saut.
\newblock Collapse of {C}hern-{S}imons-gauged matter fields.
\newblock {\em Phys. Rev. Lett.}, 74(20):3907--3911, 1995.

\bibitem{BourgainWang1997}
J.~Bourgain and W.~Wang.
\newblock Construction of blowup solutions for the nonlinear {S}chr\"odinger
  equation with critical nonlinearity.
\newblock {\em Ann. Scuola Norm. Sup. Pisa Cl. Sci. (4)}, 25(1-2):197--215
  (1998), 1997.
\newblock Dedicated to Ennio De Giorgi.

\bibitem{Collot2018MemAMS}
C.~Collot.
\newblock Type {II} blow up manifolds for the energy supercritical semilinear
  wave equation.
\newblock {\em Mem. Amer. Math. Soc.}, 252(1205):v+163, 2018.

\bibitem{CollotDuyckaertsKenigMerle2022arXiv}
C.~Collot, T.~Duyckaerts, C.~E. Kenig, and F.~Merle.
\newblock Soliton resolution for the radial quadratic wave equation in six
  space dimensions.
\newblock {\em arXiv e-prints 2201.01848}, 2022.

\bibitem{CollotGhoulMasmoudiNguyen2022CPAM}
C.~Collot, T.-E. Ghoul, N.~Masmoudi, and V.~T. Nguyen.
\newblock Refined description and stability for singular solutions of the 2{D}
  {K}eller-{S}egel system.
\newblock {\em Comm. Pure Appl. Math.}, 75(7):1419--1516, 2022.

\bibitem{Dodson2015AdvMath}
B.~Dodson.
\newblock Global well-posedness and scattering for the mass critical nonlinear
  {S}chr\"odinger equation with mass below the mass of the ground state.
\newblock {\em Adv. Math.}, 285:1589--1618, 2015.

\bibitem{Dunne1995Springer}
G.~Dunne.
\newblock {\em Self-{D}ual {C}hern-{S}imons {T}heories}, volume~36 of {\em
  Lecture Notes in Physics Monographs}.
\newblock Springer-Verlag Berline Heidelberg, 1995.

\bibitem{DuyckaertsKenigMartelMerle2022CMP}
T.~Duyckaerts, C.~Kenig, Y.~Martel, and F.~Merle.
\newblock Soliton resolution for critical co-rotational wave maps and radial
  cubic wave equation.
\newblock {\em Comm. Math. Phys.}, 391(2):779--871, 2022.

\bibitem{DuyckaertsKenigMerle2013CambJMath}
T.~Duyckaerts, C.~Kenig, and F.~Merle.
\newblock Classification of radial solutions of the focusing, energy-critical
  wave equation.
\newblock {\em Camb. J. Math.}, 1(1):75--144, 2013.

\bibitem{DuyckaertsKenigMerle2023Acta}
T.~Duyckaerts, C.~Kenig, and F.~Merle.
\newblock Soliton resolution for the radial critical wave equation in all odd
  space dimensions.
\newblock {\em Acta Math.}, 230(1):1--92, 2023.

\bibitem{GustafsonNakanishiTsai2010CMP}
S.~Gustafson, K.~Nakanishi, and T.-P. Tsai.
\newblock Asymptotic stability, concentration, and oscillation in harmonic map
  heat-flow, {L}andau-{L}ifshitz, and {S}chr\"{o}dinger maps on {$\Bbb R^2$}.
\newblock {\em Comm. Math. Phys.}, 300(1):205--242, 2010.

\bibitem{HadzicRaphael2019JEMS}
M.~Had\v{z}i\'{c} and P.~Rapha\"{e}l.
\newblock On melting and freezing for the 2{D} radial {S}tefan problem.
\newblock {\em J. Eur. Math. Soc. (JEMS)}, 21(11):3259--3341, 2019.

\bibitem{Huh2013Abstr.Appl.Anal}
H.~Huh.
\newblock Energy solution to the {C}hern-{S}imons-{S}chr\"{o}dinger equations.
\newblock {\em Abstr. Appl. Anal.}, pages Art. ID 590653, 7, 2013.

\bibitem{JackiwPi1990PRD}
R.~Jackiw and S.-Y. Pi.
\newblock Classical and quantal nonrelativistic {C}hern-{S}imons theory.
\newblock {\em Phys. Rev. D (3)}, 42(10):3500--3513, 1990.

\bibitem{JackiwPi1990PRL}
R.~Jackiw and S.-Y. Pi.
\newblock Soliton solutions to the gauged nonlinear {S}chr\"{o}dinger equation
  on the plane.
\newblock {\em Phys. Rev. Lett.}, 64(25):2969--2972, 1990.

\bibitem{JackiwPi1991PRD}
R.~Jackiw and S.-Y. Pi.
\newblock Time-dependent {C}hern-{S}imons solitons and their quantization.
\newblock {\em Phys. Rev. D (3)}, 44(8):2524--2532, 1991.

\bibitem{JackiwPi1992Progr.Theoret.}
R.~Jackiw and S.-Y. Pi.
\newblock Self-dual {C}hern-{S}imons solitons.
\newblock {\em Progr. Theoret. Phys. Suppl.}, (107):1--40, 1992.
\newblock Low-dimensional field theories and condensed matter physics (Kyoto,
  1991).

\bibitem{Jendrej2019AJM}
J.~Jendrej.
\newblock Construction of two-bubble solutions for energy-critical wave
  equations.
\newblock {\em Amer. J. Math.}, 141(1):55--118, 2019.

\bibitem{JendrejLawrie2018Invent}
J.~Jendrej and A.~Lawrie.
\newblock Two-bubble dynamics for threshold solutions to the wave maps
  equation.
\newblock {\em Invent. Math.}, 213(3):1249--1325, 2018.

\bibitem{JendrejLawrie2021arXiv}
J.~Jendrej and A.~Lawrie.
\newblock Soliton resolution for equivariant wave maps.
\newblock {\em arXiv e-prints 2106.10738, to appear in J. Amer. Math. Soc.},
  2021.

\bibitem{JendrejLawrie2023AnnPDE}
J.~Jendrej and A.~Lawrie.
\newblock Soliton resolution for the energy-critical nonlinear wave equation in
  the radial case.
\newblock {\em Ann. PDE}, 9(2):Paper No. 18, 117, 2023.

\bibitem{JendrejLawrieRodriguez2022ASENS}
J.~Jendrej, A.~Lawrie, and C.~Rodriguez.
\newblock Dynamics of bubbling wave maps with prescribed radiation.
\newblock {\em Ann. Sci. \'{E}c. Norm. Sup\'{e}r. (4)}, 55(4):1135--1198, 2022.

\bibitem{KillipLiVisanZhang2009}
R.~Killip, D.~Li, M.~Visan, and X.~Zhang.
\newblock Characterization of minimal-mass blowup solutions to the focusing
  mass-critical {NLS}.
\newblock {\em SIAM J. Math. Anal.}, 41(1):219--236, 2009.

\bibitem{Killip-Tao-Visan2009JEMS}
R.~Killip, T.~Tao, and M.~Visan.
\newblock The cubic nonlinear {S}chr\"odinger equation in two dimensions with
  radial data.
\newblock {\em J. Eur. Math. Soc. (JEMS)}, 11(6):1203--1258, 2009.

\bibitem{KimKwon2023AnnPDE}
K.~Kim and S.~Kwon.
\newblock Construction of blow-up manifolds to the equivariant self-dual
  {C}hern-{S}imons-{S}chr\"{o}dinger equation.
\newblock {\em Ann. PDE}, 9:Paper No. 6, 129, 2023.

\bibitem{KimKwon2023MAMS}
K.~Kim and S.~Kwon.
\newblock On pseudoconformal blow-up solutions to the self-dual
  {C}hern-{S}imons-{S}chr\"{o}dinger equation: existence, uniqueness, and
  instability.
\newblock {\em Mem. Amer. Math. Soc.}, 284(1409):v+128, 2023.

\bibitem{KimKwonOh2020arXiv}
K.~Kim, S.~Kwon, and S.-J. Oh.
\newblock Blow-up dynamics for smooth finite energy radial data solutions to
  the self-dual {C}hern-{S}imons-{S}chr\"{o}dinger equation.
\newblock {\em arXiv e-prints 2010.03252, to appear in Ann. Sci. \'{E}c. Norm.
  Sup\'{e}r., 80 pp.}, 2020.

\bibitem{KimKwonOh2022arXiv2}
K.~Kim, S.~Kwon, and S.-J. Oh.
\newblock Blow-up dynamics for radial self-dual
  {C}hern-{S}imons-{S}chr\"{o}dinger equation with prescribed asymptotic
  profile.
\newblock {\em in preparation}, 2022.

\bibitem{KimKwonOh2022arXiv1}
K.~Kim, S.~Kwon, and S.-J. Oh.
\newblock Soliton resolution for equivariant self-dual
  {C}hern-{S}imons-{S}chr\"{o}dinger equation in weighted {S}obolev class.
\newblock {\em arXiv e-prints 2202.07314, to appear in Amer. J. Math., 26 pp.},
  2022.

\bibitem{KriegerSchlagTataru2008Invent}
J.~Krieger, W.~Schlag, and D.~Tataru.
\newblock Renormalization and blow up for charge one equivariant critical wave
  maps.
\newblock {\em Invent. Math.}, 171(3):543--615, 2008.

\bibitem{LiZhang2012_d_geq_2}
D.~Li and X.~Zhang.
\newblock On the rigidity of solitary waves for the focusing mass-critical
  {NLS} in dimensions {$d\geq 2$}.
\newblock {\em Sci. China Math.}, 55(2):385--434, 2012.

\bibitem{LiLiu2022AIHP}
Z.~Li and B.~Liu.
\newblock On threshold solutions of the equivariant
  {C}hern-{S}imons-{S}chr\"{o}dinger equation.
\newblock {\em Ann. Inst. H. Poincar\'{e} C Anal. Non Lin\'{e}aire},
  39(2):371--417, 2022.

\bibitem{Lim2018JDE}
Z.~M. Lim.
\newblock Large data well-posedness in the energy space of the
  {C}hern-{S}imons-{S}chr\"{o}dinger system.
\newblock {\em J. Differential Equations}, 264(4):2553--2597, 2018.

\bibitem{LiuSmith2016}
B.~Liu and P.~Smith.
\newblock Global wellposedness of the equivariant
  {C}hern-{S}imons-{S}chr\"{o}dinger equation.
\newblock {\em Rev. Mat. Iberoam.}, 32(3):751--794, 2016.

\bibitem{LiuSmithTataru2014IMRN}
B.~Liu, P.~Smith, and D.~Tataru.
\newblock Local wellposedness of {C}hern-{S}imons-{S}chr\"{o}dinger.
\newblock {\em Int. Math. Res. Not. IMRN}, (23):6341--6398, 2014.

\bibitem{MartelMerleRaphael2014Acta}
Y.~Martel, F.~Merle, and P.~Rapha\"{e}l.
\newblock Blow up for the critical generalized {K}orteweg--de {V}ries equation.
  {I}: {D}ynamics near the soliton.
\newblock {\em Acta Math.}, 212(1):59--140, 2014.

\bibitem{MartelMerleRaphael2015JEMS}
Y.~Martel, F.~Merle, and P.~Rapha\"{e}l.
\newblock Blow up for the critical g{K}d{V} equation. {II}: {M}inimal mass
  dynamics.
\newblock {\em J. Eur. Math. Soc. (JEMS)}, 17(8):1855--1925, 2015.

\bibitem{MartelMerleRaphael2015AnnSci}
Y.~Martel, F.~Merle, and P.~Rapha\"{e}l.
\newblock Blow up for the critical g{K}d{V} equation {III}: exotic regimes.
\newblock {\em Ann. Sc. Norm. Super. Pisa Cl. Sci. (5)}, 14(2):575--631, 2015.

\bibitem{MerleRaphael2003GAFA}
F.~Merle and P.~Rapha\"{e}l.
\newblock Sharp upper bound on the blow-up rate for the critical nonlinear
  {S}chr\"{o}dinger equation.
\newblock {\em Geom. Funct. Anal.}, 13(3):591--642, 2003.

\bibitem{MerleRaphael2004InventMath}
F.~Merle and P.~Rapha\"{e}l.
\newblock On universality of blow-up profile for {$L^2$} critical nonlinear
  {S}chr\"{o}dinger equation.
\newblock {\em Invent. Math.}, 156(3):565--672, 2004.

\bibitem{MerleRaphael2005AnnMath}
F.~Merle and P.~Rapha\"{e}l.
\newblock The blow-up dynamic and upper bound on the blow-up rate for critical
  nonlinear {S}chr\"{o}dinger equation.
\newblock {\em Ann. of Math. (2)}, 161(1):157--222, 2005.

\bibitem{MerleRaphael2005CMP}
F.~Merle and P.~Rapha\"{e}l.
\newblock Profiles and quantization of the blow up mass for critical nonlinear
  {S}chr\"{o}dinger equation.
\newblock {\em Comm. Math. Phys.}, 253(3):675--704, 2005.

\bibitem{MerleRaphael2006JAMS}
F.~Merle and P.~Rapha\"{e}l.
\newblock On a sharp lower bound on the blow-up rate for the {$L^2$} critical
  nonlinear {S}chr\"{o}dinger equation.
\newblock {\em J. Amer. Math. Soc.}, 19(1):37--90, 2006.

\bibitem{MerleRaphaelRodnianski2013InventMath}
F.~Merle, P.~Rapha\"{e}l, and I.~Rodnianski.
\newblock Blowup dynamics for smooth data equivariant solutions to the critical
  {S}chr\"{o}dinger map problem.
\newblock {\em Invent. Math.}, 193(2):249--365, 2013.

\bibitem{MerleRaphaelRodnianski2015CambJMath}
F.~Merle, P.~Rapha\"{e}l, and I.~Rodnianski.
\newblock Type {II} blow up for the energy supercritical {NLS}.
\newblock {\em Camb. J. Math.}, 3(4):439--617, 2015.

\bibitem{MerleRaphaelSzeftel2013AJM}
F.~Merle, P.~Rapha\"{e}l, and J.~Szeftel.
\newblock The instability of {B}ourgain-{W}ang solutions for the {$L^2$}
  critical {NLS}.
\newblock {\em Amer. J. Math.}, 135(4):967--1017, 2013.

\bibitem{Mizoguchi2007MathAnn}
N.~Mizoguchi.
\newblock Rate of type {II} blowup for a semilinear heat equation.
\newblock {\em Math. Ann.}, 339(4):839--877, 2007.

\bibitem{OhPusateri2015}
S.-J. Oh and F.~Pusateri.
\newblock Decay and scattering for the {C}hern-{S}imons-{S}chr\"{o}dinger
  equations.
\newblock {\em Int. Math. Res. Not. IMRN}, (24):13122--13147, 2015.

\bibitem{Perelman2014CMP}
G.~Perelman.
\newblock Blow up dynamics for equivariant critical {S}chr\"{o}dinger maps.
\newblock {\em Comm. Math. Phys.}, 330(1):69--105, 2014.

\bibitem{Raphael2005MathAnn}
P.~Rapha\"{e}l.
\newblock Stability of the log-log bound for blow up solutions to the critical
  non linear {S}chr\"{o}dinger equation.
\newblock {\em Math. Ann.}, 331(3):577--609, 2005.

\bibitem{RaphaelRodnianski2012Publ.Math.}
P.~Rapha\"{e}l and I.~Rodnianski.
\newblock Stable blow up dynamics for the critical co-rotational wave maps and
  equivariant {Y}ang-{M}ills problems.
\newblock {\em Publ. Math. Inst. Hautes \'{E}tudes Sci.}, 115:1--122, 2012.

\bibitem{RaphaelSchweyer2013CPAM}
P.~Rapha\"{e}l and R.~Schweyer.
\newblock Stable blowup dynamics for the 1-corotational energy critical
  harmonic heat flow.
\newblock {\em Comm. Pure Appl. Math.}, 66(3):414--480, 2013.

\bibitem{RaphaelSchweyer2014AnalPDE}
P.~Rapha\"{e}l and R.~Schweyer.
\newblock Quantized slow blow-up dynamics for the corotational energy-critical
  harmonic heat flow.
\newblock {\em Anal. PDE}, 7(8):1713--1805, 2014.

\bibitem{RaphaelSzeftel2011JAMS}
P.~Rapha\"{e}l and J.~Szeftel.
\newblock Existence and uniqueness of minimal blow-up solutions to an
  inhomogeneous mass critical {NLS}.
\newblock {\em J. Amer. Math. Soc.}, 24(2):471--546, 2011.

\bibitem{RodnianskiSterbenz2010Ann.Math.}
I.~Rodnianski and J.~Sterbenz.
\newblock On the formation of singularities in the critical {${\rm O}(3)$}
  {$\sigma$}-model.
\newblock {\em Ann. of Math. (2)}, 172(1):187--242, 2010.

\bibitem{BergWilliams2013}
J.~B. van~den Berg and J.~F. Williams.
\newblock ({I}n-)stability of singular equivariant solutions to the
  {L}andau-{L}ifshitz-{G}ilbert equation.
\newblock {\em European J. Appl. Math.}, 24(6):921--948, 2013.

\end{thebibliography}

\end{document}